\documentclass[10pt]{article}
\usepackage[dvipsnames,table,xcdraw]{xcolor}
\usepackage{amsmath}
\usepackage{amsthm}
\usepackage{aliascnt}
\usepackage{mathtools}
\usepackage{amssymb}
\usepackage{graphicx}
\usepackage{titling}
\usepackage[hang,flushmargin]{footmisc}
\usepackage[hidelinks]{hyperref}
\usepackage[nameinlink,noabbrev,capitalise]{cleveref}
\usepackage{float}
\usepackage{makecell}
\usepackage{hhline}
\usepackage{nicematrix}
\usepackage{array}
\usepackage[margin=16pt]{caption}
\usepackage{subcaption}
\usepackage{enumitem}
\usepackage{xfrac}
\usepackage{booktabs}
\usepackage{ytableau}
\colorlet{yty}{yellow!60}
\usepackage{xparse}
\usepackage{comment}

\usepackage[pagewise]{lineno}
\usepackage{tkz-berge}
\tikzset{
    cell/.style={
        anchor=south west,
        draw,
        minimum size=.9cm,
    },
  }
\usetikzlibrary{arrows.meta}

\usepackage{tikz}
\usetikzlibrary{positioning,decorations.pathreplacing,arrows.meta,calc,matrix,fit,backgrounds,positioning}

\usepackage[a4paper, total={7in, 10in}]{geometry}

\usepackage{fancyhdr}
\fancypagestyle{plain}{%
  \fancyhf{}%
  \fancyfoot[C]{\thepage}%
}

\usepackage[textsize=tiny]{todonotes}

\def\R{\mathbb R}
\def\Z{\mathbb Z}
\def\Q{\mathbb Q}
\def\F{\mathbb F}
\newcommand{\core}{\mathrm{core}}
\DeclareMathOperator{\linspan}{span}
\DeclareMathOperator{\aff}{aff}
\DeclareMathOperator{\fl}{fl}
\DeclareMathOperator{\vrefl}{vrefl}
\DeclareMathOperator{\conv}{conv}
\DeclareMathOperator{\supp}{supp}
\DeclareMathOperator{\Catalan}{Cat}
\DeclareMathOperator{\ext}{ext}
\DeclareMathOperator{\extBL}{ext_{BL}}
\DeclareMathOperator{\extUL}{ext_{UL}}
\DeclareMathOperator{\extBR}{ext_{BR}}
\DeclareMathOperator{\extUR}{ext_{UR}}
\DeclareMathOperator{\extI}{\ext^{\shortmid}}
\DeclareMathOperator{\extL}{\ext^{\llcorner}}
\DeclareMathOperator{\extRL}{\ext^{\lrcorner}}
\DeclarePairedDelimiter{\floor}{\lfloor}{\rfloor}
\DeclarePairedDelimiter{\ceil}{\lceil}{\rceil}

\newcommand{\ol}{\overline}

\numberwithin{equation}{section}

\newcommand{\DeclareEquationCrefFormat}[1]{%
  \crefformat{#1}{(##2##1##3)}%
  \Crefformat{#1}{(##2##1##3)}%
  \crefrangeformat{#1}{(##3##1##4)\nobreakdash--(##5##2##6)}%
  \Crefrangeformat{#1}{(##3##1##4)\nobreakdash--(##5##2##6)}%
  \crefmultiformat{#1}{(##2##1##3)}{~and~(##2##1##3)}{,~(##2##1##3)}{,~and~(##2##1##3)}%
  \Crefmultiformat{#1}{(##2##1##3)}{~and~(##2##1##3)}{,~(##2##1##3)}{,~and~(##2##1##3)}%
  \crefrangemultiformat{#1}{(##3##1##4)\nobreakdash--(##5##2##6)}{~and~(##3##1##4)\nobreakdash--(##5##2##6)}{,~(##3##1##4)\nobreakdash--(##5##2##6)}{,~and~(##3##1##4)\nobreakdash--(##5##2##6)}%
  \Crefrangemultiformat{#1}{(##3##1##4)\nobreakdash--(##5##2##6)}{~and~(##3##1##4)\nobreakdash--(##5##2##6)}{,~(##3##1##4)\nobreakdash--(##5##2##6)}{,~and~(##3##1##4)\nobreakdash--(##5##2##6)}%
}
\DeclareEquationCrefFormat{equation}
\DeclareEquationCrefFormat{subequation}
\DeclareEquationCrefFormat{parentequation}
\DeclareEquationCrefFormat{align}
\DeclareEquationCrefFormat{alignat}
\DeclareEquationCrefFormat{xalignat}
\DeclareEquationCrefFormat{xxalignat}
\DeclareEquationCrefFormat{flalign}
\DeclareEquationCrefFormat{gather}
\DeclareEquationCrefFormat{multline}

\newcommand{\DeclareItemCrefFormat}[1]{%
  \crefformat{#1}{##2##1##3}%
  \Crefformat{#1}{##2##1##3}%
  \crefrangeformat{#1}{##3##1##4\nobreakdash--##5##2##6}%
  \Crefrangeformat{#1}{##3##1##4\nobreakdash--##5##2##6}%
  \crefmultiformat{#1}{##2##1##3}{ and~##2##1##3}{, ##2##1##3}{, and~##2##1##3}%
  \Crefmultiformat{#1}{##2##1##3}{ and~##2##1##3}{, ##2##1##3}{, and~##2##1##3}%
  \crefrangemultiformat{#1}{##3##1##4\nobreakdash--##5##2##6}{ and~##3##1##4\nobreakdash--##5##2##6}{, ##3##1##4\nobreakdash--##5##2##6}{, and~##3##1##4\nobreakdash--##5##2##6}%
  \Crefrangemultiformat{#1}{##3##1##4\nobreakdash--##5##2##6}{ and~##3##1##4\nobreakdash--##5##2##6}{, ##3##1##4\nobreakdash--##5##2##6}{, and~##3##1##4\nobreakdash--##5##2##6}%
}
\DeclareItemCrefFormat{enumi}
\DeclareItemCrefFormat{enumii}
\DeclareItemCrefFormat{enumiii}
\DeclareItemCrefFormat{enumiv}

\newaliascnt{theorem}{THM}
\newtheorem{theorem}[theorem]{Theorem}
\aliascntresetthe{theorem}

\newaliascnt{lemma}{THM}
\newtheorem{lemma}[lemma]{Lemma}
\aliascntresetthe{lemma}

\newaliascnt{claim}{THM}

\aliascntresetthe{claim}

\newaliascnt{corollary}{THM}
\newtheorem{corollary}[corollary]{Corollary}
\aliascntresetthe{corollary}

\newaliascnt{proposition}{THM}
\newtheorem{proposition}[proposition]{Proposition}
\aliascntresetthe{proposition}

\newaliascnt{conjecture}{THM}

\aliascntresetthe{conjecture}

\newaliascnt{example}{THM}

\aliascntresetthe{example}

\theoremstyle{definition}
\newaliascnt{definition}{THM}

\aliascntresetthe{definition}

\newaliascnt{remark}{THM}
\newtheorem{remark}[remark]{Remark}
\aliascntresetthe{remark}

\newaliascnt{note}{THM}

\aliascntresetthe{note}

\newaliascnt{problem}{THM}
\newtheorem{problem}[problem]{Problem}
\aliascntresetthe{problem}

\newcommand{\DeclareDashCrefRange}[2]{%
  \crefrangeformat{#1}{#2~##3##1##4\nobreakdash--##5##2##6}%
  \Crefrangeformat{#1}{#2~##3##1##4\nobreakdash--##5##2##6}%
  \crefrangemultiformat{#1}{#2~##3##1##4\nobreakdash--##5##2##6}{ and~##3##1##4\nobreakdash--##5##2##6}{, ##3##1##4\nobreakdash--##5##2##6}{, and~##3##1##4\nobreakdash--##5##2##6}%
  \Crefrangemultiformat{#1}{#2~##3##1##4\nobreakdash--##5##2##6}{ and~##3##1##4\nobreakdash--##5##2##6}{, ##3##1##4\nobreakdash--##5##2##6}{, and~##3##1##4\nobreakdash--##5##2##6}%
}
\DeclareDashCrefRange{section}{Sections}
\DeclareDashCrefRange{theorem}{Theorems}
\DeclareDashCrefRange{lemma}{Lemmas}
\DeclareDashCrefRange{proposition}{Propositions}
\DeclareDashCrefRange{corollary}{Corollaries}
\DeclareDashCrefRange{claim}{Claims}
\DeclareDashCrefRange{definition}{Definitions}
\DeclareDashCrefRange{remark}{Remarks}
\DeclareDashCrefRange{conjecture}{Conjectures}
\DeclareDashCrefRange{example}{Examples}
\DeclareDashCrefRange{note}{Notes}
\DeclareDashCrefRange{problem}{Problems}

\usepackage[yyyymmdd]{datetime}

\makeatletter
\NewDocumentEnvironment{mymatrix}{ O{0em} }{%
  \begingroup
  \edef\mymatrix@saved@boxframe{\boxframe@YT}%
  \ytableausetup{boxframe=#1}%
  \begin{bmatrix}
  \begin{ytableau}
}{%
  \end{ytableau}
  \end{bmatrix}%
  \ytableausetup{boxframe=\mymatrix@saved@boxframe}%
  \endgroup
}
\makeatother

\newcommand{\ThreeDotsAt}[4][]{%
  \node[draw=none,#1] at ($#2 + ({#4+180}:#3)$) {$\boldsymbol\cdot$};
  \node[draw=none,#1] at  #2                   {$\boldsymbol\cdot$};
  \node[draw=none,#1] at ($#2 + (#4:#3)$)      {$\boldsymbol\cdot$};
}

\newcommand{\VS}{\mathrm{VS}}
\newcommand{\HS}{\mathrm{HS}}

\newcommand{\HTS}{\mathrm{HTS}}
\newcommand{\QTS}{\mathrm{QTS}}
\newcommand{\DS}{\mathrm{DS}}
\newcommand{\DAS}{\mathrm{DAS}}
\newcommand{\TS}{\mathrm{TS}}

\newcommand{\ASM}{\mathrm{ASM}}
\newcommand{\VSASM}{\mathrm{VSASM}}

\newcommand{\VHSASM}{\mathrm{VHSASM}}
\newcommand{\HTSASM}{\mathrm{HTSASM}}
\newcommand{\QTSASM}{\mathrm{QTSASM}}
\newcommand{\DSASM}{\mathrm{DSASM}}
\newcommand{\DASASM}{\mathrm{DASASM}}
\newcommand{\TSASM}{\mathrm{TSASM}}
\newcommand{\XASM}{\mathrm{XASM}}
\begin{document}
\title{
  Polytopes of alternating sign matrices with dihedral symmetries%
}

\author{
  P\'eter Madarasi\thanks{HUN-REN Alfr\'ed R\'enyi Institute of Mathematics, Re\'altanoda u.\ 13--15, Budapest H-1053, Hungary; and Department of Operations Research, ELTE E\"otv\"os Lor\'and University, P\'azm\'any P.\ s.\ 1/c, Budapest H-1117, Hungary. E-mail: \texttt{madarasi@renyi.hu}}
}

\date{}

\maketitle

\begin{abstract}
  We study the convex hulls of $n \times n$ alternating sign matrices invariant under subgroups of the dihedral group of the square.
  For each non-trivial symmetry class, the symmetry determines the full matrix affinely from a smaller set of entries, allowing us to study the corresponding convex hull in a lower-dimensional space.
  For the vertical, vertical--horizontal, half-turn, diagonal, diagonal--antidiagonal, and total symmetry classes, we give polynomial-size linear inequality descriptions, determine the dimensions, identify all facets, and give exact facet counts.
  For the quarter-turn class, the natural fixed-point relaxation is not integral.
  We obtain the exact hull by adding parity-type Chv\'atal--Gomory inequalities, determine its dimension, and construct a family of facets indexed by Ferrers diagrams of Catalan-number cardinality, while the complete facet structure remains open.
  The formulations yield strongly polynomial-time algorithms for linear optimization over every non-empty symmetry class.
  We also show that the full-matrix polytopes of all classes except the quarter-turn class have the integer Carath\'eodory property and hence the integer decomposition property.

\end{abstract}

\section{Introduction}

An \emph{alternating sign matrix (ASM)} is an $n\times n$ matrix with entries in $\{0,\pm1\}$ such that, in each row and each column, the non-zero entries alternate in sign and the first and last non-zero entries are~$1$.
The symmetry group of the square of matrix positions is the dihedral group $\mathcal{D}_4 = \{\mathcal{I},\mathcal{V},\mathcal{H},\mathcal{D},\mathcal{A},\mathcal{R}_{\pi/2},\mathcal{R}_{\pi},\mathcal{R}_{-\pi/2}\}$.
Here $\mathcal{I}$ denotes the identity, $\mathcal{V}$, $\mathcal{H}$, $\mathcal{D}$, and $\mathcal{A}$ denote the reflections across the vertical, horizontal, diagonal, and antidiagonal axes, respectively, while $\mathcal{R}_{\pi/2}$, $\mathcal{R}_{\pi}$, and $\mathcal{R}_{-\pi/2}$ denote the three non-trivial rotations.
For a subgroup $\mathcal{G}\subseteq \mathcal{D}_4$, this paper studies the convex hull of the ASMs invariant under every element of $\mathcal{G}$.
Conjugate subgroups give the same symmetry class after relabeling matrix positions; up to this equivalence, the eight symmetry classes of ASMs are as follows~\cite{kuperberg2002symmetry}.
\begin{enumerate}
\item \emph{Unrestricted ASMs (ASMs)}:
  No additional symmetry is imposed beyond the standard ASM conditions.
  The symmetry subgroup is $\mathcal{G} = \{\mathcal{I}\}$.
\item \emph{Vertically symmetric ASMs (VSASMs)}:
  The ASMs that are invariant under reflection across the vertical axis.
  The symmetry subgroup is $\mathcal{G} = \{\mathcal{I}, \mathcal{V}\}$.
\item \emph{Vertically and horizontally symmetric ASMs (VHSASMs)}:
  The ASMs that are invariant under both vertical and horizontal reflections, and hence also under rotation by~$\pi$.
  The symmetry subgroup is $\mathcal{G} = \{\mathcal{I}, \mathcal{V}, \mathcal{H}, \mathcal{R}_{\pi}\}$.
\item \emph{Half-turn symmetric ASMs (HTSASMs)}:
  The ASMs that are invariant under rotation by~$\pi$.
  The symmetry subgroup is $\mathcal{G} = \{\mathcal{I}, \mathcal{R}_{\pi}\}$.
\item \emph{Quarter-turn symmetric ASMs (QTSASMs)}:
  The ASMs that are invariant under rotation by~$\pi/2$, and hence also under rotations by~$\pi$ and~$-\pi/2$.
  The symmetry subgroup is $\mathcal{G} = \{\mathcal{I}, \mathcal{R}_{\pi/2}, \mathcal{R}_{\pi}, \mathcal{R}_{-\pi/2}\}$.
\item \emph{Diagonally symmetric ASMs (DSASMs)}:
  The ASMs that are invariant under reflection across the main diagonal.
  The symmetry subgroup is $\mathcal{G} = \{\mathcal{I}, \mathcal{D}\}$.
\item \emph{Diagonally and antidiagonally symmetric ASMs (DASASMs)}:
  The ASMs that are invariant under reflection across both the main diagonal and the antidiagonal, and hence also under rotation by~$\pi$.
  The symmetry subgroup is $\mathcal{G} = \{\mathcal{I}, \mathcal{D}, \mathcal{A}, \mathcal{R}_{\pi}\}$.
\item \emph{Totally symmetric ASMs (TSASMs)}:
  The ASMs that are invariant under the full dihedral symmetry group of the square, that is, under all reflections and rotations.
  The symmetry subgroup is $\mathcal{G} = \mathcal{D}_4$.
\end{enumerate}

\medskip
Some symmetry classes are empty for certain sizes.
We call a size \emph{admissible} for a class when that class is non-empty.
Our central goal is to study the convex hulls associated with the eight symmetry classes, with emphasis on linear descriptions, dimensions, facets, linear optimization, and decomposition properties of integer points.

\paragraph{Enumeration background.}
Alternating sign matrices rose to prominence through the product formula conjectured by Mills, Robbins, and Rumsey~\cite{mills1983alternating}, proved by Zeilberger~\cite{zeilberger1996}, and reproved by Kuperberg using the six-vertex model~\cite{kuperberg1996another}.
Symmetry-class enumeration developed in parallel, with both existence and formulas depending delicately on the imposed symmetry and on the parity of~$n$.
Kuperberg's six-vertex-model framework expressed several symmetry-class partition functions as determinants or Pfaffians and led to product formulas for a number of classes~\cite{kuperberg2002symmetry}.
Okada evaluated several of these expressions through classical-group characters and proved the VHSASM enumeration conjecture~\cite{okada2006enumeration}.
Behrend, Fischer, and Konvalinka later proved Robbins's odd-order DASASM enumeration conjecture, the last of the classical conjectured product formulas for symmetric ASMs to be established~\cite{behrend2017diagonally}.
More recently, Behrend, Fischer, and Koutschan derived a Pfaffian formula for diagonally symmetric ASMs, giving an exact enumeration formula beyond the classical product-formula cases~\cite{behrend2023diagonally}.

\paragraph{Polyhedral background.}
The polyhedral study of alternating sign matrices began with a linear description of the ASM polytope.
Behrend and Knight~\cite{behrend2007higher} and, independently, Striker~\cite{striker2007alternating,striker2009alternating} showed in 2007 that the convex hull $P_{\ASM}$ of $n\times n$ ASMs admits a compact system of prefix-sum inequalities; we recall this characterization in \cref{sec:ASM}.
Building on this description, Knight's 2009 thesis gave a systematic polyhedral treatment of dihedral symmetry~\cite[Chapter~4]{knightPhD2009alternating}.
For each of the seven non-trivial symmetry classes, he analyzed the fixed-point polytope $P_{\ASM}\cap P_{\mathcal{G}}$, where $P_{\mathcal{G}}$ is the linear subspace of $\mathcal{G}$-invariant real matrices, obtaining vertex descriptions and Ehrhart polynomial or quasi-polynomial results.
The convex hull of the $\mathcal{G}$-invariant ASMs is contained in this fixed-point polytope.
Because some of the fixed-point polytopes have fractional vertices, however, the containment can be strict.
Knight accordingly singled out the convex hulls of invariant ASM vertices as a separate problem~\cite[Section~6.2.3]{knightPhD2009alternating}.

Subsequent work has developed the geometry of $P_{\ASM}$ and related sign-matrix polytopes in several directions.
Brualdi and Dahl studied the ASM cone, gave another proof of the linear description, and investigated faces and edges of $P_{\ASM}$~\cite{brualdi2017alternatingExtensions}.
M\'esz\'aros, Morales, and Striker identified a family of faces of $P_{\ASM}$ that are simultaneously order polytopes and flow polytopes, with consequences for Ehrhart theory and volume~\cite{meszaros2019flow}.
Solhjem and Striker studied related sign-matrix polytopes, while Heuer and Striker treated partial-ASM polytopes, including their inequality descriptions, facets, and face lattices~\cite{heuer2022partial,solhjem2019sign}.
A 2025 preprint by Dinkelman and Morris studies the general face geometry of $P_{\ASM}$, including 2-levelness and dimension-dependent bounds on the numbers of vertices and facets of its faces~\cite{dinkelman2025geometric}.
In 2026, Heuer, Solhjem, and Striker identified a further family of partial-ASM faces that are simultaneously order polytopes and flow polytopes~\cite{heuer2026nu}.
Together, these papers develop the broader geometry of the ASM polytope and related sign-matrix polytopes, but they do not describe the convex hulls of dihedrally symmetric ASMs across all seven non-trivial classes.
The present paper addresses this gap by determining those hulls for every admissible size in all seven non-trivial classes.

\paragraph{Main contributions.}
For each non-trivial dihedral symmetry class, denoted generically by XASM, and each admissible size, we use a uniform \emph{core--assembly} reduction.
We specify a set of \emph{core positions} $C$ and an affine \emph{assembly map} $\varphi : \R^{C} \to \R^{n\times n}$ such that every matrix in the class is uniquely determined by its core and $\varphi$ reconstructs the full matrix from that core.
Thus the convex hull $P_{\XASM}$ of XASMs is affinely isomorphic to a lower-dimensional \emph{core polytope} $P^\core_{\XASM}$.
The reduction allows us to translate the ASM prefix-sum constraints into core coordinates and obtain explicit linear descriptions of the core polytopes and the corresponding full-matrix hulls.
For the vertical, vertical--horizontal, half-turn, diagonal, diagonal--antidiagonal, and total symmetry classes, these descriptions have polynomial size, and we determine the dimension, identify all facet-defining inequalities, and give exact facet counts.
The quarter-turn class behaves differently: its exact description requires, in addition to the ASM constraints and quarter-turn symmetry equations, a structured family of parity-type Chv\'atal--Gomory inequalities.
These inequalities cut off the fractional vertices of the fixed-point relaxation and yield the convex hull of QTSASMs.
Although the full quarter-turn facet structure remains open, we determine its dimension and construct an explicit family of facets indexed by Ferrers diagrams whose cardinality grows exponentially with~$n$.
The descriptions also characterize the inadmissible sizes in every class.
The resulting formulations yield strongly polynomial-time algorithms for linear optimization over every non-empty symmetry class.
Finally, for every non-trivial symmetry class other than the quarter-turn class and every admissible size, we prove the integer Carath\'eodory property, and hence the integer decomposition property, for the full-matrix and core polytopes, extending the known unrestricted ASM case~\cite{borsik2026prefix}.
Whether the quarter-turn full-matrix and core polytopes have either property for every admissible size remains open.
\Cref{tab:intro:summary} summarizes the dimensions and facet information.
\begin{table}[h!]
\centering
\newcommand{\mystrut}{\rule{0pt}{3.2ex}\rule[-1.6ex]{0pt}{0pt}}
\newcommand{\mystrutt}{\rule{0pt}{2ex}\rule[-1.8ex]{0pt}{0pt}}
\renewcommand{\arraystretch}{1.5}
\begin{tabular}{|>{\columncolor{cyan!15}}l|c|c|}
\hline
\rowcolor{cyan!15}
Symmetry class & Dimension & Number of facets\\
\specialrule{.75pt}{0pt}{0pt}
\mystrut ASM
  & $(n-1)^2$\quad\cite{behrend2007higher,striker2007alternating,striker2009alternating}
  & {$4\bigl((n - 2)^2 + 1\bigr)$\quad if $n \geq 3$\quad\cite{striker2007alternating,striker2009alternating}}\\
  \hline
\mystrut VSASM ($2 \nmid n$)
  & $\frac{(n-3)^2}{2}$\quad if $n \geq 3$
  & $2n^2-19n+49$\quad if $n \geq 7$\\
\hline
\mystrut VHSASM ($2 \nmid n$)
  & $\frac{(n-5)^2}{4}$\quad if $n \geq 5$
  & $n^2-15n+60$\quad if $n \geq 9$\\
\hline
\mystrutt HTSASM
  & $\ceil*{\frac{(n-1)^2}{2}}$
  & $2\bigl((n-2)^2+\chi_{2 \mid n}\bigr)$\quad if $n \geq 4$\\
\hline
\mystrutt  QTSASM ($n \not\equiv 2 \pmod 4$)
  & $\floor*{\frac{(n-1)^2}{4}} - 2$\quad if $n \geq 5$
  & \begin{tabular}{@{}c@{}}at least\ $\Catalan_{\floor*{n/2}-1}-1 \geq \frac{2^{n-3}}{n^2}-1$ if $n \geq 4$\\ at most\ $5^{\floor*{n^2/4}} + n(n-1)$ if $n \geq 1$\end{tabular}\\
\hline
\mystrut DSASM
  & $\frac{n(n-1)}{2}$
  & {$2(n-2)^2+3$}\quad if $n \geq 3$\\
\hline
\mystrutt DASASM
  & $\floor*{\frac{n^2}{4}}$
  & $(n-2)^2 + 2$\quad if $n \geq 2$\\
\hline
\mystrut TSASM ($2 \nmid n$)
  & $\frac{(n-5)(n-3)}{8}$\quad if $n \geq 3$
  & $\frac{n^2 -15n + 62}{2}$\quad if $n \geq 9$\\
\hline
\end{tabular}
\caption{Summary of the dimensions and facet counts or bounds.
  Here $\Catalan_r$ denotes the $r^\text{th}$ Catalan number; facet formulas apply in the stated ranges.}\label{tab:intro:summary}
\end{table}

\paragraph{Organization of the paper.}
The remainder of this section fixes notation, records the polyhedral tools used in the proofs, and formalizes the core--assembly framework.
\Cref{sec:ASM} recalls the prefix-sum description of the ASM polytope.
\Crefrange{sec:VSASM}{sec:TSASM} treat the seven non-trivial symmetry classes in turn, following the uniform core--assembly approach.
\Cref{sec:linear-optimization} derives the strongly polynomial linear\nobreakdash-optimization result and proves the integer Carath\'eodory property for $P_\ASM$ and, at every admissible size, for the full-matrix and core polytopes of the six non-trivial classes other than the quarter-turn class; \cref{sec:open-problems} records the remaining questions, including the quarter-turn integer Carath\'eodory and integer decomposition questions.

\subsection{Preliminaries}\label{sec:preliminaries}
We first fix notation.
For a symmetry class XASM, we denote the set of $n \times n$ XASMs by $\XASM(n)$, and the convex hull of XASMs in $\R^{n \times n}$ by $P_{\XASM(n)}$.
We often write $P_{\XASM}$ if $n$ is clear from the context.
For integers $n \leq m$, we write $[n,m] = \{n, \dots, m\}$, and we set $[n,m] = \emptyset$ whenever $n > m$.
For an integer $n$, we use the shorthand $[n] = [1,n]$.
The elements of $\mathcal{D}_4$ act on matrix positions $[n]\times[n]$ by
\begin{align*}
  \mathcal{I}(i,j) &= (i,j), &
  \mathcal{V}(i,j) &= (i,n+1-j),\\
  \mathcal{H}(i,j) &= (n+1-i,j), &
  \mathcal{D}(i,j) &= (j,i),\\
  \mathcal{A}(i,j) &= (n+1-j,n+1-i), &
  \mathcal{R}_{\pi/2}(i,j) &= (j,n+1-i),\\
  \mathcal{R}_{\pi}(i,j) &= (n+1-i,n+1-j), &
  \mathcal{R}_{-\pi/2}(i,j) &= (n+1-j,i),
\end{align*}
for $i,j\in[n]$.
Accordingly, for $g\in \mathcal{D}_4$, a matrix $X=(x_{i,j})\in\R^{n\times n}$ is invariant under $g$ if $x_{i,j}=x_{i',j'}$ for every $i,j\in[n]$, where $g(i,j)=(i',j')$.
For $X \in \R^{m\times n}$ and $I\subseteq[m]$, $J\subseteq[n]$, let $X_{I,J} \in \R^{|I|\times|J|}$ be the submatrix formed by the rows with indices in $I$ and columns with indices in $J$.
For singleton index sets, we use the shorthand $X_{i,J} = X_{\{i\},J}, X_{I,j} = X_{I,\{j\}}, X_{i,j} = X_{\{i\},\{j\}}$.
We often write $x_{i,j}$ for the entry $X_{i,j}$.
For a finite set $S$ and a subset $T \subseteq S$, we write $\chi_T \in \{0,1\}^S$ for the characteristic vector of $T$; for a singleton, we abbreviate $\chi_s=\chi_{\{s\}}$; and for $S=[n]\times[n]$ we write $\chi_{i,j}$ for the matrix unit supported on $(i,j)$.
When the subscript is a predicate $\mathcal P$ (e.g., $\chi_{2 \mid n}$), we interpret $\chi_{\mathcal P}$ as the scalar indicator of $\mathcal P$, i.e., $\chi_{\mathcal P}=1$ if $\mathcal P$ holds and $0$ otherwise.
For a finite set $S$ and a function $f : S \to \R$, we extend $f$ to subsets $T \subseteq S$ by $f(T) = \sum_{t \in T} f(t)$.
We do not distinguish between row and column vectors; expressions are interpreted so that all products are conformable, with inner products preferred.

\medskip
We next record the polyhedral tools used repeatedly below: total unimodularity for laminar and network systems, an integer-hull description for bidirected systems, an elementary rank lemma, and an irredundancy criterion for facets.
For laminar systems, we use the following observation of Edmonds~\cite[Theorem~(39)]{edmonds1970submodular}; see also Schrijver~\cite[Theorem~41.11]{schrijver2003combinatorial}. %
\begin{theorem}[Edmonds~\cite{edmonds1970submodular}]\label{thm:edmonds-lemma}
  Let $\mathcal{L}$ be the union of two laminar families on a ground set $S$.
  Then the incidence matrix of $\mathcal{L}$ is totally unimodular.
\end{theorem}

Together with the well-known fact that a system with a totally unimodular constraint matrix and integer right-hand side has an integral feasible region, this yields the following consequence.
\begin{theorem}\label{thm:laminarSystem}
  Let $\mathcal{L}_1, \mathcal{L}_2 \subseteq 2^S$ be two laminar families on a ground set $S$, and let $f_i,g_i : \mathcal{L}_i \to \Z$ be integer-valued bounding functions for $i = 1,2$.
  Then the linear inequality system
  \begin{align}
                              x&_s \in \R      &\forall s \in S,\\
    f_1(Z) \leq \sum_{s \in Z} x&_s \leq g_1(Z) &\forall Z \in \mathcal{L}_1,\\
    f_2(Z) \leq \sum_{s \in Z} x&_s \leq g_2(Z) &\forall Z \in \mathcal{L}_2
  \end{align}
  defines an integral polyhedron.
  \medskip
\end{theorem}

Another standard source of integrality in our arguments relies on the fact that the incidence matrices of directed graphs are totally unimodular; see, for example, Schrijver~\cite[Chapter~19]{schrijver2003combinatorial}.
This implies the following.
\begin{theorem}\label{thm:digraphIncidenceTU}
  Let $M\in\{-1,0,1\}^{m \times n}$ be the node--arc incidence matrix of a digraph or its transpose.
  Then, for any integer vectors $a,b\in\Z^m$ and $c,d\in\Z^n$, the linear inequality system
  $
    \left\{x\in\R^n : c \leq x \leq d,\ a \leq M x \leq b\right\}
  $
  defines an integral polyhedron.
\end{theorem}

In the quarter-turn case, we rely on an explicit polyhedral description of the convex hull of the integer points defined by a linear inequality system with a \emph{bidirected} constraint matrix, i.e., an integer matrix $M\in\Z^{m\times n}$ such that $\sum_{i=1}^m|M_{ij}|=2$ for every $j\in[n]$.
The following theorem appears in Edmonds and Johnson~\cite{edmonds1973matching}; see also Schrijver~\cite[Chapter~36]{schrijver2003combinatorial}.
\begin{theorem}[Edmonds and Johnson~\cite{edmonds1973matching}]\label{thm:bidirPolytope}
  For a bidirected matrix $M\in\Z^{m\times n}$ and for arbitrary vectors $a,b\in\Z^m$ and $c,d\in\Z^n$, the convex hull of the integer solutions to $\{x\in\Z^n : c\leq x\leq d,\ a\leq Mx\leq b\}$ is described by the system
  \begin{subequations}
    \begin{align}\label{lp:simultan:bidirProglemHull}
             x&\in \R^n,\\
       c\leq x&\leq d,\\
      a\leq Mx&\leq b,\\
      \frac{\bigl((\chi_U-\chi_V)M+\chi_F-\chi_H\bigr)x}{2}&\leq \floor*{\frac{b(U)-a(V)+d(F)-c(H)}{2}}&\begin{aligned}
                                                                                                         &\text{for every disjoint $U,V\subseteq[m]$ and}\\
                                                                                                         &\text{partition $F,H$ of $\delta(U\cup V)$,}
                                                                                                 \end{aligned}\label{eq:simultan:bidirProglemHull:oddCon}
    \end{align}
  \end{subequations}
  where $\delta(U\cup V)=\{j\in [n] : \sum_{i\in U\cup V}|M_{ij}|=1\}$.
\end{theorem}

We will repeatedly use the following elementary rank statement when determining the dimensions of our polytopes.
Since it follows from a routine linear-algebra argument, we omit its proof.
\begin{lemma}\label{lem:prelim:rowcol-rank}
  Let $m,n \geq 1$ be integers, $r \in \R^m$, and $c \in \R^n$.
  Consider the system
  \begin{align}
                x&_{i,j}\in \R & \forall i \in [m], j \in [n],\label{eq:prelim:real}\\
    \sum_{j=1}^n x&_{i,j} = r_i & \forall i \in [m],\label{eq:prelim:eq}\\
    \sum_{i=1}^m x&_{i,j} = c_j & \forall j \in [n]\label{eq:prelim:eq2}.
  \end{align}
  Then the coefficient matrix of the equations in~\cref{eq:prelim:eq,eq:prelim:eq2} has rank $m+n-1$, and any subsystem obtained by deleting one equation is linearly independent.
\end{lemma}

\smallskip
A recurring step in our facet proofs is to show that a proposed inequality description is minimal.
For this, we invoke the following classical result: after splitting off the implicit equations of a non-empty polyhedron, the facets are exactly the supporting hyperplanes given by the non-redundant remaining inequalities, and each such inequality defines a unique facet; see, for example, Schrijver~\cite[Theorem~8.1]{schrijver1998TheoryOfLP}. %
Formally, we state the following theorem.
\begin{theorem}\label{thm:irredundant-ineq-facets}
  Let $P=\{x:Ax\leq b\}$ be non-empty.
  Partition the inequalities of $Ax\leq b$ into those that are tight for all $x\in P$ (implicit equations) and the remaining ones; denote these subsystems by $A^{=}x=b^{=}$ and $A^{+}x\leq b^{+}$, respectively.
  Assume no inequality in $A^{+}x\leq b^{+}$ is redundant in the full system $Ax\leq b$.
  Then each facet of $P$ has the form
  $
    F=\{x\in P: a_i x=b_i\}
  $
  for a unique inequality $a_i x\leq b_i$ from $A^{+}x\leq b^{+}$, yielding a bijection between facets of $P$ and inequalities in $A^{+}x\leq b^{+}$.
\end{theorem}

\subsection{General core--assembly framework}

We now formalize the core--assembly reduction used throughout the paper.
Fix a symmetry class XASM of $n\times n$ ASMs, where $n \geq 1$.
Let $C \subseteq [n] \times [n]$ be a set of positions of an $n \times n$ matrix, called the \emph{core positions} for XASMs, and let $\pi_C:\R^{n\times n} \to \R^{C}$ denote the coordinate-wise projection onto $C$.
For an $n \times n$ matrix $X$, its projection $\pi_C(X)$ is referred to as the \emph{core} of $X$.
We call an affine map $\varphi: \R^{C} \to \R^{n \times n}$ an \emph{assembly map} if $\pi_C(\varphi(Y)) = Y$ for every $Y \in \R^{C}$, and $\varphi(\pi_C(X)) = X$ for every $X \in \XASM(n)$.
Note that $\varphi$ is injective by the first condition: for every $Y, Y' \in \R^C$, if $\varphi(Y) = \varphi(Y')$, then $Y = \pi_C(\varphi(Y)) = \pi_C(\varphi(Y')) = Y'$.
We often denote an $n \times n$ matrix --- typically an XASM or an element of the convex hull of XASMs --- by $X$, and its core by $Y$.
Define the \emph{core polytope} for XASMs as $P^\core_{\XASM(n)} = \conv\{\pi_C(X) : X \in \XASM(n)\} \subseteq \R^{C}$, and also introduce the \emph{lifting} of the core polytope $\widehat P^\core_{\XASM(n)} = \{X \in \R^{n\times n} : \pi_C(X) \in P^\core_{\XASM(n)}\}$.
When the size is clear from the context, we write $P^\core_\XASM$ and $\widehat P^\core_\XASM$.

\begin{theorem}\label{thm:xasm:assembly}
  With the setup above, $P_{\XASM} = \varphi\left(P^\core_{\XASM}\right) = \widehat P^\core_{\XASM}\cap \varphi(\R^{C})$.
\end{theorem}
\begin{proof}
  First, we prove that $\varphi\left(P^\core_{\XASM}\right)\subseteq P_{\XASM}$.
  Let $Y \in P^\core_{\XASM}$ and write $Y = \sum_t\lambda_t (\pi_C(X_t))$ as a convex combination of the matrices $X_t \in \XASM(n)$ projected to the core.
  By affinity and the second property of $\varphi$,
  \[
    \varphi(Y) = \sum_t\lambda_t\varphi\left(\pi_C(X_t)\right) = \sum_t\lambda_tX_t \in \conv(\XASM(n)) = P_{\XASM}.
  \]
  For the reverse inclusion, let $X = \sum_t\lambda_t X_t$ be a convex combination with $X_t \in \XASM(n)$.
  Again by affinity and the second property of $\varphi$,
  \[
    X = \sum_t\lambda_t\varphi\left(\pi_C(X_t)\right) = \varphi \Bigl(\sum_t\lambda_t \pi_C(X_t)\Bigr) \in \varphi\left(P^\core_{\XASM}\right).
  \]
  We conclude that $P_{\XASM} = \varphi\left(P^\core_{\XASM}\right)$.

  It remains to prove that $\varphi\left(P^\core_{\XASM}\right) = \widehat P^\core_{\XASM}\cap \varphi(\R^{C})$.
  First, let $X = \varphi(Y)$ with $Y \in P^\core_{\XASM}$.
  Then $\pi_C(X) = \pi_C(\varphi(Y)) = Y \in P^\core_{\XASM}$ by the first property of $\varphi$, so $X \in \widehat P^\core_{\XASM}$.
  Conversely, let $X \in \widehat P^\core_{\XASM} \cap \varphi(\R^{C})$ and write $X = \varphi(Y)$ for some $Y \in \R^{C}$.
  Since $X \in \widehat P^\core_{\XASM}$, we have $Y = \pi_C(X) \in P^\core_{\XASM}$  by the first property of $\varphi$; hence $X = \varphi(Y) \in \varphi\left(P^\core_{\XASM}\right)$.
\end{proof}

\section{Unrestricted alternating sign matrices (ASMs)}\label{sec:ASM}
We begin with the unrestricted ASM polytope, corresponding to the symmetry subgroup $\mathcal{G}=\{\mathcal{I}\}$.
The following prefix-sum description, established independently by Behrend and Knight and by Striker in 2007, is the baseline for all later symmetry-class formulations.
\begin{theorem}[Behrend and Knight~\cite{behrend2007higher}, and Striker~\cite{striker2007alternating,striker2009alternating}]\label{thm:ASMpolytope}
  The convex hull $P_{\ASM}$ of ASMs is the set of real $n\times n$ matrices satisfying the system
  \begin{align}
                          x&_{i,j} \in \R    &\forall i, j \in [n],\label{eq:asm:real}\\
    0 \leq \sum_{j'=1}^{j} x&_{i,j'} \leq 1   &\forall i \in [n], j \in [n-1],\label{eq:asm:row-prefix}\\
    0 \leq \sum_{i'=1}^{i} x&_{i',j} \leq 1 &\forall i \in [n-1], j \in [n],\label{eq:asm:col-prefix}\\
    \sum_{j=1}^{n}         x&_{i,j} = 1      &\forall i \in [n],\label{eq:asm:row-sum}\\
    \sum_{i=1}^{n}         x&_{i,j} = 1      &\forall j \in [n].\label{eq:asm:col-sum}
  \end{align}
\end{theorem}

The constraints in~\cref{eq:asm:row-prefix,eq:asm:col-prefix} require that, for each fixed row, the row-prefix sums other than the full row sum lie between $0$ and~$1$; similarly, the column-prefix sums other than the full column sum lie between $0$ and $1$.
Meanwhile, the equations in~\cref{eq:asm:row-sum,eq:asm:col-sum} enforce that each row and each column of the matrix sums to~$1$.

\medskip
The dimension of $P_{\ASM}$ and the number of its facets are already known in the literature:
\begin{theorem}[Behrend and Knight~\cite{behrend2007higher}, and Striker~\cite{striker2007alternating,striker2009alternating}]
  The dimension of $P_{\ASM}$ is $(n-1)^2$ for $n \geq 3$.
\end{theorem}
\begin{theorem}[Striker~\cite{striker2007alternating,striker2009alternating}]
  The number of facets of $P_{\ASM}$ is $4((n - 2)^2 + 1)$ for $n \geq 3$.
\end{theorem}

\section{Vertically symmetric ASMs (VSASMs)}\label{sec:VSASM}
In this class, we impose invariance under reflection across the vertical axis.
The corresponding symmetry subgroup is $\mathcal{G} = \{\mathcal{I}, \mathcal{V}\}$.
Let $P_\VS$ denote the linear subspace of vertically symmetric real matrices, i.e.,
\[
  P_\VS
  =
  \left\{
    X \in \R^{n \times n} : x_{i,j} = x_{i,n+1-j}\ \forall i,j \in [n]
  \right\}.
\]
Clearly, any VSASM satisfies the ASM constraints in~\crefrange{eq:asm:real}{eq:asm:col-sum} and also the symmetry constraints defining $P_\VS$; thus $P_\VSASM \subseteq P_\ASM \cap P_\VS$.
We note, however, that $P_\ASM \cap P_\VS$ does not equal $P_\VSASM$.
In fact, we show that $P_\VSASM \subsetneq P_\ASM \cap P_\VS$ for every $n \geq 2$.
Let $I_n$ and $I'_n$ denote the $n \times n$ identity matrix and its reflection across the vertical axis.
It is easy to see that the fractional matrix $\frac{1}{2}(I_n + I'_n)$ is a vertex of $P_\ASM \cap P_\VS$; e.g., for $n = 3$, we have
\[
  \frac{I_3 + I'_3}{2} =
  \begin{mymatrix}
    \sfrac{1}{2} & 0 & \sfrac{1}{2}\\
    0            & 1 & 0\\
    \sfrac{1}{2} & 0 & \sfrac{1}{2}
  \end{mymatrix}.
\]

\begin{lemma}\label{lem:vsasm:noEvenN}
  There is no $n \times n$ VSASM if $n$ is even.
\end{lemma}
\begin{proof}
  Let $n$ be even and set $k = n/2$.
  In any vertically symmetric integer matrix $X \in \Z^{n \times n}$, each row $i$ is palindromic, i.e., $X_{i,[n]} = (x_{i,1}, \dots, x_{i,k}, x_{i,k}, \dots, x_{i,1})$ for $i \in [n]$.
  Therefore, the sum of the entries in row $i$ is $\sum_{j=1}^n x_{i,j} = 2 \sum_{j=1}^k x_{i,j}$, which is an even integer for every $i \in [n]$.
  On the other hand, in any ASM, every row sums to $1$.
  Hence no $n \times n$ VSASM exists for even $n$.
\end{proof}

Thus, $P_\VSASM = \emptyset$ for even $n$; hence, we assume that $n$ is odd in the rest of the section.
We need the following two lemmas before introducing the core and assembly map for VSASMs.

\begin{lemma}\label{lem:vsasm:symmetricHalfExtends}
  Let $n \geq 1$ be odd and set $k = \floor*{n/2}$.
  Let $x \in \Z^n$ be symmetric, that is, $x_j = x_{n+1-j}$ for every $j \in [k]$.
  Suppose that the first $k$ prefix sums lie in $\{0,1\}$ and that the sum of the entries of $x$ is $1$.
  Then $x_{k+1} = 1 - 2 \sum_{j=1}^k x_j \in \{\pm1\}$ and all prefix sums lie in $\{0,1\}$.
\end{lemma}
\begin{proof}
  By symmetry, $\sum_{j=1}^n x_j = 2\sum_{j=1}^k x_j + x_{k+1}$.
  Since the total sum equals $1$, we have $x_{k+1} = 1 - 2\sum_{j=1}^k x_j \in \{\pm1\}$.
  For any $r \in [k+1]$, we have
  $
  \sum_{j=1}^{k+r} x_j
  = 1 - \sum_{j=k+1+r}^n x_j
  = 1 - \sum_{j=1}^{k+1-r} x_j,
  $
  where the last equality follows from symmetry.
  For $r \in [k+1]$, the sum $\sum_{j=1}^{k+1-r} x_j$ lies in $\{0,1\}$ by assumption; hence the sum $\sum_{j=1}^{k+r} x_j$ also lies in $\{0,1\}$.
  Hence every prefix sum of $x$ lies in $\{0,1\}$.
\end{proof}

\begin{lemma}\label{lem:vsasm:middleColAlternate}
  Let $n \geq 1$ be odd and set $k = \floor*{n/2}$.
  For every VSASM $X \in \{0,\pm1\}^{n \times n}$, we have $x_{i,k+1} = (-1)^{i+1}$ for every $i \in [n]$.
\end{lemma}
\begin{proof}
  Applying \cref{lem:vsasm:symmetricHalfExtends} to each row of $X$, we obtain that the middle column has no zeros, i.e., $x_{i,k+1} \in \{\pm1\}$ for every $i \in [n]$.
  In any column of an ASM, the non-zero entries alternate in sign and the first non-zero entry is $+1$; therefore, the entry in the central column in row $i$ must be $x_{i,k+1} = (-1)^{i+1}$ for every $i \in [n]$.
\end{proof}

\paragraph{Core and assembly map.}
Assume $n$ is odd and let $k = \floor*{n/2}$.
Let the \emph{core} of a VSASM be its left $n\times k$ block, i.e.,
\[
  C = [n] \times [k]
\]
is the set of \emph{core positions}, and let $\pi_C$ be the coordinate-wise projection onto $C$.
Define the affine map $\varphi : \R^C\to\R^{n\times n}$ by
\[
  \varphi(Y)_{i,j}=
  \begin{cases}
    y_{i,j}          & \text{if } j \in [k],\\
    (-1)^{i+1}       & \text{if } j = k+1,\\
    y_{i,n+1-j}      & \text{if } j \in [k+2,n]
  \end{cases}
\]
for $Y \in \R^C$ and $i,j \in [n]$.
Thus $\varphi$ places the core $Y$ on the left, fixes the middle column to the alternating vector from \cref{lem:vsasm:middleColAlternate}, and fills the right $n \times k$ block by vertical reflection.
Clearly, the map $\varphi$ is an assembly map: it is affine, satisfies $\pi_C(\varphi(Y)) = Y$ for every $Y \in \R^C$, and $\varphi(\pi_C(X)) = X$ for every $X \in \VSASM(n)$, because $X$ is determined by its core together with the prescribed middle column and vertical symmetry.

\bigskip
We now describe the core polytope of VSASMs.
\begin{theorem}\label{thm:vsasm:corepolytope}
  Let $n \geq 1$ be odd, and set $k = \floor*{n/2}$.
  Then the core polytope $P^\core_\VSASM \subseteq \R^C$ of $n \times n$ VSASMs is described by the following system.
  \begin{align}
                          y&_{i,j} \in \R                       &\forall i \in [n], j \in [k],\label{eq:vsasm:real}\\\displaybreak[0]
    0 \leq \sum_{j'=1}^{j} y&_{i,j'} \leq 1                      &\forall i \in [n], j \in [k-1],\label{eq:vsasm:row-prefix}\\\displaybreak[0]
    0 \leq \sum_{i'=1}^{i} y&_{i',j} \leq 1                      &\forall i \in [n-1], j \in [k],\label{eq:vsasm:col-prefix}\\\displaybreak[0]
    \sum_{j=1}^{k}         y&_{i,j} = \chi_{2 \mid i}&\forall i \in [n],\label{eq:vsasm:row-sum}\\
    \sum_{i=1}^{n}         y&_{i,j} = 1                          &\forall j \in [k].\label{eq:vsasm:col-sum}
  \end{align}
\end{theorem}
\begin{proof}
  We show that the integer solutions to the system~\crefrange{eq:vsasm:real}{eq:vsasm:col-sum} are exactly the cores of VSASMs, and then we argue that the system defines an integral polytope.

  \medskip
  First, let $X$ be an $n\times n$ VSASM, and let $Y$ be its core, that is, $y_{i,j} = x_{i,j}$ for every $i \in [n]$ and $j \in [k]$.
  Since every VSASM is in particular an ASM, the ASM constraints in~\cref{eq:asm:real,eq:asm:row-prefix,eq:asm:col-prefix} of \cref{thm:ASMpolytope} directly imply the corresponding constraints in~\cref{eq:vsasm:real,eq:vsasm:row-prefix,eq:vsasm:col-prefix} for $Y$.
  For the row-sum equations in~\cref{eq:vsasm:row-sum}, \cref{lem:vsasm:middleColAlternate} gives $x_{i,k+1} = (-1)^{i+1}$ for every $i \in [n]$; hence, for every $i \in [n]$,
  \[
    \sum_{j=1}^{k} y_{i,j}
    = \sum_{j=1}^{k} x_{i,j}
    = \frac{1 - x_{i,k+1}}{2}
    =\chi_{2 \mid i}
  \]
  as required.
  Finally, the ASM column-sum equations in~\cref{eq:asm:col-sum} give the column-sum equations in~\cref{eq:vsasm:col-sum} for $Y$.
  Thus the cores of VSASMs satisfy the system~\crefrange{eq:vsasm:real}{eq:vsasm:col-sum}.

  \medskip
  Second, we show that every integer solution to the system~\crefrange{eq:vsasm:real}{eq:vsasm:col-sum} is the core of a VSASM\@.
  Let $Y \in\Z^C$ satisfy the system~\crefrange{eq:vsasm:real}{eq:vsasm:col-sum}, and set $X = \varphi(Y)$.
  By construction, $X$ is a vertically symmetric $n \times n$ integer matrix, its middle column is given by $x_{i,k+1} = (-1)^{i+1}$ for $i \in [n]$, and its core is $Y$.
  We verify that $X$ satisfies the ASM constraints given in \cref{thm:ASMpolytope}.
  For $i \in [n]$, we obtain
  \[
    \sum_{j=1}^{n} x_{i,j}
    = \sum_{j=1}^{k} y_{i,j} + x_{i, k+1} + \sum_{j=1}^{k} y_{i,j}
    = 2\chi_{2 \mid i} + (-1)^{i+1}
    = 1
  \]
  for the sum of row $i$, by~\cref{eq:vsasm:row-sum}.
  Furthermore, the prefix sums within each row of $X$ are in $\{0,1\}$ by the row-prefix bounds in~\cref{eq:vsasm:row-prefix}, the row-sum equations in~\cref{eq:vsasm:row-sum}, and \cref{lem:vsasm:symmetricHalfExtends}; thus the row-prefix bounds hold for $X$.
  For $j \in [k]$, column $j$ sums to $1$ by~\cref{eq:vsasm:col-sum}.
  For $j = k+1$, the middle column entries alternate between $+1$ and $-1$, with the first entry equal to $+1$.
  Since $n$ is odd, this column also sums to~$1$.
  For $j \in [k]$, the prefix bounds for column $j$ are precisely the bounds in~\cref{eq:vsasm:col-prefix}; in the middle column $j = k+1$, the entries alternate between $+1$ and $-1$, with the first entry equal to $+1$, so every prefix sum is either $0$ or $1$.
  For the last $k$ columns, the constraints follow by vertical symmetry.
  Thus $X \in \Z^{n \times n}$ satisfies all ASM constraints in \cref{thm:ASMpolytope}, hence $X$ is an ASM\@.
  We conclude that the integer solutions to the system~\crefrange{eq:vsasm:real}{eq:vsasm:col-sum} are exactly the cores of VSASMs.

  \medskip
  It remains to prove that the system~\crefrange{eq:vsasm:real}{eq:vsasm:col-sum} defines an integral polytope.
  The constraints in~\cref{eq:vsasm:row-prefix,eq:vsasm:col-prefix} impose lower and upper bounds on the sums of the entries within prefixes of each row and column, while the equations in~\cref{eq:vsasm:row-sum,eq:vsasm:col-sum} prescribe the full row and column sums.
  Therefore, the cores of VSASMs are so\nobreakdash-called \emph{prefix\nobreakdash-bounded matrices}, and the polytope of prefix\nobreakdash-bounded matrices is known to be described by the system above~\cite[Corollary~5.2]{borsik2026prefix}.
  Alternatively, the result also directly follows from \cref{thm:laminarSystem} applied to the laminar families of row and column prefixes.
\end{proof}

From \cref{thm:xasm:assembly,thm:vsasm:corepolytope}, we obtain the following description of the polytope $P_\VSASM$ of VSASMs.
\begin{theorem}\label{thm:vsasm:coreDescr}
  Let $n \geq 1$ be odd, and let $k = \floor*{n/2}$ and $\widehat P^\core_\VSASM = \{X \in \R^{n \times n} : \pi_C(X) \in P^\core_\VSASM\}$.
  Then
  \[
    P_\VSASM
    = \widehat P^\core_\VSASM \cap P_\VS \cap \{X \in \R^{n \times n} : x_{i,k+1} = (-1)^{i+1}\ \forall i \in [n]\}.
  \]
\end{theorem}
\begin{proof}
  By \cref{thm:xasm:assembly,thm:vsasm:corepolytope}, it suffices to prove that $\varphi(\R^C) = P_\VS \cap \{X \in \R^{n \times n} : x_{i,k+1} = (-1)^{i+1}\ \forall i \in [n]\}$.
  Let $P$ denote the right-hand side.
  By definition, $\varphi$ inserts its argument as the left half, fixes the $i^\text{th}$ entry of the middle column to $(-1)^{i+1}$ for every $i$, and fills the right half by vertical reflection.
  Thus, $\varphi(Y) \in P$ for every $Y \in \R^C$.
  Conversely, take any $X \in P$, and notice that $\varphi(\pi_C(X)) = X$; hence $X \in \varphi(\R^C)$.
\end{proof}

\begin{theorem}\label{thm:vsasm:descr}
  Let $n \geq 1$ be arbitrary, and set $k = \floor*{n/2}$.
  Then
  \[
    P_\VSASM = P_\ASM \cap P_\VS \cap \{X \in \R^{n \times n} : x_{i,k+1} = (-1)^{i+1}\ \forall i \in [n]\}.
  \]
\end{theorem}
\begin{proof}
  For odd $n$, we obtain the statement by straightforward transformations of the system given in \cref{thm:vsasm:corepolytope}.
  For even $n$, the polytope $P_\VSASM$ is empty; thus we need to show that the right-hand side is empty as well.
  Notice that $\{X \in \R^{n \times n} : x_{i,k+1} = (-1)^{i+1}\ \forall i \in [n]\}$ forces the entries in column $k+1$ to alternate between $+1$ and $-1$, with the first entry equal to $+1$, and thus $\sum_{i=1}^n x_{i,k+1} = 0$.
  On the other hand,~\cref{eq:asm:col-sum} gives $\sum_{i=1}^n x_{i,k+1} = 1$; thus the right-hand side is empty.
\end{proof}

Next, we discuss the dimension of $P_\VSASM$ for odd $n$.
\begin{theorem}\label{thm:vsasm:dim}
  For every odd $n \geq 3$, the dimension of $P_\VSASM$ is $\frac{(n-3)^2}{2}$.
\end{theorem}
\begin{proof}
  It suffices to prove that the dimension of $P^\core_\VSASM$ is $\frac{(n-3)^2}{2}$, because the assembly map $\varphi$ restricts to an affine isomorphism between $P^\core_\VSASM$ and $P_\VSASM$, which preserves dimension.
  First, we give an upper bound.
  The constraints in~\cref{eq:vsasm:col-prefix,eq:vsasm:row-sum,eq:vsasm:col-sum} force all row-prefix sums in rows $1$ and $n$ to be~$0$; hence
  \begin{align}
    y&_{1,j} = 0 &\forall j \in [k],\label{eq:vsasm:dim:firstrow}\\
    y&_{n,j} = 0 &\forall j \in [k].\label{eq:vsasm:dim:lastrow}
  \end{align}
  Using the equations in~\cref{eq:vsasm:dim:firstrow,eq:vsasm:dim:lastrow}, we eliminate the variables $y_{1,j}$ and $y_{n,j}$ from the column-sum equations in~\cref{eq:vsasm:col-sum} as follows: for each $j\in[k]$, we replace the $j^\text{th}$ equation of~\cref{eq:vsasm:col-sum} by the difference of that equation and the two equations $y_{1,j}=0$ and $y_{n,j}=0$.
  This elementary row operation does not change the solution set, so it does not change the rank of the equation system.
  After this replacement, none of the modified equations in~\cref{eq:vsasm:col-sum} involves a variable of the form $y_{1,j}$ or $y_{n,j}$.
  Moreover, the equations in~\cref{eq:vsasm:row-sum} for $i\in\{1,n\}$ become redundant.
  By \cref{lem:prelim:rowcol-rank}, $n+k-3$ independent equations remain in~\cref{eq:vsasm:row-sum,eq:vsasm:col-sum}, which --- after the elimination --- involve only variables $y_{i,j}$ with $i\in[2,n-1]$, and thus have disjoint support from~\cref{eq:vsasm:dim:firstrow,eq:vsasm:dim:lastrow}.
  We obtain $(n+k-3) + 2k = 5k-2$ independent equations.
  These define an affine subspace of dimension
  $
    |C| - (5k - 2) = nk - (5k - 2) = 2k^2 - 4k + 2 = \frac{(n-3)^2}{2},
  $
  which contains $P^\core_\VSASM$ and hence gives the bound $\dim(P^\core_\VSASM)\leq \frac{(n-3)^2}{2}$.

  \medskip
  Second, we construct $\frac{(n-3)^2}{2}+1$ affinely independent cores in $P^\core_\VSASM$ and hence obtain a matching lower bound.
  Let $\ol Y$ denote the average of the cores of VSASMs.
  We claim that $\ol Y$ reaches neither the lower nor the upper bound in~\cref{eq:vsasm:row-prefix} for any $i \in [2,n-1], j \in [k-1]$.
  It suffices to show that, for each such pair $i$ and $j$, there exists a VSASM for which the sum of the first $j$ entries in row $i$ is $1$, and there exists a VSASM for which this sum is $0$.

  To see this, place a $k \times k$ permutation matrix into the submatrix formed by the first $k$ entries of the $k$ rows with even index in such a way that $y_{i,1} = 1$ for some fixed even $i \in [2,n-1]$, and fill the rest of the core entries with $0$.
  For each even $i \in [2,n-1]$, we obtain the core of a VSASM in which, for the fixed index $i$, the sum of the first $j$ entries in row $i$ equals $1$ for every $j \in [k]$, and equals $0$ for every row with odd index and every $j \in [k]$.

  Next, set $y_{i,k} = (-1)^i$ for every $i \in [2,n-1]$.
  Place a $(k-1) \times (k-1)$ permutation matrix into the submatrix formed by the first $k-1$ entries of the $k-1$ rows with odd index in such a way that $y_{i,1} = 1$ for some fixed odd $i \in [3,n-2]$, and fill the rest of the core entries with $0$.
  This yields the core of a VSASM in which, for the fixed index $i$, the sum of the first $j$ entries in row $i$ equals $1$ for every $j \in [k-1]$, and equals $0$ for every even $i \in [2,n-1]$ and every $j \in [k-1]$.
  Thus, $\ol Y$ does not reach equality in~\cref{eq:vsasm:row-prefix} for any $i \in [2,n-1], j \in [k-1]$.
  An analogous argument shows that $\ol Y$ does not reach equality in~\cref{eq:vsasm:col-prefix} for any $i \in [2,n-2], j \in [k]$.

  For each $i \in [2,n-2], j \in [k-1]$, define
  $
    \ol Y^{i,j}
    = \ol Y
    + \varepsilon \chi_{i,j}
    - \varepsilon \chi_{i,k}
    - \varepsilon \chi_{n-1,j}
    + \varepsilon \chi_{n-1,k}
  $, where $\varepsilon$ is a small positive constant.
  By the argument above, $\ol Y^{i,j}$ does not violate any inequality in~\cref{eq:vsasm:row-prefix,eq:vsasm:col-prefix} when $\varepsilon > 0$ is small enough, and it satisfies the equations in~\cref{eq:vsasm:row-sum,eq:vsasm:col-sum} by the definition of $\ol Y^{i,j}$; thus $\ol Y^{i,j} \in P^\core_\VSASM$.
  The cores $\ol Y$ and $\ol Y^{i,j}$ for $i \in [2,n-2], j \in [k-1]$ are affinely independent: only the difference $\ol Y^{i,j}-\ol Y$ has a non-zero entry at $(i,j)$, so the cores $\{\ol Y^{i,j}-\ol Y : i \in [2,n-2], j \in [k-1]\}$ are linearly independent.
  Therefore, the dimension of $P^\core_\VSASM$ is at least $(n-3)(k-1) = \frac{(n-3)^2}{2}$.

  Combining the lower and upper bounds yields $\dim(P^\core_\VSASM)=\dim(P_\VSASM)=\frac{(n-3)^2}{2}$.
\end{proof}

\begin{theorem}\label{thm:vsasm:facets}
  Let $n \geq 7$ be odd, and set $k=\floor*{n/2}$.
  The facets of $P^\core_\VSASM$ are given by tightening the lower bound in~\cref{eq:vsasm:row-prefix} to equality for
  $
    (i,j) \in \{ (i,j) : i \in [3,n-2], j \in [k-1 - \chi_{2 \mid i}] \}
  $
  and the upper bound for
  $
    (i,j) \in \{ (i,j) : i \in [4,n-3], j \in [2,k-1-\chi_{2 \nmid i}]\}
  $;
  and by tightening the lower bound in~\cref{eq:vsasm:col-prefix} to equality for
  $
    (i,j) \in \{(2,1)\} \cup \{(i,j) : i \in [2,n-4], j \in [2,k-\chi_{2 \mid i}]\}
  $,
  and the upper bound for
  $
    (i,j) \in \{(n-2,1)\} \cup \{(i,j) : i \in [4,n-2], j \in [2,k-\chi_{2 \nmid i}]\}
  $.
  In particular, the number of facets of $P^\core_\VSASM$ is $2n^2-19n+49$.
\end{theorem}
\begin{proof}
  The facets are obtained by tightening a single inequality in one of~\cref{eq:vsasm:row-prefix,eq:vsasm:col-prefix} to equality for the index pairs listed in the statement of the theorem.
  We call the instances of the lower bounds in~\cref{eq:vsasm:row-prefix} that are tightened to equality the \emph{horizontal facet lower bounds}, and we define the \emph{horizontal facet upper bounds} analogously.
  Likewise, we call the instances of the lower bounds in~\cref{eq:vsasm:col-prefix} that are tightened to equality the \emph{vertical facet lower bounds}, and we define the \emph{vertical facet upper bounds} analogously.
  We refer to the union of these four families as the \emph{facet inequalities}.

  We proceed in two steps.
  First, we show that the facet inequalities together with the equations in~\crefrange{eq:vsasm:row-sum}{eq:vsasm:dim:lastrow} imply every inequality in~\cref{eq:vsasm:row-prefix,eq:vsasm:col-prefix}.
  Then, for every facet inequality, we construct a core of an $n \times n$ matrix violating that facet inequality and no other, thereby proving that no facet inequality is redundant.
  The core $\ol Y$ constructed in the second step of the proof of \cref{thm:vsasm:dim} shows that no facet inequality is an implicit equation; thus the two steps together imply that the facet inequalities form a minimal system that, extended with the equations in~\crefrange{eq:vsasm:row-sum}{eq:vsasm:dim:lastrow}, describes the convex hull of the cores of VSASMs, which proves the theorem.

  \medskip
  Next, we prove that the facet inequalities together with the equations in~\crefrange{eq:vsasm:row-sum}{eq:vsasm:dim:lastrow} imply every inequality in~\cref{eq:vsasm:row-prefix,eq:vsasm:col-prefix}.
  Clearly, we need to treat only those inequalities in~\cref{eq:vsasm:row-prefix,eq:vsasm:col-prefix} that are \emph{non-facet} inequalities, namely, the lower bounds in~\cref{eq:vsasm:row-prefix} for
  \[
    (i,j) \in \bigl(\{1,2,n-1,n\}\times[k-1]\bigr)\ \cup\ \bigl(\{i\in[4,n-3] : 2 \mid i\} \times \{k-1\}\bigr)
  \]
  and the upper bounds for
  \[
    (i,j) \in \bigl(\{1,2,3,n-2,n-1,n\}\times[k-1]\bigr)\ \cup\ \bigl([4,n-3]\times\{1\}\bigr)\ \cup\ \bigl(\{i\in[5,n-4] : 2 \nmid i\} \times \{k-1\}\bigr);
  \]
  and the lower bounds in~\cref{eq:vsasm:col-prefix} for
  \[
    (i,j) \in \bigl(\{1\}\times[k]\bigr)\ \cup\ \bigl([3,n-1]\times\{1\}\bigr)\ \cup\ \bigl(\{n-3,n-2,n-1\}\times[2,k]\bigr)\ \cup\ \bigl(\{i\in[2,n-5] : 2 \mid i\} \times \{k\}\bigr)
  \]
  and the upper bounds for
  \[
    (i,j) \in \bigl(\{1,2,3,n-1\}\times[k]\bigr)\ \cup\ \bigl([4,n-3]\times\{1\}\bigr)\ \cup\ \bigl(\{i\in[5,n-2] : 2 \nmid i\} \times \{k\}\bigr).
  \]

  We start with the non-facet inequalities in~\cref{eq:vsasm:row-prefix}.
  By the equations in~\cref{eq:vsasm:dim:firstrow,eq:vsasm:dim:lastrow}, we have $y_{1,j}=y_{n,j}=0$ for $j\in[k]$; hence the bounds in~\cref{eq:vsasm:row-prefix} hold for $i \in \{1,n\}, j \in [k-1]$.
  The vertical facet lower bound at $(2,1)$ gives $y_{2,1}\geq 0$, and the horizontal facet lower bounds at $(i,1)$ give $y_{i,1}\geq 0$ for all $i\in[3,n-2]$.
  Moreover, the vertical facet upper bound at $(n-2,1)$ yields $\sum_{i'=1}^{n-2}y_{i',1}\leq 1$, hence $y_{n-1,1}=1-\sum_{i'=1}^{n-2}y_{i',1}\geq 0$ by~\cref{eq:vsasm:col-sum,eq:vsasm:dim:lastrow}.
  Consequently $y_{i,1}\geq 0$ for all $i\in[n]$, which also implies $y_{i,1}\leq 1$ for all $i\in[n]$ by~\cref{eq:vsasm:col-sum} for $j=1$, and thus the non-facet row-prefix upper bounds at $(i,1)$ for $i \in [n]$ hold.

  For $i=2,j\in[k-1]$, the vertical facet lower bound gives $y_{2,j}\geq 0$.
  Moreover, we have $y_{3,k}\leq 0$ by the horizontal facet lower bound at $(3,k-1)$ and the equation in~\cref{eq:vsasm:row-sum} for $i=3$, which together with the vertical facet lower bound for $i=3,j=k$, i.e., $0\leq y_{1,k}+y_{2,k}+y_{3,k}=y_{2,k}+y_{3,k}$, imply $y_{2,k}\geq 0$.
  Thus $y_{2,j}\geq 0$ for every $j \in [k]$, and therefore both the lower and the upper bound in~\cref{eq:vsasm:row-prefix} hold for $i=2,j\in[k-1]$.
  For $i=n-1,j\in[k-1]$, the vertical facet upper bound gives $\sum_{i'=1}^{n-2}y_{i',j}\leq 1$, hence
  $
    y_{n-1,j}= \sum_{i'=1}^{n}y_{i',j}-\sum_{i'=1}^{n-2}y_{i',j} = 1-\sum_{i'=1}^{n-2}y_{i',j}\geq 0
  $
  by~\cref{eq:vsasm:col-sum,eq:vsasm:dim:lastrow}.
  Furthermore, using the horizontal facet lower bound at $(n-2,k-1)$ and the equation in~\cref{eq:vsasm:row-sum} for $i = n-2$, we obtain $y_{n-2,k}\leq 0$, which together with the vertical facet upper bound for $i=n-3,j=k$ imply $\sum_{i'=1}^{n-3}y_{i',k}\leq 1$, hence $\sum_{i'=1}^{n-2}y_{i',k}\leq 1$ and thus $y_{n-1,k}=1-\sum_{i'=1}^{n-2}y_{i',k}\geq 0$.
  Therefore, $y_{n-1,j}\geq 0$ for all $j\in[k]$.
  By the equations in~\cref{eq:vsasm:row-sum},
  $
  0\leq \sum_{j'=1}^{j} y_{n-1,j'}\leq \sum_{j'=1}^{k} y_{n-1,j'}=1
  $
  for all $j\in[k-1]$, and thus the bounds in~\cref{eq:vsasm:row-prefix} hold for $i=n-1,j\in[k-1]$.

  Next, let $i=3,j\in[k-1]$.
  For every $j'\in[j+1,k]$, the vertical facet lower bound at $(3,j')$ gives $0\leq y_{2,j'}+y_{3,j'}$, hence $-y_{3,j'}\leq y_{2,j'}$.
  Using the equations in~\cref{eq:vsasm:row-sum} for $i=2,3$, we obtain
  $
    \sum_{j'=1}^{j} y_{3,j'}
    = -\sum_{j'=j+1}^{k} y_{3,j'}
    \leq \sum_{j'=j+1}^{k} y_{2,j'}
    = 1-\sum_{j'=1}^{j} y_{2,j'}
    \leq 1,
  $
  proving the upper bound in~\cref{eq:vsasm:row-prefix} for $i=3,j\in[k-1]$.

  Now let $i=n-2$.
  We have already verified the upper bound in~\cref{eq:vsasm:row-prefix} for $j=1$.
  Fix $j\in[2,k-1]$.
  Summing the first $j$ vertical facet upper bounds for $i=n-2$ yields
  $
    \sum_{j'=1}^{j} y_{n-2,j'} \leq j-\sum_{j'=1}^{j}\sum_{i'=1}^{n-3} y_{i',j'}
  $
  after simple rearrangement.
  Summing the equations in~\cref{eq:vsasm:row-sum} over $[2,n-3]$ gives $\sum_{i'=2}^{n-3}\sum_{j'=1}^{k} y_{i',j'}=k-1$.
  Moreover, for every $j' \in [j+1,k] \subseteq [2,k]$, we have $\sum_{i'=2}^{n-3} y_{i',j'}\leq 1$ by the vertical facet upper bound at $(n-3,j')$.
  Hence $\sum_{i=2}^{n-3}\sum_{j'=1}^{j} y_{i,j'}\geq j-1$, and since the first row is zero we obtain $\sum_{j'=1}^{j}\sum_{i'=1}^{n-3} y_{i',j'}\geq j-1$.
  Plugging this into the previous display gives $\sum_{j'=1}^{j} y_{n-2,j'}\leq 1$ for all $j\in[2,k-1]$, which is exactly the upper bound in~\cref{eq:vsasm:row-prefix} for $i=n-2$.

  If $i\in[4,n-3]$ is even, then the vertical facet lower bound at $(i-1,k)$ and the vertical facet upper bound at $(i,k)$ together imply $y_{i,k}\leq 1$, hence $\sum_{j'=1}^{k-1}y_{i,j'}=1-y_{i,k}\geq 0$ by~\cref{eq:vsasm:row-sum}, which is the non-facet lower bound in~\cref{eq:vsasm:row-prefix} at $(i,k-1)$.
  If $i\in[5,n-4]$ is odd, then the vertical facet upper bound at $(i-1,k)$ and the vertical facet lower bound at $(i,k)$ together imply $y_{i,k}\geq -1$, hence $\sum_{j'=1}^{k-1}y_{i,j'}=-y_{i,k}\leq 1$ by~\cref{eq:vsasm:row-sum}, which is the non-facet upper bound in~\cref{eq:vsasm:row-prefix} at $(i,k-1)$.
  This completes the derivation of all bounds in~\cref{eq:vsasm:row-prefix}.

  \medskip
  We now turn to the non-facet inequalities in~\cref{eq:vsasm:col-prefix}.
  For $i=1,j\in[k]$, the inequalities follow because the first row is uniformly zero by~\cref{eq:vsasm:dim:firstrow}.
  We have already shown above that $y_{i',1}\geq 0$ for all $i'\in[n]$ and $\sum_{i'=1}^{n}y_{i',1}=1$ by~\cref{eq:vsasm:col-sum}, and thus the lower and upper bounds in~\cref{eq:vsasm:col-prefix} for $i\in[n],j=1$ follow.

  Let $i=2,j\in[2,k]$.
  As shown above, we have $y_{2,j'}\geq 0$ for all $j'\in[k]$, and the equation in~\cref{eq:vsasm:row-sum} for $i=2$ gives $\sum_{j'=1}^{k}y_{2,j'}=1$.
  Hence $y_{2,j}\leq 1$, and since $y_{1,j}=0$ by~\cref{eq:vsasm:dim:firstrow}, we obtain
  $
    y_{1,j} + y_{2,j} = y_{2,j}\leq 1
  $,
  proving the non-facet upper bound in~\cref{eq:vsasm:col-prefix} at $(2,j)$.

  Next, we derive the non-facet upper bound in~\cref{eq:vsasm:col-prefix} for $i=3,j\in[k]$.
  We have $y_{2,1}\geq 0$ by the vertical facet lower bound at $(2,1)$, and $y_{3,1}\geq 0$ by the horizontal facet lower bound at $(3,1)$, hence $y_{2,1}+y_{3,1}\geq 0$.
  For every $j'\in[2,k-1]$, the vertical facet lower bound at $(3,j')$ yields
  $
    0\leq y_{1,j'} + y_{2,j'} + y_{3,j'} = y_{2,j'}+y_{3,j'},
  $
  because $y_{1,j'}=0$ by~\cref{eq:vsasm:dim:firstrow}.
  Moreover, we have
  $
    \sum_{j'=1}^{k}(y_{2,j'}+y_{3,j'})
    = \sum_{j'=1}^{k}y_{2,j'}+\sum_{j'=1}^{k}y_{3,j'}
    = 1+0
    = 1
  $
  by~\cref{eq:vsasm:row-sum}.
  Since $y_{2,j'}+y_{3,j'}\geq 0$ for all $j'\in[k]$, it follows that $y_{2,j}+y_{3,j}\leq 1$, and hence
  $
    y_{1,j} + y_{2,j} + y_{3,j} = y_{2,j}+y_{3,j}\leq 1,
  $
  proving the non-facet upper bound in~\cref{eq:vsasm:col-prefix} for $i=3,j\in[k]$.

  For $i=n-1$, we have
  $
    \sum_{i'=1}^{n-1}y_{i',j} = \sum_{i'=1}^{n}y_{i',j} = 1
  $
  by~\cref{eq:vsasm:col-sum,eq:vsasm:dim:lastrow}, proving the non-facet lower and upper bound in~\cref{eq:vsasm:col-prefix} for $i=n-1,j\in[k]$.

  Fix $j\in[2,k]$.
  Summing the equations in~\cref{eq:vsasm:row-sum} over $[2,n-3]$ gives
  $
    \sum_{i=2}^{n-3}\sum_{j'=1}^{k} y_{i,j'}=k-1.
  $
  For every $j'\in[2,k]$, the vertical facet upper bound at $(n-3,j')$ yields
  $
    \sum_{i=2}^{n-3}y_{i,j'}\leq 1,
  $
  and we also have $\sum_{i=2}^{n-3}y_{i,1}\leq \sum_{i=1}^{n}y_{i,1}=1$ by~\cref{eq:vsasm:col-sum}.
  Therefore,
  $
    \sum_{i=2}^{n-3}y_{i,j}
    = (k-1) - \sum_{j'\in[k]\setminus\{j\}}\sum_{i=2}^{n-3}y_{i,j'}
    \geq (k-1) - (k-1)
    = 0,
  $
  and, since the first row is zero, we obtain
  $
    0\leq \sum_{i'=1}^{n-3}y_{i',j},
  $
  which is the non-facet lower bound in~\cref{eq:vsasm:col-prefix} at $(n-3,j)$.
  Using $\sum_{i=1}^n y_{i,k}=1$, $y_{n,k}=0$, and the vertical facet upper bound in~\cref{eq:vsasm:col-prefix} for $(n - 3,k)$, we obtain
  $
    y_{n-1,k}
    =\sum_{i=1}^n y_{i,k}-\sum_{i=1}^{n-3} y_{i,k}-y_{n-2,k}-y_{n,k}
    \geq -\,y_{n-2,k}
  $.
  Hence $y_{n-1,k}\geq 0$, since $y_{n-2,k}\leq 0$ follows from $\sum_{j=1}^{k} y_{n-2,j}=0$ together with the horizontal facet lower bound for $(n-2, k-1)$.
  Moreover, as $y_{n-1,j}\geq 0$ was already established for all $j\in[k-1]$, every summand in
  $
    \sum_{j=1}^{k} y_{n-1,j}=1
  $
  is non-negative, and therefore $y_{n-1,j}\leq 1$ for all $j \in [k]$.
  Combining this with $\sum_{i'=1}^n y_{i',j} = 1$ and $y_{n,j} = 0$ from~\cref{eq:vsasm:dim:lastrow}, we get, for each $j \in [k]$,
  $
    \sum_{i'=1}^{n-2} y_{i',j}
    = \sum_{i'=1}^n y_{i',j}-y_{n-1,j}
    = 1 - y_{n-1,j}
    \leq 1
  $.
  This proves the non-facet upper bound in~\cref{eq:vsasm:col-prefix} for $(n-2, j)$.
  
  Now we derive the non-facet bounds in~\cref{eq:vsasm:col-prefix} for $j = k$.
  For every odd $i\in[3,n-2]$,~\cref{eq:vsasm:row-sum} gives $\sum_{j'=1}^{k}y_{i,j'} = 0$ and the horizontal facet lower bound at $(i, k-1)$ gives $\sum_{j'=1}^{k-1}y_{i,j'} \geq 0$, hence $y_{i,k} = -\sum_{j'=1}^{k-1}y_{i,j'} \leq 0$.
  Now let $i\in[2,n-5]$ be even.
  Then $i+1$ is odd and lies in $[3, n-4]$, so $\sum_{i'=1}^{i+1}y_{i',k} \geq 0$ is a vertical facet lower bound.
  Together with $y_{i+1,k} \leq 0$ this implies
  $
    \sum_{i'=1}^{i}y_{i',k}
    = \sum_{i'=1}^{i+1}y_{i',k} - y_{i+1,k}
    \geq 0,
  $
  which proves the non-facet lower bounds in~\cref{eq:vsasm:col-prefix} for even $i\in[2,n-5]$ and $j=k$.

  Similarly, let $i\in[5,n-2]$ be odd.
  Then $i-1$ is even and lies in $[4,n-3]$, so $\sum_{i'=1}^{i-1}y_{i',k}\leq 1$ is a vertical facet upper bound.
  Together with $y_{i,k}\leq 0$ this implies
  $
    \sum_{i'=1}^{i}y_{i',k}
    = \sum_{i'=1}^{i-1}y_{i',k} + y_{i,k}
    \leq 1,
  $
  which proves the non-facet upper bounds in~\cref{eq:vsasm:col-prefix} for odd $i\in[5,n-2]$ and $j=k$.

  \medskip
  It remains to show that no facet inequality is redundant.
  For every horizontal facet lower bound, we construct a core $L^{n,\mathrm H}_{i,j} \in \R^C$ that violates the facet lower bound in~\cref{eq:vsasm:row-prefix} for the given $i,j$ and satisfies every other facet inequality as well as the equations in~\crefrange{eq:vsasm:row-sum}{eq:vsasm:dim:lastrow}.
  Similarly, we construct the cores $L^{n,\mathrm V}_{i,j}, U^{n,\mathrm H}_{i,j}, U^{n,\mathrm V}_{i,j} \in \R^C$ defined analogously for the vertical facet lower bounds, the horizontal and vertical facet upper bounds for the respective indices $i$ and $j$.

  Now we are ready to construct $L^{n,\mathrm H}_{i,j}$, $L^{n,\mathrm V}_{i,j}$, $U^{n,\mathrm H}_{i,j}$, and $U^{n,\mathrm V}_{i,j}$ via an inductive approach.
  For $n=7$, we set
  \[
    \mathrlap{L^{7,\mathrm H}_{3,1}}\hphantom{U^{7,\mathrm V}_{5,1}}
    =
    \vcenter{\hbox{\begin{ytableau}
          0 & 0 & 0\\
          1 & 0 & 0\\
          -1 & 1 & 0\\
          1 & 0 & 0\\
          0 & 0 & 0\\
          0 & 0 & 1\\
          0 & 0 & 0
        \end{ytableau}}}\,,
    \quad
    \mathrlap{L^{7,\mathrm H}_{3,2}}\hphantom{U^{7,\mathrm V}_{4,2}}
    =
    \vcenter{\hbox{\begin{ytableau}
          0 & 0 & 0\\
          0 & 1 & 0\\
          0 &-1 & 1\\
          0 & 1 & 0\\
          0 & 0 & 0\\
          1 & 0 & 0\\
          0 & 0 & 0
        \end{ytableau}}}\,,
    \quad
    \mathrlap{L^{7,\mathrm H}_{4,1}}\hphantom{U^{7,\mathrm V}_{5,2}}
    =
    \vcenter{\hbox{\begin{ytableau}
          0 & 0 & 0\\
          0 & 0 & 1\\
          1 & 0 &-1\\
          -1 & 1 & 1\\
          1 & 0 &-1\\
          0 & 0 & 1\\
          0 & 0 & 0
        \end{ytableau}}}\,,
    \quad
    \mathrlap{L^{7,\mathrm H}_{5,1}}\hphantom{U^{7,\mathrm V}_{4,3}}
    =
    \vcenter{\hbox{\begin{ytableau}
          0 & 0 & 0\\
          0 & 0 & 1\\
          0 & 0 & 0\\
          1 & 0 & 0\\
          -1 & 1 & 0\\
          1 & 0 & 0\\
          0 & 0 & 0
        \end{ytableau}}}\,,
    \quad
    \mathrlap{L^{7,\mathrm H}_{5,2}}\hphantom{U^{7,\mathrm H}_{4,2}}
    =
    \vcenter{\hbox{\begin{ytableau}
          0 & 0 & 0\\
          1 & 0 & 0\\
          0 & 0 & 0\\
          0 & 1 & 0\\
          0 &-1 & 1\\
          0 & 1 & 0\\
          0 & 0 & 0
        \end{ytableau}}}\,,
    \]
    \[
    L^{7,\mathrm V}_{2,1}=
    \vcenter{\hbox{\begin{ytableau}
          0 & 0 & 0\\
          -1 & 1 & 1\\
          1 &-1 & 0\\
          1 & 0 & 0\\
          0 & 0 & 0\\
          0 & 1 & 0\\
          0 & 0 & 0
        \end{ytableau}}}\,,
    \quad
    L^{7,\mathrm V}_{2,2}=
    \vcenter{\hbox{\begin{ytableau}
          0 & 0 & 0\\
          1 &-1 & 1\\
          0 & 1 &-1\\
          0 & 1 & 0\\
          0 & 0 & 0\\
          0 & 0 & 1\\
          0 & 0 & 0
        \end{ytableau}}}\,,
    \quad
    L^{7,\mathrm V}_{3,2}=
    \vcenter{\hbox{\begin{ytableau}
          0 & 0 & 0\\
          0 & 0 & 1\\
          1 &-1 & 0\\
          0 & 1 & 0\\
          0 & 0 & 0\\
          0 & 1 & 0\\
          0 & 0 & 0
        \end{ytableau}}}\,,
    \quad
    L^{7,\mathrm V}_{3,3}=
    \vcenter{\hbox{\begin{ytableau}
          0 & 0 & 0\\
          1 & 0 & 0\\
          0 & 1 &-1\\
          0 & 0 & 1\\
          0 & 0 & 0\\
          0 & 0 & 1\\
          0 & 0 & 0
        \end{ytableau}}}\,,
  \]
  \[
    U^{7,\mathrm H}_{4,2}=
    \vcenter{\hbox{\begin{ytableau}
          0 & 0 & 0\\
          0 & 0 & 1\\
          0 & 0 & 0\\
          1 & 1 &-1\\
          0 & 0 & 0\\
          0 & 0 & 1\\
          0 & 0 & 0
        \end{ytableau}}}\,,
    \quad
    U^{7,\mathrm V}_{4,2}=
    \vcenter{\hbox{\begin{ytableau}
          0 & 0 & 0\\
          0 & 1 & 0\\
          0 & 0 & 0\\
          0 & 1 & 0\\
          1 &-1 & 0\\
          0 & 0 & 1\\
          0 & 0 & 0
        \end{ytableau}}}\,,
    \quad
    U^{7,\mathrm V}_{4,3}=
    \vcenter{\hbox{\begin{ytableau}
          0 & 0 & 0\\
          0 & 0 & 1\\
          0 & 0 & 0\\
          0 & 0 & 1\\
          0 & 1 &-1\\
          1 & 0 & 0\\
          0 & 0 & 0
        \end{ytableau}}}\,,
    \quad
    U^{7,\mathrm V}_{5,1}=
    \vcenter{\hbox{\begin{ytableau}
          0 & 0 & 0\\
          1 & 0 & 0\\
          0 & 0 & 0\\
          0 & 1 & 0\\
          1 &-1 & 0\\
          -1 & 1 & 1\\
          0 & 0 & 0
        \end{ytableau}}}\,,
    \quad
    U^{7,\mathrm V}_{5,2}=
    \vcenter{\hbox{\begin{ytableau}
          0 & 0 & 0\\
          0 & 1 & 0\\
          0 & 0 & 0\\
          0 & 0 & 1\\
          0 & 1 &-1\\
          1 &-1 & 1\\
          0 & 0 & 0
        \end{ytableau}}}\,.
    \]

  \medskip
  For the case $n\geq 9$, assume that all certifying cores for size $n-2$ have already been defined.
  That is, for every facet inequality of $P^\core_\VSASM$ for size $n-2$, we already have the core $L^{n-2,\mathrm H}_{i,j}$, $L^{n-2,\mathrm V}_{i,j}$, $U^{n-2,\mathrm H}_{i,j}$, or $U^{n-2,\mathrm V}_{i,j}$ that violates that particular facet inequality and satisfies every other facet inequality as well as the equations in~\crefrange{eq:vsasm:row-sum}{eq:vsasm:dim:lastrow}.

  In order to build the certifying cores of $n\times n$ VSASMs from the certifying cores for size $n-2$, we define the four \emph{extension operators}
  \begin{center}
    \begin{tikzpicture}[scale=.8, baseline=(current bounding box.center)]
      \node at (-1.4,2.5) {$\extUL(Z)\ =$};
      \node[draw=none] at (4.15,2.5) {$,$};
      \draw[fill=cyan!15] (1,0) rectangle (4,3);
      \draw (0,0) rectangle (4,5);

      \node at (0.5,4.5) {$0$};
      \node at (1.5,4.5) {$0$};
      \node at (3.5,4.5) {$0$};
      \ThreeDotsAt{(2.5,4.5)}{.21213203}{0}

      \node at (0.5,3.5) {$1$};
      \node at (1.5,3.5) {$0$};
      \ThreeDotsAt{(2.5,3.5)}{.21213203}{0}
      \node at (3.5,3.5) {$0$};

      \node at (0.5,2.5) {$0$};
      \ThreeDotsAt{(0.5,1.5)}{.21213203}{90}
      \node at (2.5,1.5) {$Z$};
      \node at (0.5,0.5) {$0$};

      \begin{scope}[xshift=7.5cm]
        \node at (-1.4,2.5) {$\extUR(Z)\ =$};
        \node[draw=none] at (4.15,2.5) {$,$};
        \draw[fill=cyan!15] (0,0) rectangle (3,3);
        \draw (0,0) rectangle (4,5);

        \node at (0.5,4.5) {$0$};
        \node at (2.5,4.5) {$0$};
        \node at (3.5,4.5) {$0$};
        \ThreeDotsAt{(1.5,4.5)}{.21213203}{0}

        \node at (0.5,3.5) {$0$};
        \node at (2.5,3.5) {$0$};
        \ThreeDotsAt{(1.5,3.5)}{.21213203}{0}
        \node at (3.5,3.5) {$1$};

        \node at (1.5,1.5) {$Z$};
        \node at (3.5,2.5) {$0$};
        \ThreeDotsAt{(3.5,1.5)}{.21213203}{90}
        \node at (3.5,0.5) {$0$};
      \end{scope}
    \end{tikzpicture}

    \vspace{12pt}

    \begin{tikzpicture}[scale=.8, baseline=(current bounding box.center)]
      \node at (-1.4,2.5) {$\extBL(Z)\ =$};
      \node[draw=none] at (4.15,2.5) {$,$};
      \draw[fill=cyan!15] (1,2) rectangle (4,5);
      \draw (0,0) rectangle (4,5);

      \node at (0.5,4.5) {$0$};
      \node at (2.5,3.5) {$Z$};

      \ThreeDotsAt{(0.5,3.5)}{.21213203}{90}
      \node at (0.5,2.5) {$0$};

      \node at (0.5,1.5) {$1$};
      \node at (1.5,1.5) {$0$};
      \ThreeDotsAt{(2.5,1.5)}{.21213203}{0}
      \node at (3.5,1.5) {$0$};

      \node at (0.5,0.5) {$0$};
      \node at (1.5,0.5) {$0$};
      \ThreeDotsAt{(2.5,0.5)}{.21213203}{0}
      \node at (3.5,0.5) {$0$};

      \begin{scope}[xshift=7.5cm]
        \node at (-1.4,2.5) {$\extBR(Z)\ =$};
        \node[draw=none] at (4.15,2.5) {$.$};
        \draw[fill=cyan!15] (0,2) rectangle (3,5);
        \draw (0,0) rectangle (4,5);
        \node at (3.5,4.5) {$0$};
        \node at (1.5,3.5) {$Z$};
        \node at (3.5,2.5) {$0$};
        \ThreeDotsAt{(3.5,3.5)}{.21213203}{90}

        \node at (0.5,1.5) {$0$};
        \node at (2.5,1.5) {$0$};
        \ThreeDotsAt{(1.5,1.5)}{.21213203}{0}
        \node at (3.5,1.5) {$1$};

        \node at (0.5,0.5) {$0$};
        \node at (2.5,0.5) {$0$};
        \ThreeDotsAt{(1.5,0.5)}{.21213203}{0}
        \node at (3.5,0.5) {$0$};
      \end{scope}
    \end{tikzpicture}
  \end{center}
  More precisely, each operator takes a core of an $(n-2)\times (n-2)$ matrix and yields the core of an $n \times n$ matrix.
  In particular, $\extUL$ and $\extUR$ adjoin two new top rows, whereas $\extBL$ and $\extBR$ adjoin two new bottom rows; $\extBL$ and $\extUL$ adjoin a new first column, whereas $\extBR$ and $\extUR$ adjoin a new last column.
  In each case, the unique new entry at the corner cell adjacent to the old core is set to $1$, and the other new entries are set to $0$.

  \medskip
  By the definition of the four extension operators, every prefix sum in~\cref{eq:vsasm:row-prefix,eq:vsasm:col-prefix} either coincides with the corresponding prefix sum on $Z$ (after possibly a simple index shift), because all new entries in those prefixes are $0$; or lies in $\{0,1\}$ because it involves only newly inserted entries.
  Moreover, the index $(i,j)$ of the unique violated facet inequality is transported as follows.
  Under $\extBR$ the endpoint remains $(i,j)$, under $\extBL$ it becomes $(i,j+1)$, under $\extUR$ it becomes $(i+2,j)$, and under $\extUL$ it becomes $(i+2,j+1)$.
  Accordingly, the recursive definitions will invoke the smaller certificate with indices $(i,j)$, $(i,j-1)$, $(i-2,j)$, or $(i-2,j-1)$, respectively.
  Finally, it is straightforward to see that the equations in~\crefrange{eq:vsasm:row-sum}{eq:vsasm:dim:lastrow} are preserved under these extensions.

  We now give the recursive definitions of the certifying cores for $n \geq 9$.
  For $i \in [3,n-2]$ and $j \in [k-1 - \chi_{2 \mid i}]$, define
  \[
    L^{n,\mathrm H}_{i,j}=
    \begin{cases}
      \extBR\bigl(L^{n-2,\mathrm H}_{i,j}\bigr)     & \text{if } i \in [3,n-4] \text{ and } j \in [k-2-\chi_{2 \mid i}],\\
      \extBL\bigl(L^{n-2,\mathrm H}_{i,j-1}\bigr)   & \text{if } i \in [3,n-4] \text{ and } j = k-1-\chi_{2 \mid i},\\
      \extUR\bigl(L^{n-2,\mathrm H}_{i-2,j}\bigr)   & \text{if } i\in\{n-3,n-2\} \text{ and } j \in [k-2-\chi_{2 \mid i}],\\
      \extUL\bigl(L^{n-2,\mathrm H}_{i-2,j-1}\bigr) & \text{if } i\in\{n-3,n-2\} \text{ and } j = k-1-\chi_{2 \mid i}.
    \end{cases}
  \]
  For $i \in [4,n-3]$ and $j \in [2,k-1-\chi_{2 \nmid i}]$, define
  \[
    U^{n,\mathrm H}_{i,j}=
    \begin{cases}
          \vcenter{\vspace{3pt}\hbox{\begin{ytableau}
          0& 0& 0& 0\\
          0& 0& 0& 1\\
          0& 0& 0& 0\\
          0& 0& 1& 0\\
          1& 1&-1&-1\\
          0& 0& 1& 0\\
          0& 0& 0& 0\\
          0& 0& 0& 1\\
          0& 0& 0& 0
        \end{ytableau}}}                          & \text{if } n=9,i=5, \text{ and } j=2,\\[68pt]
      \extBR\bigl(U^{n-2,\mathrm H}_{i,j}\bigr)     & \text{if } i \in [4,n-5] \text{ and } j \in [2,k-2-\chi_{2 \nmid i}],\\
      \extBL\bigl(U^{n-2,\mathrm H}_{i,j-1}\bigr)   & \text{if } i \in [4,n-5] \text{ and } j = k-1-\chi_{2 \nmid i},\\
      \extUR\bigl(U^{n-2,\mathrm H}_{i-2,j}\bigr)  & \text{if } i\in\{n-4,n-3\}\setminus\{5\} \text{ and } j \in [2,k-2-\chi_{2 \nmid i}],\\
      \extUL\bigl(U^{n-2,\mathrm H}_{i-2,j-1}\bigr) & \text{if } i\in\{n-4,n-3\}\setminus\{5\} \text{ and } j = k-1-\chi_{2 \nmid i}.
    \end{cases}
  \]
  Note that we include $U^{9,\mathrm H}_{5,2}$ as an additional base case because it has no predecessor under the index shifts induced by the extension operators.
  For every $(i,j) \in \{(2,1)\} \cup \{(i,j) : i \in [2,n-4], j \in [2,k-\chi_{2 \mid i}]\}$, define
  \[
    L^{n,\mathrm V}_{i,j}=
    \begin{cases}
      \extBR\bigl(L^{n-2,\mathrm V}_{i,j}\bigr)     & \text{if } (i,j)=(2,1) \text{ or } (i \in [2,n-6] \text{ and } j \in [2,k-1-\chi_{2 \mid i}]),\\
      \extBL\bigl(L^{n-2,\mathrm V}_{i,j-1}\bigr)   & \text{if } i \in [2,n-6] \text{ and } j= k-\chi_{2 \mid i},\\
      \extUR\bigl(L^{n-2,\mathrm V}_{i-2,j}\bigr)   & \text{if } i \in \{n-5,n-4\} \text{ and } j \in [2,k-1-\chi_{2 \mid i}],\\
      \extUL\bigl(L^{n-2,\mathrm V}_{i-2,j-1}\bigr) & \text{if } i\in\{n-5,n-4\} \text{ and } j = k - \chi_{2 \mid i}.
    \end{cases}
  \]
  For every $(i,j) \in \{(n-2,1)\} \cup \{(i,j) : i \in [4,n-2], j \in [2,k-\chi_{2 \nmid i}]\}$, define
  \[
    U^{n,\mathrm V}_{i,j}=
    \begin{cases}
      \extBR\bigl(U^{n-2,\mathrm V}_{i,j}\bigr)     & \text{if } i \in [4,n-4] \text{ and } j \in [2,k-1-\chi_{2 \nmid i}],\\
      \extBL\bigl(U^{n-2,\mathrm V}_{i,j-1}\bigr)   & \text{if } i \in [4,n-4] \text{ and } j = k-\chi_{2 \nmid i},\\
      \extUR\bigl(U^{n-2,\mathrm V}_{i-2,j}\bigr)   & \text{if } (i,j) = (n-2,1) \text{ or } (i\in\{n-3,n-2\} \text{ and } j \in [2,k-1-\chi_{2 \nmid i}]),\\
      \extUL\bigl(U^{n-2,\mathrm V}_{i-2,j-1}\bigr) & \text{if } i\in\{n-3,n-2\} \text{ and } j = k-\chi_{2 \nmid i}.
    \end{cases}
  \]

  We prove by induction on $n$ that these definitions provide the desired certificates.
  The cores given explicitly for $n=7$ and $U^{9,\mathrm H}_{5,2}$ satisfy the claim by direct inspection.

  Now let $n\geq 9$, and assume that the claim holds for size $n-2$.
  By construction, every certifying core for size $n$ is obtained from a certifying core for size $n-2$ by one of the four extension operators.
  The applied extension embeds the corresponding $(n-2) \times (k-1)$ certificate as a submatrix and assigns the value $0$ to every other entry, except for a single new entry equal to $1$ at the corner cell adjacent to the embedded copy.
  In particular, every facet inequality for size $n$ comes from a row-prefix or column-prefix bound whose defining prefix lies either inside the embedded submatrix apart from new zero entries, or entirely in the newly inserted rows or columns.
  If the defining prefix lies inside the embedded copy, then its prefix sum coincides with the corresponding prefix sum of the smaller certificate after the evident index shift.
  Hence it is satisfied by the induction hypothesis, except for the single facet inequality violated by the smaller certificate.
  The extension operator was chosen so that this unique violated facet inequality is transported to the intended facet inequality for size $n$, via the endpoint shifts described above.
  If the defining prefix lies in one of the newly inserted rows or in the newly inserted column, then its prefix sum belongs to $\{0,1\}$, because all new entries are $0$ except for the single new entry equal to $1$.
  Therefore, all such facet inequalities are satisfied.
  Finally, the equations in~\crefrange{eq:vsasm:row-sum}{eq:vsasm:dim:lastrow} are preserved under the extension operators.

  Note that $L^{7,\mathrm V}_{2,1}$ violates not only the facet lower bound in~\cref{eq:vsasm:col-prefix} at $(2,1)$, but also the non-facet lower bound in~\cref{eq:vsasm:row-prefix} at $(2,1)$; similarly, $U^{7,\mathrm V}_{5,1}$ also violates the non-facet inequality in~\cref{eq:vsasm:row-prefix} at $(6,1)$.
  It is straightforward to verify that the extra violations remain non-facet under the index transports throughout the recursion.

  Therefore, for every facet inequality for size $n$, we have constructed a core that violates that facet inequality and satisfies every other facet inequality as well as the equations in~\crefrange{eq:vsasm:row-sum}{eq:vsasm:dim:lastrow}.
  This shows that no facet inequality is redundant.
  Together with the first part of the proof and the fact that no facet inequality is an implicit equation, the facet inequalities, together with the equations in~\crefrange{eq:vsasm:row-sum}{eq:vsasm:dim:lastrow}, form a minimal description of $P^\core_\VSASM$.
\end{proof}

\section{Vertically and horizontally symmetric ASMs (VHSASMs)}\label{sec:VHSASM}
In this class, we impose invariance under reflections across both the vertical and horizontal axes; hence the matrix is symmetric both left-to-right and top-to-bottom.
As a consequence, the matrix is also invariant under rotation by $\pi$.
The corresponding symmetry subgroup is $\mathcal{G} = \{\mathcal{I}, \mathcal{V}, \mathcal{H}, \mathcal{R}_{\pi}\}$.
Let $P_\HS$ denote the linear subspace of horizontally symmetric real matrices, i.e.,
\[
  P_\HS
  =
  \left\{
    X \in \R^{n \times n} : x_{i,j} = x_{n+1-i,j}\ \forall i,j \in [n]
  \right\}.
\]
Clearly, any VHSASM satisfies the VSASM constraints and also the symmetry constraints defining $P_\HS$; thus $P_\VHSASM \subseteq P_\VSASM \cap P_\HS$.
We note, however, that $P_\VSASM \cap P_\HS$ does not equal $P_\VHSASM$.
In fact, we show that $P_\VHSASM \subsetneq P_\VSASM \cap P_\HS$ for every odd $n \geq 5$.
It is straightforward to verify that the fractional matrix defined below is a vertex of $P_\VSASM\cap P_{\HS}$.
In this matrix, the first two and last two columns have entries $\sfrac{1}{2}$ in rows $(n+1) / 2 \pm1$, and $0$ in every other row.
All remaining columns follow the diamond pattern: for each $j \in [3,n-2]$, the first and last $|(n+1) / 2 - j|$ entries are $0$, and the entries alternate between $+1$ and $-1$, with the first entry equal to $+1$.
For example, we obtain the following matrix for $n = 7$
\[
  \begin{mymatrix}
    0 & 0 & 0 & 1 & 0 & 0 & 0 \\
    0 & 0 & 1 & -1 & 1 & 0 & 0 \\
    \sfrac{1}{2} & \sfrac{1}{2} & -1 & 1 & -1 & \sfrac{1}{2} & \sfrac{1}{2} \\
    0 & 0 & 1 & -1 & 1 & 0 & 0 \\
    \sfrac{1}{2} & \sfrac{1}{2} & -1 & 1 & -1 & \sfrac{1}{2} & \sfrac{1}{2} \\
    0 & 0 & 1 & -1 & 1 & 0 & 0 \\
    0 & 0 & 0 & 1 & 0 & 0 & 0
  \end{mymatrix}.
\]

\Cref{lem:vsasm:noEvenN} immediately implies the following.
\begin{lemma}\label{lem:vhsasm:noEvenN}
  There is no $n \times n$ VHSASM if $n$ is even.
\end{lemma}

Thus, $P_\VHSASM = \emptyset$ for even $n$; hence, we assume that $n$ is odd in the rest of the section.
Applying \cref{lem:vsasm:middleColAlternate} to a VHSASM and its transpose, we obtain the following.

\begin{lemma}\label{lem:vhsasm:middleColAndRowAlternate}
  Let $n \geq 1$ be odd and set $k = \floor*{n/2}$.
  For every VHSASM $X \in \{0,\pm1\}^{n \times n}$, we have $x_{i,k+1} = (-1)^{i+1}$ for every $i \in [n]$ and $x_{k+1,j} = (-1)^{j+1}$ for every $j \in [n]$.
\end{lemma}

\paragraph{Core and assembly map.}
Assume $n$ is odd and let $k = \floor*{n/2}$.
Let the \emph{core} of a VHSASM be its upper-left $k \times k$ block, i.e.,
\[
  C = [k] \times [k]
\]
is the set of \emph{core positions}.
Define the affine map $\varphi : \R^C \to \R^{n\times n}$ by
\[
  \varphi(Y)_{i,j} =
  \begin{cases}
    y_{i,j}        & \text{if } i \in [k], j \in [k],\\
    (-1)^{i+1}     & \text{if } j = k+1,\\
    (-1)^{j+1}     & \text{if } i = k+1,\\
    y_{i,n+1-j}    & \text{if } i \in [k], j \in [k+2, n],\\
    y_{n+1-i,j}    & \text{if } i \in [k+2, n], j \in [k],\\
    y_{n+1-i,n+1-j}& \text{if } i \in [k+2, n], j \in [k+2, n]
  \end{cases}
\]
for $Y \in \R^C$ and $i,j \in [n]$.
Note that the second and third cases both apply for $(i,j) = (k+1,k+1)$; however, they give the same value $(-1)^{k+2}$, so $\varphi(Y)$ is well defined.
By definition, $\varphi$ places the core $Y$ in the upper-left $k\times k$ block, fixes the middle row and the middle column to the alternating pattern from \cref{lem:vhsasm:middleColAndRowAlternate}, and completes the matrix by reflecting $Y$ across the middle column and across the middle row, thereby filling all four quadrants by vertical and horizontal reflections.
Clearly, the map $\varphi$ is an assembly map: it is affine, satisfies $\pi_C(\varphi(Y)) = Y$ for every $Y \in \R^C$, and $\varphi(\pi_C(X)) = X$ for every $X \in \VHSASM(n)$, because $X$ is determined by its core together with the prescribed middle row and column and the imposed vertical and horizontal symmetries, where $\pi_C$ is the coordinate-wise projection onto $C$.

\bigskip
We now describe the core polytope of VHSASMs.
\begin{theorem}\label{thm:vhsasm:corepolytope}
  Let $n \geq 1$ be odd, and set $k = \floor*{n/2}$.
  Then the core polytope $P^\core_\VHSASM \subseteq \R^C$ of $n \times n$ VHSASMs is described by the following system.
  \begin{align}
                          y&_{i,j} \in \R                        &\forall i,j \in [k],\label{eq:vhsasm:real}\\
    0 \leq \sum_{j'=1}^{j} y&_{i,j'} \leq 1                       &\forall i \in [k], j \in [k-1],\label{eq:vhsasm:row-prefix}\\
    0 \leq \sum_{i'=1}^{i} y&_{i',j} \leq 1                       &\forall i \in [k-1], j \in [k],\label{eq:vhsasm:col-prefix}\\
    \sum_{j=1}^{k}         y&_{i,j} = \chi_{2 \mid i}              &\forall i \in [k],\label{eq:vhsasm:row-sum}\\
    \sum_{i=1}^{k}         y&_{i,j} = \chi_{2 \mid j}              &\forall j \in [k].\label{eq:vhsasm:col-sum}
  \end{align}
\end{theorem}
\begin{proof}
  We show that the integer solutions to the system~\crefrange{eq:vhsasm:real}{eq:vhsasm:col-sum} are exactly the cores of VHSASMs, and then we argue that the system defines an integral polytope.

  \medskip
  First, let $X$ be an $n\times n$ VHSASM, and let $Y$ be its core, that is, $y_{i,j} = x_{i,j}$ for every $i,j \in [k]$.
  Since every VHSASM is in particular a VSASM, we obtain that $Y$ is the first $k$ rows of the core of a VSASM; thus the constraints in~\cref{eq:vsasm:real,eq:vsasm:row-prefix,eq:vsasm:col-prefix,eq:vsasm:row-sum} of \cref{thm:vsasm:corepolytope} directly imply the corresponding constraints in~\cref{eq:vhsasm:real,eq:vhsasm:row-prefix,eq:vhsasm:col-prefix,eq:vhsasm:row-sum} for $Y$.
  Since the transpose of every VHSASM is again a VSASM, we obtain that $Y^\top$ is the first $k$ rows of the core of a VSASM, and thus the equations in~\cref{eq:vsasm:row-sum} directly imply the equations in~\cref{eq:vhsasm:col-sum} for $Y$.
  Thus the cores of VHSASMs satisfy the system~\crefrange{eq:vhsasm:real}{eq:vhsasm:col-sum}.

  \medskip
  Second, we show that every integer solution to the system~\crefrange{eq:vhsasm:real}{eq:vhsasm:col-sum} is the core of a VHSASM\@.
  Let $Y \in \Z^C$ satisfy the system~\crefrange{eq:vhsasm:real}{eq:vhsasm:col-sum}, and set $X = \varphi(Y)$.
  By construction, $X$ is vertically and horizontally symmetric, its middle column is given by $x_{i,k+1} = (-1)^{i+1}$ for $i \in [n]$, its middle row is given by $x_{k+1,j} = (-1)^{j+1}$  for $j \in [n]$, and its core is $Y$.
  We verify that $X$ satisfies the ASM constraints given in \cref{thm:ASMpolytope}.
  For $i \in [k]$, we obtain
  \[
    \sum_{j=1}^{n} x_{i,j}
    = \sum_{j=1}^{k} y_{i,j} + x_{i,k+1} + \sum_{j=1}^{k} y_{i,j}
    = 2\chi_{2 \mid i} + (-1)^{i+1}
    = 1
  \]
  for the sum of row $i$, by~\cref{eq:vhsasm:row-sum}.
  By horizontal symmetry, the same holds for row $n+1-i$.
  In the middle row $i = k+1$, the entries alternate between $+1$ and $-1$, with the first entry equal to $+1$.
  Since $n$ is odd, this row also sums to $1$.
  Furthermore, all row-prefix sums within the first $k$ rows are in $\{0,1\}$ by the row-prefix bounds in~\cref{eq:vhsasm:row-prefix}, the row-sum equations in~\cref{eq:vhsasm:row-sum}, and \cref{lem:vsasm:symmetricHalfExtends}; the same holds for every row by horizontal symmetry.
  Repeating the argument for columns, using the column-prefix bounds in~\cref{eq:vhsasm:col-prefix}, the column-sum equations in~\cref{eq:vhsasm:col-sum}, and \cref{lem:vsasm:symmetricHalfExtends}, shows that $X$ satisfies the ASM constraints on column-prefix sums and column sums.
  Thus $X \in \Z^{n \times n}$ satisfies all ASM constraints in \cref{thm:ASMpolytope}, hence $X$ is an ASM\@.
  Together with the construction, this shows that $X$ is a VHSASM\@.
  We conclude that the integer solutions to the system~\crefrange{eq:vhsasm:real}{eq:vhsasm:col-sum} are exactly the cores of VHSASMs.

  \medskip
  It remains to prove that the system~\crefrange{eq:vhsasm:real}{eq:vhsasm:col-sum} defines an integral polytope.
  The constraints in~\cref{eq:vhsasm:row-prefix,eq:vhsasm:col-prefix} impose lower and upper bounds on the sums of the entries within prefixes of each row and column of the $k\times k$ core, while the equations in~\cref{eq:vhsasm:row-sum,eq:vhsasm:col-sum} fix the row and column sums to $0$ or $1$.
  Therefore, the cores of VHSASMs are prefix\nobreakdash-bounded matrices with prescribed (integer) row and column sums, and the polytope of such prefix\nobreakdash-bounded matrices is known to be described by the system above~\cite[Corollary~5.2]{borsik2026prefix}.
  Alternatively, the result also directly follows from \cref{thm:laminarSystem} applied to the laminar families of row and column prefixes.
\end{proof}

From \cref{thm:xasm:assembly,thm:vhsasm:corepolytope}, we obtain the following description of the polytope $P_\VHSASM$ of VHSASMs.
\begin{theorem}\label{thm:vhsasm:coreDescr}
  Let $n \geq 1$ be odd, and let $k = \floor*{n/2}$ and $\widehat P^\core_\VHSASM = \{X \in \R^{n \times n} : \pi_C(X) \in P^\core_\VHSASM\}$.
  Then
  \[
    P_\VHSASM
    = \widehat P^\core_\VHSASM \cap P_\VS \cap P_\HS \cap \left\{X \in \R^{n \times n} : x_{i,k+1} = x_{k+1,i} = (-1)^{i+1}\ \forall i \in [n]\right\}.
  \]
\end{theorem}
\begin{proof}
  By \cref{thm:xasm:assembly,thm:vhsasm:corepolytope}, it suffices to prove that
  $
    \varphi(\R^C)
    = P_\VS \cap P_\HS \cap \{X \in \R^{n \times n} : x_{i,k+1} = x_{k+1,i} = (-1)^{i+1}\ \forall i \in [n]\}.
  $
  Let $P$ denote the right-hand side.
  By definition, $\varphi$ inserts its argument as the upper-left $k \times k$ block, fixes the $i^\text{th}$ entries of the middle row and the middle column to $(-1)^{i+1}$ for every $i$, and fills the remaining entries by vertical and horizontal reflection.
  Thus, $\varphi(Y) \in P$ for every~$Y \in \R^C$.

  Conversely, take any $X \in P$.
  The equations defining $P_\VS \cap P_\HS$ together with the prescribed middle row and column imply that $X$ is completely determined by its upper-left $k \times k$ block, that is, by its core $\pi_C(X)$.
  Therefore, $\varphi(\pi_C(X)) = X$, and hence $X \in \varphi(\R^C)$.
  This shows $\varphi(\R^C) = P$, and the statement follows.
\end{proof}

\begin{theorem}\label{thm:vhsasm:descr}
  Let $n \geq 1$ be arbitrary, and set $k = \floor*{n/2}$.
  Then
  \[
    P_\VHSASM = P_\ASM \cap P_\VS \cap P_\HS \cap \left\{X \in \R^{n \times n} : x_{i,k+1} = x_{k+1,i} = (-1)^{i+1}\ \forall i \in [n]\right\}.
  \]
\end{theorem}
\begin{proof}
  For odd $n$, we obtain the statement by straightforward transformations of the system given in \cref{thm:vhsasm:coreDescr}, in complete analogy with the passage from \cref{thm:vsasm:coreDescr} to \cref{thm:vsasm:descr}.

  For even $n$, the polytope $P_\VHSASM$ is empty; thus we need to show that the right-hand side is empty as well.
  Notice that
  $
  \{X \in \R^{n \times n} : x_{i,k+1} = x_{k+1,i} = (-1)^{i+1}\ \forall i \in [n]\}
  $
  forces the entries of column $k+1$ to alternate between $+1$ and $-1$, with the first entry equal to $+1$.
  Hence $\sum_{i=1}^n x_{i,k+1} = 0$ because $n$ is even.
  On the other hand, if $X \in P_\ASM$, then by~\cref{eq:asm:col-sum} we have $\sum_{i=1}^n x_{i,k+1} = 1$, a contradiction.
  Thus the right-hand side is empty for even $n$.
\end{proof}

\begin{theorem}\label{thm:vhsasm:dim}
  For every odd $n \geq 5$, the dimension of $P_\VHSASM$ is $\frac{(n-5)^2}{4}$.
\end{theorem}
\begin{proof}
  Let $n\geq 5$ be odd and set $k=\floor*{n/2}$.
  It suffices to prove that $\dim(P^\core_\VHSASM)=\frac{(n-5)^2}{4}$, because the assembly map $\varphi$ restricts to an affine isomorphism between $P^\core_\VHSASM$ and $P_\VHSASM$, which preserves dimension.
  First, we give an upper bound.
  By \cref{thm:vhsasm:corepolytope}, the core polytope $P^\core_\VHSASM\subseteq\R^C$ is described by the system~\crefrange{eq:vhsasm:real}{eq:vhsasm:col-sum}.
  The case $i=1$ of the inequalities in~\cref{eq:vhsasm:col-prefix} gives $0\leq y_{1,j}\leq 1$ for every $j\in[k]$.
  Since the case $i=1$ of the equations in~\cref{eq:vhsasm:row-sum} gives $\sum_{j=1}^k y_{1,j}=0$, it follows that
  \begin{align}
    y&_{1,j}=0 &\forall j\in[k].\label{eq:vhsasm:dim:firstrow}
  \intertext{Likewise, the case $j=1$ of the inequalities in~\cref{eq:vhsasm:row-prefix} gives $0\leq y_{i,1}\leq 1$ for every $i\in[k]$, and the case $j=1$ of the equations in~\cref{eq:vhsasm:col-sum} gives $\sum_{i=1}^k y_{i,1}=0$, hence}
    y&_{i,1}=0 &\forall i\in[k].\label{eq:vhsasm:dim:firstcol}
  \end{align}
  Using the equations in~\cref{eq:vhsasm:dim:firstrow,eq:vhsasm:dim:firstcol}, we eliminate the variables in the first row and first column as follows: for each $i\in[2,k]$, we replace the $i^\text{th}$ equation of~\cref{eq:vhsasm:row-sum} by the difference of that equation and the equation $y_{i,1}=0$; for each $j\in[2,k]$, we replace the $j^\text{th}$ equation of~\cref{eq:vhsasm:col-sum} by the difference of that equation and the equation $y_{1,j}=0$.
  These elementary row operations do not change the solution set, and thus do not change the rank of the equation system.
  After this replacement, the $i=1$ equation in~\cref{eq:vhsasm:row-sum} and the $j=1$ equation in~\cref{eq:vhsasm:col-sum} become redundant, and none of the remaining modified equations involves a variable from the first row or the first column.

  The remaining row- and column-sum equations are those for $i,j\in[2,k]$.
  By \cref{lem:prelim:rowcol-rank} applied for $m=n=k-1$, they contain $2(k-1)-1=2k-3$ independent equations.
  Moreover, these $2k-3$ equations involve only variables $y_{i,j}$ with $i,j\in[2,k]$ after the elimination, and hence have disjoint support from~\cref{eq:vhsasm:dim:firstrow,eq:vhsasm:dim:firstcol}.
  Therefore, the two families of equations are linearly independent, and so we obtain $(2k-3)+(2k-1)=4k-4$ independent equations in total.
  Since $|C| = k^2$, these define an affine subspace of dimension
  $
    k^2-(4k-4)=\frac{(n-5)^2}{4},
  $
  which contains $P^\core_\VHSASM$ and hence gives the bound $\dim(P^\core_\VHSASM)\leq \frac{(n-5)^2}{4}$.

  \medskip
  Second, we construct $\frac{(n-5)^2}{4} + 1 = (k-2)^2 + 1$ affinely independent cores in $P^\core_\VHSASM$.
  Let $\ol Y$ denote the average of the cores of all VHSASMs.
  We claim that $\ol Y$ reaches neither the lower nor the upper bound in~\cref{eq:vhsasm:row-prefix} for any $i \in [2,k]$ and $j \in [2,k-1]$.
  It suffices to show that, for each such $i,j$, there exists a core for which the sum of the first $j$ entries in row $i$ is $1$, and there exists a core for which this sum is $0$.
  To see this, place a permutation matrix into the submatrix formed by the rows and columns with even index of the $k\times k$ core in such a way that $y_{i,2}=1$ for some fixed even $i \in [2,k]$; fill the remaining core entries with $0$.
  This yields an integer solution to the system~\crefrange{eq:vhsasm:real}{eq:vhsasm:col-sum}, and hence the core of a VHSASM\@.
  For the fixed index $i$, the sum of the first $j$ entries in row $i$ equals $1$ for every $j \in [2,k-1]$, while for every odd row index the same sum equals $0$ for every $j \in [2,k-1]$.
  Next, set $y_{r,k}=(-1)^r$ for every $r \in [2,k]$, and place a permutation matrix into the submatrix formed by the rows with odd index in $[2,k]$ and the columns with even index in $[2,k-1]$ in such a way that $y_{i,2}=1$ for some fixed odd $i \in [3,k]$; fill the remaining core entries with $0$.
  Again we obtain an integer solution to the system~\crefrange{eq:vhsasm:real}{eq:vhsasm:col-sum}.
  For the fixed index $i$, the sum of the first $j$ entries in row $i$ equals $1$ for every $j \in [2,k-1]$, while for every even row index the same sum equals $0$ for every $j \in [2,k-1]$.
  Thus, for every $i \in [2,k]$ and $j \in [2,k-1]$, the row-prefix sum in~\cref{eq:vhsasm:row-prefix} attains both values $0$ and $1$ on VHSASM cores, and therefore $\ol Y$ does not reach equality in~\cref{eq:vhsasm:row-prefix} for any such $i,j$.
  An analogous argument applied to transposed cores shows that $\ol Y$ reaches neither bound in~\cref{eq:vhsasm:col-prefix} for any $i\in[2,k-1]$ and $j\in[2,k]$.

  For each $i,j\in[2,k-1]$, define
  $
    \ol Y^{i,j}
    = \ol Y
    + \varepsilon \chi_{i,j}
    - \varepsilon \chi_{i,k}
    - \varepsilon \chi_{k,j}
    + \varepsilon \chi_{k,k},
  $
  where $\varepsilon$ is a small positive constant.
  By definition, $\ol Y^{i,j}$ satisfies the equations in~\cref{eq:vhsasm:row-sum,eq:vhsasm:col-sum}.
  By the claim above, choosing $\varepsilon>0$ small enough ensures that $\ol Y^{i,j}$ violates no inequality in~\cref{eq:vhsasm:row-prefix,eq:vhsasm:col-prefix}, hence $\ol Y^{i,j}\in P^\core_\VHSASM$.
  The cores $\ol Y$ and $\ol Y^{i,j}$ for $i,j\in[2,k-1]$ are affinely independent: only the difference $\ol Y^{i,j}-\ol Y$ has a non-zero entry at $(i,j)$, so the cores $\{\ol Y^{i,j}-\ol Y : i,j\in[2,k-1]\}$ are linearly independent.
  Therefore, the dimension of $P^\core_\VHSASM$ is at least $(k-2)^2=\frac{(n-5)^2}{4}$.

  Combining the lower and upper bounds yields $\dim(P^\core_\VHSASM)=\dim(P_\VHSASM)=\frac{(n-5)^2}{4}$.
\end{proof}

\begin{theorem}\label{thm:vhsasm:facets}
  Let $n \geq 9$ be odd, and set $k=\floor*{n/2}$.
  The facets of $P^\core_\VHSASM$ are given by tightening the lower bound in~\cref{eq:vhsasm:row-prefix} to equality for
  $
    (i,j)\in
    \{(2,2)\}
    \cup
    \{(i,j) : i \in [3,k-1], j \in [2,k-1-\chi_{2 \mid i}]\}
    \cup
    \bigl(\{k\} \times \{j\in[3,k-1] : 2 \nmid j\}\bigr)
  $,
  and the upper bound for
  $
    (i,j)\in
    \{(i,j) : i \in [4,k-1], j \in [4,k-1-\chi_{2 \nmid i}]\}
    \cup
    \bigl(\{k\} \times \{j\in[4,k-1] : 2 \mid j\}\bigr)
  $;
  and by tightening the lower bound in~\cref{eq:vhsasm:col-prefix} to equality for
  $
    (i,j)\in
    \{(i,j): j\in[3,k-1], i\in[2,k-1-\chi_{2 \mid j}]\}
    \cup
    \bigl(\{i \in [3,k-2] : 2 \nmid i\} \times \{k\}\bigr)
  $,
  and the upper bound for
  $
    (i,j)\in
    \{(i,j): j\in[4,k-1], i\in[4,k-1-\chi_{2 \nmid j}]\}
    \cup
    \bigl(\{i \in [4,k-2] : 2 \mid i\} \times \{k\}\bigr)
  $.
  In particular, the number of facets of $P^\core_\VHSASM$ is $n^2-15n+60$.
\end{theorem}
\begin{proof}
  The facets are obtained by tightening a single inequality in one of~\cref{eq:vhsasm:row-prefix,eq:vhsasm:col-prefix} to equality for the index pairs listed in the statement of the theorem.
  We call the instances of the lower bounds in~\cref{eq:vhsasm:row-prefix} that are tightened to equality the \emph{horizontal facet lower bounds}, and we define the \emph{horizontal facet upper bounds} analogously.
  Likewise, we call the instances of the lower bounds in~\cref{eq:vhsasm:col-prefix} that are tightened to equality the \emph{vertical facet lower bounds}, and we define the \emph{vertical facet upper bounds} analogously.
  We refer to the union of these four families as the \emph{facet inequalities}.

  We proceed in two steps.
  First, we show that the facet inequalities together with the equations in~\crefrange{eq:vhsasm:row-sum}{eq:vhsasm:dim:firstcol} imply every inequality in~\cref{eq:vhsasm:row-prefix,eq:vhsasm:col-prefix}.
  Then, for every facet inequality, we construct a core of an $n\times n$ matrix violating that facet inequality and no other, thereby proving that no facet inequality is redundant.
  The core $\ol Y$ constructed in the second step of the proof of \cref{thm:vhsasm:dim} shows that no facet inequality is an implicit equation; thus the two steps together imply that the facet inequalities form a minimal system that, extended with the equations in~\crefrange{eq:vhsasm:row-sum}{eq:vhsasm:dim:firstcol}, describes the convex hull of the cores of VHSASMs, which proves the theorem.

  \medskip
  Now we prove that the facet inequalities together with the equations in~\crefrange{eq:vhsasm:row-sum}{eq:vhsasm:dim:firstcol} imply every inequality in~\cref{eq:vhsasm:row-prefix,eq:vhsasm:col-prefix}.
  Clearly, we need to treat only those inequalities in~\cref{eq:vhsasm:row-prefix,eq:vhsasm:col-prefix} that are \emph{non-facet} inequalities, namely, the lower bounds in~\cref{eq:vhsasm:row-prefix} for
  \begin{align*}
    (i,j)\in\;&
                (\{1\}\times [k-1])
                \ \cup\
                (\{2\}\times(\{1\}\cup[3,k-1]))
                \ \cup\
                ([3,k]\times\{1\})\\
              &\cup\ (\{i\in[4,k-1]:2 \mid i\}\times\{k-1\})
                \ \cup\ (\{k\} \times \{j\in[2,k-1]:2 \mid j\}),
  \end{align*}
  and the upper bounds for
  \begin{align*}
    (i,j)\in\;&
                ([k]\times[3])
                \ \cup\ ([3]\times[4,k-1])
                \ \cup\ (\{i\in[5,k-1]:2 \nmid i\}\times\{k-1\})\\
              &\cup\ (\{k\}\times\{j\in[5,k-1]:2 \nmid j\});
  \end{align*}
  and the lower bounds in~\cref{eq:vhsasm:col-prefix} for
  \begin{align*}
    (i,j)\in\;&
                ([k-1]\times[2])
                \ \cup\ (\{1\}\times[3,k])
                \ \cup\ (\{k-1\}\times\{j\in[4,k-1]:2 \mid j\})\\
              &\cup\ \Bigl(\bigl(\{i\in[2,k-2]:2 \mid i\} \cup \{k-1\}\bigr)\times\{k\}\Bigr),
  \end{align*}
  and the upper bounds for
  \begin{align*}
    (i,j)\in\;&
                ([k-1]\times[3])
                \ \cup\ ([3]\times[4,k])
                \ \cup\ (\{k-1\}\times\{j\in[5,k-1]:2 \nmid j\})\\
              &\cup\ \Bigl(\bigl(\{i \in [5,k-2] : 2 \nmid i\} \cup \{k-1\}\bigr)\times\{k\}\Bigr).
  \end{align*}

  \medskip
  We start with the non-facet inequalities in~\cref{eq:vhsasm:row-prefix}.
  The equations in~\cref{eq:vhsasm:dim:firstrow,eq:vhsasm:dim:firstcol} give $y_{1,j}=0$ for all $j\in[k]$ and $y_{i,1}=0$ for all $i\in[k]$; hence the bounds in~\cref{eq:vhsasm:row-prefix} hold whenever $i=1$ or $j=1$.

  Next, consider $i=2$.
  The horizontal facet lower bound at $(2,2)$ gives $0\leq y_{2,1}+y_{2,2}=y_{2,2}$.
  Moreover, for every $j\in[3,k-1]$, the vertical facet lower bound at $(2,j)$ yields $0\leq y_{1,j}+y_{2,j}=y_{2,j}$.
  Hence $y_{2,j}\geq 0$ for all $j\in[2,k-1]$, and therefore the non-facet lower bounds in~\cref{eq:vhsasm:row-prefix} at $(2,j)$ with $j\in\{1\}\cup[3,k-1]$ follow.

  To obtain the corresponding non-facet upper bounds for $i=2$, we first note that
  $
    \sum_{j'=1}^{k} y_{2,j'}=\chi_{2 \mid 2}=1
  $
  by~\cref{eq:vhsasm:row-sum}.
  Furthermore, using the horizontal facet lower bound at $(3,k-1)$ and the equation in~\cref{eq:vhsasm:row-sum} for $i=3$, we obtain
  $
    y_{3,k} = -\sum_{j'=1}^{k-1}y_{3,j'} \leq 0.
  $
  If $k\geq 5$, then the vertical facet lower bound at $(3,k)$ gives $0\leq y_{1,k}+y_{2,k}+y_{3,k}=y_{2,k}+y_{3,k}$, hence $y_{2,k}\geq 0$.
  For $n=9$, the remaining non-facet bounds can be checked directly.
  Thus $y_{2,j}\geq 0$ for all $j\in[k]$ and $\sum_{j'=1}^k y_{2,j'}=1$, which implies
  $
  0\leq \sum_{j'=1}^{j} y_{2,j'} \leq 1
  $
  for every $j\in[k-1]$, and hence all non-facet upper bounds in~\cref{eq:vhsasm:row-prefix} at $(2,j)$ follow.

  Next, we derive the non-facet upper bounds in~\cref{eq:vhsasm:row-prefix} for $i=3$.
  Fix $j \in [2, k-1]$.
  For every $j'\in[j+1,k-1]$, the vertical facet lower bound at $(3,j')$ yields $0\leq y_{1,j'}+y_{2,j'}+y_{3,j'}=y_{2,j'}+y_{3,j'}$, hence $-y_{3,j'}\leq y_{2,j'}$.
  If $k\geq 5$, the same inequality also holds for $j'=k$ by the vertical facet lower bound at $(3,k)$.
  Using the equations in~\cref{eq:vhsasm:row-sum} for $i=2,3$, we obtain
  $
    \sum_{j'=1}^{j} y_{3,j'}
    = -\sum_{j'=j+1}^{k} y_{3,j'}
    \leq \sum_{j'=j+1}^{k} y_{2,j'}
    = 1-\sum_{j'=1}^{j} y_{2,j'}
    \leq 1
  $,
  proving the required non-facet row-prefix upper bounds for $i=3$.

  Next let $i\in[5,k-1]$ be odd.
  The corresponding equation in~\cref{eq:vhsasm:row-sum} gives $\sum_{j'=1}^k y_{i,j'}=0$, so $\sum_{j'=1}^{k-1} y_{i,j'}=-y_{i,k}$.
  Hence it suffices to show $y_{i,k}\geq -1$.
  If $i\leq k-2$, then the vertical facet upper bound at $(i-1,k)$ and the vertical facet lower bound at $(i,k)$ give
  $
  \sum_{r=1}^{i-1} y_{r,k}\leq 1
  $
  and
  $
    \sum_{r=1}^{i} y_{r,k}\geq 0
  $,
  and subtracting yields $y_{i,k}\geq -1$.
  If $k$ is even and $i=k-1$, then the horizontal facet lower bound at $(k,k-1)$ implies $y_{k,k}=1-\sum_{j'=1}^{k-1}y_{k,j'}\leq 1$ by~\cref{eq:vhsasm:row-sum}, hence $\sum_{r=1}^{k-1} y_{r,k}=1-y_{k,k}\geq 0$ by~\cref{eq:vhsasm:col-sum}.
  Together with the vertical facet upper bound at $(k-2,k)$ this gives $y_{k-1,k}\geq -1$.
  Therefore in all cases $\sum_{j'=1}^{k-1} y_{i,j'}=-y_{i,k}\leq 1$.

  Now let $i\in[4,k-1]$ be even.
  We show the non-facet lower bound in~\cref{eq:vhsasm:row-prefix} at $(i,k-1)$.
  If $i\leq k-2$, then the vertical facet lower bound at $(i-1,k)$ and the vertical facet upper bound at $(i,k)$ imply $y_{i,k}\leq 1$, hence
  $
  \sum_{j'=1}^{k-1} y_{i,j'} = \chi_{2 \mid i}-y_{i,k} = 1-y_{i,k}\geq 0
  $
  by~\cref{eq:vhsasm:row-sum}.
  If $k$ is odd and $i=k-1$, we use $\sum_{j'=1}^{k}y_{k,j'}=0$ from~\cref{eq:vhsasm:row-sum} and the horizontal facet upper bound at $(k,k-1)$ to get $y_{k,k}\geq -1$, hence $\sum_{i'=1}^{k-1}y_{i',k}=-y_{k,k}\leq 1$; together with the vertical facet lower bound at $(k-2,k)$ this gives $y_{k-1,k}\leq 1$ and thus $\sum_{j'=1}^{k-1} y_{i,j'} \geq 0$.

  For the remaining non-facet upper bounds in~\cref{eq:vhsasm:row-prefix} with $j\in\{2,3\}$, note that $y_{i,1}=0$ for all $i$ by~\cref{eq:vhsasm:dim:firstcol}.
  Thus the horizontal facet lower bound at $(i,2)$ yields $y_{i,2}\geq 0$ for all $i\in[3,k]$, and together with the already established $y_{2,2}\geq 0$ we obtain $y_{i,2}\geq 0$ for all $i\in[k]$.
  Using the $j=2$ equation in~\cref{eq:vhsasm:col-sum}, we have $\sum_{i=1}^k y_{i,2}=1$, hence $y_{i,2}\leq 1$ for all $i$, proving the upper bound at $(i,2)$.
  Similarly, the horizontal facet lower bound at $(i,3)$ gives $y_{i,2}+y_{i,3}\geq 0$ for all $i\in[3,k]$, and with the already established $y_{2,2}+y_{2,3}\geq 0$ we get $y_{i,2}+y_{i,3}\geq 0$ for all $i\in[k]$.
  Summing over $i$ and using the equations in~\cref{eq:vhsasm:col-sum} for $j=2,3$ yields
  $
    \sum_{i=1}^k (y_{i,2}+y_{i,3}) = \sum_{i=1}^k y_{i,2} + \sum_{i=1}^k y_{i,3} = 1
  $,
  so each summand satisfies $y_{i,2}+y_{i,3}\leq 1$, which proves the upper bound at $(i,3)$.

  Finally, we treat the case $i=k$.
  For every odd $j'\in[3,k-1]$, the vertical facet lower bound at $(k-1,j')$ gives $\sum_{i'=1}^{k-1}y_{i',j'}\geq 0$, and since $\sum_{i'=1}^{k}y_{i',j'}=\chi_{2 \mid j'}=0$ by~\cref{eq:vhsasm:col-sum}, we obtain
  $
    y_{k,j'} = -\sum_{i'=1}^{k-1}y_{i',j'}\leq 0
  $.
  Likewise, for every even $j'\in[4,k-1]$, the vertical facet upper bound at $(k-1,j')$ gives $\sum_{i'=1}^{k-1}y_{i',j'}\leq 1$, and
  $\sum_{i'=1}^{k}y_{i',j'}=1$, hence
  $
    y_{k,j'} = 1-\sum_{i'=1}^{k-1}y_{i',j'}\geq 0
  $.
  Now let $j\in[2,k-1]$ be even.
  If $j\leq k-2$, then $j+1$ is odd and the horizontal facet lower bound at $(k,j+1)$ gives
  $\sum_{j'=1}^{j+1}y_{k,j'}\geq 0$, so
  $
    \sum_{j'=1}^{j}y_{k,j'}
    = \sum_{j'=1}^{j+1}y_{k,j'} - y_{k,j+1}
    \geq 0
  $,
  because $y_{k,j+1}\leq 0$ as shown above.
  If $k$ is odd and $j=k-1$, then $j$ is even and the vertical facet upper bound at $(k-1,k-1)$ together with $\sum_{i'=1}^{k}y_{i',k-1}=1$ gives $y_{k,k-1}\geq 0$, hence $\sum_{j'=1}^{k-1}y_{k,j'}\geq \sum_{j'=1}^{k-2}y_{k,j'}\geq 0$.
  This proves the non-facet lower bounds in~\cref{eq:vhsasm:row-prefix} at $(k,j)$ with $j$ even.

  For the non-facet upper bounds in~\cref{eq:vhsasm:row-prefix} in row $k$ at odd $j\in[5,k-1]$, we use that $j+1$ is even.
  If $j\leq k-2$, then the horizontal facet upper bound at $(k,j+1)$ gives
  $
    \sum_{j'=1}^{j+1} y_{k,j'}\leq 1
  $,
  and since $y_{k,j+1}\geq 0$ as above, we obtain
  $
    \sum_{j'=1}^{j} y_{k,j'}
    = \sum_{j'=1}^{j+1} y_{k,j'} - y_{k,j+1}
    \leq 1
  $.
  If $k$ is even and $j=k-1$, then by~\cref{eq:vhsasm:col-sum}, we have $\sum_{i'=1}^{k} y_{i',k}=1$, and the vertical facet upper bound at $(k-1,k)$ yields $\sum_{i'=1}^{k-1} y_{i',k}\leq 1$, hence
  $
    y_{k,k}=1-\sum_{i'=1}^{k-1} y_{i',k}\geq 0
  $.
  Together with $\sum_{j'=1}^{k} y_{k,j'}=1$ from~\cref{eq:vhsasm:row-sum}, this implies
  $
    \sum_{j'=1}^{k-1} y_{k,j'} = 1-y_{k,k}\leq 1
  $, which completes the derivation of all bounds in~\cref{eq:vhsasm:row-prefix}.

  The non-facet inequalities in~\cref{eq:vhsasm:col-prefix} follow analogously by applying the argument above with rows and columns interchanged.

  \medskip
  It remains to show that no facet inequality is redundant.
  For every horizontal facet lower bound, we construct a core $L^{n,\mathrm H}_{i,j}\in\R^{C}$ that violates that facet inequality and satisfies every other facet inequality as well as the equations defining $P^\core_\VHSASM$.
  Similarly, we construct certifying cores $U^{n,\mathrm H}_{i,j}$ for the horizontal facet upper bounds.
  Since the vertical certificates $L^{n,\mathrm V}_{i,j}$ and $U^{n,\mathrm V}_{i,j}$ can be obtained by transposition, i.e., $L^{n,\mathrm V}_{i,j}=\bigl(L^{n,\mathrm H}_{j,i}\bigr)^\top$ and $U^{n,\mathrm V}_{i,j}=\bigl(U^{n,\mathrm H}_{j,i}\bigr)^\top$, we focus on the horizontal certificates.

  In order to build the certifying cores of $n\times n$ VHSASMs from the certifying cores for smaller sizes, we define the four \emph{extension operators}
  \begin{center}
    \begin{tikzpicture}[scale=.8, baseline=(current bounding box.center)]
      \node at (-1.4,2.5) {$\extUL(Z)\ =$};
      \node[draw=none] at (5.15,2.5) {$,$};
      \draw[fill=cyan!15] (2,0) rectangle (5,3);
      \draw (0,0) rectangle (5,5);

      \node at (0.5,4.5) {$0$};
      \node at (0.5,3.5) {$0$};
      \ThreeDotsAt{(0.5,1.5)}{.21213203}{90}
      \node at (0.5,2.5) {$0$};
      \node at (0.5,0.5) {$0$};

      \node at (1.5,4.5) {$0$};
      \node at (1.5,3.5) {$1$};
      \ThreeDotsAt{(1.5,1.5)}{.21213203}{90}
      \node at (1.5,2.5) {$0$};
      \node at (1.5,0.5) {$0$};

      \node at (2.5,3.5) {$0$};
      \ThreeDotsAt{(3.5,3.5)}{.21213203}{0}
      \node at (4.5,3.5) {$0$};

      \node at (2.5,4.5) {$0$};
      \ThreeDotsAt{(3.5,4.5)}{.21213203}{0}
      \node at (4.5,4.5) {$0$};

      \node at (3.5,1.5) {$Z$};

      \begin{scope}[xshift=8.5cm]
        \node at (-1.4,2.5) {$\extUR(Z)\ =$};
        \node[draw=none] at (5.15,2.5) {$,$};
        \draw[fill=cyan!15] (0,0) rectangle (4,4);
        \draw (0,0) rectangle (5,5);

        \node at (0.5,4.5) {$0$};
        \ThreeDotsAt{(2,4.5)}{.21213203}{0}
        \node at (4.5,4.5) {$0$};
        \node at (3.5,4.5) {$0$};

        \node at (4.5,3.5) {$1$};
        \node at (4.5,2.5) {$-1$};
        \node at (4.5,1.5) {$1$};
        \ThreeDotsAt{(4.5,0.5)}{.21213203}{90}

        \node at (2,2) {$Z$};
      \end{scope}
    \end{tikzpicture}

    \vspace{18pt}

    \begin{tikzpicture}[scale=.8, baseline=(current bounding box.center)]
      \node at (-1.4,2.5) {$\extBL(Z)\ =$};
      \node[draw=none] at (5.15,2.5) {$,$};
      \draw[fill=cyan!15] (1,1) rectangle (5,5);
      \draw (0,0) rectangle (5,5);

      \node at (0.5,4.5) {$0$};
      \node at (0.5,1.5) {$0$};
      \node at (3,3) {$Z$};

      \ThreeDotsAt{(0.5,3)}{.21213203}{90}

      \node at (0.5,0.5) {$0$};
      \node at (1.5,0.5) {$1$};
      \node at (2.5,0.5) {$-1$};
      \node at (3.5,0.5) {$1$};
      \ThreeDotsAt{(4.5,0.5)}{.21213203}{0}

      \begin{scope}[xshift=8.5cm]
        \node at (-1.4,2.5) {$\extBR(Z)\ =$};
        \node[draw=none] at (5.15,2.5) {$.$};
        \draw[fill=cyan!15] (0,1) rectangle (4,5);
        \draw (0,0) rectangle (5,5);
        \node at (2,3) {$Z$};
        \node at (4.5,4.5) {$0$};
        \ThreeDotsAt{(4.5,3)}{.21213203}{90}
        \node at (4.5,1.5) {$0$};

        \node at (0.5,0.5) {$0$};
        \ThreeDotsAt{(2,0.5)}{.21213203}{0}
        \node at (3.5,0.5) {$0$};

        \node at (4.5,0.5) {$\chi_{2 \mid k}$};
      \end{scope}
    \end{tikzpicture}
  \end{center}
  More precisely, the operator $\extUL$ takes a $(k-2)\times (k-2)$ core $Z$ and yields a $k\times k$ core by placing $Z$ in the lower-right block, adjoining two new top rows and two new left columns, all zero except for a single $1$ at the corner cell adjacent to the embedded copy.
  The other three operators take a $(k-1)\times (k-1)$ core $Z$ and yield a $k\times k$ core.
  In particular, $\extUR$ extends a $(k-1)\times (k-1)$ core $Z$ by inserting a new top row of zeros and appending a new last column whose entries alternate between $+1$ and $-1$, with the first entry equal to $+1$; $\extBL$ extends $Z$ by inserting a new first column of zeros and appending a new bottom row whose entries alternate between $+1$ and $-1$, with the first entry equal to $+1$; and $\extBR$ extends $Z$ by appending a new last column and a new bottom row, all zero except possibly for the bottom-right corner entry, which is set to $\chi_{2 \mid k}$.

  By the definition of our four extension operators, every prefix sum in~\cref{eq:vhsasm:row-prefix,eq:vhsasm:col-prefix} either coincides with the corresponding prefix sum on the embedded core (after a simple index shift), because all newly inserted entries contributing to that prefix are $0$; or belongs to $\{0,1\}$ because the defining prefix lies entirely in the newly inserted boundary rows or columns, and the alternating patterns yield prefix sums alternating between $0$ and $1$.
  Moreover, the endpoint $(i,j)$ of the unique violated facet inequality is transported as follows.
  Under $\extUL$ it becomes $(i+2,j+2)$, under $\extUR$ it becomes $(i+1,j)$, under $\extBL$ it becomes $(i,j+1)$, and under $\extBR$ the endpoint remains $(i,j)$.
  Accordingly, the recursive definitions invoke the smaller certificate with indices $(i-2,j-2)$, $(i-1,j)$, $(i,j-1)$, or $(i,j)$, respectively.
  Finally, it is straightforward to see that the equations in~\crefrange{eq:vhsasm:row-sum}{eq:vhsasm:dim:firstcol} are preserved under $\extUL$, $\extUR$, $\extBL$, and $\extBR$.

  Now we are ready to construct $L^{n,\mathrm H}_{i,j}$, $L^{n,\mathrm V}_{i,j}$, $U^{n,\mathrm H}_{i,j}$, and $U^{n,\mathrm V}_{i,j}$ via an inductive approach.
  For $n=9$, we set
  \[
    L^{9,\mathrm H}_{2,2}=
    \vcenter{\hbox{\begin{ytableau}
          0& 0& 0& 0\\
          0&-1& 1& 1\\
          0& 1&-1& 0\\
          0& 1& 0& 0
        \end{ytableau}}}\,,
    \quad
    L^{9,\mathrm H}_{3,2}=
    \vcenter{\hbox{\begin{ytableau}
          0& 0& 0& 0\\
          0& 1& 0& 0\\
          0&-1& 1& 0\\
          0& 1&-1& 1
        \end{ytableau}}}\,,
    \quad
    L^{9,\mathrm H}_{3,3}=
    \vcenter{\hbox{\begin{ytableau}
          0& 0& 0& 0\\
          0& 0& 1& 0\\
          0& 0&-1& 1\\
          0& 1& 0& 0
        \end{ytableau}}}\,,
    \quad
    L^{9,\mathrm H}_{4,3}=
    \vcenter{\hbox{\begin{ytableau}
          0& 0& 0& 0\\
          0& 1& 0& 0\\
          0& 0& 1&-1\\
          0& 0&-1& 2
        \end{ytableau}}}\,.
  \]
  We define $L^{9,\mathrm V}_{i,j}=(L^{9,\mathrm H}_{j,i})^\top$ for $(i,j)\in\{(2,3),(3,3)\}$.

  For $n=11$, we define $L^{11,\mathrm H}_{i,j}$ as follows.
  For $(i,j)\in\{(2,2),(3,2),(3,3)\}$, set
  $
  L^{11,\mathrm H}_{i,j}=\extBR\bigl(L^{9,\mathrm H}_{i,j}\bigr)
  $;
  for $(i,j)\in\{(4,2),(4,3),(5,3)\}$, set
  $
  L^{11,\mathrm H}_{i,j}=\extUR\bigl(L^{9,\mathrm H}_{i-1,j}\bigr)
  $;
  for $(i,j)=(3,4)$, set
  $
  L^{11,\mathrm H}_{3,4}=\extBL\bigl(L^{9,\mathrm H}_{3,3}\bigr)
  $.
  We define
  \[
    U^{11,\mathrm H}_{4,4}=
    \vcenter{\hbox{\begin{ytableau}
          0& 0& 0& 0& 0\\
          0& 0& 0& 0& 1\\
          0& 0& 0& 0& 0\\
          0& 1& 0& 1&-1\\
          0& 0& 0& 0& 0
        \end{ytableau}}}\,,
    \quad
    U^{11,\mathrm H}_{5,4}=
    \vcenter{\hbox{\begin{ytableau}
          0& 0& 0& 0& 0\\
          0& 0& 0& 0& 1\\
          0& 0& 0& 0& 0\\
          0& 0& 0& 0& 1\\
          0& 1& 0& 1&-2
        \end{ytableau}}}\,.
  \]
  We define $L^{11,\mathrm V}_{i,j}=(L^{11,\mathrm H}_{j,i})^\top$ for $(i,j) \in \{(2,3),(2,4),(3,3),(3,4),(3,5),(4,3)\}$ and $U^{11,\mathrm V}_{4,4}=(U^{11,\mathrm H}_{4,4})^\top$.

  \medskip
  We now give the recursive definitions of the certifying cores for odd $n\geq 13$.
  Assume that the certificates have already been defined for sizes $n-2$ and $n-4$.
  For horizontal facet lower bounds, define $L^{n,\mathrm H}_{i,j}$ for every index $(i,j)$ appearing in the horizontal facet lower family as follows:
  \[
    L^{n,\mathrm H}_{i,j}=
    \begin{cases}
      \extBR\bigl(L^{n-2,\mathrm H}_{2,2}\bigr)     & \text{if } (i,j) = (2,2),\\
      \extBR\bigl(L^{n-2,\mathrm H}_{i,j}\bigr)     & \text{if } i \in [3, k-2] \text{ and } j \in [2, k-2-\chi_{2 \mid i}],\\
      \extBL\bigl(L^{n-2,\mathrm H}_{i,j-1}\bigr)   & \text{if } i \in [3, k-2] \text{ and } j = k-1-\chi_{2 \mid i},\\
      \extUR\bigl(L^{n-2,\mathrm H}_{k-2,j}\bigr)   & \text{if } i = k-1 \text{ and } j \in [2, k-3-\chi_{2 \nmid k}],\\
      \extUL\bigl(L^{n-4,\mathrm H}_{k-3,j-2}\bigr) & \text{if } i = k-1 \text{ and } k-2-\chi_{2 \nmid k} \leq j,\\
      \extUR\bigl(L^{n-2,\mathrm H}_{k-1,j}\bigr)   & \text{if } i = k \text{ and } j \in [3, k-2-\chi_{2 \nmid k}],\\
      \extUL\bigl(L^{n-4,\mathrm H}_{k-2,j-2}\bigr) & \text{if } i = k \text{ and } j = k-1-\chi_{2 \nmid k}.
    \end{cases}
  \]
  For horizontal facet upper bounds, define $U^{n,\mathrm H}_{i,j}$ for every index $(i,j)$ appearing in the horizontal facet upper family as follows:
  \[
    U^{n,\mathrm H}_{i,j}=
    \begin{cases}
      \vcenter{\vspace{3pt}\hbox{\begin{ytableau}
            0& 0& 0& 0& 0& 0\\
            0& 0& 0& 0& 0& 1\\
            0& 0& 0& 0& 0& 0\\
            0& 0& 0& 0& 1& 0\\
            0& 0& 0& 0& 0& 0\\
            0& 1& 0& 1&-1& 0
          \end{ytableau}}}                        & \text{if } n=13, (i,j)=(k,k-2),\\[46pt]
      \extUL\bigl(U^{n-4,\mathrm H}_{i-2,j-2}\bigr) & \text{if } n \geq 17,\ 2 \mid k \text{ and } (i,j)=(k,k-2),\\
      \extBR\bigl(U^{n-2,\mathrm H}_{i,j}\bigr)     & \text{if } i \in [4, k-2] \text{ and } j \in [4, k-2-\chi_{2 \nmid i}],\\
      \extBL\bigl(U^{n-2,\mathrm H}_{i,j-1}\bigr)   & \text{if } i \in [4, k-2] \text{ and } j = k-1-\chi_{2 \nmid i},\\
      \extUR\bigl(U^{n-2,\mathrm H}_{k-2,j}\bigr)   & \text{if } i = k-1 \text{ and } j \in [4, k-2-\chi_{2 \nmid k}],\\
      \extBL\bigl(U^{n-2,\mathrm H}_{k-1,k-3}\bigr) & \text{if } 2 \nmid k \text{ and } (i,j)=(k-1,k-2),\\
      \extUL\bigl(U^{n-4,\mathrm H}_{k-3,k-3}\bigr) & \text{if } 2 \nmid k \text{ and } (i,j)=(k-1,k-1),\\
      \extUR\bigl(U^{n-2,\mathrm H}_{k-1,j}\bigr)   & \text{if } i = k \text{ and } j \in [4, k-2-\chi_{2 \mid k}],\\
      \extUL\bigl(U^{n-4,\mathrm H}_{k-2,k-3}\bigr) & \text{if } 2 \nmid k \text{ and } (i,j)=(k,k-1).
    \end{cases}
  \]

  Finally, inspection of the index sets in~\cref{thm:vhsasm:facets} shows that, for every vertical lower (respectively, upper) facet index $(i,j)$, the transposed index $(j,i)$ is a horizontal lower (respectively, upper) facet index.
  Hence all certificates on the right-hand side below have already been defined, and we set
  \[
    L^{n,\mathrm V}_{i,j} = \bigl(L^{n,\mathrm H}_{j,i}\bigr)^\top,
    \qquad
    U^{n,\mathrm V}_{i,j} = \bigl(U^{n,\mathrm H}_{j,i}\bigr)^\top.
  \]

  \medskip
  We claim that these definitions provide the desired certificates.
  We prove this by induction on~$n$.
  The explicitly listed cores for $n=9$ (horizontal lower), for $n=11$ (horizontal upper), and the additional base case $U^{13,\mathrm H}_{k,k-2}$ satisfy the claim by direct inspection.
  Now let $n\geq 13$ and assume that the claim holds for sizes $n-2$ and $n-4$.

  \smallskip
  By construction, every certifying core for size $n$ is obtained from a smaller certifying core (for size $n-2$ or size $n-4$) by applying one of the four extension operators.
  Each extension embeds the smaller $k'\times k'$ core as a submatrix and fills the remaining new rows and columns with entries of the fixed boundary patterns used in the definition of $\extUL,\extUR,\extBL,\extBR$ (zeros, or alternatingly $1$ and $-1$, or a single corner entry), so that every row-prefix or column-prefix sum whose endpoint lies entirely in the newly inserted part takes a value in $\{0,1\}$.
  Consequently, every facet inequality whose defining prefix lies in the newly inserted rows or columns is satisfied.

  \smallskip
  If the endpoint $(i,j)$ of a facet inequality lies in the embedded copy of the smaller certificate, then its defining prefix is contained in that copy (after the evident index shift), and its prefix sum coincides with the corresponding prefix sum of the smaller certificate.
  Hence it is satisfied by the induction hypothesis, except for the unique facet inequality violated by the smaller certificate.
  The extension operator in the recursive definition is chosen so that this uniquely violated facet inequality is transported to the intended facet inequality for size $n$: under $\extBR$ the endpoint remains $(i,j)$, under $\extBL$ it becomes $(i,j+1)$, under $\extUR$ it becomes $(i+1,j)$, and under $\extUL$ it becomes $(i+2,j+2)$.
  Accordingly, the recursive definitions invoke the predecessor certificate with indices $(i,j)$, $(i,j-1)$, $(i-1,j)$, or $(i-2,j-2)$, respectively (and size $n-2$ or size $n-4$ as indicated).

  \smallskip
  Finally, each of the four extension operators preserves the equations in~\crefrange{eq:vhsasm:row-sum}{eq:vhsasm:dim:firstcol}; this follows directly from their definitions.
  Therefore, for every facet inequality for size $n$ we have constructed a core that violates that facet inequality and satisfies every other facet inequality as well as all defining equations.
  This shows that no facet inequality is redundant.
\end{proof}

\section{Half-turn symmetric ASMs (HTSASMs)}\label{sec:HTSASM}
In this class, we impose invariance under rotation by $\pi$.
The corresponding symmetry subgroup is $\mathcal{G} = \{\mathcal{I}, \mathcal{R}_{\pi}\}$.
Let $P_\HTS$ denote the linear subspace of half-turn symmetric real matrices, i.e.,
\[
  P_\HTS
  =
  \left\{
    X \in \R^{n \times n} : x_{i,j} = x_{n+1-i,n+1-j}\ \forall i, j \in [n]
  \right\}.
\]
Clearly, any HTSASM satisfies the ASM constraints in~\crefrange{eq:asm:real}{eq:asm:col-sum} and also the symmetry constraints defining $P_\HTS$; thus $P_\HTSASM \subseteq P_\ASM \cap P_\HTS$.
We will prove that these constraints are in fact sufficient, i.e., $P_\HTSASM = P_\ASM \cap P_\HTS$.

\paragraph{Core and assembly map.}
Let the \emph{core} of an HTSASM be given by the first $k = \floor*{n/2}$ columns, together with the entries strictly above the central row in the middle column when $n$ is odd, i.e.,
\[
  C =
  \begin{cases}
    [n] \times [k]                             & \text{if } 2 \mid n,\\
    ([n] \times [k]) \cup ([k] \times \{k+1\}) & \text{if } 2 \nmid n
  \end{cases}
\]
is the set of \emph{core positions}, and let $\pi_C$ be the coordinate-wise projection onto $C$.
Define the affine map $\varphi : \R^C \to \R^{n \times n}$ by
\[
  \varphi(Y)_{i,j} =
  \begin{cases}
    y_{i,j}                        & \text{if } (i,j) \in C,\\
    1 - 2\sum_{j'=1}^{k} y_{k+1,j'} & \text{if } 2 \nmid n\text{ and } i=j=k+1,\\
    y_{n+1-i,n+1-j}                 & \text{otherwise}
  \end{cases}
\]
for $Y \in \R^C$ and $i,j \in [n]$.
Thus $\varphi$ places the core $Y$ on the core positions, assigns the central entry according to the second case when $n$ is odd, and fills the remaining entries by a rotation by $\pi$, yielding a half-turn symmetric matrix.
Clearly, the map $\varphi$ is an assembly map: it is affine, satisfies $\pi_C(\varphi(Y)) = Y$ for every $Y \in \R^C$, and $\varphi(\pi_C(X)) = X$ for every $X \in \HTSASM(n)$, because $X$ is completely determined by its entries in $C$ together with the imposed half-turn symmetry, including the middle entry for odd $n$, which is uniquely determined in an HTSASM by the rest of the palindromic middle row.

\bigskip
We now describe the core polytope of HTSASMs.
\begin{theorem}\label{thm:htsasm:corepolytope}
  Let $n \geq 1$, $k = \floor*{n/2}$, and $\ell = \ceil*{n/2}$.
  Then the core polytope $P^\core_\HTSASM \subseteq \R^C$ of $n \times n$ HTSASMs is described by the following system.
  \begin{align}
                                     y&_{i,j} \in \R                                       &\forall (i,j) \in C, \label{eq:htsasm:real}\\
    0 \leq \sum_{j'=1}^{\min\{\ell,j\}} y&_{i,j'} + \sum_{j'=\ell+1}^{j} y_{n+1-i,n+1-j'} \leq 1 &\forall i \in [k], j \in [n-1], \label{eq:htsasm:row-prefix}\displaybreak[0]\\
    0 \leq \sum_{i'=1}^{i}            y&_{i',j} \leq 1                                      &\forall i \in [n-1], j \in [k], \label{eq:htsasm:col-prefix}\\
    \sum_{j'=1}^{\ell}                 y&_{i,j'} + \sum_{j'=\ell+1}^{n} y_{n+1-i,n+1-j'} = 1     &\forall i \in [k], \label{eq:htsasm:row-sum}\\
    \sum_{i=1}^{n}                    y&_{i,j} = 1                                          &\forall j \in [k]; \label{eq:htsasm:col-sum}\\
    \intertext{if $n$ is odd, then we also add the constraints}
    0 \leq \sum_{j'=1}^{j}            y&_{\ell,j'} \leq 1                                    &\forall j \in [k], \label{eq:htsasm:mid-row-pref}\\
    0 \leq \sum_{i'=1}^{i}            y&_{i',\ell} \leq 1                                    &\forall i \in [k]. \label{eq:htsasm:mid-col-pref}
  \end{align}
\end{theorem}
\begin{proof}
  We prove that, when $n$ is even, the integer solutions to the system~\crefrange{eq:htsasm:real}{eq:htsasm:col-sum} are exactly the cores of HTSASMs, while for odd $n$ the cores arise precisely as the integer solutions to the extended system~\crefrange{eq:htsasm:real}{eq:htsasm:mid-col-pref}.
  We then show that the system defines an integral polytope in both cases.

  \medskip
  First, let $X$ be an $n\times n$ HTSASM, and let $Y = \pi_C(X)$ be its core.
  Since $X$ is half-turn symmetric, we have $X = \varphi(Y)$.
  By the definition of $\varphi$, for every $i \in [k]$ and $j \in [n]$,
  \begin{equation}\label{eq:htsasm:phi-row-prefix-used}
    \sum_{j'=1}^{j} x_{i,j'}
    = \sum_{j'=1}^{\min\{\ell,j\}} y_{i,j'}
    + \sum_{j'=\ell+1}^{j} y_{n+1-i,n+1-j'}.
  \end{equation}
  For the first $k$ columns, we similarly have, for every $j\in[k]$ and $i\in[n]$,
  \begin{equation}\label{eq:htsasm:phi-col-prefix-used}
    \sum_{i'=1}^{i} x_{i',j} = \sum_{i'=1}^{i} y_{i',j}.
  \end{equation}
  If $n$ is odd, then $(\ell,j)\in C$ for every $j\in[k]$, and $(i,\ell)\in C$ for every $i\in[k]$, so for the middle row and column we also obtain, for every $i,j\in[k]$,
  \begin{equation}\label{eq:htsasm:phi-mid-prefix-used}
    \sum_{j'=1}^{j} x_{\ell,j'} = \sum_{j'=1}^{j} y_{\ell,j'}
    \qquad\text{ and }\qquad
    \sum_{i'=1}^{i} x_{i',\ell} = \sum_{i'=1}^{i} y_{i',\ell}.
  \end{equation}
  For later use, note that in the odd case both the middle row and the middle column of $X$ are palindromic, and therefore
  $\sum_{j=1}^{n} x_{\ell,j} = 2\sum_{j=1}^{k} x_{\ell,j} + x_{\ell,\ell}$, and $\sum_{i=1}^{n} x_{i,\ell} = 2\sum_{i=1}^{k} x_{i,\ell} + x_{\ell,\ell}$.

  Since every HTSASM is in particular an ASM, the ASM row-prefix bounds in~\cref{eq:asm:row-prefix}, together with the identity~\cref{eq:htsasm:phi-row-prefix-used} for $j \in [n-1]$, yield the constraints in~\cref{eq:htsasm:row-prefix}.
  The case $j = n$ of the same identity, together with the ASM row-sum constraint in~\cref{eq:asm:row-sum}, gives the row-sum equations in~\cref{eq:htsasm:row-sum}.
  In the odd case, combining the ASM row-prefix bounds in~\cref{eq:asm:row-prefix} with the identity~\cref{eq:htsasm:phi-mid-prefix-used} yields the middle-row-prefix bounds in~\cref{eq:htsasm:mid-row-pref}.
  Likewise, the ASM column-prefix bounds in~\cref{eq:asm:col-prefix}, together with the identity~\cref{eq:htsasm:phi-col-prefix-used} for $i \in [n-1]$, yield the constraints in~\cref{eq:htsasm:col-prefix}.
  The case $i = n$ of the same identity, together with the ASM column-sum constraint in~\cref{eq:asm:col-sum}, gives the column-sum equations in~\cref{eq:htsasm:col-sum};
  in the odd case, combining the ASM column-prefix bounds in~\cref{eq:asm:col-prefix} with the identity~\cref{eq:htsasm:phi-mid-prefix-used} yields the middle-column-prefix bounds in~\cref{eq:htsasm:mid-col-pref}.
  Thus the cores of HTSASMs satisfy the system~\crefrange{eq:htsasm:real}{eq:htsasm:mid-col-pref}.

  \medskip
  Second, let $Y\in\Z^C$ satisfy the system~\crefrange{eq:htsasm:real}{eq:htsasm:mid-col-pref}, and set $X=\varphi(Y)$.
  By construction, $X$ is an $n\times n$ integer matrix, it is invariant under rotation by~$\pi$, and its core is~$Y$.
  For every $i\in[k]$ and $j\in[n]$, the identity~\cref{eq:htsasm:phi-row-prefix-used} translates the constraints in~\cref{eq:htsasm:row-prefix,eq:htsasm:row-sum} exactly into the ASM row-prefix bounds in~\cref{eq:asm:row-prefix} and row-sum constraints in~\cref{eq:asm:row-sum} for the first $k$ rows of~$X$.
  For the last $k$ rows, the half-turn symmetry of $X$ implies that each row is the reverse of some row in the top half; using the row-sum constraints just established, this shows that all row-prefix sums lie between $0$ and $1$, and every row sums to~$1$.
  Similarly, for each $j \in [k]$ and $i \in [n]$, the identity~\cref{eq:htsasm:phi-col-prefix-used} translates the constraints in~\cref{eq:htsasm:col-prefix,eq:htsasm:col-sum} into the ASM column-prefix bounds in~\cref{eq:asm:col-prefix} and column sums in~\cref{eq:asm:col-sum} for the first $k$ columns of~$X$.
  For the last $k$ columns, the constraints follow by half-turn symmetry.
  When $n$ is odd, subtracting the sum of the equations in~\cref{eq:htsasm:row-sum} from the sum of the equations in~\cref{eq:htsasm:col-sum} gives $\sum_{j=1}^{k}y_{\ell,j}=\sum_{i=1}^{k}y_{i,\ell}$.
  By the definition of the central entry in $\varphi$ and the palindromic identities above, the middle row and the middle column therefore both sum to~$1$.
  Furthermore, the additional constraints in~\cref{eq:htsasm:mid-row-pref,eq:htsasm:mid-col-pref}, together with half-turn symmetry, ensure that the middle row and middle column satisfy the ASM prefix bounds.
  Altogether, $X$ satisfies all ASM constraints from \cref{thm:ASMpolytope}.
  Since $X$ is also half-turn symmetric, we conclude that $X$ is an HTSASM\@.
  Hence the integer solutions to the system~\crefrange{eq:htsasm:real}{eq:htsasm:mid-col-pref} are precisely the cores of HTSASMs.

  \medskip
  It remains to prove that the system~\crefrange{eq:htsasm:real}{eq:htsasm:col-sum} and the system~\crefrange{eq:htsasm:real}{eq:htsasm:mid-col-pref} define integral polytopes when $n$ is even and odd, respectively.
  We first work out the details for the case when $n$ is even.
  Consider the index sets $A \subseteq C$ that appear on the left-hand sides of the row constraints in~\cref{eq:htsasm:row-prefix,eq:htsasm:row-sum}, so that each of these constraints can be written in the form $\alpha_A \leq \sum_{(i,j)\in A} y_{i,j} \leq \beta_A$ with $\alpha_A \in \{0,1\}$ and $\beta_A = 1$.
  For fixed $i$, the sets corresponding to the inequalities in~\cref{eq:htsasm:row-prefix} form a chain as $j$ increases that ends with the index set of the variables in the left-hand side of~\cref{eq:htsasm:row-sum}; for different $i$, these chains are supported on disjoint pairs of rows, so they form a laminar family $\mathcal{L}_{\mathrm{row}} \subseteq 2^C$.
  Likewise, the index sets occurring in the column constraints in~\cref{eq:htsasm:col-prefix,eq:htsasm:col-sum} form a second laminar family $\mathcal{L}_{\mathrm{col}} \subseteq 2^C$.
  Thus the system~\crefrange{eq:htsasm:real}{eq:htsasm:col-sum} is exactly of the form treated in \cref{thm:laminarSystem}, with variables $y_{i,j}$ for $(i,j)\in C$ and laminar families $\mathcal{L}_{\mathrm{row}}$ and $\mathcal{L}_{\mathrm{col}}$, and therefore defines an integral polytope.

  When $n$ is odd, the additional constraints in~\cref{eq:htsasm:mid-row-pref,eq:htsasm:mid-col-pref} fit into the same framework as follows.
  The index sets in~\cref{eq:htsasm:mid-row-pref} form a chain supported on the middle row, which is disjoint from every set in $\mathcal{L}_{\mathrm{row}}$; thus adjoining them yields a laminar family $\mathcal{L}'_{\mathrm{row}}$.
  Likewise, the index sets in~\cref{eq:htsasm:mid-col-pref} form a chain supported on the middle column, which is disjoint from every set in $\mathcal{L}_{\mathrm{col}}$; adjoining them yields a laminar family $\mathcal{L}'_{\mathrm{col}}$.
  Therefore, the system~\crefrange{eq:htsasm:real}{eq:htsasm:mid-col-pref} is of the form treated in \cref{thm:laminarSystem} with laminar families $\mathcal{L}'_{\mathrm{row}}$ and $\mathcal{L}'_{\mathrm{col}}$, and hence defines an integral polytope.
\end{proof}

By \cref{thm:xasm:assembly,thm:htsasm:corepolytope}, we obtain the following description of $P_\HTSASM$.
\begin{theorem}\label{thm:htsasm:coreDescr}
  Let $n \geq 1$, $k=\floor*{n/2}$, and $\widehat P^\core_\HTSASM = \{X \in \R^{n \times n} : \pi_C(X) \in P^\core_\HTSASM\}$.
  Then
  \[
    P_\HTSASM = \widehat P^\core_\HTSASM \cap P_\HTS \cap \Bigl\{X \in \R^{n \times n} : \sum_{j=1}^nx_{k+1,j} = 1\Bigr\}.
  \]
\end{theorem}

Using the assembly map $\varphi$ and the identities~\crefrange{eq:htsasm:phi-row-prefix-used}{eq:htsasm:phi-mid-prefix-used} to translate the HTSASM core constraints in \cref{thm:htsasm:coreDescr} to and from the ASM constraints in \cref{thm:ASMpolytope}, we obtain the following.
\begin{theorem}\label{thm:htsasm:descr}
  For every $n \geq 1$,
  \[
    P_\HTSASM = P_\ASM \cap P_\HTS.
  \]
\end{theorem}

\begin{theorem}\label{thm:htsasm:dim}
  For every $n \geq 1$, the dimension of $P_\HTSASM$ is $\ceil*{\frac{(n-1)^2}{2}}$.
\end{theorem}
\begin{proof}
  For $n=1$, the polytope $P_\HTSASM$ consists of a single point, so the assertion is immediate.
  Assume from now on that $n\geq 2$.
  It suffices to prove that the dimension of $P^\core_\HTSASM$ is $\ceil*{\frac{(n-1)^2}{2}}$, because the assembly map $\varphi$ restricts to an affine isomorphism between $P^\core_\HTSASM$ and $P_\HTSASM$, which preserves dimension.
  First, we give an upper bound.
  Let $k=\floor*{n/2}$ and $\ell=\ceil*{n/2}$.

  Assume $n$ is even, so $C=[n]\times[k]$.
  Consider the coefficient matrix of the row-sum equations in~\cref{eq:htsasm:row-sum} and the column-sum equations in~\cref{eq:htsasm:col-sum}.
  Delete all columns corresponding to variables $y_{i,j}$ with $i\in[k+1,2k]$.
  The remaining matrix is the coefficient matrix of the row- and column-sum system for a $k\times k$ matrix in the variables $y_{i,j}$ with $i,j\in[k]$.
  By \cref{lem:prelim:rowcol-rank}, this subsystem has rank $2k-1$.
  Hence the original system has rank at least $2k-1$, and therefore $\dim(P^\core_\HTSASM)\leq |C|-(2k-1)=2k^2-(2k-1)=\ceil*{\frac{(n-1)^2}{2}}$.

  Assume $n$ is odd, and recall that $C=([n]\times[k])\cup([k]\times\{\ell\})$.
  Let $A$ be the coefficient matrix for the $2k$ equations in~\cref{eq:htsasm:row-sum,eq:htsasm:col-sum}.
  Deleting all columns of $A$ except those corresponding to variables $y_{i,j}$ with $i,j\in[k]$ yields the coefficient matrix of the row- and column-sum system for a $k\times k$ matrix, and hence has rank $2k-1$ by \cref{lem:prelim:rowcol-rank}.
  Moreover, keeping in addition the column corresponding to $y_{1,\ell}$ increases the rank by $1$: for every variable $y_{i,j}$ with $i,j\in[k]$, its column has one non-zero entry among the equations in~\cref{eq:htsasm:row-sum} and one non-zero entry among the equations in~\cref{eq:htsasm:col-sum}, and hence every linear combination of these columns has the same total sum in the rows corresponding to the equations in~\cref{eq:htsasm:row-sum} as in those corresponding to the equations in~\cref{eq:htsasm:col-sum}.
  In contrast, the column of $y_{1,\ell}$ has total sum $1$ in the former and $0$ in the latter rows; thus it is not contained in the span of the previously retained columns.
  Therefore, the system consisting of the equations in~\cref{eq:htsasm:row-sum,eq:htsasm:col-sum} has rank at least $2k$.
  Consequently, $\dim(P^\core_\HTSASM)\leq |C|-2k=(2k^2+2k)-2k=2k^2=\ceil*{\frac{(n-1)^2}{2}}$.
  Thus we obtain the bound $\dim(P^\core_\HTSASM)\leq \ceil*{\frac{(n-1)^2}{2}}$ for every $n \geq 1$.

  \medskip
  Second, we construct $\ceil*{\frac{(n-1)^2}{2}}+1$ affinely independent cores in $P^\core_\HTSASM$.
  Let $\ol Y\in\R^C$ be the core with all entries equal to $1/n$.
  It is easy to see that $\ol Y$ satisfies all inequalities in \cref{thm:htsasm:corepolytope} strictly, and it satisfies the equations in~\cref{eq:htsasm:row-sum,eq:htsasm:col-sum}.

  Assume $n$ is even.
  For $i,j\in[k-1]$, define
  $
    \ol Y^{i,j}=\ol Y+\varepsilon\chi_{i,j}-\varepsilon\chi_{i,k}-\varepsilon\chi_{k,j}+\varepsilon\chi_{k,k}
  $,
  and for $i\in[k+1,2k]$ and $j\in[k]$, define
  $
    \ol Z^{i,j}=\ol Y+\varepsilon\chi_{n+1-i,j}-\varepsilon\chi_{i,j}
  $,
  where $\varepsilon>0$ is a small constant.
  By construction, in each of the equations in~\cref{eq:htsasm:row-sum,eq:htsasm:col-sum} the increment(s) are cancelled by the decrement(s), hence every $\ol Y^{i,j}$ with $i,j\in[k-1]$ and every $\ol Z^{i,j}$ with $i\in[k+1,2k]$, $j\in[k]$ satisfies the equations in~\cref{eq:htsasm:row-sum,eq:htsasm:col-sum}.
  Since $\ol Y$ satisfies all inequalities strictly, choosing $\varepsilon>0$ small enough guarantees that $\ol Y^{i,j} \in P^\core_\HTSASM$ for all $i,j\in[k-1]$ and $\ol Z^{i,j}\in P^\core_\HTSASM$ for all $i\in[k+1,2k]$, $j\in[k]$.
  We claim that the cores $\ol Y$ and $\ol Y^{i,j}$ for $i,j\in[k-1]$ are affinely independent: only the difference $\ol Y^{i,j}-\ol Y$ has a non-zero entry among the positions $[k-1]\times[k-1]$, namely at $(i,j)$.
  Hence the cores $\{\ol Y^{i,j}-\ol Y : i,j\in[k-1]\}$ are linearly independent.
  Likewise, the cores $\ol Y$ and $\ol Z^{i,j}$ for $i\in[k+1,2k]$ and $j\in[k]$ are affinely independent: only the difference $\ol Z^{i,j}-\ol Y$ has a non-zero entry among the positions $[k+1,2k]\times[k]$, namely at $(i,j)$.
  Hence the cores $\{\ol Z^{i,j}-\ol Y : i\in[k+1,2k], j\in[k]\}$ are linearly independent.
  Moreover, these two sets of differences are linearly independent together: every core $\ol Y^{i,j}-\ol Y$ has all its non-zero entries in rows $[k]$, while every core $\ol Z^{i,j}-\ol Y$ has a non-zero entry in a row from $[k+1,2k]$, namely at $(i,j)$.
  We obtain $(k-1)^2+k^2+1=2k^2-2k+2=\ceil*{\frac{(n-1)^2}{2}}+1$ affinely independent cores, and therefore $\dim(P^\core_\HTSASM)\geq \ceil*{\frac{(n-1)^2}{2}}$.

  Assume $n$ is odd.
  For $i,j\in[k]$, define
  $
    \ol Y^{i,j}=\ol Y+\varepsilon\chi_{i,j}-\varepsilon\chi_{i,\ell}-\varepsilon\chi_{\ell,j}
  $,
  and for $i\in[\ell+1,n]$ and $j\in[k]$, define
  $
    \ol Z^{i,j}=\ol Y+\varepsilon\chi_{n+1-i,j}-\varepsilon\chi_{i,j}
  $.
  Again, by construction every $\ol Y^{i,j}$ with $i,j\in[k]$ and every $\ol Z^{i,j}$ with $i\in[\ell+1,n]$, $j\in[k]$ satisfies the equations in~\cref{eq:htsasm:row-sum,eq:htsasm:col-sum}, and since $\ol Y$ satisfies all inequalities strictly, choosing $\varepsilon>0$ small enough ensures that $\ol Y^{i,j} \in P^\core_\HTSASM$ for all $i,j\in[k]$ and $\ol Z^{i,j}\in P^\core_\HTSASM$ for all $i\in[\ell+1,n]$, $j\in[k]$.
  The cores $\ol Y$ and $\ol Y^{i,j}$ for $i,j\in[k]$ are affinely independent: the difference $\ol Y^{i,j}-\ol Y$ has a unique non-zero entry among the positions $[k]\times[k]$, namely at $(i,j)$, so the cores $\{\ol Y^{i,j}-\ol Y : i,j\in[k]\}$ are linearly independent.
  Likewise, the cores $\ol Y$ and $\ol Z^{i,j}$ for $i\in[\ell+1,n]$ and $j\in[k]$ are affinely independent: only the difference $\ol Z^{i,j}-\ol Y$ has a non-zero entry at $(i,j)$, so the cores $\{\ol Z^{i,j}-\ol Y : i\in[\ell+1,n], j\in[k]\}$ are linearly independent.
  Moreover, these two sets of differences are linearly independent together: every difference $\ol Y^{i,j}-\ol Y$ with $i,j\in[k]$ has all its non-zero entries outside $[\ell+1,n]\times[k]$, while every difference $\ol Z^{i,j}-\ol Y$ with $i\in[\ell+1,n]$ has a non-zero entry among the positions $[\ell+1,n]\times[k]$, namely at $(i,j)$.
  Hence $\dim(P^\core_\HTSASM)\geq k^2+k^2=2k^2=\ceil*{\frac{(n-1)^2}{2}}$.

  Combining the lower and upper bounds yields $\dim(P^\core_\HTSASM)=\dim(P_\HTSASM)=\ceil*{\frac{(n-1)^2}{2}}$.
\end{proof}

\begin{theorem}\label{thm:htsasm:facets}
  Let $n \geq 4$, and set $k=\floor*{n/2}$.
  The facets of $P^\core_\HTSASM$ are given by tightening the lower bound in~\cref{eq:htsasm:row-prefix} to equality for
  $
    (i,j)\in \{(1,1)\} \cup \bigl([2,k]\times[n-2]\bigr)
  $,
  and the upper bound for
  $
    (i,j)\in \{(1,n-1)\} \cup \bigl([2,k]\times[2,n-1]\bigr)
  $;
  and by tightening the lower bound in~\cref{eq:htsasm:col-prefix} to equality for
  $
    (i,j)\in [n-2]\times[2,k]
  $,
  and the upper bound for
  $
    (i,j)\in [2,n-1]\times[2,k]
  $.
  If $n$ is odd, then in addition the facets include those obtained by tightening the lower bound in~\cref{eq:htsasm:mid-row-pref} to equality for
  $
    j\in [k]
  $,
  and the upper bound for
  $
    j \in [2,k]
  $;
  and by tightening the lower bound in~\cref{eq:htsasm:mid-col-pref} to equality for
  $
    i \in [k-1]
  $,
  and the upper bound for
  $
    i \in [2,k-1]
  $.
  In particular, the number of facets of $P^\core_\HTSASM$ is $2\bigl((n-2)^2+\chi_{2 \mid n}\bigr)$.
\end{theorem}
\begin{proof}
  The facets are obtained by tightening a single inequality in one of~\cref{eq:htsasm:row-prefix,eq:htsasm:col-prefix} to equality for the index pairs listed in the statement of the theorem; furthermore, if $n$ is odd, then we additionally obtain facets by tightening a single inequality in one of~\cref{eq:htsasm:mid-row-pref,eq:htsasm:mid-col-pref} to equality.
  We call the instances of the lower bounds in~\cref{eq:htsasm:row-prefix} that are tightened to equality the \emph{horizontal facet lower bounds}, and we define the \emph{horizontal facet upper bounds} analogously.
  Likewise, we call the instances of the lower bounds in~\cref{eq:htsasm:col-prefix} that are tightened to equality the \emph{vertical facet lower bounds}, and we define the \emph{vertical facet upper bounds} analogously.
  When $n$ is odd, we also call the instances of the lower bounds in~\cref{eq:htsasm:mid-row-pref} that are tightened to equality the \emph{middle-row facet lower bounds}, and we define the \emph{middle-row facet upper bounds} analogously.
  Likewise, tightening the bounds in~\cref{eq:htsasm:mid-col-pref} yields the \emph{middle-column facet lower bounds} and \emph{middle-column facet upper bounds}.
  We refer to the union of these families as the \emph{facet inequalities}.

  We proceed in two steps.
  First, we show that the facet inequalities together with the equations in~\cref{eq:htsasm:row-sum,eq:htsasm:col-sum} imply every inequality in~\cref{eq:htsasm:row-prefix,eq:htsasm:col-prefix}.
  When $n$ is odd, they also imply every inequality in~\cref{eq:htsasm:mid-row-pref,eq:htsasm:mid-col-pref}.
  Then, for every facet inequality, we construct a core of an $n\times n$ matrix violating that facet inequality and no other, thereby proving that no facet inequality is redundant.
  The core $\ol Y$ constructed in the second step of the proof of \cref{thm:htsasm:dim} satisfies all inequalities in the core description strictly, hence no facet inequality is an implicit equation.
  Thus, the two steps together imply that the facet inequalities form a minimal system that, extended with the defining equations, describes $P^\core_\HTSASM$, which proves the theorem.

  \medskip
  Now we prove that the facet inequalities together with the equations in~\cref{eq:htsasm:row-sum,eq:htsasm:col-sum} imply every inequality in~\cref{eq:htsasm:row-prefix,eq:htsasm:col-prefix}.
  When $n$ is odd, they also imply every inequality in~\cref{eq:htsasm:mid-row-pref,eq:htsasm:mid-col-pref}.
  Clearly, we need to treat only those inequalities that are \emph{non-facet} inequalities, namely, the lower bounds in~\cref{eq:htsasm:row-prefix} for
  \[
    (i,j)\in \bigl(\{1\}\times[2,n-1]\bigr) \cup \bigl([2,k]\times\{n-1\}\bigr)
  \]
  and the upper bounds for
  \[
    (i,j)\in \bigl(\{1\}\times[n-2]\bigr) \cup \bigl([2,k]\times\{1\}\bigr);
  \]
  and the lower bounds in~\cref{eq:htsasm:col-prefix} for
  \[
    (i,j)\in \bigl([n-1]\times\{1\}\bigr) \cup \bigl(\{n-1\}\times[2,k]\bigr)
  \]
  and the upper bounds for
  \[
    (i,j)\in \bigl([n-1]\times\{1\}\bigr) \cup \bigl(\{1\}\times[2,k]\bigr).
  \]
  If $n$ is odd, then in addition the only non-facet inequalities in~\cref{eq:htsasm:mid-row-pref,eq:htsasm:mid-col-pref} are the upper bound in~\cref{eq:htsasm:mid-row-pref} for $j=1$, the lower bound in~\cref{eq:htsasm:mid-col-pref} for $i=k$, and the upper bounds in~\cref{eq:htsasm:mid-col-pref} for $i\in\{1,k\}$.

  \smallskip
  We start with the non-facet inequalities in~\cref{eq:htsasm:row-prefix}.
  For $i\in[k]$ and $j\in[n-1]$, let $z_{i,j}$ denote the left-hand side of~\cref{eq:htsasm:row-prefix}, and extend the notation by $z_{i,n}=1$ using the equations in~\cref{eq:htsasm:row-sum}.

  We first record that every entry in the first core column is non-negative.
  Indeed, the horizontal facet lower bound at $(1,1)$ gives $y_{1,1}\geq 0$, and for each $i\in[2,k]$ the horizontal facet lower bound at $(i,1)$ gives $y_{i,1}\geq 0$.
  Moreover, for $i\in[2,k]$ the horizontal facet upper bound at $(i,n-1)$ yields $z_{i,n-1}\leq1$, hence $y_{n+1-i,1}=z_{i,n}-z_{i,n-1}=1-z_{i,n-1}\geq0$.
  Likewise, the horizontal facet upper bound at $(1,n-1)$ gives $y_{n,1}=1-z_{1,n-1}\geq0$.
  If $n$ is odd, then the middle-row facet lower bound in~\cref{eq:htsasm:mid-row-pref} for $j=1$ gives $y_{\ell,1}\geq 0$, where $\ell=\ceil*{n/2}$.
  Consequently $y_{i',1}\geq 0$ for all $i'\in[n]$, and by~\cref{eq:htsasm:col-sum} for $j=1$, we also have $y_{i',1}\leq 1$ for all $i'\in[n]$.
  In particular, the non-facet upper bounds in~\cref{eq:htsasm:row-prefix} at $(i,1)$ for $i\in[2,k]$ hold.

  Next, let $i\in[2,k]$.
  Using again $z_{i,n}=1$, we obtain $z_{i,n-1}=1-y_{n+1-i,1}\geq0$ because $y_{n+1-i,1}\leq1$.
  This proves the non-facet lower bounds in~\cref{eq:htsasm:row-prefix} at $(i,n-1)$ for $i\in[2,k]$.

  Now we treat the non-facet bounds in~\cref{eq:htsasm:row-prefix} for $i=1$.
  For $j\in[k]$, we have $y_{1,j}\geq 0$: for $j=1$ this is the horizontal facet lower bound at $(1,1)$, and for $j\in[2,k]$ it is the vertical facet lower bound in~\cref{eq:htsasm:col-prefix} at $(1,j)$.
  If $n$ is odd, then additionally $y_{1,\ell}\geq 0$ by the middle-column facet lower bound in~\cref{eq:htsasm:mid-col-pref} for $i=1$.
  Furthermore, for each $j'\in[2,k]$, the vertical facet upper bound in~\cref{eq:htsasm:col-prefix} at $(n-1,j')$ gives $\sum_{i'=1}^{n-1}y_{i',j'}\leq 1$, hence $y_{n,j'}=1-\sum_{i'=1}^{n-1}y_{i',j'}\geq 0$ by~\cref{eq:htsasm:col-sum}.
  The expression for $z_{1,j}$ is obtained from that of $z_{1,j-1}$ by adding $y_{1,j}$ if $j\leq \ell$ and $y_{n,n+1-j}$ if $j>\ell$.
  In the former case, this new term is non-negative by the previous paragraph; in the latter case, we have $j\geq \ell+1$ and $j\leq n-1$, hence $2\leq n+1-j\leq n-\ell = k$, so the added term is among the entries $y_{n,j'}$ with $j'\in[2,k]$, which are non-negative.
  Therefore, $z_{1,j}\geq z_{1,j-1}$ for every $j\in[2,n-1]$, and since the horizontal facet lower bound in~\cref{eq:htsasm:row-prefix} at $(1,1)$ gives $z_{1,1}\geq 0$, we obtain $z_{1,j}\geq 0$ for every $j\in[2,n-1]$.

  Finally, for $j\in[n-2]$, the expression for $z_{1,j}$ is obtained from that of $z_{1,j+1}$ by removing a single term, namely $y_{1,j+1}$ if $j+1\leq \ell$ and $y_{n,n-j}$ if $j+1>\ell$.
  In both cases this removed term is non-negative by the previous paragraph.
  Hence $z_{1,j}\leq z_{1,n-1}$ for all $j\in[n-2]$, and since the horizontal facet upper bound at $(1,n-1)$ gives $z_{1,n-1}\leq 1$, we conclude $z_{1,j}\leq 1$ for all $j\in[n-2]$.
  This proves the non-facet upper bounds in~\cref{eq:htsasm:row-prefix} at $(1,j)$ for $j\in[n-2]$.
  Altogether, we have derived all inequalities in~\cref{eq:htsasm:row-prefix}.

  \smallskip
  We now turn to the non-facet inequalities in~\cref{eq:htsasm:col-prefix}.
  For $j=1$ and $i\in[n-1]$, we have
  $
    0 \leq \sum_{i'=1}^{i} y_{i',1} \leq \sum_{i'=1}^{n} y_{i',1} = 1
  $
  by $y_{i',1}\geq 0$ for all $i'\in[n]$ and by~\cref{eq:htsasm:col-sum} for $j=1$.
  This proves all non-facet lower and upper bounds in~\cref{eq:htsasm:col-prefix} with $j=1$.

  Next, let $j\in[2,k]$.
  For the non-facet upper bound in~\cref{eq:htsasm:col-prefix} at $(1,j)$, observe that $y_{1,j}$ is one of the summands of $z_{1,n}=1$, and all summands of $z_{1,n}$ are non-negative; hence $y_{1,j}\leq 1$.
  For the non-facet lower bound in~\cref{eq:htsasm:col-prefix} at $(n-1,j)$, we write
  $
    \sum_{i'=1}^{n-1} y_{i',j} = 1-y_{n,j}
  $
  using the equations in~\cref{eq:htsasm:col-sum}.
  It remains to show that $y_{n,j}\leq 1$.
  For $j\in[2,k]$, the entry $y_{n,j}$ appears as a summand of $z_{1,n-1}$ (take $j'=n+1-j\in[\ell+1,n-1]$ in the second sum), and every other summand of $z_{1,n-1}$ is non-negative.
  Thus $y_{n,j}\leq z_{1,n-1}\leq 1$ by the horizontal facet upper bound at $(1,n-1)$, and therefore $1-y_{n,j}\geq 0$ as desired.
  This completes the derivation of all inequalities in~\cref{eq:htsasm:col-prefix}.

  \smallskip
  Assume now that $n$ is odd.
  The only non-facet inequality in~\cref{eq:htsasm:mid-row-pref} is the upper bound for $j=1$, i.e., $y_{\ell,1}\leq 1$, which follows from $\sum_{i'=1}^{n}y_{i',1}=1$, since every summand is non-negative.
  For the bounds in~\cref{eq:htsasm:mid-col-pref}, the non-facet upper bound for $i=1$ is $y_{1,\ell}\leq 1$, which follows since every summand in the definition of $z_{1,n}$ is non-negative and this sum is prescribed to be $1$.
  Finally, for the remaining non-facet bounds in~\cref{eq:htsasm:mid-col-pref} at $i=k$, we use the identity
  $
    \sum_{j'=1}^{k} y_{\ell,j'} = \sum_{i'=1}^{k} y_{i',\ell},
  $
  which is obtained by subtracting the sum of the equations in~\cref{eq:htsasm:row-sum} over $i\in[k]$ from the sum of the equations in~\cref{eq:htsasm:col-sum} over $j\in[k]$.
  Consequently, the middle-row facet inequalities in~\cref{eq:htsasm:mid-row-pref} for $j=k$ give $0\leq \sum_{i'=1}^{k} y_{i',\ell}\leq 1$, exactly the missing lower and upper bounds in~\cref{eq:htsasm:mid-col-pref} for $i=k$.
  This completes the derivation of all non-facet inequalities.

  \medskip
  It remains to show that no facet inequality is redundant.
  For every facet inequality in \cref{thm:htsasm:facets}, we construct a certifying core in~$\R^C$ that violates exactly this facet inequality and satisfies every other facet inequality as well as all defining equations of~$P^\core_\HTSASM$.
  For the horizontal facet lower bounds, we construct certificates $L^{n,\mathrm H}_{i,j}\in\R^C$, and for the horizontal facet upper bounds, we construct certificates $U^{n,\mathrm H}_{i,j}\in\R^C$.
  When $n$ is odd, we also construct horizontal certificates $L^{n,\mathrm H}_{\ell,j}$ and $U^{n,\mathrm H}_{\ell,j}$.

  To obtain vertical certificates from horizontal ones, we use a \emph{core flipping operator} $\fl:\R^C\to\R^C$ defined as follows.
  If $n$ is even, then a core is an $n\times k$ matrix, and $\fl$ acts by transposing the upper $k\times k$ block and reflecting the lower $k \times k$ block across its antidiagonal.
  If $n$ is odd, then a core consists of the upper $(k+1) \times (k+1)$ block with the corner $(k+1,k+1)$ missing, together with the lower $k \times k$ block; in this case $\fl$ transposes the upper block and reflects the lower block across its antidiagonal.
  Using~$\fl$, we define the vertical certificates by
  \[
    L^{n,\mathrm V}_{i,j}=\fl\bigl(L^{n,\mathrm H}_{j,i}\bigr),
    \qquad
    U^{n,\mathrm V}_{i,j}=\fl\bigl(U^{n,\mathrm H}_{j,i}\bigr).
  \]
  Since $\fl$ turns horizontal certificates into vertical ones and every vertical facet has a horizontal counterpart, it suffices to construct and verify the horizontal certificates.

  In order to build a certifying core for size $n$ from a certifying core $Z$ for size $\tilde n=n-2$, we set $\tilde k=\floor*{\tilde n/2}$ and define four \emph{extension operators}, for odd $n$, as
  \begin{center}
    \begin{tikzpicture}[scale=.975, baseline=(current bounding box.center)]
      \node at (-1.4,4.5) {$\extUL(Z)\ =$};
      \node[draw=none] at (5.15,4.5) {$,$};
      \draw[fill=cyan!15]  (1,1) -- (1,8) -- (5,8) -- (5,5) -- (4,5) -- (4,1) -- (1,1);
      \draw (0,0) -- (0,9) -- (5,9) -- (5,5) -- (4,5) -- (4,0) -- (0,0);

      \node at (0.5,8.5) {$1$};
      \node at (1.5,8.5) {$0$};
      \ThreeDotsAt{(2.5,8.5)}{.21213203}{0}
      \node at (3.5,8.5) {$0$};
      \node at (4.5,8.5) {$0$};

      \node at (0.5,7.5) {$0$};
      \ThreeDotsAt{(0.5,6.5)}{.21213203}{90}
      \node at (0.5,5.5) {$0$};

      \node at (1.5,7.5) {$z_{1,1}$};
      \ThreeDotsAt{(2.5,7.5)}{.21213203}{0}
      \node at (3.5,7.5) {$z_{1,\tilde k}$};
      \node at (4.5,7.5) {$z_{1,\tilde k+1}$};
      \ThreeDotsAt{(1.5,6.5)}{.21213203}{90}
      \ThreeDotsAt{(4.5,6.5)}{.21213203}{90}
      \ThreeDotsAt{(3.5,6.5)}{.21213203}{90}
      \node at (1.5,5.5) {$z_{\tilde k,1}$};
      \ThreeDotsAt{(2.5,5.5)}{.21213203}{0}
      \node at (3.5,5.5) {$z_{\tilde k,\tilde k}$};
      \node at (4.5,5.5) {$z_{\tilde k,\tilde k+1}$};

      \node at (0.5,0.5) {$0$};
      \node at (1.5,0.5) {$0$};
      \node at (3.5,0.5) {$0$};
      \ThreeDotsAt{(2.5,0.5)}{.21213203}{0}

      \node at (0.5,4.5) {$0$};
      \node at (1.5,4.5) {$z_{\tilde k+1,1}$};
      \ThreeDotsAt{(2.5,4.5)}{.21213203}{0}
      \node at (3.5,4.5) {$z_{\tilde k+1,\tilde k}$};

      \node at (0.5,3.5) {$0$};
      \ThreeDotsAt{(0.5,2.5)}{.21213203}{90}
      \node at (0.5,1.5) {$0$};

      \node at (1.5,3.5) {$z_{\tilde n-\tilde k,1}$};
      \ThreeDotsAt{(2.5,3.5)}{.21213203}{0}
      \node at (3.5,3.5) {$z_{\tilde n-\tilde k,\tilde k}$};
      \ThreeDotsAt{(1.5,2.5)}{.21213203}{90}
      \ThreeDotsAt{(3.5,2.5)}{.21213203}{90}
      \node at (1.5,1.5) {$z_{\tilde n,1}$};
      \ThreeDotsAt{(2.5,1.5)}{.21213203}{0}
      \node at (3.5,1.5) {$z_{\tilde n,\tilde k}$};

      \begin{scope}[xshift=8cm]
        \node at (-1.4,4.5) {$\extUR(Z)\ =$};
        \node[draw=none] at (5.15,4.5) {$,$};
        \draw[fill=cyan!15]  (0,1) rectangle (3,8);
        \draw[fill=cyan!15]  (4,5) rectangle (5,8);
        \draw (0,0) -- (0,9) -- (5,9) -- (5,5) -- (4,5) -- (4,0) -- (0,0);

        \node at (3.5,8.5) {$1$};
        \node at (0.5,8.5) {$0$};
        \node at (2.5,8.5) {$0$};
        \node at (4.5,8.5) {$0$};
        \ThreeDotsAt{(1.5,8.5)}{.21213203}{0}

        \node at (3.5,7.5) {$0$};
        \ThreeDotsAt{(3.5,6.5)}{.21213203}{90}
        \node at (3.5,5.5) {$0$};

        \node at (0.5,7.5) {$z_{1,1}$};
        \ThreeDotsAt{(1.5,7.5)}{.21213203}{0}
        \node at (2.5,7.5) {$z_{1,\tilde k}$};
        \node at (4.5,7.5) {$z_{1,\tilde k+1}$};
        \ThreeDotsAt{(0.5,6.5)}{.21213203}{90}
        \ThreeDotsAt{(4.5,6.5)}{.21213203}{90}
        \ThreeDotsAt{(2.5,6.5)}{.21213203}{90}
        \node at (0.5,5.5) {$z_{\tilde k,1}$};
        \ThreeDotsAt{(1.5,5.5)}{.21213203}{0}
        \node at (2.5,5.5) {$z_{\tilde k,\tilde k}$};
        \node at (4.5,5.5) {$z_{\tilde k,\tilde k+1}$};

        \node at (0.5,0.5) {$0$};
        \ThreeDotsAt{(1.5,0.5)}{.21213203}{0}
        \node at (2.5,0.5) {$0$};
        \node at (3.5,0.5) {$0$};

        \node at (3.5,4.5) {$0$};
        \node at (0.5,4.5) {$z_{\tilde k+1,1}$};
        \ThreeDotsAt{(1.5,4.5)}{.21213203}{0}
        \node at (2.5,4.5) {$z_{\tilde k+1,\tilde k}$};

        \node at (3.5,3.5) {$0$};
        \ThreeDotsAt{(3.5,2.5)}{.21213203}{90}
        \node at (3.5,1.5) {$0$};

        \node at (0.5,3.5) {$z_{\tilde n-\tilde k,1}$};
        \ThreeDotsAt{(1.5,3.5)}{.21213203}{0}
        \node at (2.5,3.5) {$z_{\tilde n-\tilde k,\tilde k}$};
        \ThreeDotsAt{(0.5,2.5)}{.21213203}{90}
        \ThreeDotsAt{(2.5,2.5)}{.21213203}{90}
        \node at (0.5,1.5) {$z_{\tilde n,1}$};
        \ThreeDotsAt{(1.5,1.5)}{.21213203}{0}
        \node at (2.5,1.5) {$z_{\tilde n,\tilde k}$};
      \end{scope}
    \end{tikzpicture}

    \vspace{18pt}

    \begin{tikzpicture}[scale=.975, baseline=(current bounding box.center)]
      \node at (-1.4,4.5) {$\extBL(Z)\ =$};
      \node[draw=none] at (5.15,4.5) {$,$};
      \draw[fill=cyan!15] (1,6) rectangle (5,9);
      \draw[fill=cyan!15] (1,4) rectangle (4,5);
      \draw[fill=cyan!15] (1,0) rectangle (4,3);
      \draw (0,0) -- (0,9) -- (5,9) -- (5,5) -- (4,5) -- (4,0) -- (0,0);

      \node at (0.5,8.5) {$0$};
      \ThreeDotsAt{(0.5,7.5)}{.21213203}{90}
      \node at (0.5,6.5) {$0$};

      \node at (1.5,8.5) {$z_{1,1}$};
      \ThreeDotsAt{(2.5,8.5)}{.21213203}{0}
      \node at (3.5,8.5) {$z_{1,\tilde k}$};
      \node at (4.5,8.5) {$z_{1,\tilde k+1}$};
      \ThreeDotsAt{(1.5,7.5)}{.21213203}{90}
      \ThreeDotsAt{(3.5,7.5)}{.21213203}{90}
      \ThreeDotsAt{(4.5,7.5)}{.21213203}{90}
      \node at (1.5,6.5) {$z_{\tilde k,1}$};
      \ThreeDotsAt{(2.5,6.5)}{.21213203}{0}
      \node at (3.5,6.5) {$z_{\tilde k,\tilde k}$};
      \node at (4.5,6.5) {$z_{\tilde k,\tilde k+1}$};

      \node at (0.5,5.5) {$1$};
      \node at (1.5,5.5) {$0$};
      \node at (3.5,5.5) {$0$};
      \node at (4.5,5.5) {$0$};
      \ThreeDotsAt{(2.5,5.5)}{.21213203}{0}

      \node at (0.5,3.5) {$0$};
      \node at (1.5,3.5) {$0$};
      \node at (3.5,3.5) {$0$};
      \ThreeDotsAt{(2.5,3.5)}{.21213203}{0}

      \node at (0.5,4.5) {$0$};
      \node at (1.5,4.5) {$z_{\tilde k+1,1}$};
      \ThreeDotsAt{(2.5,4.5)}{.21213203}{0}
      \node at (3.5,4.5) {$z_{\tilde k+1,\tilde k}$};

      \node at (0.5,2.5) {$0$};
      \ThreeDotsAt{(0.5,1.5)}{.21213203}{90}
      \node at (0.5,0.5) {$0$};

      \node at (1.5,2.5) {$z_{\tilde n-\tilde k,1}$};
      \ThreeDotsAt{(2.5,2.5)}{.21213203}{0}
      \node at (3.5,2.5) {$z_{\tilde n-\tilde k,\tilde k}$};
      \ThreeDotsAt{(1.5,1.5)}{.21213203}{90}
      \ThreeDotsAt{(3.5,1.5)}{.21213203}{90}
      \node at (1.5,0.5) {$z_{\tilde n,1}$};
      \ThreeDotsAt{(2.5,0.5)}{.21213203}{0}
      \node at (3.5,0.5) {$z_{\tilde n,\tilde k}$};

      \begin{scope}[xshift=8cm]
        \node at (-1.4,4.5) {$\extBR(Z)\ =$};
        \node[draw=none] at (5.15,4.5) {$.$};
        \draw[fill=cyan!15] (0,6) rectangle (3,9);
        \draw[fill=cyan!15] (4,6) rectangle (5,9);
        \draw[fill=cyan!15] (0,4) rectangle (3,5);
        \draw[fill=cyan!15] (0,0) rectangle (3,3);
        \draw (0,0) -- (0,9) -- (5,9) -- (5,5) -- (4,5) -- (4,0) -- (0,0);

        \node at (3.5,8.5) {$0$};
        \ThreeDotsAt{(3.5,7.5)}{.21213203}{90}
        \node at (3.5,6.5) {$0$};

        \node at (0.5,8.5) {$z_{1,1}$};
        \ThreeDotsAt{(1.5,8.5)}{.21213203}{0}
        \node at (2.5,8.5) {$z_{1,\tilde k}$};
        \node at (4.5,8.5) {$z_{1,\tilde k+1}$};
        \ThreeDotsAt{(0.5,7.5)}{.21213203}{90}
        \ThreeDotsAt{(2.5,7.5)}{.21213203}{90}
        \ThreeDotsAt{(4.5,7.5)}{.21213203}{90}
        \node at (0.5,6.5) {$z_{\tilde k,1}$};
        \ThreeDotsAt{(1.5,6.5)}{.21213203}{0}
        \node at (2.5,6.5) {$z_{\tilde k,\tilde k}$};
        \node at (4.5,6.5) {$z_{\tilde k,\tilde k+1}$};

        \node at (3.5,5.5) {$1$};
        \node at (0.5,5.5) {$0$};
        \node at (4.5,5.5) {$0$};
        \node at (2.5,5.5) {$0$};
        \ThreeDotsAt{(1.5,5.5)}{.21213203}{0}

        \node at (0.5,3.5) {$0$};
        \ThreeDotsAt{(1.5,3.5)}{.21213203}{0}
        \node at (2.5,3.5) {$0$};
        \node at (3.5,3.5) {$0$};

        \node at (0.5,4.5) {$z_{\tilde k+1,1}$};
        \node at (3.5,4.5) {$0$};
        \ThreeDotsAt{(1.5,4.5)}{.21213203}{0}
        \node at (2.5,4.5) {$z_{\tilde k+1,\tilde k}$};

        \node at (3.5,2.5) {$0$};
        \ThreeDotsAt{(3.5,1.5)}{.21213203}{90}
        \node at (3.5,0.5) {$0$};

        \node at (0.5,2.5) {$z_{\tilde n-\tilde k,1}$};
        \ThreeDotsAt{(1.5,2.5)}{.21213203}{0}
        \node at (2.5,2.5) {$z_{\tilde n-\tilde k,\tilde k}$};
        \ThreeDotsAt{(0.5,1.5)}{.21213203}{90}
        \ThreeDotsAt{(2.5,1.5)}{.21213203}{90}
        \node at (0.5,0.5) {$z_{\tilde n,1}$};
        \ThreeDotsAt{(1.5,0.5)}{.21213203}{0}
        \node at (2.5,0.5) {$z_{\tilde n,\tilde k}$};
      \end{scope}
    \end{tikzpicture}
  \end{center}
  To obtain the corresponding definition for even $n$, remove the middle row and the half-column on the right.
  More precisely, for a core of an $(n-2) \times (n-2)$ matrix, the operators $\extUL$ and $\extUR$ adjoin a new first row and a new last row to the core; furthermore, $\extUL$ adjoins a new first column, whereas $\extUR$ adjoins a new last column of height $n$ (before the half column if $n$ is odd).
  All newly created entries are set to $0$, except for a single new entry equal to $1$ in position $(1,1)$ for $\extUL$ and in position $(1,k)$ for $\extUR$.
  Hence both $\extUL$ and $\extUR$ yield the core of an $n\times n$ matrix.

  The operators $\extBL$ and $\extBR$ adjoin two new middle rows to the core of an $(n-2) \times (n-2)$ matrix: if $n$ is odd, then the new rows are inserted to the positions $k$ and $k+2$; if $n$ is even, then they are inserted to the positions $k$ and $k+1$.
  Moreover, $\extBL$ adjoins a new first column, whereas $\extBR$ adjoins a new last column.
  Again, all newly created entries are set to $0$, except for a single new entry equal to $1$ in position $(k,1)$ for $\extBL$ and in position $(k,k)$ for $\extBR$.
  Therefore, $\extBL$ and $\extBR$ again yield the core of an $n\times n$ matrix.

  Now we are ready to construct $L^{n,\mathrm H}_{i,j}$, $L^{n,\mathrm V}_{i,j}$, $U^{n,\mathrm H}_{i,j}$, and $U^{n,\mathrm V}_{i,j}$ via an inductive approach.
  For $n=4$, set
  \begin{align*}
    L^{4,\mathrm H}_{1,1}
    &=
      \vcenter{\hbox{\begin{ytableau}
            -1 & 1\\
            1 & 0\\
            0 & 0\\
            1 & 0
          \end{ytableau}}}\,,
    &
      L^{4,\mathrm H}_{2,1}
    &=
      \vcenter{\hbox{\begin{ytableau}
            1 & 0\\
            -1 & 1\\
            1 & 0\\
            0 & 0
          \end{ytableau}}}\,,
    &
      L^{4,\mathrm H}_{2,2}
    &=
      \vcenter{\hbox{\begin{ytableau}
            0 & 1\\
            0 &-1\\
            1 & 1\\
            0 & 0
          \end{ytableau}}}\,,
    \\
    U^{4,\mathrm H}_{1,3}
    &=
      \vcenter{\hbox{\begin{ytableau}
            0 & 1\\
            1 &-1\\
            1 & 0\\
            -1 & 1
          \end{ytableau}}}\,,
    &
      U^{4,\mathrm H}_{2,2}
    &=
      \vcenter{\hbox{\begin{ytableau}
            0 & 0\\
            1 & 1\\
            0 &-1\\
            0 & 1
          \end{ytableau}}}\,,
    &
      U^{4,\mathrm H}_{2,3}
    &=
      \vcenter{\hbox{\begin{ytableau}
            0 & 0\\
            1 & 0\\
            -1 & 1\\
            1 & 0
          \end{ytableau}}}\,.
  \end{align*}
  For $n=5$, set
  \begin{align*}
    L^{5,\mathrm H}_{1,1}
    &=
      \vcenter{\hbox{\begin{ytableau}
            -1 & 1 & 0\\
            1 & 0 & 0\\
            0 & 0 & \none\\
            0 & 0 & \none\\
            1 & 0 & \none
          \end{ytableau}}}\,,
    &
      L^{5,\mathrm H}_{2,1}
    &=
      \vcenter{\hbox{\begin{ytableau}
            1 & 0 & 0\\
            -1 & 1 & 0\\
            0 & 0 & \none\\
            1 & 0 & \none\\
            0 & 0 & \none
          \end{ytableau}}}\,,
    &
      L^{5,\mathrm H}_{2,2}
    &=
      \vcenter{\hbox{\begin{ytableau}
            0 & 1 & 0\\
            0 &-1 & 1\\
            0 & 1 & \none\\
            1 & 0 & \none\\
            0 & 0 & \none
          \end{ytableau}}}\,,
    \\
    L^{5,\mathrm H}_{2,3}
    &=
      \vcenter{\hbox{\begin{ytableau}
            0 & 0 & 1\\
            0 & 0 &-1\\
            0 & 0 & \none\\
            1 & 1 & \none\\
            0 & 0 & \none
          \end{ytableau}}}\,,
    &
      L^{5,\mathrm H}_{3,1}
    &=
      \vcenter{\hbox{\begin{ytableau}
            0 & 0 & 0\\
            1 & 0 & 0\\
            -1 & 1 & \none\\
            0 & 0 & \none\\
            1 & 0 & \none
          \end{ytableau}}}\,,
    &
      L^{5,\mathrm H}_{3,2}
    &=
      \vcenter{\hbox{\begin{ytableau}
            0 & 0 & 0\\
            0 & 1 &-1\\
            0 &-1 & \none\\
            1 & 0 & \none\\
            0 & 1 & \none
          \end{ytableau}}}\,,\\
    U^{5,\mathrm H}_{1,4}
    &=
      \vcenter{\hbox{\begin{ytableau}
            1 & 0 & 0\\
            0 & 0 & 0\\
            0 & 0 & \none\\
            1 & 0 & \none\\
            -1 & 1 & \none
          \end{ytableau}}}\,,
    &
      U^{5,\mathrm H}_{2,2}
    &=
      \vcenter{\hbox{\begin{ytableau}
            0 & 0 & 1\\
            1 & 1 &-1\\
            0 & 0 & \none\\
            0 & 0 & \none\\
            0 & 0 & \none
          \end{ytableau}}}\,,
    &
      U^{5,\mathrm H}_{2,3}
    &=
      \vcenter{\hbox{\begin{ytableau}
            0 & 0 & 0\\
            0 & 1 & 1\\
            1 & 0 & \none\\
            0 &-1 & \none\\
            0 & 1 & \none
          \end{ytableau}}}\,,
    \\
    U^{5,\mathrm H}_{2,4}
    &=
      \vcenter{\hbox{\begin{ytableau}
            0 & 0 & 0\\
            1 & 0 & 0\\
            0 & 0 & \none\\
            -1 & 1 & \none\\
            1 & 0 & \none
          \end{ytableau}}}\,,
    &
      U^{5,\mathrm H}_{3,2}
    &=
      \vcenter{\hbox{\begin{ytableau}
            0 & 0 & 1\\
            0 & 0 & 1\\
            1 & 1 & \none\\
            0 & 0 & \none\\
            0 & 0 & \none
          \end{ytableau}}}\,.
  \end{align*}
  Furthermore, for $n=6$, we include the further explicit certificates
  \[
    L^{6,\mathrm H}_{2,4}=
    \vcenter{\hbox{\begin{ytableau}
          0&0&0\\
          0&0&0\\
          0&0&0\\
          0&0&1\\
          1&1&-1\\
          0&0&1
        \end{ytableau}}}\,,
    \quad
    L^{6,\mathrm H}_{3,4}=
    \vcenter{\hbox{\begin{ytableau}
          0&0&1\\
          0&0&0\\
          0&0&0\\
          1&1&-1\\
          0&0&1\\
          0&0&0
        \end{ytableau}}}\,,
    \quad
    U^{6,\mathrm H}_{2,2}=
    \vcenter{\hbox{\begin{ytableau}
          0&0&1\\
          1&1&-1\\
          0&0&1\\
          0&0&0\\
          0&0&0\\
          0&0&0
        \end{ytableau}}}\,,
    \quad
    U^{6,\mathrm H}_{3,2}=
    \vcenter{\hbox{\begin{ytableau}
          0&0&0\\
          0&0&1\\
          1&1&-1\\
          0&0&0\\
          0&0&0\\
          0&0&1
        \end{ytableau}}}\,,
  \]
  and, for $n=7$, we define
  \[
    L^{7,\mathrm H}_{2,5}=
    \vcenter{\hbox{\begin{ytableau}
          0&0&0&0\\
          0&0&0&0\\
          0&0&0&0\\
          0&0&0&\none\\
          0&0&1&\none\\
          1&1&-1&\none\\
          0&0&1&\none
        \end{ytableau}}}\,,
    \quad
    L^{7,\mathrm H}_{3,5}=
    \vcenter{\hbox{\begin{ytableau}
          0&0&0&1\\
          0&0&0&0\\
          0&0&0&0\\
          0&0&1&\none\\
          1&1&-1&\none\\
          0&0&1&\none\\
          0&0&0&\none
        \end{ytableau}}}\,,
  \]
  \[
    U^{7,\mathrm H}_{2,2}=
    \vcenter{\hbox{\begin{ytableau}
          0&0&1&0\\
          1&1&-1&0\\
          0&0&1&0\\
          0&0&0&\none\\
          0&0&0&\none\\
          0&0&0&\none\\
          0&0&0&\none
        \end{ytableau}}}\,,
    \quad
    U^{7,\mathrm H}_{3,2}=
    \vcenter{\hbox{\begin{ytableau}
          0&0&0&0\\
          0&0&1&0\\
          1&1&-1&0\\
          0&0&0&\none\\
          0&0&0&\none\\
          0&0&0&\none\\
          0&0&1&\none
        \end{ytableau}}}\,,
    \quad
    U^{7,\mathrm H}_{4,2}=
    \vcenter{\hbox{\begin{ytableau}
          0&0&0&1\\
          0&0&0&0\\
          0&0&1&0\\
          1&1&-1&\none\\
          0&0&0&\none\\
          0&0&1&\none\\
          0&0&0&\none
        \end{ytableau}}}\,.
  \]
  Let $n\geq 6$, and assume that all certificates for size $n-2$ are already defined.
  For $n\in\{6,7\}$, some certificates are given explicitly above.
  All remaining certificates for $n \geq 6$ are obtained by the recursive rules below, where the explicitly listed certificates for $n=6$ and $n=7$ are omitted from the recursive cases.
  For every remaining $(i,j)\in \{(1,1)\}\cup([2,k]\times[n-2])$, and for odd $n$ also for every remaining $(i,j)=(k+1,j)$ with $j\in[k]$, define $L^{n,\mathrm H}_{i,j}$ recursively by
  \[
    L^{n,\mathrm H}_{i,j}=
    \begin{cases}
      \extBR\bigl(L^{n-2,\mathrm H}_{i,1}\bigr)    & \text{if } j=1 \text{ and } i\in[k-1],\\
      \extUR\bigl(L^{n-2,\mathrm H}_{i-1,1}\bigr)  & \text{if } j=1 \text{ and } i\in\{k,k+1\},\\
      \extBL\bigl(L^{n-2,\mathrm H}_{i,j-1}\bigr)  & \text{if } i\in[2,k-1] \text{ and } j\in[2,n-3],\\
      \extBR\bigl(L^{n-2,\mathrm H}_{i,n-4}\bigr)  & \text{if } i\in[2,k-1] \text{ and } j=n-2,\\
      \extUR\bigl(L^{n-2,\mathrm H}_{k-1,j-2}\bigr)& \text{if } i=k \text{ and } j\in[\ell+2,n-2],\\
      \extUL\bigl(L^{n-2,\mathrm H}_{k-1,j-1}\bigr)& \text{if } i=k \text{ and } j\in[2,\ell+1],\\
      \extUL\bigl(L^{n-2,\mathrm H}_{k,j-1}\bigr)  & \text{if } 2 \nmid n,\ i=k+1, \text{ and } j\in[2,k].
    \end{cases}
  \]
  For every remaining $(i,j)\in \{(1,n-1)\}\cup([2,k]\times[2,n-1])$, and for odd $n$ also for every remaining $(i,j)=(k+1,j)$ with $j\in[2,k]$, define $U^{n,\mathrm H}_{i,j}$ recursively by
  \[
    U^{n,\mathrm H}_{i,j}=
    \begin{cases}
      \extBR\bigl(U^{n-2,\mathrm H}_{i,2}\bigr)    & \text{if } j=2 \text{ and } i\in[k-1],\\
      \extUR\bigl(U^{n-2,\mathrm H}_{i-1,2}\bigr)  & \text{if } j=2 \text{ and } i\in\{k,k+1\},\\
      \extBR\bigl(U^{n-2,\mathrm H}_{i,j-2}\bigr)  & \text{if } (i \in [2,k-1] \text{ and } j\in[\ell+2,n-1]) \text{ or } (i, j) = (1, n-1),\\
      \extBL\bigl(U^{n-2,\mathrm H}_{i,j-1}\bigr)  & \text{if } i\in[2,k-1] \text{ and } j\in[3,\ell+1],\\
      \extUR\bigl(U^{n-2,\mathrm H}_{k-1,j-2}\bigr)& \text{if } i=k \text{ and } j\in[\ell+2,n-1],\\
      \extUL\bigl(U^{n-2,\mathrm H}_{k-1,j-1}\bigr)& \text{if } i=k \text{ and } j\in[3,\ell+1],\\
      \extUL\bigl(U^{n-2,\mathrm H}_{k,j-1}\bigr)  & \text{if } 2 \nmid n,\ i=k+1, \text{ and } j\in[3,k].
    \end{cases}
  \]

  \medskip
  We prove by induction on~$n$ that the cores defined above are the desired certificates.
  For the explicitly listed cores with $n\in[4,7]$, this is checked by direct inspection.
  This completes the proof for $n\in\{4,5\}$.
  Now let $n\geq 6$ and assume that the claim holds for size~$n-2$.
  Let $Y$ be a certifying core for size~$n$ that is defined recursively, and let $Z$ be its predecessor for size~$n-2$ so that $Y=\ext(Z)$, where $\ext$ is one of $\extUL,\extUR,\extBL,\extBR$.
  By the definition of the extension operators, $Y$ agrees with $Z$ on an embedded copy of the core positions for size~$n-2$, and every newly created entry of~$Y$ is~$0$ except for a single new entry equal to~$1$ at the corner cell adjacent to the embedded copy.

  \smallskip
  We now check the facet inequalities for size~$n$.
  Every facet inequality is one of the row-prefix or column-prefix bounds in~\cref{eq:htsasm:row-prefix,eq:htsasm:col-prefix}; when $n$ is odd, it may also be one of the middle-row-prefix or middle-column-prefix bounds in~\cref{eq:htsasm:mid-row-pref,eq:htsasm:mid-col-pref}.
  Fix such a facet inequality and consider its left-hand side.

  If the endpoint of the defining prefix lies in the embedded copy, then any newly created entries in that prefix are equal to~$0$, since the unique new~$1$ lies outside the embedded copy by construction.
  Hence the prefix sum on~$Y$ coincides with the corresponding prefix sum on~$Z$ after the evident index shift determined by the extension operator.
  Therefore, the inequality holds for~$Y$ by the induction hypothesis, except possibly for the unique facet inequality violated by~$Z$.
  Moreover, the case distinctions in the recursive definitions of the certificates were chosen precisely so that this uniquely violated facet inequality of~$Z$ is transported by the extension operator to the intended facet inequality for~$Y$.

  If the endpoint of the defining prefix lies in a newly created row or column, then either the defining prefix involves only newly created entries, or any newly created entries in that prefix are equal to~$0$.
  In the former case, all such entries are~$0$ except for the single inserted~$1$, so the corresponding prefix sum belongs to~$\{0,1\}$ and satisfies the facet lower and upper bounds.
  In the latter case, the prefix sum on~$Y$ coincides with the corresponding prefix sum on~$Z$ after the evident index shift determined by the extension operator, and the construction ensures that this prefix sum on $Z$ is non-violating.

  Finally, it is straightforward to verify from the definitions of $\extUL,\extUR,\extBL,\extBR$ that all defining equations of $P^\core_\HTSASM$ are preserved under these extensions: the equations corresponding to the embedded part are inherited from~$Z$, and the newly created rows and columns clearly satisfy the required equations.

  Thus, for every facet inequality for size~$n$, we have constructed a core that violates that facet inequality and satisfies every other facet inequality as well as all defining equations.
  This shows that no facet inequality is redundant; hence the facet inequalities in \cref{thm:htsasm:facets}, together with the defining equations, form a minimal description of $P^\core_\HTSASM$.
\end{proof}

\section{Quarter-turn symmetric ASMs (QTSASMs)}\label{sec:QTSASM}
In this class, we impose invariance under rotation by $\pi/2$.
The corresponding symmetry subgroup is $\mathcal{G} = \{\mathcal{I}, \mathcal{R}_{\pi/2}, \mathcal{R}_{\pi},\allowbreak \mathcal{R}_{-\pi/2}\}$.
Let $P_\QTS$ denote the linear subspace of quarter-turn symmetric real matrices, i.e.,
\[
  P_\QTS
  =
  \left\{
    X \in \R^{n \times n} : x_{i,j} = x_{j,n+1-i}\ \forall i,j \in [n]
  \right\}.
\]
Clearly, any QTSASM satisfies the ASM constraints in~\crefrange{eq:asm:real}{eq:asm:col-sum} and also the symmetry constraints defining $P_\QTS$; thus $P_\QTSASM \subseteq P_\ASM \cap P_\QTS$.
In contrast to the half-turn symmetric case in \cref{sec:HTSASM}, these constraints are not sufficient in general: the polytope $P_\ASM \cap P_\QTS$ can have fractional vertices and strictly contain $P_\QTSASM$.
For example, it is easy to see that the fractional matrix $\frac{1}{2}(I_n + I'_n)$ is a fractional vertex of $P_\ASM \cap P_\QTS$ when $n \geq 2$; e.g., for $n = 4$, we have
\[
  \frac{I_4 + I'_4}{2} =
  \begin{mymatrix}
    \sfrac{1}{2} &           0  &           0  & \sfrac{1}{2} \\
    0            & \sfrac{1}{2} & \sfrac{1}{2} & 0 \\
    0            & \sfrac{1}{2} & \sfrac{1}{2} & 0 \\
    \sfrac{1}{2} &           0  &           0  & \sfrac{1}{2}
  \end{mymatrix}.
\]
To obtain the actual hull, we first introduce a core and derive an exact parity-augmented inequality description; we then determine its dimension in height coordinates and construct an exponentially large family of Ferrers facets.
Before defining the core, we record two simple structural facts.
The first is the following parity obstruction, included for completeness.
\begin{lemma}\label{lem:QTSASM-parity}
  There is no $n \times n$ QTSASM if $n \equiv 2 \pmod 4$.
\end{lemma}
\begin{proof}
  Let $X$ be an $n\times n$ QTSASM, where $n=2k$ is even.
  The lines between rows $k$ and $k+1$ and between columns $k$ and $k+1$ divide the square into four disjoint $k \times k$ blocks:
  the upper-left, upper-right, lower-right, and lower-left blocks.
  By rotational symmetry, the total sum of entries is four times the sum of the entries in the upper-left $k \times k$ block and is therefore a multiple of $4$.
  On the other hand, in any ASM, each row sums to~$1$, so the total sum of all entries of $X$ is $n = 2k$.
  Thus $2k$ must be divisible by $4$, which means that $n$ is a multiple of~$4$.
\end{proof}

\begin{lemma}\label{lem:qtsasm:center}
  Let $n \geq 1$ be odd, and set $k = \floor*{n/2}$.
  For every QTSASM $X \in \{0,\pm1\}^{n \times n}$, the central entry is $x_{k+1,k+1} = (-1)^k$.
\end{lemma}
\begin{proof}
  Let $X$ be a QTSASM\@.
  By quarter-turn symmetry, every entry other than the central entry belongs to an orbit of four matrix positions.
  Hence the total sum of entries is $4t + x_{k+1,k+1}$ for an integer $t$.
  Since the total sum of entries is $n$ by the ASM constraints, we have $4t + x_{k+1,k+1} = n$.
  If $n \equiv 1 \pmod 4$, then we must have $x_{k+1,k+1} = 1$; if $n \equiv 3 \pmod 4$, then $x_{k+1,k+1} = -1$.
  In the former case, $k$ is even; in the latter, $k$ is odd.
  Therefore, we can equivalently write $x_{k+1,k+1} = (-1)^k$.
\end{proof}

\paragraph{Core and assembly map.}
Let $n \geq 1$, $k = \floor*{n/2}$, and $\ell = \ceil*{n/2}$.
Let the \emph{core} of a QTSASM be its upper-left $\ell \times k$ block, i.e.,
\[
  C = [\ell] \times [k]
\]
is the set of \emph{core positions}, and let $\pi_C : \R^{n \times n} \to \R^C$ be the coordinate-wise projection onto~$C$.
We now define the affine map $\varphi : \R^C \to \R^{n \times n}$ by
\[
  \varphi(Y)_{i,j} =
  \begin{cases}
    y_{i,j}         & \text{if } i \in [\ell], j \in [k],\\
    (-1)^k          & \text{if } 2 \nmid n, i = j = \ell,\\
    y_{j,n+1-i}      & \text{if } i \in [\ell+1,n], j \in [\ell],\\
    y_{n+1-i,n+1-j}  & \text{if } i \in [k+1,n], j \in [\ell+1,n],\\
    y_{n+1-j,i}      & \text{if } i \in [k], j \in [k+1,n]
  \end{cases}
\]
for $Y \in \R^C$ and $i,j \in [n]$.
By construction, $\varphi(Y)$ is quarter-turn symmetric, and for odd $n$ its central entry agrees with \cref{lem:qtsasm:center}.
Thus $\varphi$ places the core $Y$ into the upper-left rectangle $C$, fixes the central entry when $n$ is odd, and fills the remaining entries by successive quarter-turn rotations.
It is immediate from the definition that $\pi_C(\varphi(Y)) = Y$ for every $Y \in \R^C$.
Conversely, if $X \in \QTSASM(n)$, then $X$ is completely determined by its entries in $C$ together with the imposed quarter-turn symmetry and, for odd $n$, the central entry, which is uniquely determined by \cref{lem:qtsasm:center}.
Hence $\varphi(\pi_C(X)) = X$ for every $X \in \QTSASM(n)$, so $\varphi$ is an assembly map.

\medskip
Now we provide a valid inequality that cuts off the fractional vertex $\frac{I_4+I'_4}{2}$ of $P_\ASM \cap P_\QTS$.
Let $n=4$, and let $y \in \R^C$.
By rotational symmetry and the ASM constraints, the following inequalities are valid for the convex hull of integer QTSASM cores:
\begin{gather*}
  2y_{1,1} + y_{1,2} + y_{2,1} \leq 1,\\
  y_{2,1} + 2y_{2,2} + y_{1,2} \geq 1,\\
  y_{2,1} \geq 0,\\
  y_{2,1} + 2y_{2,2} \leq 1.\\
  \intertext{Turning all inequalities into upper bounds and adding them, we obtain}
  2 y_{1,1} \leq 1.\\
  \intertext{Since $y_{1,1}$ is integral for integer QTSASM cores, we can strengthen this bound to}
  y_{1,1} \leq \floor*{\sfrac{1}{2}} = 0.
\end{gather*}
Thus $y_{1,1} \leq 0$ is valid for the convex hull of integer cores of $4 \times 4$ QTSASMs.
However, the core of $\frac{I_4+I'_4}{2}$ violates this inequality.

\medskip
We prove that a family of parity inequalities derived in a similar manner cuts off every fractional vertex of $P_\ASM \cap P_\QTS$, and thus yields a description of the core polytope of QTSASMs.
We start with the description of the core polytope.
\begin{theorem}\label{thm:qtsasm:corepolytope}
  Let $n \geq 1$ and set $k = \floor*{n/2}$.
  Define
  \[
    \widetilde C =
    \begin{cases}
      [k] \times [n]                              & \text{if } 2 \mid n,\\
      [k] \times [n] \cup \{(k+1,j) : j \in [k]\} & \text{if } 2 \nmid n,
    \end{cases}
  \]
  and
  \[
    \mathcal{S} = \Bigl\{S \in \{0,\pm1\}^{\widetilde C} : 2 \mid \sum_{j'=j}^{n_i} s_{i,j'} + \sum_{i'=n+1-i}^ns_{j,i'}\ \forall (i,j) \in C\Bigr\},
  \]
  where $n_i = n$ for all $i \in [k]$, and, when $n$ is odd, $n_{k+1} = k$.
  Then the core polytope $P^\core_\QTSASM \subseteq \R^C$ of $n \times n$ QTSASMs is described by the following system.
  \begin{align}
                                                                       y_{i,j} &\in \R &\forall (i,j) \in C, \label{eq:qtsasm:real}\\
    0 \leq \sum_{j'=1}^{\min\{j,k\}}      y_{i,j'} + \sum_{j'=k+1}^{j} y_{n+1-j',i} &\leq 1 &\hspace*{-3cm}\forall (i,j) \in \widetilde C \setminus ([k] \times \{n\}), \label{eq:qtsasm:prefix}\\
                    \sum_{j'=1}^{k}      y_{i,j'} + \sum_{j'=k+1}^{n} y_{n+1-j',i} &= 1    &\forall i \in [k], \label{eq:qtsasm:full}\\
    \nonumber\dfrac{\displaystyle \sum_{(i,j) \in \widetilde C} s_{i,j} \left(\sum_{j'=1}^{\min\{j,k\}}  y_{i,j'} + \sum_{j'=k+1}^{j} y_{n+1-j',i} \right)}{2} &&\\
            &\hspace*{-1.5cm}\leq\floor*{\dfrac{|\{(i,j) \in \widetilde C : s_{i,j}=1\}| - |\{i \in [k] : s_{i,n}=-1\}|}{2}} &\forall S \in \mathcal{S}. \label{eq:qtsasm:cut}
  \end{align}
\end{theorem}
\begin{proof}
  We prove that the integer solutions to the system~\crefrange{eq:qtsasm:real}{eq:qtsasm:full} are exactly the cores of QTSASMs.
  We then show that the inequalities in~\cref{eq:qtsasm:cut} do not cut any integer solution, and the system~\crefrange{eq:qtsasm:real}{eq:qtsasm:cut} defines an integral polytope.

  \medskip
  First, let $X$ be an $n \times n$ QTSASM, and let $Y = \pi_C(X)$ be its core.
  Since $X$ is quarter-turn symmetric, we have $X = \varphi(Y)$ for the assembly map $\varphi$ defined above.
  By the definition of $\varphi$, for every $(i,j) \in \widetilde C$, the row-prefix sum within row $i$ up to column $j$
  can be written in terms of the core entries as
  \begin{equation}\label{eq:qtsasm:phi-row-prefix-used}
    \sum_{j'=1}^{j} x_{i,j'}
    = \sum_{j'=1}^{\min\{j,k\}} y_{i,j'}
    + \sum_{j'=k+1}^{j} y_{n+1-j',i}.
  \end{equation}
  Since every QTSASM is in particular an ASM, the ASM row-prefix bounds in~\cref{eq:asm:row-prefix}, together with the identity~\cref{eq:qtsasm:phi-row-prefix-used} for $(i,j) \in \widetilde C \setminus ([k] \times \{n\})$, yield the inequalities in~\cref{eq:qtsasm:prefix}.
  For $(i,j) \in [k] \times \{n\}$, the same identity, together with the ASM row-sum constraint in~\cref{eq:asm:row-sum}, gives the equations in~\cref{eq:qtsasm:full}.
  Thus the cores of QTSASMs satisfy the system~\crefrange{eq:qtsasm:real}{eq:qtsasm:full}.

  \medskip
  Conversely, let $Y$ be an integer solution to the system~\crefrange{eq:qtsasm:real}{eq:qtsasm:full}, and put $X = \varphi(Y)$.
  By construction, $X$ is an $n\times n$ integer matrix and its core is $Y$.
  The identity~\cref{eq:qtsasm:phi-row-prefix-used} shows that the inequalities in~\cref{eq:qtsasm:prefix} give the ASM row-prefix bounds for the first $k$ rows of $X$.
  For $j = n$, the same identity, together with the equations in~\cref{eq:qtsasm:full}, gives the ASM row-sum constraints for those rows.
  When $n$ is odd, the first $k$ prefixes in row $k+1$ lie in $\{0,1\}$ by~\cref{eq:qtsasm:prefix}.
  Although the sum of row $k+1$ is not prescribed to be one in~\cref{eq:qtsasm:full} when $n$ is odd, we now show that it is implied by the other constraints.
  Summing all equations in~\cref{eq:qtsasm:full}, every $y$-variable appears with coefficient $2$, except the variables $y_{k+1,j}$ for $j \in [k]$, which appear with coefficient $1$; hence $\sum_{j=1}^k y_{k+1,j} \equiv k \pmod 2$.
  Moreover, we have $\sum_{j=1}^k y_{k+1,j} \in \{0,1\}$ by~\cref{eq:qtsasm:prefix}.
  In particular, $\sum_{j=1}^k y_{k+1,j}$ is uniquely determined, and using the definition of $\varphi$, the sum of row $k+1$ becomes $2\sum_{j=1}^k y_{k+1,j} + (-1)^k = 1$.
  Thus the row-sum constraint for row $k+1$ follows from the system when $n$ is odd.
  Furthermore, by \cref{lem:vsasm:symmetricHalfExtends}, every prefix of row $k+1$ lies in $\{0,1\}$.
  The row constraints for the last $k$ rows and all column constraints follow by quarter-turn symmetry of $X$.
  Hence $X$ satisfies all ASM constraints from \cref{thm:ASMpolytope}, so $X$ is a QTSASM and $Y = \pi_C(X)$ is its core.
  This proves that the integer solutions to the system~\crefrange{eq:qtsasm:real}{eq:qtsasm:full} are exactly the cores of QTSASMs.

  \medskip
  It remains to prove that the cores of QTSASMs satisfy the inequalities in~\cref{eq:qtsasm:cut} and the system~\crefrange{eq:qtsasm:real}{eq:qtsasm:cut} defines an integral polytope.
  For $(i,j) \in \widetilde C$, define
  \begin{equation}\label{eq:qtsasm:z-def}
    z_{i,j} = \sum_{j'=1}^{\min\{j,k\}}  y_{i,j'} + \sum_{j'=k+1}^{j} y_{n+1-j',i},
  \end{equation}
  that is, $z_{i,j}$ denotes the left-hand side of~\cref{eq:qtsasm:prefix} for $(i,j) \in \widetilde C \setminus ([k] \times \{n\})$, and of~\cref{eq:qtsasm:full} when $(i,j) \in [k] \times \{n\}$.
  We record how $z_{i,j}$ changes with $j$.
  For $(i,j) \in \widetilde C$, we have
  \begin{equation}\label{eq:qtsasm:z-recursion:cases}
    z_{i,j} =
    \begin{cases}
      y_{i,1}                 & \text{if } j = 1,\\
      z_{i,j-1} + y_{i,j}     & \text{if } j \in [2,k],\\
      z_{i,j-1} + y_{n+1-j,i} & \text{if } j \in [k+1,n].
    \end{cases}
  \end{equation}
  Indeed, increasing $j$ by $1$ adds exactly one new term to the defining sum for $z_{i,j}$: when $j \leq k$, this term is $y_{i,j}$ in the first sum; when $j > k$, it is $y_{n+1-j,i}$ in the second sum.

  By definition, the prefix constraints in~\cref{eq:qtsasm:prefix} become
  \begin{align}
          0 \leq z&_{i,j} \leq 1                       & \forall (i,j) \in \widetilde C \setminus ([k] \times \{n\}),\label{eq:qtsasm:z-bounds}\\
    \intertext{likewise, the equations in~\cref{eq:qtsasm:full} turn into}
                 z&_{i,n} = 1                          &\forall i \in [k].\label{eq:qtsasm:z-rowsum}\\
    \intertext{Furthermore, we rearrange the recursion in~\cref{eq:qtsasm:z-recursion:cases} to the equivalent form}
                 z&_{i,1} - y_{i,1} = 0                 &\forall i \in [\ell],\label{eq:qtsasm:z-recursion1}\\
                 z&_{i,j} - z_{i,j-1} - y_{i,j} = 0      &\forall i \in [\ell], j \in [2,k],\label{eq:qtsasm:z-recursion2}\\
                 z&_{i,j} - z_{i,j-1} - y_{n+1-j,i} = 0  &\forall i \in [k], j \in [k+1,n],\label{eq:qtsasm:z-recursion3}\\
    \intertext{where $\ell = \ceil*{n/2}$. Finally, we add the redundant constraint}
    -\infty \leq z&_{i,n} \leq \infty                  &\forall i \in [k];\label{eq:qtsasm:redund}
    \intertext{if $n$ is odd, then we also add}
    -\infty \leq z&_{\ell,k} \leq \infty.\label{eq:qtsasm:redund2}
  \end{align}
  Here $\infty$ denotes a sufficiently large positive integer and $-\infty$ its negative, chosen so that every cut involving one of these bounds is redundant.
  Let the auxiliary system consist of the constraints in~\crefrange{eq:qtsasm:z-bounds}{eq:qtsasm:redund}, together with the constraint in~\cref{eq:qtsasm:redund2} when $n$ is odd.
  Clearly, the constraints in~\cref{eq:qtsasm:prefix,eq:qtsasm:full} in the $y$-variables are equivalent to this auxiliary system in the $y$- and $z$-variables.
  We consider the bounds in~\cref{eq:qtsasm:z-bounds,eq:qtsasm:z-rowsum} as entry bounds, and the remaining rows as our linear constraints.
  There are no effective entry bounds for the $y$-variables, which we express as $-\infty\leq y_{i,j}\leq\infty$ for every $(i,j)\in C$.

  We claim that the coefficient matrix of these linear constraints is bidirected.
  To see this, observe that each variable $y_{i,j}$ for $(i,j) \in C$ appears exactly twice with coefficient $-1$, namely in the equations defining $z_{i,j}$ and $z_{j,n+1-i}$.
  For $i \in [k],j \in [n-1]$ and, when $n$ is odd, also $i = \ell,j \in [k-1]$, the variable $z_{i,j}$ appears once with coefficient $1$ and once with coefficient $-1$, in the equations defining $z_{i,j}$ and $z_{i,j+1}$, respectively.
  For $i \in [k]$, the variable $z_{i,n}$ appears once with coefficient $1$ in~\cref{eq:qtsasm:z-recursion3} when defining $z_{i,n}$ itself and once with coefficient $1$ in~\cref{eq:qtsasm:redund}.
  When $n$ is odd, $z_{\ell,k}$ appears with coefficient $1$ in~\cref{eq:qtsasm:z-recursion2} and with coefficient $1$ in~\cref{eq:qtsasm:redund2}.
  Thus every column of the coefficient matrix has two $\pm1$ entries, so, by definition, the matrix is bidirected.

  Applying \cref{thm:bidirPolytope}, we obtain that the convex hull of the integer solutions to the auxiliary system is described by adding the cuts
  \begin{align}
    \dfrac{\bigl((\chi_U-\chi_V)M+\chi_F-\chi_H\bigr)(y,z)}{2}
    &\leq \floor*{\dfrac{b(U)-a(V)+d(F)-c(H)}{2}}&&\begin{aligned}
                                                  &\text{for every disjoint subsets $U$}\\
                                                  &\text{and $V$ of the rows, and every}\\
                                                  &\text{partition $F,H$ of $\delta(U\cup V)$,}
                                                 \end{aligned}\label{eq:qtsasm:edmonds-johnson-cut}
  \intertext{where $\delta(U\cup V)$ is the set of those $y$- and $z$-variables that appear in exactly one row in $U\cup V$.
  We may assume that $U\cup V$ does not contain any row corresponding to the constraint in~\cref{eq:qtsasm:redund} or, when $n$ is odd, to the constraint in~\cref{eq:qtsasm:redund2}; otherwise, the cut is redundant.
  For every other row, the lower and upper bounds $a$ and $b$ are both zero.
  Thus, $(\chi_U-\chi_V)M(y,z) = 0$ on the feasible region, and our cuts simplify to}
    \dfrac{(\chi_F-\chi_H)(y,z)}{2}
    &\leq \floor*{\dfrac{d(F)-c(H)}{2}}&&\begin{aligned}
                                        &\text{for every subset $W$ of the rows,}\\
                                        &\text{and partition $F,H$ of $\delta(W)$.}
                                      \end{aligned}\label{eq:qtsasm:simplified-cut}
  \end{align}

  Notice that if a $y$-variable lies in $\delta(W)$, then the corresponding cut is redundant.
  Therefore, we may and shall assume that $W$ contains either both or none of the rows in which a given $y$-variable appears.
  Then, $\delta(W)$ consists of those $z$-variables that occur in exactly one row contained in $W$.

  \medskip
  Now we prove that every cut in~\cref{eq:qtsasm:simplified-cut} appears in~\cref{eq:qtsasm:cut}.
  Each partition $F,H$ of $\delta(W)$ determines a sign $s_{i,j} \in \{0,\pm1\}$ for the variable $z_{i,j}$, namely,
  \begin{equation}\label{eq:qtsasm:s-def}
    s_{i,j} =
    \begin{cases}
      +1 & \text{if } z_{i,j} \in F,\\
      -1 & \text{if } z_{i,j} \in H,\\
       0 & \text{if } z_{i,j} \notin \delta(W),
    \end{cases}
  \end{equation}
  for $(i,j) \in \widetilde C$.
  As observed above, we may assume that $\delta(W)$ contains only $z$-variables, so for the numerator of the left-hand side of~\cref{eq:qtsasm:simplified-cut}, we have
  \begin{equation}\label{eq:qtsasm:expressz-with-s}
    (\chi_F-\chi_H)(y,z) = \sum_{(i,j) \in \widetilde C} s_{i,j} z_{i,j}.
  \end{equation}

  Now we express every $z$-variable in a given cut from~\cref{eq:qtsasm:simplified-cut} in terms of $y$-variables by~\cref{eq:qtsasm:z-def}, and we claim that the coefficients of all $y$-variables in the numerator are even integers.
  To see this, add up all equations in~\crefrange{eq:qtsasm:z-recursion1}{eq:qtsasm:z-recursion3} whose rows lie in $W$.
  Each $y$-variable appears in exactly two such equations, and $W$ contains either both of these rows or none.
  Since its coefficient is $-1$ in each such equation, the coefficient of any $y$-variable in this sum is therefore $0$ or $-2$, in particular an even integer.
  Moreover, each $z$-variable appears in at most two recursion equations, with opposite signs, so in the sum it has coefficient $\pm1$ if and only if it lies in $\delta(W)$, and coefficient $0$ otherwise.

  Rearranging this equation, we can express $\sum_{(i,j) \in \widetilde C} s_{i,j} z_{i,j}$, i.e., the numerator of the left-hand side of~\cref{eq:qtsasm:simplified-cut}, in terms of $y$-variables by possibly flipping the sign of the coefficient of $z_{i,j}$ for some $z_{i,j}\in\delta(W)$.
  Changing the sign of the coefficient of a given $z_{i,j}$ can be achieved by adding or subtracting the defining equation for $z_{i,j}$ in~\cref{eq:qtsasm:z-def}, multiplied by two.
  This operation adjusts the coefficient of $z_{i,j}$ by $\pm2$ and changes the coefficients of the $y$-variables by $\pm2$, so all $y$-coefficients remain even.
  Hence, the numerator of the left-hand side of~\cref{eq:qtsasm:simplified-cut} can be expressed as a linear combination of the $y$-variables with even integer coefficients.

  Furthermore, direct inspection of~\cref{eq:qtsasm:z-def} yields that the coefficient of $y_{i,j}$ in the numerator of~\cref{eq:qtsasm:simplified-cut} is exactly
  \[
    \sum_{j'=j}^{n_i} s_{i,j'} + \sum_{i'=n+1-i}^ns_{j,i'},
  \]
  where we set $n_i = n$ for all $i \in [k]$, and, when $n$ is odd, $n_{k+1} = k$.
  By the parity argument above, this integer is even for every $(i,j)\in C$, which is exactly the defining condition for $S=(s_{i,j})$ to belong to the set $\mathcal{S}$ defined in the statement of the theorem.

  The bounds in~\cref{eq:qtsasm:z-bounds,eq:qtsasm:z-rowsum} on the $z$-variables give $c_{i,j}=0$ and $d_{i,j}=1$ for all $(i,j) \in \widetilde C \setminus ([k] \times \{n\})$, while $c_{i,n}=d_{i,n}=1$ for all $i\in[k]$.
  Hence
  \[
    d(F)-c(H) = |\{(i,j) \in \widetilde C : s_{i,j}=1\}| - |\{i \in [k] : s_{i,n}=-1\}|.
  \]

  Thus each choice of $W$, $F$, and $H$ determines a sign matrix $S=(s_{i,j})\in\mathcal{S}$ via the definition in~\cref{eq:qtsasm:s-def}, and the corresponding cut from \cref{thm:bidirPolytope} is precisely the inequality in~\cref{eq:qtsasm:simplified-cut} associated with this $S$.
  In particular, every non-redundant cut produced by \cref{thm:bidirPolytope} for the auxiliary system appears among the family of inequalities in~\cref{eq:qtsasm:cut}.

  \medskip
  Conversely, fix $S\in\mathcal{S}$.
  The bounds on the $z$-variables give the valid inequality
  \[
    \sum_{(i,j)\in \widetilde C} s_{i,j} z_{i,j}
    \leq
    |\{(i,j) \in \widetilde C : s_{i,j}=1\}|-|\{i \in [k] : s_{i,n}=-1\}|.
  \]
  After expressing the left-hand side in terms of the $y$-variables using the identity~\cref{eq:qtsasm:z-def}, the parity condition defining $\mathcal{S}$ ensures that all coefficients in the numerator of the left-hand side of~\cref{eq:qtsasm:cut} are even.
  Since this left-hand side is even for every integer solution, dividing the inequality by~$2$ and taking the floor of the right-hand side gives a valid inequality for the convex hull of the integer solutions to the system~\crefrange{eq:qtsasm:real}{eq:qtsasm:full}, i.e., the convex hull of QTSASMs.

  \medskip
  In conclusion, every inequality in~\cref{eq:qtsasm:simplified-cut} appears among the inequalities in~\cref{eq:qtsasm:cut}, and every inequality in~\cref{eq:qtsasm:cut} is valid for the convex hull of the integer solutions to the system~\crefrange{eq:qtsasm:real}{eq:qtsasm:full}.
  By \cref{thm:bidirPolytope}, adding all inequalities in~\cref{eq:qtsasm:cut} to the auxiliary system yields the convex hull of its integer solutions.
  Since the constraints in~\cref{eq:qtsasm:redund} and, when $n$ is odd, the constraint in~\cref{eq:qtsasm:redund2} are redundant, translating back from $(y,z)$ to $y$ using the identity~\cref{eq:qtsasm:z-def} gives the claimed description.
  We conclude that the system~\crefrange{eq:qtsasm:real}{eq:qtsasm:cut} defines the core polytope $P^\core_\QTSASM$.
\end{proof}

By \cref{thm:xasm:assembly,thm:qtsasm:corepolytope}, we obtain the following description of $P_\QTSASM$.
\begin{theorem}\label{thm:qtsasm:coreDescr}
  Let $n \geq 1$, $k = \floor*{n/2}$, and $\widehat P^\core_\QTSASM = \{X \in \R^{n \times n} : \pi_C(X) \in P^\core_\QTSASM\}$.
  Then
  \[
    P_\QTSASM =
    \begin{cases}
      \widehat P^\core_\QTSASM \cap P_\QTS                                                    & \text{if } 2 \mid n,\\
      \widehat P^\core_\QTSASM \cap P_\QTS \cap \{X \in \R^{n \times n} : x_{k+1,k+1} = (-1)^k\} & \text{if } 2 \nmid n.
    \end{cases}
  \]
\end{theorem}

Using the assembly map $\varphi$ and the identity~\cref{eq:qtsasm:phi-row-prefix-used} to translate the QTSASM core constraints in \cref{thm:qtsasm:corepolytope} to and from the ASM constraints in \cref{thm:ASMpolytope}, we obtain the following.

\begin{theorem}\label{thm:qtsasm:descr}
  For every $n \geq 1$, $P_\QTSASM \subseteq \R^{n \times n}$ is described by the system
  \begin{align}
                                                                          X &\in P_\ASM \cap P_\QTS,\\
    \dfrac{\displaystyle \sum_{(i,j) \in \widetilde C} s_{i,j} \sum_{j'=1}^j  x_{i,j'}}{2} &\leq\floor*{\dfrac{|\{(i,j) \in \widetilde C : s_{i,j}=1\}| - |\{i \in [k] : s_{i,n}=-1\}|}{2}}  &\forall S \in \mathcal{S},
  \end{align}
  where $\widetilde C$ and $\mathcal{S}$ are as defined in \cref{thm:qtsasm:corepolytope}.
\end{theorem}

The cut inequalities in~\cref{eq:qtsasm:cut} are indexed by the sign patterns $S \in \mathcal{S}$.
Ignoring the parity condition in the definition of $\mathcal{S}$ gives at most $3^{|\widetilde C|}=9^{\floor*{n^2/4}}$ possible sign patterns, and hence the same crude bound for the number of cut inequalities.
The following theorem uses the parity condition to improve this bound.

\begin{theorem}\label{thm:qtsasm:corepolytope-inequality-count}
  The description of $P^\core_\QTSASM$ in \cref{thm:qtsasm:corepolytope} contains at most $5^{\floor*{n^2/4}} + n(n-1)$ linear inequalities, counting the lower and upper bounds in~\cref{eq:qtsasm:prefix} separately.
  In particular, both $P^\core_\QTSASM$ and $P_\QTSASM$ have at most this many facets.
\end{theorem}
\begin{proof}
  Let $m = |C| = \floor*{n^2/4}$.
  We first count the inequalities coming from the prefix bounds.
  In both parity cases, $|\widetilde C| = 2m$.
  The index set in~\cref{eq:qtsasm:prefix} is $\widetilde C \setminus ([k]\times\{n\})$, and $[k]\times\{n\}$ has size $k$.
  Each prefix bound in~\cref{eq:qtsasm:prefix} contributes one lower and one upper inequality.
  Hence the prefix bounds contribute $2(|\widetilde C|-k)=2(2m-k)=n(n-1)$ linear inequalities.
  It remains to bound the number of cut inequalities indexed by $\mathcal{S}$.

  \medskip
  We count the sign patterns in $\mathcal{S}$ by first counting their support vectors.
  For $S=(s_{a,b}) \in \{0,\pm1\}^{\widetilde C}$, let $t_{a,b} = \chi_{s_{a,b} \neq 0}$, and regard $t$ as an element of $\F_2^{\widetilde C}$.
  Since $1 \equiv -1 \pmod 2$, the parity condition defining $\mathcal{S}$ depends only on~$t$.
  Set $U = \widetilde C\setminus C$, and write $x = t|_C$ and $u = t|_U$.
  For $(i,j) \in C$, consider the congruence modulo $2$ obtained from the divisibility condition defining $\mathcal{S}$ for this pair $(i,j)$.
  In this congruence, the terms whose indices lie in $C$ are exactly $x_{i,j}+x_{i,j+1}+\dots+x_{i,k}$.
  All other terms have indices in $U$, and are fixed once $u$ is fixed.
  Therefore, after fixing $u$, the congruence for $(i,j)$ determines the suffix sum $x_{i,j}+x_{i,j+1}+\dots+x_{i,k}$.
  For fixed $i$, read these congruences in the order $j=k,k-1,\dots,1$.
  The first congruence determines $x_{i,k}$, and each subsequent congruence determines the next coordinate to the left after the coordinates to its right are known.
  Thus, for each $u$, there is a unique $x$ such that $(x,u)$ satisfies the parity condition.
  This defines a map $M : \F_2^U \to \F_2^C$, where $Mu$ is this unique vector $x$.
  Since the congruences are homogeneous modulo $2$, the map $M$ is linear.
  Hence the support vectors satisfying the parity condition are precisely the vectors $(Mu,u)$.

  For a finite set $I$ and a vector $v$ with coordinates indexed by $I$, write $\supp(v)=\{i\in I : v_i\neq 0\}$.
  A fixed support vector $(Mu,u)$ gives $2^{|\supp(u)|+|\supp(Mu)|}$ sign patterns in $\mathcal{S}$, because each non-zero entry may independently have sign $+1$ or $-1$.
  Therefore
  \[
    |\mathcal{S}|=\sum_{u\in\F_2^U}2^{|\supp(u)|+|\supp(Mu)|}.
  \]
  The Cauchy--Schwarz inequality gives
  \[
    |\mathcal{S}|^2
    \leq
    \left(\sum_{u\in\F_2^U}4^{|\supp(u)|}\right)
    \left(\sum_{u\in\F_2^U}4^{|\supp(Mu)|}\right).
  \]
  Since $|U| = m$, the first factor is $(1+4)^m=5^m$.
  We prove that the second factor is at most $5^m$ in both parity cases.

  \medskip
  Suppose first that $n = 2k+1$.
  We show that $M$ is injective.
  It is enough to prove that $Mu=0$ implies $u=0$.
  Let $u \in \F_2^U$, and suppose that $Mu=0$.
  The congruence modulo $2$ obtained from the divisibility condition defining $\mathcal{S}$ for the pair $(k+1,j)$ becomes
  \[
    \sum_{p=k+1}^{2k+1}u_{j,p}\equiv 0 \pmod2
  \]
  for every $j \in [k]$.
  In the congruence for the pair $(i,j)$ with $i \in [k]$, the terms from $C$ vanish and the terms from $U$ in row $i$ form the full row sum shown above to be zero.
  The remaining terms therefore give
  \[
    \sum_{p=2k+2-i}^{2k+1}u_{j,p}\equiv 0 \pmod2
  \]
  for all $i,j \in [k]$.
  For fixed $j$, the cases $i=1,\dots,k$ successively force $u_{j,2k+1},u_{j,2k},\dots,u_{j,k+2}$ to be zero.
  The full row-sum congruence then forces $u_{j,k+1}=0$.
  Thus $u=0$, and $M$ is injective.
  Since $|U| = |C| = m$, the map $M$ is bijective, and hence
  \[
    \sum_{u\in\F_2^U}4^{|\supp(Mu)|}
    =
    \sum_{x\in\F_2^C}4^{|\supp(x)|}=5^m.
  \]

  \medskip
  Suppose now that $n = 2k$.
  For $u \in \F_2^U$ and $i \in [k]$, let $A_i \in \{0,1\}$ be such that
  \[
    A_i\equiv\sum_{q=k+1}^{2k}u_{i,q}\pmod2.
  \]
  We first show that, for each $x \in \F_2^C$, there are at most two vectors $u \in \F_2^U$ with $Mu=x$.
  Let $u \in \F_2^U$, and suppose that $Mu=0$.
  The congruence modulo $2$ obtained from the divisibility condition defining $\mathcal{S}$ for the pair $(k,j)$ gives $A_k+A_j\equiv0\pmod2$, so $A_j=A_k$ for all $j \in [k]$.
  In the congruence for the pair $(i,j)$, the terms from $C$ vanish and the terms from $U$ in row $i$ contribute $A_i=A_k$.
  The remaining terms therefore satisfy
  \[
    \sum_{p=2k+1-i}^{2k}u_{j,p}\equiv A_k \pmod2
  \]
  for all $i,j \in [k]$.
  For fixed $j$, the cases $i=1,\dots,k$ successively give $u_{j,2k}=A_k$ and $u_{j,2k-1}=\dots=u_{j,k+1}=0$.
  Thus every $u$ satisfying $Mu=0$ is determined by the single value $A_k$, so there are at most two such vectors.
  Now fix $x \in \F_2^C$.
  If there is no vector $u_0$ with $Mu_0=x$, there is nothing to prove.
  Otherwise, fix one such $u_0$.
  Since $M$ is linear, every vector $u$ with $Mu=x$ is of the form $u_0+w$ with $Mw=0$.
  Hence there are at most two vectors $u \in \F_2^U$ with $Mu=x$.

  We next show that every vector in $M(\F_2^U)$ has zero $(k,k)$-coordinate.
  Let $u \in \F_2^U$, and set $x=Mu$.
  The congruence modulo $2$ obtained from the divisibility condition defining $\mathcal{S}$ for the pair $(k,k)$ gives
  \[
    x_{k,k}+A_k+A_k\equiv 0 \pmod2.
  \]
  Here the first sum in the definition of $\mathcal{S}$ contributes $x_{k,k}+A_k$, and the second sum contributes $A_k$.
  Hence $x_{k,k}=0$ for every $x \in M(\F_2^U)$.
  Therefore
  \[
    \sum_{u\in\F_2^U}4^{|\supp(Mu)|}
    = \sum_{x\in M(\F_2^U)} |\{u\in\F_2^U : Mu=x\}|4^{|\supp(x)|}
    \leq 2\sum_{x\in\F_2^C:x_{k,k}=0}4^{|\supp(x)|}
    =2\cdot5^{m-1}
    \leq 5^m.
  \]

  We have shown in both parity cases that the second factor in the Cauchy--Schwarz bound is at most $5^m$.
  Hence $|\mathcal{S}|^2\leq 5^{2m}$, and so $|\mathcal{S}|\leq 5^m$.
  Together with the count of prefix bounds, this proves the asserted bound on the number of inequalities.
  Consequently, $P^\core_\QTSASM$ has at most $5^{\floor*{n^2/4}} + n(n-1)$ facets.
  By \cref{thm:xasm:assembly}, the QTSASM assembly map restricts to an affine isomorphism from $P^\core_\QTSASM$ onto $P_\QTSASM$.
  Thus $P_\QTSASM$ also has at most $5^{\floor*{n^2/4}} + n(n-1)$ facets.
\end{proof}

\begin{remark}\label{rem:qtsasm:facet-data}
  The preceding theorem bounds the number of facets but does not determine it exactly.
  For $n=7,8,9,11,12,13$, computer calculations give $11,24,45,221,551,1508$ facets of $P_\QTSASM$, respectively.
  The corresponding numbers of non-prefix facets arising from the cut inequalities are $4,8,25,175,483,1428$.
  These data are consistent with exponential growth in $n$; the Ferrers inequalities constructed below provide an explicit exponential lower bound.
\end{remark}

\subsection{The dimension of the QTSASM polytope}\label{subsec:qtsasm-dimension}

We next determine the dimension of the convex hull of quarter-turn symmetric ASMs for all $n\geq5$ with $n\not\equiv2\pmod4$.
The proof differs from the dimension arguments in the previous sections since the description of $P_\QTSASM$ in \cref{thm:qtsasm:descr} contains additional parity inequalities, which makes direct perturbations of core variables less convenient.
Affine equations in the core variables still give the upper bound.
For the lower bound, we instead work in height coordinates: the map sending a core $Y$ to the height matrix $H^{\varphi(Y)}$ of its assembly is affine and injective, and the required affinely independent family is easier to construct in these coordinates.
For each fixed $n$, write $k=\floor*{n/2}$, $\ell=\ceil*{n/2}$, and $C_n=[\ell]\times[k]$ for the QTSASM core position set.
We now recall the standard corner-sum and height-function coordinates; see, for example, Propp's survey~\cite{propp2001many}.

For an $n\times n$ real matrix $X=(x_{i,j})$, define its \emph{corner-sum matrix} $\Sigma^X=(\sigma_{p,q}^X) \in \R^{[0,n]\times[0,n]}$ by
\[
  \sigma_{p,q}^X=\sum_{i=1}^p\sum_{j=1}^q x_{i,j}
\]
for $p,q\in[0,n]$.
The \emph{height matrix} of $X$ is the matrix $H^X=(h_{p,q}^X) \in \R^{[0,n]\times[0,n]}$ given by
\[
  h_{p,q}^X=p+q-2\sigma_{p,q}^X
\]
for $p,q\in[0,n]$.
Equivalently, $\sigma_{p,q}^X=(p+q-h_{p,q}^X)/2$.
For fixed $n$, let $R$ be the quarter-turn action on the height grid $[0,n]\times[0,n]$, given by $R(p,q)=(q,n-p)$.
For $t\in\Z$, write $R^t$ for the corresponding iterate of $R$, with exponents read modulo $4$.
For two height-grid positions $(p,q)$ and $(p',q')$, we call $|p-p'|+|q-q'|$ their \emph{grid distance}.
The following lemma translates between ASMs and their height matrices: the ASM prefix-sum constraints become local conditions between adjacent height-grid positions, and quarter-turn symmetry becomes an affine relation.

\begin{lemma}\label{lem:asm:height-functions}
A matrix $H=(h_{p,q}) \in \Z^{[0,n]\times[0,n]}$ is the height matrix of an $n\times n$ ASM if and only if it satisfies the boundary conditions
\begin{align*}
  h_{0,q}&=q                         &\forall q\in[0,n],\\
  h_{p,0}&=p                         &\forall p\in[0,n],\\
  h_{n,q}&=n-q                       &\forall q\in[0,n],\\
  h_{p,n}&=n-p                       &\forall p\in[0,n],
  \intertext{the parity condition}
  h_{p,q}&\equiv p+q\pmod2           &\forall p,q\in[0,n],
  \intertext{and the local conditions}
  |h_{p,q}-h_{p-1,q}|&=1             &\forall p\in[n],\ q\in[0,n],\\
  |h_{p,q}-h_{p,q-1}|&=1             &\forall p\in[0,n],\ q\in[n].
  \intertext{Moreover, the corresponding $n\times n$ ASM is quarter-turn symmetric if and only if its height matrix $H$ satisfies}
  h_{q,n-p}&=n-h_{p,q}               &\forall p,q\in[0,n].
\end{align*}
\end{lemma}

\begin{proof}
Let $X$ be an $n\times n$ ASM, and write $H=H^X=(h_{p,q})$ and $\Sigma=\Sigma^X=(\sigma_{p,q})$.
The boundary conditions for $H$ follow from $\sigma_{0,q}=\sigma_{p,0}=0$, $\sigma_{n,q}=q$, and $\sigma_{p,n}=p$.
The parity condition follows directly from $h_{p,q}=p+q-2\sigma_{p,q}$.
We check the local conditions.
For $p\in[n]$ and $q\in[0,n]$, we have $h_{p,q}-h_{p-1,q}=1-2(\sigma_{p,q}-\sigma_{p-1,q})$.
For $p\in[0,n]$ and $q\in[n]$, we have $h_{p,q}-h_{p,q-1}=1-2(\sigma_{p,q}-\sigma_{p,q-1})$.
The differences $\sigma_{p,q}-\sigma_{p-1,q}$ and $\sigma_{p,q}-\sigma_{p,q-1}$ are row-prefix and column-prefix sums of $X$.
Since these prefix sums lie in $\{0,1\}$ for an ASM, these height differences have absolute value~$1$.

Conversely, suppose that $H=(h_{p,q})$ satisfies the stated conditions.
For $p,q\in[0,n]$, define $\sigma_{p,q}=(p+q-h_{p,q})/2$.
The parity condition implies that $\sigma_{p,q}$ is integral.
The boundary conditions on $H$ give $\sigma_{0,q}=\sigma_{p,0}=0$, $\sigma_{n,q}=q$, and $\sigma_{p,n}=p$.
Rearranging the identities for $h_{p,q}-h_{p-1,q}$ and $h_{p,q}-h_{p,q-1}$ used above, the local conditions imply $\sigma_{p,q}-\sigma_{p-1,q}\in\{0,1\}$ for $p\in[n]$, $q\in[0,n]$, and $\sigma_{p,q}-\sigma_{p,q-1}\in\{0,1\}$ for $p\in[0,n]$, $q\in[n]$.
Define an $n\times n$ matrix $X=(x_{i,j})$ by
\[
  x_{i,j}=\sigma_{i,j}-\sigma_{i-1,j}-\sigma_{i,j-1}+\sigma_{i-1,j-1}
\]
for $i,j\in[n]$.
We show that $X$ is an ASM\@.
Then the row-prefix sums of $X$ are $\sigma_{i,q}-\sigma_{i-1,q}$, and the column-prefix sums of $X$ are $\sigma_{p,j}-\sigma_{p,j-1}$.
Hence all row-prefix and column-prefix sums lie in $\{0,1\}$.
The boundary values of $\sigma$ also give that every row and every column of $X$ sums to $1$.
In each row, the prefix sums start at $0$, end at $1$, and take only the values $0$ and $1$.
Therefore the non-zero entries in the row are the successive changes between $0$ and $1$, so they alternate in sign, with first and last non-zero entries equal to $1$.
The same argument applies to every column.
Thus $X$ is an ASM\@.
By construction, the matrix $\Sigma=(\sigma_{p,q})$ is the corner-sum matrix of $X$, and hence the height matrix of $X$ is $H$.

\medskip
We now prove the statement about quarter-turn symmetry.
Let $X$ be an $n\times n$ ASM, and write $H=H^X=(h_{p,q})$ and $\Sigma=\Sigma^X=(\sigma_{p,q})$.
Assume first that $X$ is quarter-turn symmetric.
For $p,q\in[0,n]$, the sum of all entries in the top-left $q\times(n-p)$ rectangle is the sum of all entries in the first $q$ rows minus the sum of all entries in the top-right $q\times p$ rectangle.
The latter rectangle is the image of the top-left $p\times q$ rectangle under quarter-turn rotation.
Hence $\sigma_{q,n-p}=q-\sigma_{p,q}$.
It follows that $h_{q,n-p}=q+n-p-2\sigma_{q,n-p}=q+n-p-2(q-\sigma_{p,q})=n-h_{p,q}$.
Thus $h_{q,n-p}=n-h_{p,q}$ for all $p,q\in[0,n]$.

Conversely, suppose that $h_{q,n-p}=n-h_{p,q}$ for all $p,q\in[0,n]$.
Equivalently, the corner sums satisfy $\sigma_{q,n-p}=q-\sigma_{p,q}$ for all $p,q\in[0,n]$.
For $i,j\in[n]$, applying the corner-sum equality at $(i,j)$, $(i-1,j)$, $(i,j-1)$, and $(i-1,j-1)$ gives
\[
  \sigma_{j,n-i}-\sigma_{j,n-i+1}-\sigma_{j-1,n-i}+\sigma_{j-1,n-i+1}
  =
  -\sigma_{i,j}+\sigma_{i-1,j}+\sigma_{i,j-1}-\sigma_{i-1,j-1}.
\]
The left-hand side is $-x_{j,n+1-i}$, and the right-hand side is $-x_{i,j}$.
Therefore $x_{i,j}=x_{j,n+1-i}$ for all $i,j\in[n]$.
Hence $X$ is invariant under quarter-turn rotation.
This proves the reverse direction and completes the proof.
\end{proof}

We next provide an extension tool for partial \emph{height specifications}, that is, height values specified only on a subset of the grid $[0,n]\times[0,n]$.
An \emph{$R$-invariant} subset of the grid $[0,n]\times[0,n]$ is a set of positions closed under $R$.
For a height-grid position $u$, write $\mathcal O(u)=\{R^t u:t\in[0,3]\}$ for its $R$-orbit.
When $u=(p,q)$, we also write $\mathcal O(p,q)$ for this orbit.
The following lemma shows that a partial height specification on an $R$-invariant set extends to the height matrix of a quarter-turn symmetric ASM whenever the displayed compatibility conditions hold.

\begin{lemma}\label{lem:qtsasm:equivariant-height-extension}
Let $n\not\equiv2\pmod4$.
Let $P\subseteq[0,n]\times[0,n]$ be $R$-invariant and contain the boundary of $[0,n]\times[0,n]$.
Let $G=(g_{p,q})\in\Z^P$ be a partial height specification on $P$.
Suppose that $G$ satisfies the boundary, parity, and quarter-turn conditions
\begin{align*}
  g_{0,q}&=q                            &\forall q\in[0,n],\\
  g_{p,0}&=p                            &\forall p\in[0,n],\\
  g_{n,q}&=n-q                          &\forall q\in[0,n],\\
  g_{p,n}&=n-p                          &\forall p\in[0,n],\\
  g_{p,q}&\equiv p+q\pmod2              &\forall (p,q)\in P,\\
  g_{q,n-p}&=n-g_{p,q}                   &\forall (p,q)\in P,\\
  \intertext{and the grid-distance condition}
  |g_{p,q}-g_{p',q'}|&\leq |p-p'|+|q-q'| &\forall (p,q),(p',q')\in P.
\end{align*}
Then $G$ extends to the height matrix of a quarter-turn symmetric ASM, that is, there exists an ASM height matrix $H=(h_{p,q}) \in \Z^{[0,n]\times[0,n]}$ satisfying $h_{p,q}=g_{p,q}$ for every $(p,q)\in P$ and $h_{q,n-p}=n-h_{p,q}$ for every $p,q\in[0,n]$.
\end{lemma}
\begin{proof}
Let $\overline P=([0,n]\times[0,n])\setminus P$.
We prove the statement by induction on the size of $\overline P$.
If $\overline P=\emptyset$, then set $H=G$.
The boundary, parity, and quarter-turn conditions are part of the assumptions.
For adjacent height-grid positions, the grid-distance condition gives that the corresponding height difference has absolute value at most $1$; parity makes this difference odd, so its absolute value is exactly $1$.
Thus $H$ is an ASM height matrix by \cref{lem:asm:height-functions}, and the quarter-turn condition gives the required symmetry.

Assume now that $\overline P$ is non-empty and that the statement holds whenever the number of unspecified positions is smaller.
Choose $(r,s)\in\overline P$.
Since $P$ is $R$-invariant, the orbit $\mathcal O(r,s)$ is disjoint from $P$.
We extend $G$ to $P\cup\mathcal O(r,s)$ while preserving the assumptions of the lemma.
We shall find an integer $a\equiv r+s\pmod2$ for which the following assignment is valid: the orbit positions $(r,s)$ and $(n-r,n-s)$ receive height value $a$, and the orbit positions $(s,n-r)$ and $(n-s,r)$ receive height value $n-a$.
The grid-distance inequalities between $(r,s)$ and the positions in $P$ are equivalent to $a$ lying in the integer interval
\[
  I=
  \bigcap_{(p,q)\in P}
  [\,g_{p,q}-|p-r|-|q-s|,\,g_{p,q}+|p-r|+|q-s|\,].
\]
We claim that the same interval also accounts for the grid-distance inequalities between $P$ and the positions in $\mathcal O(r,s)$.
Fix $(p,q)\in P$.
The inequality involving $(s,n-r)$, whose prescribed height is $n-a$, and $(p,q)$ is equivalent to the inequality between $(r,s)$ and $(n-q,p)\in P$, since $g_{n-q,p}=n-g_{p,q}$ and the two grid distances are equal.
Similarly, the inequality involving $(n-r,n-s)$ and $(p,q)$ is equivalent to the inequality between $(r,s)$ and $(n-p,n-q)\in P$, since $g_{n-p,n-q}=g_{p,q}$ and the two grid distances are equal.
The inequality involving $(n-s,r)$, whose prescribed height is $n-a$, and $(p,q)$ is equivalent to the inequality between $(r,s)$ and $(q,n-p)\in P$, since $g_{q,n-p}=n-g_{p,q}$ and the two grid distances are equal.
Consequently, for the proposed height values, $a\in I$ is equivalent to all grid-distance inequalities with one endpoint in $P$ and the other in $\mathcal O(r,s)$.

The interval $I$ is non-empty since, for any two positions in $P$, the grid-distance condition on $P$ and the triangle inequality show that the corresponding two intervals have non-empty intersection, and pairwise intersecting intervals have non-empty total intersection.
Moreover, since the intervals defining $I$ all have endpoints congruent to $r+s$ modulo $2$, so does $I$.

It remains to impose the conditions involving only the positions in $\mathcal O(r,s)$.
These are the grid-distance inequalities inside the orbit, together with the requirement that the prescription be consistent at any position that occurs more than once among the four positions in the orbit; this can happen only when $\mathcal O(r,s)$ is a singleton.
Set $\Delta=|r-s|+|r+s-n|$.
By quarter-turn invariance of grid distance, whenever the four orbit positions are distinct, every orbit position assigned the value $a$ is at grid distance $\Delta$ from every orbit position assigned the value $n-a$; pairs assigned the same value impose no restriction.
If the four orbit positions are not distinct, then they all coincide with $(n/2,n/2)$, so $\Delta=0$, and the single compatibility requirement is $a=n-a$.
Thus, in all cases, these remaining conditions are captured by $|a-(n-a)|\leq\Delta$, or equivalently $a\in J$, where $J=[(n-\Delta)/2,(n+\Delta)/2]$.
The interval $J$ is an integer interval because $\Delta=|r-s|+|r+s-n|\equiv n\pmod2$.
We now show that $I\cap J\neq\emptyset$, so the conditions connecting $\mathcal O(r,s)$ to $P$ and the conditions internal to the orbit can be satisfied simultaneously.
For every $(p,q)\in P$, the quarter-turn condition and the grid-distance condition on $P$ give $|2g_{p,q}-n|=|g_{p,q}-g_{q,n-p}|\leq |p-q|+|p+q-n|$.
Here $|p-q|+|p+q-n|$ is the grid distance from $(p,q)$ to its quarter-turn image $(q,n-p)$.
Along the chain $(p,q)$, $(r,s)$, $(s,n-r)$, $(q,n-p)$, the first and last successive distances are both $|p-r|+|q-s|$, since quarter-turn rotation preserves grid distance, and the middle successive distance is $\Delta$.
The triangle inequality therefore gives $|p-q|+|p+q-n| \leq 2(|p-r|+|q-s|)+\Delta$.
Hence $|2g_{p,q}-n|\leq 2(|p-r|+|q-s|)+\Delta$.
After division by two, the left-hand side is the distance between the centers of
$[\,g_{p,q}-|p-r|-|q-s|,\,g_{p,q}+|p-r|+|q-s|\,]$ and $J$, while the right-hand side is the sum of their radii; hence $J$ intersects every interval appearing in the definition of~$I$.
Consequently, $J$ and the intervals defining $I$ are pairwise intersecting, so their total intersection $I\cap J$ is non-empty.

We need an integer $a \in I\cap J$ that satisfies the parity condition, that is, $a \equiv r+s\pmod2$.
If $J$ is a singleton, then $\Delta=0$, so $(r,s)=(n/2,n/2)$.
Since $n\not\equiv2\pmod4$, this can occur only when $n\equiv0\pmod4$, and the unique point $n/2$ of $J$ is congruent to $n$, hence to $r+s$, modulo $2$.
If $J$ is not a singleton and $I\cap J$ is a singleton, then the unique point in the intersection is an endpoint of $I$, and therefore has the required parity.
Finally, if $I\cap J$ is not a singleton, then it contains two consecutive integers, one of which satisfies the parity condition.

Choose an integer $a\in I\cap J$ with $a\equiv r+s\pmod2$.
The preceding discussion also covers possible coincidences among the four orbit positions: they can occur only when $\Delta=0$, and then $a=n/2$ by $a\in J$.
Let $\widetilde P=P\cup\mathcal O(r,s)$ and define $\widetilde G=(\widetilde g_{u,v})\in\Z^{\widetilde P}$ by setting $\widetilde g_{p,q}=g_{p,q}$ for $(p,q)\in P$, $\widetilde g_{r,s}=\widetilde g_{n-r,n-s}=a$, and $\widetilde g_{s,n-r}=\widetilde g_{n-s,r}=n-a$.
The boundary conditions are unchanged, and the parity and the quarter-turn conditions hold by construction.
The grid-distance condition inside $P$ is inherited from $G$.
By the choice $a\in I$ and the observation above, all grid-distance inequalities between $P$ and $\mathcal O(r,s)$ hold, while the grid-distance condition inside the orbit and the consistency of the prescription when $\mathcal O(r,s)$ is a singleton follow from $a\in J$.
Thus $\widetilde G$ satisfies the assumptions of the lemma on $\widetilde P$, and the induction hypothesis gives the desired extension.
This completes the proof.
\end{proof}

We now choose the height-grid positions that will provide the additional directions in the lower-bound induction:
\[
  Q_n=
  \{(1,q):q\in[2,n-2]\}
  \cup
  \{(2,q):q\in[2,n-3]\}
  \cup
  \{(3,3)\}.
\]
The set $Q_n$ is meant as a set of \emph{orbit representatives}.
For $n\geq 9$, every $R$-orbit meets $Q_n$ in at most one point; equivalently, distinct elements of $Q_n$ lie in distinct $R$-orbits.
The next lemma realizes, for each $(r,s)\in Q_n$, the corresponding height-coordinate direction on its $R$-orbit.
For a height-grid position $u=(p,q)$, let $\chi_u=\chi_{p,q}\in \R^{[0,n]\times[0,n]}$ denote the matrix unit supported at $u$, and set
\[
  \mathcal E_u=\chi_u-\chi_{Ru}+\chi_{R^2u}-\chi_{R^3u}.
\]
When $u=(p,q)$, we also write $\mathcal E_{p,q}$ for $\mathcal E_u$.
The following lemma constructs one \emph{height-coordinate direction} for each representative in $Q_n$.

\begin{lemma}\label{lem:qtsasm:orbit-realization}
Let $n\geq9$ and $n\not\equiv2\pmod4$.
For every $(r,s)\in Q_n$, there exist two $n\times n$ QTSASMs $X_{r,s}^-$ and $X_{r,s}^+$ such that $H^{X_{r,s}^+}-H^{X_{r,s}^-}=2\mathcal E_{r,s}$.
Consequently, $\mathcal E_{r,s}$ is a direction of the affine hull of the height matrices of $n\times n$ QTSASMs.
\end{lemma}
\begin{proof}
Fix $(r,s)\in Q_n$.
We first construct a height matrix $H_{r,s}^-$ in which, at each of the four positions in the $R$-orbit of $(r,s)$, the value is smaller by $1$ than all adjacent height-grid values or larger by $1$ than all adjacent height-grid values.
We will then change these four values by $2$, turning the local minima into local maxima and the local maxima into local minima, to obtain $H_{r,s}^+$.
Since $(r,s)\in Q_n$, we have $r\in[3]$, $r\leq s\leq n-r$, and $r+s\leq n-1$.
Let $a$ be the smallest integer such that $r\leq a$, $s-r\leq a$, and $a\equiv r+s\pmod2$.
Equivalently, $a=s-1$ if $r=1$; $a=2$ if $(r,s)=(2,2)$; $a=3$ if $(r,s)=(2,3)$; $a=s-2$ if $r=2$ and $4\leq s$; and $a=4$ if $(r,s)=(3,3)$.
In the argument below we use only the following consequences of this choice: the congruence $a\equiv r+s\pmod2$, the lower bounds $s-r\leq a$ and $r\leq a$, and the upper bounds $a\leq r+s-2$ and $a\leq n-r-2$.

Write $R^t(r,s)=(r_t,s_t)$ for $t\in[0,3]$.
Thus $(r_0,s_0)=(r,s)$, $(r_1,s_1)=(s,n-r)$, $(r_2,s_2)=(n-r,n-s)$, and $(r_3,s_3)=(n-s,r)$.
For $t\in[0,3]$, let $N_t= \{(p,q)\in[0,n]\times[0,n]: |p-r_t|+|q-s_t|\leq1\}$.
Let $P$ be the union of the boundary of $[0,n]\times[0,n]$ and the neighborhoods $N_t$ for $t \in [0,3]$.
We define a partial height specification $G=(g_{p,q})\in\Z^P$ as follows.
On the boundary, prescribe the values required by the boundary condition in \cref{lem:qtsasm:equivariant-height-extension}, that is, $g_{0,q}=q$, $g_{p,0}=p$, $g_{n,q}=n-q$, and $g_{p,n}=n-p$.
For each $t\in[0,3]$ and each $(p,q)\in N_t$, prescribe
\[
  g_{p,q}=
  \begin{cases}
    a+|p-r_t|+|q-s_t|      & \text{if } 2 \mid t,\\
    n-a-|p-r_t|-|q-s_t|    & \text{if } 2 \nmid t.
  \end{cases}
\]

We first check that this prescription is well defined.
The grid distance between consecutive orbit positions $(r_t,s_t)$ and $(r_{t+1},s_{t+1})$, with indices taken modulo $4$, is $n-2r$; the distance between opposite orbit positions is $|n-2r|+|n-2s|$, and thus at least $n-2r$.
Since $n-2r\geq3$, the neighborhoods $N_0,N_1,N_2,N_3$ are pairwise disjoint.
The only possible overlap between one of these neighborhoods and the boundary occurs when $r=1$.
In that case, we have $a=s-1$, so the boundary position $(0,s)\in N_0$ receives the value $a+1=s$, which matches the top boundary value.
Rotation by quarter-turn gives the corresponding boundary compatibility for $N_1$, $N_2$, and $N_3$.
Hence $P$ is $R$-invariant, and there is no conflict among the prescribed values.

We record the local form of the prescription on each neighborhood, since it will be used both in the verification below and after the extension is obtained.
For even $t$, the value at $(r_t,s_t)$ is $a$, while every other position in $N_t$ has value $a+1$; for odd $t$, the value at $(r_t,s_t)$ is $n-a$, while every other position in $N_t$ has value $n-a-1$.
Hence the prescribed values differ by exactly $1$ between $(r_t,s_t)$ and each adjacent prescribed height-grid position in $N_t$.
Later, this same local form will allow us to add $2$ at the even orbit positions and subtract $2$ at the odd orbit positions without changing the absolute value of any affected local height difference.

We now verify the hypotheses of \cref{lem:qtsasm:equivariant-height-extension}.
The boundary condition holds by definition.
The prescribed boundary values satisfy the parity condition.
For $(p,q)\in N_0$, using the congruence $a\equiv r+s\pmod2$, we get
$
  g_{p,q}
  =
  a+|p-r|+|q-s|
  \equiv
  r+s+(p-r)+(q-s)
  \equiv
  p+q
  \pmod2
$,
so the parity condition holds on $N_0$.
The parity condition on the remaining neighborhoods follows because the quarter-turn prescription sends
$g_{p,q}$ to $n-g_{p,q}$, and
$
  n-g_{p,q}\equiv n-p-q\equiv q+n-p\pmod2
$.
By construction, $g_{q,n-p}=n-g_{p,q}$ for every $(p,q)\in P$, so the quarter-turn condition also holds.
It remains only to check the grid-distance condition.
There are four types of pairs to consider.
First, the prescribed boundary values satisfy the grid-distance condition among boundary positions by definition.
Second, consider two positions that both lie in a single neighborhood $N_t$.
If one of them is the orbit position $(r_t,s_t)$, then the other is adjacent to $(r_t,s_t)$, and their prescribed values differ by exactly $1$.
If neither position is $(r_t,s_t)$, then both prescribed values are equal.
Thus the grid-distance condition holds inside $N_t$.
Third, compare a prescribed neighborhood value with the boundary.
For a position $(p,q)$ with prescribed value $b$, the grid-distance condition with the four boundary sides is equivalent to $|p-q|\leq b\leq \min\{p+q,2n-p-q\}$.
By the quarter-turn condition and the fact that $R$ preserves grid distance, it suffices to verify these inequalities for $(p,q)\in N_0$.
Put $d=|p-r|+|q-s|$.
Then $d\leq1$ and $g_{p,q}=a+d$.
Since $q-p\leq s-r+d\leq a+d$ and $p-q\leq r+d\leq a+d$, we have $|p-q|\leq g_{p,q}$.
Moreover, $p+q\geq r+s-1$, while $g_{p,q}\leq a+1\leq r+s-1$; hence $g_{p,q}\leq p+q$.
For the remaining boundary inequality, $p+q\leq r+s+1\leq n$, and therefore $g_{p,q}\leq p+q\leq2n-p-q$.
Thus every neighborhood value satisfies the grid-distance condition with the boundary.
Fourth, compare values in different neighborhoods.
If the indices differ by $2$ modulo $4$, then both neighborhoods have prescribed values in either $\{a,a+1\}$ or $\{n-a-1,n-a\}$.
The absolute difference of two such values is at most $1$, and the neighborhoods are disjoint, so the grid-distance condition holds for these pairs.
We next compare adjacent neighborhoods in the cyclic order of the orbit.
Since $G$ satisfies the quarter-turn condition and $R$ preserves grid distance, it suffices to compare $N_0$ with $N_1$.
Let $(p,q)\in N_0$ and $(p',q')\in N_1$, and put $\alpha=|p-r|+|q-s|$ and $\beta=|p'-s|+|q'-(n-r)|$.
Then $\alpha,\beta\in\{0,1\}$, the prescribed values are $a+\alpha$ and $n-a-\beta$, and the distance between the two positions is at least $n-2r-\alpha-\beta$.
We claim that the difference of the two prescribed values is bounded by this same quantity.
Indeed, $r\leq a$ gives $n-2a-\alpha-\beta\leq n-2r-\alpha-\beta$, while $a\leq n-r-2$ and $\alpha+\beta\leq2$ give $2a+\alpha+\beta-n\leq n-2r-\alpha-\beta$.
These two inequalities are exactly $|(a+\alpha)-(n-a-\beta)|\leq n-2r-\alpha-\beta$.
Thus the grid-distance condition holds between $N_0$ and $N_1$.
Consequently, $G$ satisfies all assumptions of \cref{lem:qtsasm:equivariant-height-extension}.

By that lemma, $G$ extends to a height matrix $H_{r,s}^-=H^{X_{r,s}^-}$ of some $n\times n$ QTSASM $X_{r,s}^-$.
Define $H_{r,s}^+$ by adding $2$ at $(r_0,s_0)$ and $(r_2,s_2)$, subtracting $2$ at $(r_1,s_1)$ and $(r_3,s_3)$, and leaving all other height-grid values unchanged.
We show that $H_{r,s}^+$ satisfies the conditions of \cref{lem:asm:height-functions}.
The boundary entries are unchanged, and parity is preserved because each modified value is changed by $2$ or $-2$.
For the local condition, note that every neighbor of a changed position lies in the corresponding neighborhood $N_t$.
Since the neighborhoods $N_t$ are pairwise disjoint, no adjacent entries are both changed.
At the orbit positions $(r_0,s_0)$ and $(r_2,s_2)$, the change turns a local minimum into a local maximum; at the orbit positions $(r_1,s_1)$ and $(r_3,s_3)$, it turns a local maximum into a local minimum.
Thus the only affected local height differences are those between an orbit position and its adjacent height-grid positions, and these still have absolute value $1$ after the change.
We show that the quarter-turn condition is also preserved.
The changes are $+2$ at $(r_0,s_0)$ and $(r_2,s_2)$, and $-2$ at $(r_1,s_1)$ and $(r_3,s_3)$.
Since $R(r_t,s_t)=(r_{t+1},s_{t+1})$, with the index $t+1$ taken modulo $4$, the change at $R(r_t,s_t)$ is the negative of the change at $(r_t,s_t)$ for each $t\in[0,3]$.
Hence $H^+_{r,s}$ satisfies the quarter-turn condition on the modified orbit, and it agrees with $H^-_{r,s}$ outside this orbit.
Together with the local check above, \cref{lem:asm:height-functions} gives $H^+_{r,s}=H^{X^+_{r,s}}$ for some $n\times n$ QTSASM $X^+_{r,s}$.
By construction, $H^{X^+_{r,s}}-H^{X^-_{r,s}}=2\mathcal E_{r,s}$.
\end{proof}

We now introduce the affine subspace that will be shown to be the affine hull of $P^\core_{\QTSASM(n)}$.
Assume $n\geq5$ with $n\not\equiv2\pmod4$.
For $Y\in\R^{C_n}$ and $i\in[k]$, define
\[
  \rho_i(Y)=
  \sum_{j=1}^k y_{i,j}
  +
  \sum_{j=k+1}^n y_{n+1-j,i}.
\]
When $n$ is odd, also define
\[
  \rho_{k+1}(Y)=\sum_{j=1}^k y_{k+1,j}.
\]
Let $A_n\subseteq\R^{C_n}$ be the affine subspace given by the following equations.
If $n$ is even, then $A_n$ is given by $\rho_i(Y)=1$ for $i\in[k]$, by $y_{1,1}=0$, and by $y_{k,k}=0$.
If $n$ is odd, then $A_n$ is given by $\rho_i(Y)=1$ for $i\in[k]$, by $y_{1,1}=0$, and by $\rho_{k+1}(Y)=\chi_{2 \nmid k}$.
\begin{theorem}
\label{thm:qtsasm:dimension}
Let $n\geq5$ and $n\not\equiv2\pmod4$.
Then $\aff(P^\core_{\QTSASM(n)})=A_n$ and $\dim(P_{\QTSASM(n)})=\floor*{\frac{(n-1)^2}{4}}-2$.
\end{theorem}
\begin{proof}
It suffices to prove that $\aff(P^\core_{\QTSASM(n)})=A_n$ and to compute $\dim(A_n)$.
Since the assembly map $\varphi$ is affine, injective, and satisfies $P_{\QTSASM(n)}=\varphi(P^\core_{\QTSASM(n)})$, this gives the asserted dimension of $P_{\QTSASM(n)}$.
We first prove the upper bound by showing $P^\core_{\QTSASM(n)}\subseteq A_n$ and computing $\dim(A_n)$.
Then we prove the reverse inequality by induction, using QTSASMs of size $n-4$ together with the height-coordinate directions constructed in \cref{lem:qtsasm:orbit-realization}.

\smallskip
We now show that $P^\core_{\QTSASM(n)}\subseteq A_n$.
The equations $\rho_i(Y)=1$ are precisely the row-sum equations in the QTSASM core description.
It remains to justify the equation $y_{1,1}=0$.
Let $X$ be a QTSASM with core $Y$.
Since boundary entries of an ASM are non-negative, $y_{1,1}=x_{1,1}$ is either $0$ or $1$.
If $y_{1,1}=1$, then quarter-turn symmetry gives $x_{1,n}=1$ as well.
Since all entries in the first row are boundary entries, the first row would have sum at least $2$, contradicting the ASM row-sum condition.
Thus $y_{1,1}=0$ for every QTSASM core.

Suppose first that $n$ is even.
Since $n\not\equiv2\pmod4$, the integer $k$ is even.
Quarter-turn symmetry gives $x_{k,k}=x_{k,k+1}$.
Since these two entries are adjacent in row $k$, the alternating-sign condition in that row forces their common value to be zero.
Thus every QTSASM core satisfies $y_{k,k}=0$.
Now suppose that $n$ is odd.
By \cref{lem:qtsasm:center}, the central entry is $x_{k+1,k+1}=(-1)^k$.
The middle row is palindromic by half-turn symmetry, and therefore $\sum_{j=1}^n x_{k+1,j}=2\sum_{j=1}^k y_{k+1,j}+(-1)^k=1$.
Hence every QTSASM core satisfies $\rho_{k+1}(Y)=\chi_{2 \nmid k}$.
This proves $P^\core_{\QTSASM(n)}\subseteq A_n$.

\smallskip
We compute $\dim A_n$ by showing that the left-hand sides of the defining equations of $A_n$ are linearly independent.
First, suppose that $n$ is even.
Then $C_n=[k]\times[k]$ and, for every $i\in[k]$,
$
  \rho_i(Y)=\sum_{j=1}^k y_{i,j}+\sum_{j=1}^k y_{j,i}
$.
Here $k\geq4$.
Consider a linear relation
$
  \sum_{i=1}^k \alpha_i\rho_i+\lambda y_{1,1}+\mu y_{k,k}=0
$
on $\R^{C_n}$.
The coefficient of $y_{2,2}$ is $2\alpha_2$, hence $\alpha_2=0$.
For every $s\in[k]\setminus\{2\}$, the coefficient of $y_{2,s}$ is $\alpha_2+\alpha_s$, so $\alpha_s=0$.
Thus all $\alpha_i$ are zero, and the coefficients of $y_{1,1}$ and $y_{k,k}$ give $\lambda=\mu=0$.
Therefore the equations are independent in the even case.

Second, suppose that $n$ is odd.
Consider a linear relation
$
  \sum_{i=1}^{k+1} \alpha_i\rho_i+\lambda y_{1,1}=0
$
on $\R^{C_n}$.
For each $j\in[k]$, the variable $y_{k+1,j}$ appears only in $\rho_j$ and $\rho_{k+1}$.
Thus $\alpha_j+\alpha_{k+1}=0$ for every $j\in[k]$.
Since $k\geq2$, the variable $y_{1,2}$ belongs to the core.
Its coefficient is $\alpha_1+\alpha_2$, and hence $0=\alpha_1+\alpha_2=-2\alpha_{k+1}$.
Thus $\alpha_{k+1}=0$, and the relations $\alpha_j+\alpha_{k+1}=0$ give $\alpha_j=0$ for every $j\in[k]$; the coefficient of $y_{1,1}$ then gives $\lambda=0$.
Therefore the equations are independent in the odd case.

In both cases, $A_n$ is defined by $k+2$ independent equations in $\R^{C_n}$, so
$
  \dim(A_n)
  =
  |C_n|-(k+2)
  =
  \floor*{\frac{(n-1)^2}{4}}-2
$.
Since $P^\core_{\QTSASM(n)}\subseteq A_n$, it follows that
\[
  \dim(P^\core_{\QTSASM(n)})
  \leq
  \dim(A_n)
  =
  \floor*{\frac{(n-1)^2}{4}}-2.
\]

\smallskip
It remains to prove the reverse inequality.
The lower-bound argument proceeds by induction on $n$ with step size $4$, using explicit base cases and the height-coordinate directions constructed in \cref{lem:qtsasm:orbit-realization} for the induction step.
Since the induction step increases the size by $4$, we need one base case in each admissible residue class modulo $4$.
By the upper bound proved above, it is enough to exhibit $\dim(A_n)+1$ affinely independent QTSASM cores in each case.

For $n=5$, the following three QTSASM cores are affinely independent by a direct rank check:
\begin{align*}
Y^5_0
&=
\begin{mymatrix}
0&0\\
0&1\\
1&-1
\end{mymatrix}\,,
&
Y^5_1
&=
\begin{mymatrix}
0&0\\
1&0\\
0&0
\end{mymatrix}\,,
&
Y^5_2
&=
\begin{mymatrix}
0&1\\
0&0\\
0&0
\end{mymatrix}\,.
\end{align*}
Since $\dim(A_5)=2$, we get $\aff(P^\core_{\QTSASM(5)})=A_5$.
For $n=7$, the following eight QTSASM cores are affinely independent by a direct rank check:
\begin{align*}
Y^7_0
&=
\begin{mymatrix}
0&0&0\\
0&0&0\\
1&0&0\\
0&1&0
\end{mymatrix}\,,
&
Y^7_1
&=
\begin{mymatrix}
0&0&0\\
0&0&1\\
0&1&-1\\
1&-1&1
\end{mymatrix}\,,
&
Y^7_2
&=
\begin{mymatrix}
0&0&0\\
0&0&1\\
1&0&-1\\
0&0&1
\end{mymatrix}\,,
&
Y^7_3
&=
\begin{mymatrix}
0&0&0\\
0&1&0\\
1&-1&0\\
0&0&1
\end{mymatrix}\,,
\displaybreak[0]\\[2ex]
Y^7_4
&=
\begin{mymatrix}
0&0&0\\
1&0&0\\
0&0&0\\
0&0&1
\end{mymatrix}\,,
&
Y^7_5
&=
\begin{mymatrix}
0&0&1\\
0&0&0\\
0&0&0\\
0&1&0
\end{mymatrix}\,,
&
Y^7_6
&=
\begin{mymatrix}
0&0&1\\
0&1&-1\\
0&0&0\\
0&0&1
\end{mymatrix}\,,
&
Y^7_7
&=
\begin{mymatrix}
0&1&0\\
0&0&0\\
0&0&0\\
0&0&1
\end{mymatrix}\,.
\end{align*}
Since $\dim(A_7)=7$, we get $\aff(P^\core_{\QTSASM(7)})=A_7$.
For $n=8$, the following eleven QTSASM cores are affinely independent by a direct rank check:
\begin{align*}
Y^8_0
&=
\begin{mymatrix}
0&0&0&0\\
0&0&0&1\\
1&0&0&0\\
0&0&0&0
\end{mymatrix}\,,
&
Y^8_1
&=
\begin{mymatrix}
0&0&0&0\\
0&0&1&0\\
0&0&0&0\\
1&0&0&0
\end{mymatrix}\,,
&
Y^8_2
&=
\begin{mymatrix}
0&0&0&0\\
0&0&1&0\\
0&1&-1&1\\
1&-1&0&0
\end{mymatrix}\,,
\displaybreak[0]\\[2ex]
Y^8_3
&=
\begin{mymatrix}
0&0&0&0\\
1&0&0&0\\
0&0&0&1\\
0&0&0&0
\end{mymatrix}\,,
&
Y^8_4
&=
\begin{mymatrix}
0&0&0&1\\
0&0&0&0\\
0&1&0&0\\
0&0&0&0
\end{mymatrix}\,,
&
Y^8_5
&=
\begin{mymatrix}
0&0&0&1\\
0&0&1&0\\
0&0&0&0\\
0&0&0&0
\end{mymatrix}\,,
\displaybreak[0]\\[2ex]
Y^8_6
&=
\begin{mymatrix}
0&0&1&0\\
0&0&0&0\\
0&0&0&0\\
0&1&0&0
\end{mymatrix}\,,
&
Y^8_7
&=
\begin{mymatrix}
0&0&1&0\\
0&0&0&1\\
0&0&0&0\\
0&0&0&0
\end{mymatrix}\,,
&
Y^8_8
&=
\begin{mymatrix}
0&0&1&0\\
0&1&-1&0\\
0&0&0&1\\
0&0&0&0
\end{mymatrix}\,,
\\[2ex]
Y^8_9
&=
\begin{mymatrix}
0&1&0&0\\
0&0&0&0\\
0&0&0&0\\
0&0&1&0
\end{mymatrix}\,,
&
Y^8_{10}
&=
\begin{mymatrix}
0&1&0&0\\
0&0&0&0\\
0&0&0&1\\
0&0&0&0
\end{mymatrix}\,.
\end{align*}
Since $\dim(A_8)=10$, we get $\aff(P^\core_{\QTSASM(8)})=A_8$.

\smallskip
We now prove the induction step for $n\geq9$.
Assume that $n\not\equiv2\pmod4$ and that $\aff(P^\core_{\QTSASM(n-4)})=A_{n-4}$.
The map sending a core $Y$ to the height matrix $H^{\varphi(Y)}$ of its assembly is affine and injective.
Affineness follows from the definition of $H^X$ and the affineness of $\varphi$.
Injectivity follows because a height matrix uniquely determines its corner-sum matrix, hence the underlying matrix, and projection to $C_n$ then recovers the core.
Therefore, linear independence of core differences can be checked after applying this height-coordinate map.

We first describe a way to turn an $(n-4)\times(n-4)$ QTSASM into an $n\times n$ QTSASM\@.
Given an $(n-4)\times(n-4)$ QTSASM $X$, define an $n\times n$ matrix $\ol X$ by
\[
  \ol x_{i,j}
  =
  \begin{cases}
    x_{i-2,j-2} & \text{if } i,j\in[3,n-2],\\
    1 & \text{if } (i,j)\in\{(1,2),(2,n),(n,n-1),(n-1,1)\},\\
    0 & \text{otherwise}.
  \end{cases}
\]
The only non-zero entries outside the embedded copy of $X$ are the entries equal to $1$ at positions $(1,2)$, $(2,n)$, $(n,n-1)$, and $(n-1,1)$.
Thus the rows and columns with indices in $\{1,2,n-1,n\}$ have the required ASM pattern, while every remaining row and column is obtained from the corresponding row or column of $X$ by inserting zeros outside the interval $[3,n-2]$.
Therefore $\ol X$ is again an ASM, and the construction is quarter-turn symmetric.
Thus $\ol X$ is an $n\times n$ QTSASM\@.

For the corresponding height matrices, we have $h_{p+2,q+2}^{\ol X}=h_{p,q}^X+2$ for every $p,q \in [0,n-4]$.
The $(p+2)\times(q+2)$ corner of $\ol X$ contains the $p\times q$ corner of $X$, shifted two rows down and two columns to the right, together with the added $1$ at $(1,2)$.
Thus $\sigma_{p+2,q+2}^{\ol X}=\sigma_{p,q}^X+1$.
Hence the height value of $\ol X$ at $(p+2,q+2)$ is the height value of $X$ at $(p,q)$ increased by $2$.

Choose $\dim(A_{n-4})+1$ affinely independent cores of $(n-4)\times(n-4)$ QTSASMs, which exist by the induction hypothesis.
Applying the construction above to the corresponding QTSASMs gives $\dim(A_{n-4})+1$ QTSASMs of size $n$.
Their height matrices are affinely independent, because after restricting to the coordinates $(p+2,q+2)$ with $p,q\in[0,n-4]$ and subtracting the fixed constant $2$, one recovers the height matrices of the original QTSASMs.
Therefore this construction gives $\dim(A_{n-4})$ independent height-coordinate directions realized by $n\times n$ QTSASMs.

Every height-coordinate direction obtained from the embedded construction for size $n-4$ vanishes on all coordinates $(1,q)$ and $(2,q)$ with $q\in[0,n]$, because the first two rows of the height matrix of $\ol X$ do not depend on $X$.
We also check that these directions vanish at $(3,3)$.
By the formula above, $h_{3,3}^{\ol X}=h_{1,1}^X+2$.
Since every core in $A_{n-4}$ satisfies $y_{1,1}=0$, the entry $x_{1,1}$ of the assembled matrix is fixed at $0$, and hence $h_{1,1}^X=2-2x_{1,1}=2$ is fixed as well.
Thus $h_{3,3}^{\ol X}=4$ is fixed, so every embedded height-coordinate direction vanishes at $(3,3)$.
Therefore the embedded height-coordinate directions vanish on all positions in $Q_n$.

By \cref{lem:qtsasm:orbit-realization}, every matrix $\mathcal E_{r,s}$ with $(r,s)\in Q_n$ is a realized height-coordinate direction.
Distinct elements of $Q_n$ lie in distinct $R$-orbits.
Thus the restrictions of the vectors $\mathcal E_{r,s}$ for $(r,s)\in Q_n$ to the coordinates indexed by $Q_n$ form the standard basis of $\R^{Q_n}$.
Since the embedded height-coordinate directions vanish on $Q_n$, these directions together with the directions $\mathcal E_{r,s}$ for $(r,s)\in Q_n$ are linearly independent.
By the injectivity established above, linear independence of these height-coordinate differences implies linear independence of the corresponding core differences.
Hence $\dim(P^\core_{\QTSASM(n)})\geq \dim(A_{n-4})+|Q_n|$.
Finally, $|Q_n|=(n-3)+(n-4)+1=2n-6$, and
\[
  \dim(A_n)-
  \dim(A_{n-4})
  =
  \floor*{\frac{(n-1)^2}{4}}
  -
  \floor*{\frac{(n-5)^2}{4}}
  =
  2n-6.
\]
Thus $\aff(P^\core_{\QTSASM(n)})$ has dimension at least $\dim(A_n)$.
Since $P^\core_{\QTSASM(n)}\subseteq A_n$, it follows that $\aff(P^\core_{\QTSASM(n)})=A_n$.
This proves the induction step.

\smallskip
Together with the base cases $n=5,7,8$, the induction step gives $\aff(P^\core_{\QTSASM(n)})=A_n$ for every $n\geq5$ with $n\not\equiv2\pmod4$.
Hence
\[
  \dim(P^\core_{\QTSASM(n)})
  =
  \dim(A_n)
  =
  \floor*{\frac{(n-1)^2}{4}}-2.
\]
The QTSASM assembly map $\varphi$ is affine and injective, so $\dim(P_{\QTSASM(n)}) = \dim(P^\core_{\QTSASM(n)})$.
This proves the theorem.
\end{proof}

\subsection{Ferrers inequalities and QTSASM facets}\label{subsec:qtsasm-ferrers-facets}

Throughout this subsection, fix $n\geq5$ with $n\not\equiv2\pmod4$, and put $k=\floor*{n/2}$, $\ell=\ceil*{n/2}$, and $C_n=[\ell]\times[k]$.
Consider the staircase Ferrers shape with row lengths $k-2,k-3,\dots,0$.
A Ferrers diagram contained in this staircase is a weakly decreasing sequence $\lambda=(\lambda_1,\dots,\lambda_{k-1})$ of non-negative integers such that $\lambda_i\leq k-1-i$ for every $i\in[k-1]$.
Thus row $i$ of the diagram consists of the cells $(i,j)$ with $j\in[\lambda_i]$.
For $Y\in\R^{C_n}$, define
\[
  F_\lambda(Y)=\sum_{i=1}^{k-1}\sum_{j=1}^{\lambda_i} y_{i+1,j+1}.
\]
The inequality $F_\lambda(Y)\geq0$ is the Ferrers inequality associated with $\lambda$.
The zero diagram $\lambda=(0,\dots,0)$ gives the tautology $0\geq0$, so all facet statements below concern non-zero Ferrers diagrams.
We will prove that every non-zero Ferrers diagram in the staircase gives a facet of $P^\core_{\QTSASM(n)}$, that distinct diagrams give distinct facets, and that these facets transfer through the assembly map to $P_{\QTSASM(n)}$.
Since the staircase contains $\Catalan_{k-1}$ Ferrers diagrams, including the zero diagram, this will give at least $\Catalan_{k-1}-1$ facets of $P_{\QTSASM(n)}$.

The proof has two main components.
First, we derive the Ferrers inequality from the cut system in \cref{eq:qtsasm:cut} and characterize its tight QTSASM cores by selected row-prefix coordinates associated with the boundary of the shifted diagram.
Second, we work in height coordinates to show that these tight cores span a hyperplane in the affine hull.
The selected row-prefix coordinates correspond to the edges of a quarter-turn-invariant grid cycle following the Ferrers boundary.
Height modifications away from this cycle realize all affine-hull directions on which the selected row-prefix coordinates are fixed.
Differences between tight cores for which the value $1$ occurs at different selected row-prefix coordinates provide the remaining directions on which $F_\lambda$ vanishes.
Together with \cref{thm:qtsasm:dimension}, this proves that the tight face has codimension one.

We begin by describing the height-coordinate directions in the affine hull.
Let $A_n$ denote the affine subspace defined in \cref{subsec:qtsasm-dimension}.
By \cref{thm:qtsasm:dimension}, $A_n=\aff(P^\core_{\QTSASM(n)})$.
Let $\mathcal B_n=\{(p,q)\in[0,n]\times[0,n]:p\in\{0,n\}\text{ or }q\in\{0,n\}\}$ be the boundary of the height grid.
Recall that $R(p,q)=(q,n-p)$ is the quarter-turn action on the height grid.
Define
\[
  \Psi_n:\R^{C_n}\to\R^{[0,n]\times[0,n]},
  \qquad
  \Psi_n(Y)=H^{\varphi(Y)},
\]
to be the affine map that sends a core to the height matrix of its assembly.
We use the height-coordinate directions $\mathcal E_u=\chi_u-\chi_{Ru}+\chi_{R^2u}-\chi_{R^3u}$ introduced in \cref{subsec:qtsasm-dimension}.

For a set $S$ of height-grid positions, write $R^t S=\{R^t u:u\in S\}$.
We continue to write $\mathcal O(u)=\{R^t u:t\in[0,3]\}$ for the $R$-orbit of $u$.
The following set records the interior height positions that are fixed throughout $\Psi_n(A_n)$:
\[
  \mathcal Z_n=
  \begin{cases}
    \mathcal O(1,1)\cup\mathcal O(k-1,k-1)\cup\{(k,k)\} & \text{if } n=2k,\\
    \mathcal O(1,1)\cup\mathcal O(k,k) & \text{if } n=2k+1.
  \end{cases}
\]
The next proposition identifies the entire height-difference space and shows that no other height coordinates are fixed on $A_n$.

\begin{proposition}\label{prop:qtsasm:height-tangent-space}
Let $n\geq5$ and $n\not\equiv2\pmod4$.
Then the height-difference space $\Psi_n(A_n)-\Psi_n(A_n)$ consists of the height-coordinate differences $D\in\R^{[0,n]\times[0,n]}$ such that $D$ vanishes on $\mathcal B_n\cup\mathcal Z_n$ and satisfies $D_{Rv}=-D_v$ for every $v\in[0,n]\times[0,n]$.
Moreover, if one representative $u$ is chosen from each four-element $R$-orbit disjoint from $\mathcal B_n\cup\mathcal Z_n$, then the height-coordinate directions $\mathcal E_u$ form a basis of $\Psi_n(A_n)-\Psi_n(A_n)$.
\end{proposition}
\begin{proof}
Write $k=\floor*{n/2}$.
We first show that the height values at the positions in $\mathcal Z_n$ are fixed on $A_n$.
Let $Y\in A_n$, set $X=\varphi(Y)$, and let $H=\Psi_n(Y)$.
The matrix $X$ is quarter-turn symmetric by the definition of the assembly map.
Moreover, all row and column sums of $X$ are equal to $1$.
For the first $k$ rows, this is exactly the system $\rho_i(Y)=1$.
When $n$ is odd, the middle row has sum $2\rho_{k+1}(Y)+(-1)^k=1$.
The remaining row sums follow by half-turn symmetry, and the column sums follow by quarter-turn symmetry.
Consequently, $H$ has the boundary height values in \cref{lem:asm:height-functions} and satisfies $h_{Rv}=n-h_v$.

Since $y_{1,1}=x_{1,1}=0$ on $A_n$, we have $h_{1,1}=2-2x_{1,1}=2$.
The rest of $\mathcal O(1,1)$ is fixed by the quarter-turn height relation.

Suppose $n=2k$.
The height-grid position $(k,k)$ is fixed by $R$, so $h_{k,k}=n-h_{k,k}$ and $h_{k,k}=k$.
Thus the even central height position is fixed.
Since $n\not\equiv2\pmod4$, $k$ is even.
Quarter-turn symmetry makes the four $k\times k$ quadrant sums equal.
Since the total sum of $X$ is $2k$, each quadrant sum is $k/2$, and hence $\sigma_{k,k}=k/2$.
Also $y_{k,k}=x_{k,k}=0$ on $A_n$.
Using $x_{k,k}=\sigma_{k,k}-\sigma_{k-1,k}-\sigma_{k,k-1}+\sigma_{k-1,k-1}$ and the quarter-turn corner-sum relation $\sigma_{k-1,k}=k-1-\sigma_{k,k-1}$, we obtain $\sigma_{k-1,k-1}=k/2-1$.
Therefore $h_{k-1,k-1}=2k-2-2\sigma_{k-1,k-1}=k$.
The other positions in $\mathcal O(k-1,k-1)$ are fixed by the quarter-turn relation.

Now suppose $n=2k+1$.
The middle-row affine-hull equation is $\sum_{j=1}^k y_{k+1,j}=\chi_{2\nmid k}$.
Equivalently, $z_{k+1,k}=\chi_{2\nmid k}$.
In height coordinates, $z_{k+1,k}=\sigma_{k+1,k}-\sigma_{k,k}=(1+h_{k,k}-h_{k+1,k})/2$.
Since $R(k+1,k)=(k,k)$, the quarter-turn height relation gives $h_{k,k}=n-h_{k+1,k}$.
Thus $\chi_{2\nmid k}=(1+2h_{k,k}-n)/2$, so $h_{k,k}$ is fixed.
The rest of $\mathcal O(k,k)$ is fixed by quarter-turn symmetry.
This proves that every height matrix in $\Psi_n(A_n)$ has the same values on $\mathcal Z_n$.

Now let $Y,Y'\in A_n$ and set $D=\Psi_n(Y)-\Psi_n(Y')$.
The boundary height values are independent of $Y\in A_n$, so $D$ vanishes on $\mathcal B_n$.
The quarter-turn height relation gives $D_{Rv}=-D_v$.
The common values on $\mathcal Z_n$ give $D_v=0$ for $v\in\mathcal Z_n$.
Thus $\Psi_n(A_n)-\Psi_n(A_n)$ is contained in the stated space.

Conversely, let $W$ be the stated space.
Every $R$-orbit disjoint from $\mathcal B_n\cup\mathcal Z_n$ has four elements, and the vectors $\mathcal E_u$ for a set of orbit representatives form a basis of $W$.
Counting these representatives gives $\dim W=\floor*{\frac{(n-1)^2}{4}}-2$.
If $n=2k$, then the interior positions other than the fixed point $(k,k)$ form $k(k-1)$ four-element $R$-orbits.
Excluding the two additional orbits $\mathcal O(1,1)$ and $\mathcal O(k-1,k-1)$ leaves $k(k-1)-2=\floor*{(n-1)^2/4}-2$ orbit directions.
If $n=2k+1$, then the interior positions form $k^2$ four-element $R$-orbits.
Excluding $\mathcal O(1,1)$ and $\mathcal O(k,k)$ leaves $k^2-2=\floor*{(n-1)^2/4}-2$ orbit directions.
By \cref{thm:qtsasm:dimension}, $A_n=\aff(P^\core_{\QTSASM(n)})$ and $\dim P_{\QTSASM(n)}=\floor*{(n-1)^2/4}-2$.
By \cref{thm:xasm:assembly}, the assembly map sends $P^\core_{\QTSASM(n)}$ onto $P_{\QTSASM(n)}$, and it is affine and injective on cores.
Thus $\dim A_n=\floor*{(n-1)^2/4}-2$.
Since $\Psi_n$ is affine and injective, $\dim(\Psi_n(A_n)-\Psi_n(A_n))=\dim(A_n-A_n)=\dim W$.
The containment proved above is therefore equality, and the basis statement follows.
\end{proof}

Every Ferrers functional $F_\lambda$ is supported in the shifted staircase $T_k=\{(i,j):i,j\in[2,k],\ i+j\leq k+1\}$.
For a linear combination of core coordinates $c(Y)=\sum_{(i,j)\in C_n}c_{i,j}y_{i,j}$, we say that $c$ is supported in $S\subseteq C_n$ if $c_{i,j}=0$ for every $(i,j)\notin S$.
The next lemma will later ensure that a non-zero Ferrers functional cuts the affine hull in codimension one.

\begin{lemma}\label{lem:qtsasm:no-shifted-staircase-affine-hull-equation}
Let $n\geq5$, $n\not\equiv2\pmod4$, and put $k=\floor*{n/2}$.
No non-zero linear combination of core coordinates supported entirely in $T_k$ is constant on $A_n$.
\end{lemma}
\begin{proof}
The left-hand sides of the defining equations of $A_n$ span all linear combinations of core coordinates that are constant on $A_n$.
Thus, in the even case, every such combination lies in the span of $\rho_1,\dots,\rho_k,y_{1,1},y_{k,k}$; in the odd case, it lies in the span of $\rho_1,\dots,\rho_k,\rho_{k+1},y_{1,1}$.

Suppose first that $n=2k$.
Let $c=\sum_{i=1}^{k}\alpha_i\rho_i+\xi y_{1,1}+\mu y_{k,k}$ be supported in $T_k$.
Since $y_{1,b}$ lies outside $T_k$ for every $b\in[2,k]$, the coefficients of these coordinates give $\alpha_1+\alpha_b=0$ for $b\in[2,k]$.
In particular, $\alpha_2=\alpha_k=-\alpha_1$.
The coordinate $y_{2,k}$ also lies outside $T_k$, and its coefficient is $\alpha_2+\alpha_k$.
Thus $0=\alpha_2+\alpha_k=-2\alpha_1$.
Hence all $\alpha_i$ are zero.
The coordinates $y_{1,1}$ and $y_{k,k}$ lie outside $T_k$, so their vanishing coefficients force $\xi=\mu=0$.
Thus $c=0$.

Now suppose that $n=2k+1$.
Let $c=\sum_{i=1}^{k}\alpha_i\rho_i+\beta\rho_{k+1}+\xi y_{1,1}$ be supported in $T_k$.
The coordinates $y_{k+1,b}$ lie outside $T_k$, so their coefficients give $\alpha_b+\beta=0$ for $b\in[k]$.
Also $y_{1,2}$ lies outside $T_k$, and its coefficient gives $\alpha_1+\alpha_2=0$.
Since $\alpha_1=\alpha_2=-\beta$, this implies $\beta=0$.
Hence all $\alpha_i$ are zero.
Finally, the coordinate $y_{1,1}$ lies outside $T_k$, so its vanishing coefficient forces $\xi=0$.
Thus $c=0$.
\end{proof}

We now fix a non-zero Ferrers diagram $\lambda$ contained in the staircase and select row-prefix coordinates corresponding to the boundary edges of the shifted diagram.
These coordinates will characterize the QTSASM cores that satisfy $F_\lambda(Y)=0$.
For $a\in[k]$, set
\[
  j_a^{\mathrm L}=
  \begin{cases}
    1 & \text{if } a=1,\\
    \lambda_{a-1}+1 & \text{if } a\in[2,k].
  \end{cases}
\]
Let $\lambda'$ be the column-length sequence of $\lambda$, so $\lambda'_c=|\{i\in[k-1]:\lambda_i\geq c\}|$ for $c\geq1$, and for $b\in[k]$ set
\[
  j_b^{\mathrm R}=
  \begin{cases}
    n-1 & \text{if } b=1,\\
    n-\lambda'_{b-1}-1 & \text{if } b\in[2,k].
  \end{cases}
\]
For $Y\in A_n$, let $z=z(Y)$ be the row-prefix variables from \cref{eq:qtsasm:z-def}.
For $a\in[2,k]$, define
\[
  \omega_a^{\mathrm L}(Y)=z_{a,j_a^{\mathrm L}},
  \qquad
  \omega_a^{\mathrm R}(Y)=1-z_{a,j_a^{\mathrm R}},
\]
and, when $n$ is odd, define $\omega^{\mathrm M}(Y)=z_{\ell,1}$.
The superscripts indicate the two families of boundary edges that will be identified explicitly when the Ferrers cycle is defined below.
Let
\[
  \mathcal T_\lambda=
  \{(a,{\mathrm L}):a\in[2,k]\}
  \cup
  \{(a,{\mathrm R}):a\in[2,k]\},
\]
and, when $n$ is odd, adjoin the symbol ${\mathrm M}$ to $\mathcal T_\lambda$.
For $\tau\in\mathcal T_\lambda$, write $\omega_\tau$ for the corresponding boundary coordinate.
When $\omega_\tau$ is evaluated on a direction $V\in A_n-A_n$, we use its linear part; equivalently, $\omega_\tau(V)=\omega_\tau(Y)-\omega_\tau(Y')$ whenever $V=Y-Y'$ with $Y,Y'\in A_n$.
For $\tau\in\mathcal T_\lambda$, define
\[
  \mathcal U_{\lambda,\tau}^n
  =
  \{Y\in\pi_C(\QTSASM(n)):
      \omega_\tau(Y)=1,
      \ \omega_{\tau'}(Y)=0\text{ for every }\tau'\neq\tau
    \}.
\]
Thus $\mathcal U_{\lambda,\tau}^n$ consists of the QTSASM cores for which $\omega_\tau$ equals $1$ and every other boundary coordinate equals $0$.

Let $\widetilde C_n$ denote the row-prefix index set appearing as $\widetilde C$ in \cref{thm:qtsasm:corepolytope}, namely
\[
  \widetilde C_n=
  \begin{cases}
    [k]\times[n] & \text{if } 2\mid n,\\
    [k]\times[n]\cup\{(k+1,j):j\in[k]\} & \text{if } 2\nmid n.
  \end{cases}
\]
Define the Ferrers cut sign pattern $S^\lambda=(s^\lambda_{u,v})\in\{0,\pm1\}^{\widetilde C_n}$ by
\[
  s^\lambda_{u,v}=
  \begin{cases}
    -1 & \text{if } (u,v)=(a,j_a^{\mathrm L}) \text{ for some } a\in[k],\\
     1 & \text{if } (u,v)=(a,j_a^{\mathrm R}) \text{ for some } a\in[k],\\
    -1 & \text{if } (u,v)=(a,n) \text{ for some } a\in[2,k],\\
    -1 & \text{if } 2\nmid n \text{ and } (u,v)=(\ell,1),\\
     0 & \text{otherwise.}
  \end{cases}
\]
No position satisfies two of the non-zero cases in this definition.
To see this, the positions $(a,j_a^{\mathrm L})$ lie in columns $[k]$.
Moreover, $j_1^{\mathrm R}=n-1$, and the staircase bounds give $\lambda'_{b-1}\leq k-b$ for $b\in[2,k]$, so $j_b^{\mathrm R}\in[k+1,n-1]$ for every $b\in[k]$.
The positions $(a,n)$ lie in column $n$, while the last position, present only in odd size, lies in row $\ell=k+1$.
Thus $S^\lambda$ is well defined.

\begin{lemma}\label{lem:qtsasm:ferrers-defect-identity}
Let $n\geq5$, $n\not\equiv2\pmod4$, and let $\lambda$ be a non-zero Ferrers diagram contained in the staircase.
Then $S^\lambda\in\mathcal{S}$, and the cut inequality in \cref{eq:qtsasm:cut} indexed by $S^\lambda$ is $-F_\lambda(Y)\leq0$, equivalently the Ferrers inequality $F_\lambda(Y)\geq0$.
Moreover, every $Y\in A_n$ satisfies
\[
  1+2F_\lambda(Y)=\sum_{\tau\in\mathcal T_\lambda}\omega_\tau(Y).
\]
Consequently, $F_\lambda(Y)\geq0$ is valid on $P^\core_{\QTSASM(n)}$, and equality holds at a QTSASM core if and only if exactly one boundary coordinate equals $1$ and all the remaining boundary coordinates equal $0$.
\end{lemma}
\begin{proof}
For $Y\in\R^{C_n}$, let $z=z(Y)$ be the row-prefix expressions from \cref{eq:qtsasm:z-def}.
In the following coefficient calculation, we use the convention $s^\lambda_{u,v}=0$ for $(u,v)\notin \widetilde C_n$.
The coefficient of $y_{a,b}$ in $\sum_{(u,v)\in \widetilde C_n}s^\lambda_{u,v}z_{u,v}$ is
\[
  \sum_{v=b}^{n}s^\lambda_{a,v}
  +
  \sum_{u=n+1-a}^{n}s^\lambda_{b,u}.
\]
Let $a,b\in[2,k]$.
The coordinate $y_{a,b}$ occurs in $F_\lambda$ exactly when $b-1\leq\lambda_{a-1}$.
By conjugacy of the row- and column-length sequences, this condition is equivalent to $a-1\leq\lambda'_{b-1}$.
In the first sum of the coefficient formula, the signs at $(a,j_a^{\mathrm R})$ and $(a,n)$ cancel, and the remaining sign at $(a,j_a^{\mathrm L})$ contributes $-1$ exactly when $b\leq j_a^{\mathrm L}$, or equivalently $b-1\leq\lambda_{a-1}$.
In the second sum, the sign at $(b,n)$ always contributes $-1$, while the sign at $(b,j_b^{\mathrm R})$ is also included exactly when $a\geq\lambda'_{b-1}+2$.
Hence the second sum equals $-1$ exactly when $a-1\leq\lambda'_{b-1}$ and equals $0$ otherwise.
It follows that the coefficient of $y_{a,b}$ is $-2$ precisely when $y_{a,b}$ occurs in $F_\lambda$, and is $0$ otherwise.

It remains to check the coordinates in the first row, the first column, and, when $n$ is odd, the last core row.
For $a=1$ and $b\in[2,k]$, the first sum contributes $s^\lambda_{1,j_1^{\mathrm R}}=1$ and the second sum contributes $s^\lambda_{b,n}=-1$.
If $a\in[2,k]$ and $b=1$, then the first sum contributes $s^\lambda_{a,j_a^{\mathrm L}}+s^\lambda_{a,j_a^{\mathrm R}}+s^\lambda_{a,n}=-1$ and the second sum contributes $s^\lambda_{1,j_1^{\mathrm R}}=1$.
For $(a,b)=(1,1)$, the two signs $s^\lambda_{1,j_1^{\mathrm L}}$ and $s^\lambda_{1,j_1^{\mathrm R}}$ cancel in the first sum, and the second sum is zero.
When $n$ is odd, the coefficient of every coordinate in row $\ell$ is also zero.
For $b=1$, the terms $s^\lambda_{\ell,1}$ and $s^\lambda_{1,j_1^{\mathrm R}}$ cancel, while for $b\in[2,k]$ the terms $s^\lambda_{b,j_b^{\mathrm R}}$ and $s^\lambda_{b,n}$ cancel.
Thus, as an identity in the core variables,
\begin{equation}\label{eq:qtsasm:ferrers-sign-sum}
  \sum_{(u,v)\in \widetilde C_n}s^\lambda_{u,v}z_{u,v}=-2F_\lambda(Y).
\end{equation}
For each $(a,b)\in C_n$, the coefficient displayed above is the integer whose parity is prescribed in the definition of $\mathcal{S}$.
Since each such coefficient is $0$ or $-2$, the sign pattern $S^\lambda$ belongs to $\mathcal{S}$.
The positive entries of $S^\lambda$ are the $k$ positions $(a,j_a^{\mathrm R})$ with $a\in[k]$, and its negative entries in column $n$ are the $k-1$ positions $(a,n)$ with $a\in[2,k]$.
Hence the right-hand side of the cut inequality in \cref{eq:qtsasm:cut} is $\floor*{(k-(k-1))/2}=0$.
By \cref{eq:qtsasm:ferrers-sign-sum}, the left-hand side of that cut inequality is $-F_\lambda(Y)$.

We next derive the boundary-coordinate identity from the same sign pattern.
Assume that $Y\in A_n$.
Since $z_{a,n}=1$ for $a\in[k]$ on $A_n$, the definition of $S^\lambda$ gives
\[
  \sum_{(u,v)\in \widetilde C_n}s^\lambda_{u,v}z_{u,v}
  =
  \sum_{a=1}^k z_{a,j_a^{\mathrm R}}
  -
  \sum_{a=1}^k z_{a,j_a^{\mathrm L}}
  -(k-1)
  -
  \begin{cases}
    0 & \text{if } 2 \mid n,\\
    z_{\ell,1} & \text{if } 2 \nmid n.
  \end{cases}
\]
Because $Y\in A_n$, we have $y_{1,1}=0$, and hence $z_{1,j_1^{\mathrm L}}=z_{1,1}=0$ and $z_{1,j_1^{\mathrm R}}=z_{1,n-1}=z_{1,n}=1$.
Combining the last display with \cref{eq:qtsasm:ferrers-sign-sum} and rearranging gives
\[
  1+2F_\lambda(Y)=\sum_{\tau\in\mathcal T_\lambda}\omega_\tau(Y).
\]
For a QTSASM core, every row-prefix variable belongs to $\{0,1\}$, and therefore every boundary coordinate belongs to $\{0,1\}$.
Since $F_\lambda(Y)$ is integral on QTSASM cores, the left-hand side $1+2F_\lambda(Y)$ is odd.
The right-hand side is non-negative, so the identity implies $1+2F_\lambda(Y)\geq1$ and hence $F_\lambda(Y)\geq0$.
Convexity gives validity on $P^\core_{\QTSASM(n)}$.
Equality holds exactly when one boundary coordinate equals $1$ and all the remaining boundary coordinates equal~$0$.
\end{proof}

We next interpret the selected boundary coordinates in height coordinates.
Their associated height-grid edges form a monotone path in the northwest part of the grid.
Define the northwest Ferrers arc to be the unique monotone grid path
\[
  \Pi_\lambda=(v_0,\dots,v_M),
  \qquad
  v_0=
  \begin{cases}
    (k,1) & \text{if } n=2k,\\
    (k+1,1) & \text{if } n=2k+1,
  \end{cases}
  \qquad
  v_M=(1,k),
\]
with steps $(-1,0)$ and $(0,1)$ whose labelled edges are as follows.
For $a\in[2,k]$, the vertical edge $[(a-1,j_a^{\mathrm L}),(a,j_a^{\mathrm L})]$ has label $(a,{\mathrm L})$, and the horizontal edge $[(\lambda'_{a-1}+1,a-1),(\lambda'_{a-1}+1,a)]$ has label $(a,{\mathrm R})$.
When $n$ is odd, the middle edge $[(k,1),(k+1,1)]$ has label ${\mathrm M}$.
These labelled edges exhaust the path, and its length is $M=n-2$.
Define the Ferrers cycle by $\Gamma_\lambda=\bigcup_{t=0}^3 R^t(\Pi_\lambda)$.

\begin{lemma}\label{lem:qtsasm:ferrers-defect-cycle}
Let $n\geq5$, $n\not\equiv2\pmod4$, and let $\lambda$ be a non-zero Ferrers diagram contained in the staircase.
Then $\Gamma_\lambda$ is a connected $R$-invariant simple grid cycle.
\end{lemma}
\begin{proof}
The vertical and horizontal edges labelled above are exactly the vertical and horizontal boundary steps of the shifted Ferrers boundary in the northwest quadrant.
Since $\lambda$ is weakly decreasing, these steps concatenate to a simple monotone path.
Every vertex of $\Pi_\lambda$ lies in $[k]\times[k]$, except that in the odd case the initial endpoint is $v_0=(k+1,1)$.
The only vertex with second coordinate $k$ is $v_M=(1,k)$; in the even case the only vertex with first coordinate $k$ is $v_0$, and in the odd case the only vertex with first coordinate $k+1$ is $v_0$.
Consequently $\Pi_\lambda$ meets $R\Pi_\lambda$ only at $v_M=Rv_0$, meets $R^3\Pi_\lambda$ only at $v_0=R^3v_M$, and is disjoint from $R^2\Pi_\lambda$; the same statements hold after rotating.
In the even case this path runs from $(k,1)$ to $(1,k)$, and $R(k,1)=(1,k)$.
In the odd case the additional edge extends the same path to $(k+1,1)$, and $R(k+1,1)=(1,k)$.
Thus the four quarter-turn images of the northwest arc close to one connected simple cycle.
\end{proof}

For each $\tau\in\mathcal T_\lambda$, we now prescribe height values on $\Gamma_\lambda$ so that every QTSASM height matrix agreeing with these values has core in $\mathcal U_{\lambda,\tau}^n$.
Along $\Pi_\lambda$, the coordinate sum changes by $s_{i+1}-s_i\in\{\pm1\}$ across the edge $[v_i,v_{i+1}]$.
The prescription follows this change on every edge except the edge labeled by $\tau$, where the change is reversed.
The initial value is chosen so that the two endpoints, which are quarter-turn images of one another, receive complementary heights.
Continue to write $\Pi_\lambda=(v_0,\dots,v_M)$, where $v_i=(x_i,y_i)$.
Let $\tau_i$ be the label of the edge $[v_i,v_{i+1}]$, and set $s_i=x_i+y_i$.
For $\tau\in\mathcal T_\lambda$, define
\[
  \nu_i^\tau=
  \begin{cases}
    s_i-s_{i+1} & \text{if } \tau_i=\tau,\\
    s_{i+1}-s_i & \text{if } \tau_i\neq\tau.
  \end{cases}
\]
Set $c^\tau(v_0)=(n-\sum_{i=0}^{M-1}\nu_i^\tau)/2$, and recursively define $c^\tau(v_{i+1})=c^\tau(v_i)+\nu_i^\tau$ for $i\in[0,M-1]$.
Extend $c^\tau$ to all of $V(\Gamma_\lambda)$ by $c^\tau(Rv)=n-c^\tau(v)$.

\begin{lemma}\label{lem:qtsasm:single-defect-cycle-data}
Let $n\geq5$, $n\not\equiv2\pmod4$, and let $\lambda$ be a non-zero Ferrers diagram contained in the staircase.
For every $\tau\in\mathcal T_\lambda$, the function $c^\tau$ is well defined on $V(\Gamma_\lambda)$, satisfies $c^\tau(Rv)=n-c^\tau(v)$ and $c^\tau(v)\equiv v_1+v_2\pmod2$ for every $v\in V(\Gamma_\lambda)$, and has height difference $\pm1$ across every edge of $\Gamma_\lambda$.
If a QTSASM height matrix agrees with $c^\tau$ on $V(\Gamma_\lambda)$, then the core of the corresponding QTSASM lies in $\mathcal U_{\lambda,\tau}^n$.
\end{lemma}
\begin{proof}
Along the northwest arc, $s_{i+1}-s_i\in\{\pm1\}$, and hence $\nu_i^\tau\in\{\pm1\}$.
Let $m$ be the unique index with $\tau_m=\tau$, and put $\varepsilon=s_{m+1}-s_m$.
Then $\sum_{i=0}^{M-1}\nu_i^\tau=s_M-s_0-2\varepsilon$.
Since $v_M=Rv_0=(y_0,n-x_0)$ and $v_0=(x_0,1)$, we get $n-s_M+s_0=2x_0$.
Thus $n-\sum_{i=0}^{M-1}\nu_i^\tau=2x_0+2\varepsilon$, so $c^\tau(v_0)$ is an integer and $c^\tau(v_0)=x_0+\varepsilon\equiv x_0+1\equiv v_{0,1}+v_{0,2}\pmod2$.
The same calculation gives $c^\tau(v_0)+c^\tau(v_M)=n$, which is the compatibility condition needed to extend the prescription by quarter-turn symmetry.
Since $v_M=Rv_0$, the rule $c^\tau(Rv)=n-c^\tau(v)$ is consistent at the endpoints and then on all four rotated arcs.
The parity condition is preserved along the arc because every step changes both grid parity and height parity, and it is preserved under quarter-turn because $n-c^\tau(v)\equiv n-v_1-v_2\equiv (Rv)_1+(Rv)_2\pmod2$.
The recursive definition changes the height by $\nu_i^\tau\in\{\pm1\}$ on each northwest-arc edge, and the quarter-turn rule gives height difference $\pm1$ on the rotated edges.
Along the northwest arc, the recursive definition uses the coordinate-sum change $s_{i+1}-s_i$ on every edge except the edge labeled by $\tau$, where it uses the negative of that change.
For a height matrix $H$ and a vertical row-prefix edge $[(a-1,j),(a,j)]$, the corresponding prefix variable is $z_{a,j}=(H_{a-1,j}-H_{a,j}+1)/2$.
Hence, on the edge labeled $(a,{\mathrm L})$, the ordinary coordinate-sum change gives $\omega_a^{\mathrm L}=0$, while the reversed change gives $\omega_a^{\mathrm L}=1$.
For the coordinate $\omega_b^{\mathrm R}$, applying $R^3$ to the vertical row-prefix edge $[(b-1,j_b^{\mathrm R}),(b,j_b^{\mathrm R})]$ gives the northwest horizontal edge $[(\lambda'_{b-1}+1,b-1),(\lambda'_{b-1}+1,b)]$.
The same endpoint-height calculation shows that the ordinary change gives $\omega_b^{\mathrm R}=0$ and the reversed change gives $\omega_b^{\mathrm R}=1$.
When $n$ is odd, the middle edge is handled by the vertical-edge calculation.
Thus the core of a QTSASM height matrix agreeing with $c^\tau$ on $V(\Gamma_\lambda)$ lies in $\mathcal U_{\lambda,\tau}^n$.
\end{proof}

The next lemma translates equality of the boundary coordinates into a condition on the height values along the Ferrers cycle.

\begin{lemma}\label{lem:qtsasm:fixed-defect-height-differences}
Let $n\geq5$, $n\not\equiv2\pmod4$, and let $\lambda$ be a non-zero Ferrers diagram contained in the staircase.
Let $Y,Y'\in A_n$ and set $D=\Psi_n(Y)-\Psi_n(Y')$.
Then $\omega_\tau(Y-Y')=0$ for every $\tau\in\mathcal T_\lambda$ if and only if $D_v=0$ for every $v\in V(\Gamma_\lambda)$.
\end{lemma}
\begin{proof}
For every row-prefix variable, $z_{i,j}=\sigma_{i,j}-\sigma_{i-1,j}$.
Thus for a height difference $D$, $z_{i,j}(Y)-z_{i,j}(Y')=(D_{i-1,j}-D_{i,j})/2$.
For a label $(a,{\mathrm L})$, the equality $\omega_a^{\mathrm L}(Y-Y')=0$ is equivalent to equality of $D$ at the endpoints of the vertical edge $[(a-1,j_a^{\mathrm L}),(a,j_a^{\mathrm L})]$.
For a label $(b,{\mathrm R})$, the constant term in $\omega_b^{\mathrm R}=1-z_{b,j_b^{\mathrm R}}$ cancels on the direction $Y-Y'$.
Thus $\omega_b^{\mathrm R}(Y-Y')=0$ is equivalent to equality of $D$ at the endpoints of the vertical row-prefix edge $[(b-1,j_b^{\mathrm R}),(b,j_b^{\mathrm R})]$.
Applying $R^3$ sends this edge to the horizontal edge $[(\lambda'_{b-1}+1,b-1),(\lambda'_{b-1}+1,b)]$ of $\Pi_\lambda$, and the relation $D_{Rv}=-D_v$ preserves equality of the endpoint values.
When $n$ is odd, the middle coordinate is attached directly to the edge $[(k,1),(k+1,1)]$.
The quarter-turn orbits of these edges are exactly the edges of $\Gamma_\lambda$.
Therefore all boundary coordinates vanish on $Y-Y'$ if and only if $D$ has equal values at the endpoints of every edge of $\Gamma_\lambda$.
Since the cycle is connected, this means that $D$ is constant on $V(\Gamma_\lambda)$.
The cycle is $R$-invariant and $D_{Rv}=-D_v$, so the constant must be zero.
The converse follows from the same endpoint formulas.
\end{proof}

\begin{lemma}\label{lem:qtsasm:single-defect-sets-non-empty}
Let $n\geq5$, $n\not\equiv2\pmod4$, and let $\lambda$ be a non-zero Ferrers diagram contained in the staircase.
For every $\tau\in\mathcal T_\lambda$, the set $\mathcal U_{\lambda,\tau}^n$ is non-empty.
\end{lemma}
\begin{proof}
Prescribe the ASM boundary height values on $\mathcal B_n$ and the values $c^\tau$ on $V(\Gamma_\lambda)$.
We verify the hypotheses of \cref{lem:qtsasm:equivariant-height-extension}.
By \cref{lem:qtsasm:ferrers-defect-cycle}, the prescribed set $\mathcal B_n\cup V(\Gamma_\lambda)$ is $R$-invariant.
The boundary conditions hold by construction, and the parity and quarter-turn conditions on the cycle follow from \cref{lem:qtsasm:single-defect-cycle-data}.
It remains to verify the grid-distance condition.
The grid-distance condition holds for every pair of boundary positions.
If two cycle vertices lie on the same rotated arc, then their prescribed values differ by at most the length of the intervening subpath, which equals their grid distance because the arc is monotone.
For cycle vertices on different rotated arcs, quarter-turn symmetry reduces the verification to pairs with one vertex on $\Pi_\lambda$ and the other on $R\Pi_\lambda$ or $R^2\Pi_\lambda$; the comparison with $R^3\Pi_\lambda$ is symmetric to the comparison with $R\Pi_\lambda$.
Let $[v_m,v_{m+1}]$ be the edge labeled by $\tau$, and write $c^\tau(v_i)=s_i-2\varepsilon_i$ with $\varepsilon_i\in\{0,1\}$.
If $s_{m+1}-s_m=1$, then $\varepsilon_i=0$ for $i\in[0,m]$ and $\varepsilon_i=1$ for $i\in[m+1,M]$.
If $s_{m+1}-s_m=-1$, then $\varepsilon_i=1$ for $i\in[0,m]$ and $\varepsilon_i=0$ for $i\in[m+1,M]$.
Thus the edge labeled by $\tau$ is the unique transition between the two possible levels $s_i$ and $s_i-2$ along the northwest arc.

First compare $v_i=(x,y)$ with $Rv_j=(y',n-x')$, where $v_j=(x',y')$, and put $\varepsilon=\varepsilon_i+\varepsilon_j$.
The grid distance from $v_i$ to $Rv_j$ is $|x-y'|+n-x'-y$, while $c^\tau(v_i)-c^\tau(Rv_j)=s_i+s_j-n-2\varepsilon$.
Thus the two grid-distance inequalities are equivalent to $\varepsilon\leq\max\{x,y'\}$ and
\[
  \begin{cases}
    y+s_j\leq n+\varepsilon & \text{if } x\geq y',\\
    s_i+x'\leq n+\varepsilon & \text{if } x<y'.
  \end{cases}
\]
The first inequality can fail only when $x=y'=1$ and $\varepsilon_i=\varepsilon_j=1$.
This combination is impossible.
If the edge labeled by $\tau$ is horizontal, then $\varepsilon_r=0$ before that edge and $\varepsilon_r=1$ after it; a vertex with second coordinate $1$ occurs before every horizontal edge, so $\varepsilon_j=0$.
If the edge labeled by $\tau$ is vertical, then $\varepsilon_r=1$ before that edge and $\varepsilon_r=0$ after it; a vertex with first coordinate $1$ occurs after every vertical edge, so $\varepsilon_i=0$.
For the second display, suppose first that $x<y'$.
If $x'\geq k$, then $y'=1$ on the northwest arc, contradicting $x<y'$; hence $x'\leq k-1$, and $s_i+x'\leq(\ell+1)+(k-1)=n$.
Now suppose that $x\geq y'$.
If $y<k$, then $y+s_j\leq(k-1)+(\ell+1)=n$.
If $y=k$, then $v_i=(1,k)$, so $x\geq y'$ forces $y'=1$.
Thus $y+s_j\leq n+1$, with equality only when $v_j=v_0$.
For the endpoint pair $v_i=v_M$ and $v_j=v_0$ we have $\varepsilon_i+\varepsilon_j=1$, again by the preceding description of the $\varepsilon_i$, and the inequality $y+s_j\leq n+\varepsilon$ follows.
This proves all comparisons between adjacent arcs.

For opposite arcs, compare $v_i$ with $R^2v_j$.
The value at $R^2v_j$ is $c^\tau(v_j)$, and $|c^\tau(v_i)-c^\tau(v_j)|\leq |s_i-s_j|+2$.
On the other hand, the grid distance from $v_i$ to $R^2v_j$ is $|n-x_i-x_j|+|n-y_i-y_j|$, and this is at least $2n-s_i-s_j$.
Since every northwest-arc vertex has coordinate sum at most $\ell+1\leq n-1$,
\[
  2n-s_i-s_j-|s_i-s_j|
  =2\bigl(n-\max\{s_i,s_j\}\bigr)
  \geq2.
\]
Thus the opposite-arc comparisons also satisfy the grid-distance condition.
Finally, the comparison with the boundary is equivalent to
\[
  |x-y|\leq c^\tau(x,y)\leq\min\{x+y,2n-x-y\},
\]
which follows on the northwest arc from $c^\tau(x,y)\in\{x+y,x+y-2\}$; the other arcs follow by quarter-turn symmetry.
Thus the partial height specification satisfies the hypotheses of \cref{lem:qtsasm:equivariant-height-extension}.
The extension is a QTSASM height matrix.
Since the extension has the cycle values $c^\tau$, \cref{lem:qtsasm:single-defect-cycle-data} shows that its core lies in $\mathcal U_{\lambda,\tau}^n$.
\end{proof}

By \cref{prop:qtsasm:height-tangent-space,lem:qtsasm:fixed-defect-height-differences}, the directions in $A_n-A_n$ on which every boundary coordinate vanishes correspond to the orbit directions $\mathcal E_u$ whose orbits avoid $\mathcal B_n\cup V(\Gamma_\lambda)\cup\mathcal Z_n$.
We will realize each such direction as the height difference of two QTSASMs in the same set $\mathcal U_{\lambda,\tau}^n$.
The basic height modification changes the four positions in an $R$-orbit by alternating amounts $+2,-2,+2,-2$.
The next lemma states the local conditions under which this modification preserves the ASM height conditions.

\begin{lemma}\label{lem:qtsasm:local-flip-specification}
Let $n\geq5$, $n\not\equiv2\pmod4$, and let $\lambda$ be a non-zero Ferrers diagram contained in the staircase.
Let $u$ be a height-grid position, write $u_t=R^t u$ for $t\in[0,3]$, and let $N_t$ be the height-grid neighborhood consisting of $u_t$ and its height-grid neighbors.
Assume that $\mathcal O(u)\cap\bigl(\mathcal B_n\cup V(\Gamma_\lambda)\bigr)=\emptyset$ and that the grid distance from $u_s$ to $u_t$ is at least $2$ for all $s,t\in[0,3]$ with $s<t$.
Suppose there are $\tau\in\mathcal T_\lambda$, $a\in\Z$, and a partial height specification
\[
  g:\mathcal B_n\cup V(\Gamma_\lambda)\cup\bigcup_{t=0}^3 N_t\to\Z
\]
satisfying the hypotheses of \cref{lem:qtsasm:equivariant-height-extension}, such that $g$ has the ASM boundary values on $\mathcal B_n$, agrees with $c^\tau$ on $V(\Gamma_\lambda)$, and, for every $t\in[0,3]$ and $w\in N_t\setminus\{u_t\}$,
\[
  g(u_t)=
  \begin{cases}
    a & \text{if } t \text{ is even},\\
    n-a & \text{if } t \text{ is odd},
  \end{cases}
  \qquad
  \qquad
  g(w)=
  \begin{cases}
    a+1 & \text{if } t \text{ is even},\\
    n-a-1 & \text{if } t \text{ is odd}.
  \end{cases}
\]
Then there are two QTSASMs $X^-_u$ and $X^+_u$ whose cores both lie in $\mathcal U_{\lambda,\tau}^n$ and such that $H^{X^+_u}-H^{X^-_u}=2\mathcal E_u$.
\end{lemma}
\begin{proof}
By \cref{lem:qtsasm:equivariant-height-extension}, $g$ extends to a QTSASM height matrix $H^-=H^{X^-_u}$.
Define $H^+$ by adding $2$ at $u_0$ and $u_2$, subtracting $2$ at $u_1$ and $u_3$, and leaving all other height values unchanged.
The disjointness assumption ensures that the changed positions are neither boundary positions nor vertices of $\Gamma_\lambda$.
Thus the boundary values are unchanged and $H^+$ still agrees with $H^-$ on $V(\Gamma_\lambda)$.
The parity condition is unchanged, and the quarter-turn relation is preserved because the changes alternate in sign.
Only local conditions on edges incident to one of the four changed positions can change.
By the separation assumption, no such edge has both endpoints changed.
For even $t$, the definition of $g$ gives $H^-_{u_t}=a$ and $H^-_w=a+1$ for every neighbor $w$ of $u_t$.
If such a neighbor belongs to another part of the prescribed domain, the assumption that $g$ is a function ensures that the prescriptions agree.
Since $w$ is not changed, after adding $2$ at $u_t$ the adjacent difference across $[u_t,w]$ still has absolute value $1$.
For odd $t$, the same argument uses $H^-_{u_t}=n-a$ and $H^-_w=n-a-1$ for every neighbor $w$ of $u_t$.
After subtracting $2$ at $u_t$, the adjacent difference across $[u_t,w]$ still has absolute value $1$.
Thus $H^+$ is the height matrix of a QTSASM by \cref{lem:asm:height-functions}.
The identity $H^{X^+_u}-H^{X^-_u}=2\mathcal E_u$ follows from the definition of $\mathcal E_u$.
Since $H^-$ agrees with $c^\tau$ on $V(\Gamma_\lambda)$, \cref{lem:qtsasm:single-defect-cycle-data} shows that the core of $X^-_u$ lies in $\mathcal U_{\lambda,\tau}^n$.
Both cores lie in $A_n$ by \cref{thm:qtsasm:dimension}, and the two height matrices agree on $V(\Gamma_\lambda)$.
By \cref{lem:qtsasm:fixed-defect-height-differences}, their boundary coordinates are equal, so the core of $X^+_u$ also lies in $\mathcal U_{\lambda,\tau}^n$.
\end{proof}

To apply the preceding lemma, we must construct a partial height specification that is compatible with the boundary, the Ferrers cycle, and the four orbit neighborhoods.
We do this in three steps.
First, we choose a northwest representative of the orbit and an admissible height at its center.
Second, we choose $\tau\in\mathcal T_\lambda$ so that the cycle prescription $c^\tau$ is compatible with that center height.
Third, we verify all grid-distance inequalities for the combined prescription.

\begin{lemma}\label{lem:qtsasm:local-flip-center-level}
Let $n\geq5$, $n\not\equiv2\pmod4$, and let $\lambda$ be a non-zero Ferrers diagram contained in the staircase.
Let $u$ be a height-grid position whose $R$-orbit is disjoint from $\mathcal B_n\cup V(\Gamma_\lambda)\cup\mathcal Z_n$.
Then there exist $u'=(p,q)\in\mathcal O(u)$ with $p,q\in[k]$ and an integer $a$ such that, writing $u_t=R^t u'$, the orbit $\mathcal O(u')$ is disjoint from $\mathcal B_n\cup V(\Gamma_\lambda)$ and the grid distance from $u_s$ to $u_t$ is at least $2$ for all $s,t\in[0,3]$ with $s<t$.
Set $m=\min\{p,q\}$.
The choice can be made so that $p+q\leq n-1$ and $a\equiv p+q\pmod2$.
It also satisfies $\max\{|p-q|,m\}\leq a\leq\min\{p+q-2,n-m-2\}$.
Set $\theta=(p+q-2-a)/2$, and write $\Pi_\lambda=(v_0,\dots,v_M)$.
For $i\in[0,M-1]$, let $\tau_i$ be the label of the edge $[v_i,v_{i+1}]$.
Writing $v_i=(x_i,y_i)$, set $s_i=x_i+y_i$.
Then $\theta$ is a non-negative integer, $s_i\leq\ell+1$ for every $i$, and the only vertex of $\Pi_\lambda$ with second coordinate $k$ is $v_M=(1,k)$.
For $t\in\Z$, write $t^+=\max\{t,0\}$ and $t^-=\max\{-t,0\}$, and for every vertex $v_i$ set $\Delta_i^\pm=(x_i-p)^\pm +(y_i-q)^\pm$.
For $i\in[0,M]$ and $t\in[0,3]$, let $d_i^t$ denote the grid distance from $u'$ to $R^t v_i$.
Then
\begin{equation}\label{eq:qtsasm:finite-check-Delta-identities}
  s_i-d_i^0=p+q-2\Delta_i^-,
  \qquad
  s_i+d_i^0=p+q+2\Delta_i^+.
\end{equation}
Finally,
\begin{equation}\label{eq:qtsasm:Deltaminus-lower-bound}
  \Delta_i^-\geq \theta\quad\text{for every }i\text{ whenever }\theta>0.
\end{equation}
\end{lemma}
\begin{proof}
Choose a representative $u'=(p,q)\in\mathcal O(u)$ with $p,q\in[k]$; such a representative exists because every four-element quarter-turn orbit has one height-grid position in the closed northwest quadrant $[k]\times[k]$.
The orbit of $u'$ is the orbit of $u$, so it is disjoint from $\mathcal B_n$ and from $V(\Gamma_\lambda)$.
The orbit of $u'$ also satisfies the grid-distance separation condition in \cref{lem:qtsasm:local-flip-specification}.
After rotating, it is enough to compare $u'$ with $Ru'$ and $R^2u'$.
Since $p,q\in[k]$, these distances are $|p-q|+n-p-q$ and $2n-2p-2q$, respectively.
The first distance is at most $1$ only when $u'=(k,k)$, either as the fixed central position for $n=2k$ or as a position in $\mathcal O(k,k)$ for $n=2k+1$.
The second distance is at most $1$ only when $n=2k$ and $u'=(k,k)$.
All such positions lie in $\mathcal Z_n$, contrary to the hypothesis.
For the inequality $p+q\leq n-1$, if $n=2k$, then equality $p+q=2k$ forces $(p,q)=(k,k)$, the excluded even central height position; if $n=2k+1$, the choice $p,q\in[k]$ already gives $p+q\leq2k=n-1$.

Put $U=\min\{p+q-2,n-m-2\}$, and choose $a$ to be the largest integer at most $U$ that is congruent to $p+q$ modulo $2$.
This gives $a\equiv p+q\pmod2$ and the two upper bounds on $a$.
We next verify that $U\geq |p-q|$ and $U\geq m$.
For the first upper candidate, $p+q-2-|p-q|=2m-2\geq0$.
Also, $p+q-2-m=\max\{p,q\}-2\geq0$, where the inequality uses the exclusion of $(p,q)=(1,1)\in\mathcal Z_n$.
For the second upper candidate, the inequality $n-m-2\geq |p-q|$ is equivalent to $\max\{p,q\}\leq n-2$, which follows from $\max\{p,q\}\leq k$.
The inequality $n-m-2\geq m$ follows from $m\leq k-1$; equality $m=k$ would force $(p,q)=(k,k)$, which is excluded in both parity cases.
Thus $U$ is at least both lower endpoints.

The integer $a$ is either $U$ or $U-1$.
It remains to show that the possible decrease by one cannot cross a lower endpoint.
The candidate $p+q-2$ already has the prescribed parity.
The equality $n-m-2=|p-q|$ would force $\max\{p,q\}=n-2$, impossible because $\max\{p,q\}\leq k$.
If $n-m-2=m$, then $n=2k$ and $m=k-1$.
The position $(k-1,k-1)$ lies in $\mathcal Z_n$, so the only eligible positions are $(k-1,k)$ and $(k,k-1)$; for both, $n-m-2=k-1\equiv p+q\pmod2$ because $k$ is even.
In every remaining case, $U$ is at least one larger than each lower endpoint.
Therefore the parity adjustment preserves both lower bounds.

The integer $\theta$ is non-negative by the upper bound $a\leq p+q-2$, and the congruence $a\equiv p+q\pmod2$ makes it integral.
Every vertex of $\Pi_\lambda$ lies in the region $x+y\leq\ell+1$, so $s_i\leq\ell+1$.
Since the second coordinate increases along the path and reaches $k$ only at its terminal endpoint, $v_M=(1,k)$ is the unique vertex with second coordinate $k$.
The identities in \eqref{eq:qtsasm:finite-check-Delta-identities} follow by separating the positive and negative parts of $x_i-p$ and $y_i-q$ in the expression $d_i^0=|p-x_i|+|q-y_i|$.
It remains to prove \eqref{eq:qtsasm:Deltaminus-lower-bound}.
Since $p+q-2\equiv p+q\pmod2$, the inequality $\theta>0$ implies that the upper bound $p+q-2$ is not active in the definition of $U$; hence $U=n-m-2$.
As $a\geq U-1$, we have $\theta=(p+q-2-a)/2\leq(\max\{p,q\}+2m-n+1)/2$.
If $n=2k$ and $\max\{p,q\}\leq2$, then $U=p+q-2$ and hence $\theta=0$; thus $\theta>0$ forces $\max\{p,q\}\geq3$ in the even case.
If $n=2k+1$ and $\max\{p,q\}\leq3$, then again $U=p+q-2$ and $\theta=0$, except possibly for $u'=(k,k)$ when $k\in\{2,3\}$; this exceptional height-grid position lies in $\mathcal Z_n$ and is excluded.
Thus $\theta>0$ forces $\max\{p,q\}\geq4$ in the odd case.
Therefore this upper bound is at most $\max\{p,q\}+m-\ell-1=p+q-\ell-1$.
On the other hand, since $s_i\leq\ell+1$ and $t^-\geq -t$ for every integer $t$,
\[
  \Delta_i^-=(x_i-p)^- +(y_i-q)^-
  \geq p+q-s_i
  \geq p+q-\ell-1
  \geq \theta.\qedhere
\]
\end{proof}

\begin{lemma}\label{lem:qtsasm:local-flip-defect-label}
Let $n\geq5$, $n\not\equiv2\pmod4$, and let $\lambda$ be a non-zero Ferrers diagram contained in the staircase.
Assume that $u'=(p,q)$ and $a$ satisfy the conclusions of \cref{lem:qtsasm:local-flip-center-level}, and use the notation $m$, $\theta$, $\Pi_\lambda=(v_0,\dots,v_M)$, $v_i=(x_i,y_i)$, $s_i$, $\Delta_i^-$, and $\Delta_i^+$ from that lemma.
Set
\[
  \Delta_{\min}^-=\min_{i\in[0,M]}\Delta_i^-,
  \qquad
  \Delta_{\min}^+=\min_{i\in[0,M]}\Delta_i^+.
\]
Then $\tau\in\mathcal T_\lambda$ can be chosen according to exactly one of the following cases.
\begin{enumerate}[label=\textup{(\roman*)}]
\item
Suppose that $\Delta_{\min}^-=\theta$, and let
\[
  I^- = \{i\in[0,M]:\Delta_i^-=\theta\},
  \qquad
  \alpha=\min I^-,
  \qquad
  \beta=\max I^-.
\]
If $\alpha>0$, choose $\tau=\tau_{\alpha-1}$; then
\[
  c^\tau(v_i)=
  \begin{cases}
    s_i & \text{if } i\in[0,\alpha-1],\\
    s_i-2 & \text{if } i\in[\alpha,M].
  \end{cases}
\]
If $\alpha=0$, choose $\tau=\tau_\beta$; then $\beta<M$ and
\[
  c^\tau(v_i)=
  \begin{cases}
    s_i-2 & \text{if } i\in[0,\beta],\\
    s_i & \text{if } i\in[\beta+1,M].
  \end{cases}
\]

\item
Suppose that $\theta>0$ and $\Delta_{\min}^->\theta$.
Choose $\tau=\tau_0$; then
\[
  c^\tau(v_i)=
  \begin{cases}
    s_i-2 & \text{if } i=0,\\
    s_i & \text{if } i\in[1,M].
  \end{cases}
\]

\item
Suppose that $\theta=0$, $\Delta_{\min}^->0$, and $\Delta_{\min}^+=0$, and let
\[
  I^+ = \{i\in[0,M]:\Delta_i^+=0\},
  \qquad
  \alpha=\min I^+,
  \qquad
  \beta=\max I^+.
\]
If $\alpha>0$, choose $\tau=\tau_{\alpha-1}$; then
\[
  c^\tau(v_i)=
  \begin{cases}
    s_i-2 & \text{if } i\in[0,\alpha-1],\\
    s_i & \text{if } i\in[\alpha,M].
  \end{cases}
\]
If $\alpha=0$, choose $\tau=\tau_\beta$; then $\beta<M$ and
\[
  c^\tau(v_i)=
  \begin{cases}
    s_i & \text{if } i\in[0,\beta],\\
    s_i-2 & \text{if } i\in[\beta+1,M].
  \end{cases}
\]

\item
Suppose that $\theta=0$, $\Delta_{\min}^->0$, and $\Delta_{\min}^+>0$.
Choose $\tau=\tau_0$; then
\[
  c^\tau(v_i)=
  \begin{cases}
    s_i-2 & \text{if } i=0,\\
    s_i & \text{if } i\in[1,M].
  \end{cases}
\]
\end{enumerate}
\end{lemma}
\begin{proof}
The four cases are exhaustive.
If $\theta>0$, then \eqref{eq:qtsasm:Deltaminus-lower-bound} gives $\Delta_{\min}^-\geq\theta$, so either case~\textup{(i)} or case~\textup{(ii)} applies.
If $\theta=0$, then case~\textup{(i)} applies when $\Delta_{\min}^-=0$; otherwise $\Delta_{\min}^->0$, and exactly one of cases~\textup{(iii)} and~\textup{(iv)} applies according to whether $\Delta_{\min}^+$ is zero or positive.

We next verify that the specified choices of $\tau$ correspond to edges of $\Pi_\lambda$ and give the stated height levels.
Along $\Pi_\lambda$, the first coordinate is non-increasing and the second coordinate is non-decreasing.
Across a vertical edge, and across the middle edge in odd size, $\Delta_i^-$ cannot decrease; across a horizontal edge, it cannot increase.
Consequently, if the set $I^-$ in case~\textup{(i)} begins after $v_0$, then the edge entering $I^-$ is horizontal, whereas the edge leaving $I^-$ is vertical or, in odd size, the middle edge.
Similarly, $\Delta_i^+=0$ means $x_i\leq p$ and $y_i\leq q$.
Hence the set $I^+$ in case~\textup{(iii)} is an interval, entered through a vertical or middle edge and left through a horizontal edge.

Reversing the coordinate-sum change on a horizontal edge changes the cycle levels from $s_i$ before that edge to $s_i-2$ after it.
Reversing the change on a vertical or middle edge changes the levels from $s_i-2$ before that edge to $s_i$ after it.
These observations give all level assignments in cases~\textup{(i)} and~\textup{(iii)}.
The first edge of $\Pi_\lambda$ is vertical or, in odd size, the middle edge, so choosing $\tau_0$ gives the assignments in cases~\textup{(ii)} and~\textup{(iv)}.

It remains to verify that $\tau_\beta$ exists in the subcases with $\alpha=0$.
In case~\textup{(i)}, suppose for contradiction that $\beta=M$.
Then $\Delta_0^-=\Delta_M^-=\theta$.
If $\theta=0$, the endpoint formulas force $p=q=1$, so $u'=(1,1)\in\mathcal Z_n$.
If $\theta>0$, they give $p=q=\theta+1$, and hence $a=p+q-2-2\theta=0$, contradicting $m\leq a$.
Thus $\beta<M$ in case~\textup{(i)}.
In case~\textup{(iii)}, the equality $\beta=M$ would give $\Delta_0^+=\Delta_M^+=0$.
This forces $n=2k$ and $p=q=k$, so $u'=(k,k)\in\mathcal Z_n$.
Therefore $\beta<M$ in case~\textup{(iii)} as well.
\end{proof}

\begin{lemma}\label{lem:qtsasm:adjacent-arc-grid-distance-checks}
Let $n\geq5$, $n\not\equiv2\pmod4$, and let $\lambda$ be a non-zero Ferrers diagram contained in the staircase.
Assume that $u'=(p,q)$ and $a$ satisfy the conclusions of \cref{lem:qtsasm:local-flip-center-level}, and choose $\tau$ as in \cref{lem:qtsasm:local-flip-defect-label}.
Use the notation $\theta$, $\Pi_\lambda=(v_0,\dots,v_M)$, $v_i=(x_i,y_i)$, $s_i$, $\Delta_i^-$, $\Delta_i^+$, and $d_i^t$ from \cref{lem:qtsasm:local-flip-center-level}.
Let $N_0$ be the height-grid neighborhood of $u'$, put $c_i=c^\tau(v_i)$, and prescribe value $a$ at $u'$, value $a+1$ on $N_0\setminus\{u'\}$, and value $n-c_i$ at $Rv_i$ and $R^3v_i$.
Then these prescriptions agree on overlaps between $N_0$ and $R\Pi_\lambda\cup R^3\Pi_\lambda$ and satisfy all grid-distance inequalities between these sets.
\end{lemma}
\begin{proof}
We first compare $N_0$ with $R\Pi_\lambda$.
Put $c_i=c^\tau(v_i)$.
Then $Rv_i=(y_i,n-x_i)$, $c^\tau(Rv_i)=n-c_i$, and the grid distance from $u'$ to $Rv_i$ is $d_i^1=|p-y_i|+n-x_i-q$.
The grid-distance inequalities between the prescribed values on $N_0$ and the cycle value at $Rv_i$, including the case $Rv_i\in N_0$ where the two prescriptions must agree, are equivalent to
\begin{equation}\label{eq:qtsasm:R-arc-interval}
  n-c_i-d_i^1\leq a\leq n-c_i+d_i^1-2.
\end{equation}
The lower bound is tested at the center $u'$, and the upper bound is tested at a neighbor of $u'$ on a shortest path to $Rv_i$; if that neighbor is $Rv_i$, then the upper bound is exactly the condition that the neighborhood value and the cycle value at this common vertex agree.
We check the two possible levels separately.
If $c_i=s_i$, the left inequality in \eqref{eq:qtsasm:R-arc-interval} becomes $q-y_i-|p-y_i|\leq a$, which follows from $|p-q|\leq a$.
For the same level, the right inequality is equivalent, using $a=p+q-2-2\theta$, to
\[
  \begin{cases}
    q\leq n-s_i+\theta & \text{if } y_i\leq p,\\
    p+q\leq n-x_i+\theta & \text{if } y_i>p.
  \end{cases}
\]
It remains to identify the indices for which these inequalities could fail.
First suppose $y_i>p$.
Then $x_i+p\leq s_i-1$, so a failure of the second inequality would imply $s_i+q\geq n+\theta+2$.
But $s_i\leq k+1$ for $n=2k$ and $s_i\leq k+2$ for $n=2k+1$, while $q\leq k$; hence $s_i+q\leq n+1$, a contradiction.
Thus a failure can only occur with $y_i\leq p$.
In that case a failure gives $s_i+q\geq n+\theta+1$.
The same bounds on $s_i$ and $q$ force $\theta=0$, $q=k$, and $s_i=k+1$ when $n=2k$, respectively $s_i=k+2$ when $n=2k+1$.
Since $\theta=0$, we have $a=p+k-2$, and the upper bound $a\leq n-p-2$ gives $2p\leq k$ in the even case and $2p\leq k+1$ in the odd case.
Together with $y_i\leq p$ this gives $x_i=s_i-y_i>p$.
Thus $\Delta_i^+>0$.
Since $q=k$, the set $I^+=\{i:\Delta_i^+=0\}$ consists of the vertices for which $x_i\leq p$.
It is non-empty because $v_M=(1,k)$ belongs to it, while the inequality $x_i>p$ places every possible failing index before $I^+$.
Moreover, $p>1$, since $p=1$ would give $u'=(1,k)=v_M\in V(\Gamma_\lambda)$.
The only northwest-arc vertex with second coordinate $k$ is $v_M$, so $\Delta_{\min}^->0$.
Also $\Delta_0^+>0$, whereas $\Delta_M^+=0$.
Thus case~\textup{(iii)} of \cref{lem:qtsasm:local-flip-defect-label} applies with $\alpha>0$, and it assigns the level $s_i-2$ before $I^+$ and the level $s_i$ on and after $I^+$.
Therefore every vertex assigned the level $s_i$ satisfies \eqref{eq:qtsasm:R-arc-interval}.

If $c_i=s_i-2$, the right inequality in \eqref{eq:qtsasm:R-arc-interval} follows from $a\leq p+q-2$ once we prove $p+q-2\leq n-s_i+d_i^1$.
If $y_i\leq p$, this is equivalent to $s_i+q\leq n+1$, which follows from $s_i\leq \ell+1$ and $q\leq k$.
If $y_i>p$, it is equivalent to $x_i+p+q\leq n+1$; this follows from $x_i+p\leq s_i-1\leq \ell$ and $q\leq k$.
The left inequality becomes $q-y_i-|p-y_i|+2\leq a$.
We determine the remaining cases in which this inequality could fail.
If $y_i>p$, then the left-hand side is $p+q-2y_i+2\leq q-p\leq |p-q|\leq a$.
If $y_i\leq p$ and $q<p$, then the left-hand side is $q-p+2\leq p-q=|p-q|\leq a$.
Thus a failure requires $y_i\leq p\leq q$.
In this remaining case the left-hand side is $q-p+2$, while $a\geq q-p$ and $a\equiv q-p\pmod2$; hence failure is possible only if $a=q-p$.
The choice of $a$ then forces $p=1$.
To see this, suppose $p>1$.
Then the two upper candidates defining $U$ are at least $q-p+2$; the only case in which $n-p-2<q-p+2$ is $n=5$ and $p=q=2$, which lies in $\mathcal Z_n$ and is excluded.
Since $q-p+2\equiv q-p\pmod2$, the maximal choice of $a$ would give $a\geq q-p+2$, a contradiction.
Consequently, failure can occur only in the boundary case $p=1$, $a=q-1$, and $y_i=1$.
Here $\theta=0$.
The value $q=1$ would put $u'=(1,1)$ in $\mathcal Z_n$, so $q>1$.
The set $I^-=\{i:\Delta_i^-=0\}$ is non-empty because $v_M=(1,k)$ belongs to it, and any vertex with $y_i=1<q$ lies strictly before $I^-$ along the monotone arc.
Thus case~\textup{(i)} of \cref{lem:qtsasm:local-flip-defect-label} applies with $\theta=0$ and $\alpha>0$.
That case assigns the level $s_i$ before $I^-$ and the level $s_i-2$ from $I^-$ onward.
Therefore every vertex assigned the level $s_i-2$ satisfies \eqref{eq:qtsasm:R-arc-interval}.
We have proved all inequalities between $N_0$ and $R\Pi_\lambda$.

For $R^3\Pi_\lambda$, we repeat the calculation with the two coordinates interchanged.
Now $R^3v_i=(n-y_i,x_i)$, $c^\tau(R^3v_i)=n-c_i$, and the grid distance from $u'$ to $R^3v_i$ is $d_i^3=n-y_i-p+|q-x_i|$.
The grid-distance inequalities between the prescribed values on $N_0$ and the cycle value at $R^3v_i$, including the case $R^3v_i\in N_0$ where the two prescriptions must agree, are equivalent to
\begin{equation}\label{eq:qtsasm:R3-arc-interval}
  n-c_i-d_i^3\leq a\leq n-c_i+d_i^3-2.
\end{equation}
Here again the interval tests the center and a neighbor of $u'$ on a shortest path; if that neighbor is $R^3v_i$, the upper bound is the equality of the two prescribed values at this common vertex.
This interval is obtained from the $R\Pi_\lambda$ interval by interchanging $p$ with $q$ and $x_i$ with $y_i$.
Thus, for $c_i=s_i$, the left inequality becomes $p-x_i-|q-x_i|\leq a$, which follows from $|p-q|\leq a$, and the right inequality is equivalent to
\[
  \begin{cases}
    p\leq n-s_i+\theta & \text{if } x_i\leq q,\\
    p+q\leq n-y_i+\theta & \text{if } x_i>q.
  \end{cases}
\]
It remains to identify the indices for which this right inequality could fail.
If $x_i>q$, then a failure of the second inequality would imply $s_i+p\geq n+\theta+2$, because $y_i+q\leq s_i-1$.
This contradicts $s_i\leq\ell+1$ and $p\leq k$.
Thus a failure can only occur with $x_i\leq q$.
Then a failure gives $s_i+p\geq n+\theta+1$.
The bounds on $s_i$ and $p$ force $\theta=0$, $p=k$, and $s_i=k+1$ in the even case, respectively $s_i=k+2$ in the odd case.
In the odd case the equality $s_i=k+2$ occurs only at the initial endpoint $v_0=(k+1,1)$, which is incompatible with $x_i\leq q\leq k$.
Hence only the even case remains.
There $a=k+q-2$, and the upper bound $a\leq n-q-2$ gives $2q\leq k$; together with $x_i\leq q$ this gives $y_i=s_i-x_i>q$.
Thus the failing vertex lies strictly after the interval on which $\Delta_i^+=0$.
Also $q>1$, since $q=1$ would give $u'=(k,1)=v_0\in V(\Gamma_\lambda)$.
Therefore $\Delta_{\min}^->0$, while $\Delta_{\min}^+=0$ and the interval $I^+=\{i:\Delta_i^+=0\}$ begins at $0$.
Case~\textup{(iii)} of \cref{lem:qtsasm:local-flip-defect-label} applies with $\alpha=0$ and assigns the level $s_i-2$ after $I^+$.
Every possible failing index lies after $I^+$, so every vertex assigned the level $s_i$ satisfies \eqref{eq:qtsasm:R3-arc-interval}.
For $c_i=s_i-2$, the right inequality follows from $a\leq p+q-2$ and the estimate $p+q-2\leq n-s_i+d_i^3$.
If $x_i\leq q$, this is equivalent to $s_i+p\leq n+1$, which follows from $s_i\leq\ell+1$ and $p\leq k$.
If $x_i>q$, it is equivalent to $y_i+p+q\leq n+1$; this follows from $y_i+q\leq s_i-1\leq\ell$ and $p\leq k$.
The left inequality for $c_i=s_i-2$ becomes $p-x_i-|q-x_i|+2\leq a$.
As above with the coordinates interchanged, the cases $x_i>q$ and $p<q$ satisfy the inequality.
Any failure must have $x_i\leq q\leq p$.
The parity condition then makes failure possible only when $a=p-q$, and the maximal choice of $a$ forces $q=1$.
To see the last implication, suppose $q>1$.
Then the two upper candidates defining $U$ are at least $p-q+2$; the only case in which $n-q-2<p-q+2$ is $n=5$ and $p=q=2$, which lies in $\mathcal Z_n$ and is excluded.
Since $p-q+2\equiv p-q\pmod2$, the maximal choice of $a$ would give $a\geq p-q+2$, a contradiction.
Thus failure can occur only in the boundary case $q=1$, $a=p-1$, and $x_i=1$.
Here $\theta=0$.
Also $p>1$, since $p=q=1$ would put $u'=(1,1)$ in $\mathcal Z_n$.
Because $q=1$ and $p\leq k$, the initial vertex $v_0$ satisfies $\Delta_0^-=0$ in both parities: $v_0=(k,1)$ when $n$ is even and $v_0=(k+1,1)$ when $n$ is odd.
Thus the interval $I^-=\{i:\Delta_i^-=0\}$ begins at $v_0$, while every vertex with $x_i=1<p$ lies strictly after $I^-$ along the monotone arc.
Case~\textup{(i)} of \cref{lem:qtsasm:local-flip-defect-label} applies with $\theta=0$ and $\alpha=0$.
It assigns the level $s_i-2$ on $I^-$ and the level $s_i$ after $I^-$.
Therefore every vertex assigned the level $s_i-2$ satisfies \eqref{eq:qtsasm:R3-arc-interval}.
Thus all inequalities between $N_0$ and $R^3\Pi_\lambda$ hold.
\end{proof}

\begin{lemma}\label{lem:qtsasm:ferrers-local-flip-extension-checks}
Let $n\geq5$, $n\not\equiv2\pmod4$, and let $\lambda$ be a non-zero Ferrers diagram contained in the staircase, with $\Gamma_\lambda$, $\mathcal T_\lambda$, and the boundary coordinates $\omega_\tau$ as above.
Let $u$ be a height-grid position whose $R$-orbit is disjoint from $\mathcal B_n\cup V(\Gamma_\lambda)\cup\mathcal Z_n$.
Then there exist $u'\in\mathcal O(u)$, $\tau\in\mathcal T_\lambda$, and an integer $a$ such that, with $u'$ in place of $u$, the disjointness and separation hypotheses of \cref{lem:qtsasm:local-flip-specification} hold and the boundary, cycle, and neighborhood values prescribed in that lemma satisfy the hypotheses of \cref{lem:qtsasm:equivariant-height-extension}.
\end{lemma}
\begin{proof}
Choose $u'=(p,q)$ and $a$ by \cref{lem:qtsasm:local-flip-center-level}, and choose $\tau$ by \cref{lem:qtsasm:local-flip-defect-label}.
Use the notation $m$, $\theta$, $\Pi_\lambda=(v_0,\dots,v_M)$, $v_i=(x_i,y_i)$, $s_i$, $\Delta_i^-$, $\Delta_i^+$, and $d_i^t$ from \cref{lem:qtsasm:local-flip-center-level}.
The disjointness and separation hypotheses in \cref{lem:qtsasm:local-flip-specification} hold by \cref{lem:qtsasm:local-flip-center-level}.
Let $u_t=R^t u'$ for $t\in[0,3]$, and let $N_t$ be the height-grid neighborhood consisting of $u_t$ and its height-grid neighbors.
Let $g$ be the partial height specification with the ASM boundary values, the cycle values $c^\tau$, and the following neighborhood values.
On $N_t$, the center $u_t$ has value $a$ for even $t$ and $n-a$ for odd $t$, and each height-grid neighbor of $u_t$ has value $a+1$ for even $t$ and $n-a-1$ for odd $t$.
We verify that $g$ satisfies the hypotheses of \cref{lem:qtsasm:equivariant-height-extension}.

First compare $N_0$ with the boundary.
Let $w=(r,s)\in N_0$, and let $d\in\{0,1\}$ be the grid distance from $w$ to $u'$.
Then $g_w=a+d$.
The grid-distance inequalities from $w$ to the four sides are $|r-s|\leq a+d\leq r+s$ and $a+d\leq2n-r-s$.
These are exactly the inequalities against the top, left, bottom, and right boundary values.
Since $|r-s|\leq |p-q|+d\leq a+d$, the lower inequality holds.
Also $r+s\geq p+q-1$ and $a+d\leq a+1\leq p+q-1$, so $a+d\leq r+s$.
Finally, $p+q\leq n-1$ gives $r+s\leq p+q+1\leq n$, while $a+d\leq n-m-1\leq n-2$; hence $a+d\leq2n-r-s$.
If a height-grid position in $N_0$ also lies on the boundary, then the grid-distance condition with that boundary position has distance zero, so it says that the neighborhood value equals the boundary value.
The other boundary cases follow by quarter-turn symmetry.

Next compare neighborhoods.
The grid distance between adjacent neighborhood centers is $\Delta=|p-q|+|p+q-n|=n-2m$.
Let $w\in N_0$ and $w'\in N_1$, and let $d$ and $d'$ be their grid distances from $u_0$ and $u_1$, respectively.
Then the grid distance from $w$ to $w'$ is at least $\Delta-d-d'$.
The prescribed values are $a+d$ and $n-a-d'$.
The lower bound $a\geq m$ gives $n-2a-d-d'\leq\Delta-d-d'$, while the upper bound $a\leq n-m-2$ gives $2a+d+d'-n\leq\Delta-4+d+d'\leq\Delta-d-d'$.
Thus the absolute difference between the two prescribed values is at most the grid distance from $w$ to $w'$.
The other adjacent pairs are quarter-turn images.
If $\Delta=2$, then $n=2k$, $m=k-1$, and $a=k-1$; every height-grid position common to two adjacent neighborhoods is at distance one from both centers and receives the common value $k$.
If $\Delta\leq1$, then $u'$ lies in $\mathcal Z_n$, which is excluded.
For opposite neighborhoods, the prescribed values differ by at most one.
If opposite neighborhoods have a common height-grid position, that position is at distance one from both centers and receives the same value.
Hence all neighborhood-neighborhood inequalities hold, and the neighborhood prescriptions agree at common height-grid positions.

We now compare $N_0$ with the northwest arc.
For all $w\in N_0$, the grid-distance inequalities between $w$ and $v_i$ are equivalent to
\begin{equation}\label{eq:qtsasm:neighborhood-cycle-center-form}
  c^\tau(v_i)-d_i^0\leq a\leq c^\tau(v_i)+d_i^0-2.
\end{equation}
The lower bound comes from $w=u'$, and the upper bound comes from the neighbor of $u'$ on a shortest path to $v_i$; if that neighbor is $v_i$, this is exactly the condition that the neighborhood value and the cycle value at $v_i$ agree.
Using \eqref{eq:qtsasm:finite-check-Delta-identities}, the two possible levels give the following criteria.
If $c^\tau(v_i)=s_i$, then \eqref{eq:qtsasm:neighborhood-cycle-center-form} holds exactly when $\Delta_i^-\geq \theta+1$.
If $c^\tau(v_i)=s_i-2$, then \eqref{eq:qtsasm:neighborhood-cycle-center-form} holds exactly when $\Delta_i^-\geq \theta$ and $\Delta_i^++\theta\geq1$.
Suppose first that $\theta>0$.
By \cref{lem:qtsasm:local-flip-defect-label}, every vertex with $\Delta_i^-=\theta$ is assigned the level $s_i-2$, while the level $s_i$ is assigned only when $\Delta_i^-\geq\theta+1$.
The bound \eqref{eq:qtsasm:Deltaminus-lower-bound} gives $\Delta_i^-\geq\theta$ at every vertex, and $\Delta_i^++\theta\geq1$ is automatic.
Thus \eqref{eq:qtsasm:neighborhood-cycle-center-form} holds when $\theta>0$.

Now suppose that $\theta=0$.
If case~\textup{(i)} of \cref{lem:qtsasm:local-flip-defect-label} applies with $\alpha>0$, then $I^-=[\alpha,\beta]$ and the level $s_i-2$ is assigned for every $i\in[\alpha,M]$.
On $I^-$, the inequality $\Delta_i^+>0$ follows from $v_i\neq u'$, which holds because the orbit of $u'$ is disjoint from the cycle.
For $i>\beta$, monotonicity gives $y_i\geq q$.
Equality $y_i=q$ would force the edge leaving $I^-$ to end at $(p,q)=u'$, so in fact $y_i>q$ and again $\Delta_i^+>0$.
The vertices before $I^-$ are assigned the level $s_i$ and satisfy $\Delta_i^->0$.
If case~\textup{(i)} applies with $\alpha=0$, then the level $s_i-2$ is assigned exactly on $I^-=[0,\beta]$, where $\Delta_i^+>0$ because $v_i\neq u'$, and every later vertex is assigned $s_i$ and has $\Delta_i^->0$.
In case~\textup{(iii)}, the level $s_i-2$ is assigned only where $\Delta_i^+>0$, while the level $s_i$ is assigned only where $\Delta_i^->0$.
In case~\textup{(iv)}, both $\Delta_i^-$ and $\Delta_i^+$ are positive at every vertex, so the two criteria also hold.
Hence \eqref{eq:qtsasm:neighborhood-cycle-center-form} holds for every $i$ when $\theta=0$.
This proves all inequalities between $N_0$ and $\Pi_\lambda$.

For $R^2\Pi_\lambda$, let $c_i=c^\tau(v_i)\in\{s_i,s_i-2\}$.
Then $R^2v_i=(n-x_i,n-y_i)$, $c^\tau(R^2v_i)=c_i$, and the grid distance from $u'$ to $R^2v_i$ is $d_i^2=2n-s_i-p-q$.
The grid-distance inequalities are equivalent to $c_i-d_i^2\leq a\leq c_i+d_i^2-2$.
If $c_i=s_i$, the lower endpoint is $2s_i+p+q-2n$; if $c_i=s_i-2$, it is $2s_i+p+q-2n-2$.
For the lower endpoints, use the parity-specific bound on $s_i$.
If $n=2k$, then $s_i\leq k+1$, while if $n=2k+1$, then $s_i\leq k+2$.
Together with $p+q\leq m+k$, this gives $2s_i+p+q-2n\leq m$ in both cases, and the lower endpoint for $c_i=s_i-2$ is smaller by $2$.
Thus both lower endpoints are at most $m\leq a$.
For the upper endpoints, $a\leq p+q-2$ and $p+q\leq n-1$ give $p+q-2\leq2n-p-q-4\leq2n-p-q-2$.
Hence all inequalities between $N_0$ and $R^2\Pi_\lambda$ hold, including cases where a position lies in both sets.

The adjacent-arc inequalities between $N_0$ and $R\Pi_\lambda\cup R^3\Pi_\lambda$, together with agreement at common vertices, are the conclusions of \cref{lem:qtsasm:adjacent-arc-grid-distance-checks}.

The proof of \cref{lem:qtsasm:single-defect-sets-non-empty} verifies the grid-distance inequalities involving only boundary and cycle vertices for these boundary values and cycle values $c^\tau$.
The preceding paragraphs have checked all inequalities involving at least one neighborhood position, including the cases where a position receives two prescriptions.
Because $R$ is an isometry and the prescription satisfies $g_{Rv}=n-g_v$, the corresponding inequalities for $N_1,N_2,N_3$ follow by quarter-turn symmetry.
For the parity condition, let $w\in N_0$ and let $d\in\{0,1\}$ be the grid distance from $w$ to $u'$.
Then $g_w=a+d\equiv p+q+d\equiv w_1+w_2\pmod2$.
The other neighborhoods follow by quarter-turn symmetry, and the parity on $V(\Gamma_\lambda)$ follows from \cref{lem:qtsasm:single-defect-cycle-data}.
The prescribed values satisfy the quarter-turn condition by construction on the neighborhoods and by the relation $c^\tau(Rv)=n-c^\tau(v)$ on the cycle.
The preceding checks also show that every position belonging to more than one prescribed set receives the same value from all prescriptions that contain it.
Thus $g$ satisfies the hypotheses of \cref{lem:qtsasm:equivariant-height-extension}.
\end{proof}

\begin{lemma}\label{lem:qtsasm:local-flips-away-from-ferrers-cycle}
Let $n\geq5$, $n\not\equiv2\pmod4$, and let $\lambda$ be a non-zero Ferrers diagram contained in the staircase.
Let $u$ be a height-grid position whose $R$-orbit is disjoint from $\mathcal B_n\cup V(\Gamma_\lambda)\cup\mathcal Z_n$.
Then there exist $\tau\in\mathcal T_\lambda$ and two QTSASMs $X^-_u$ and $X^+_u$ whose cores both lie in $\mathcal U_{\lambda,\tau}^n$ such that $H^{X^+_u}-H^{X^-_u}=2\mathcal E_u$.
\end{lemma}
\begin{proof}
By \cref{lem:qtsasm:ferrers-local-flip-extension-checks}, choose $u'\in\mathcal O(u)$, $\tau\in\mathcal T_\lambda$, and $a\in\Z$ so that, with $u'$ in place of $u$, all hypotheses of \cref{lem:qtsasm:local-flip-specification} are satisfied.
Applying that lemma gives a pair with height difference $2\mathcal E_{u'}$.
If $u'=R^{2m}u$ for some integer $m$, then $\mathcal E_{u'}=\mathcal E_u$.
If $u'=R^{2m+1}u$ for some integer $m$, then $\mathcal E_{u'}=-\mathcal E_u$, and we swap $X^-_u$ and $X^+_u$.
\end{proof}

We now combine the height directions realized above with differences between the sets $\mathcal U_{\lambda,\tau}^n$.
The first spanning lemma treats directions on which every boundary coordinate vanishes.

\begin{lemma}\label{lem:qtsasm:fixed-defect-spanning}
Let $n\geq5$, $n\not\equiv2\pmod4$, and let $\lambda$ be a non-zero Ferrers diagram contained in the staircase.
For every $\tau\in\mathcal T_\lambda$, the set $\mathcal U_{\lambda,\tau}^n$ is non-empty.
Moreover,
\[
  \sum_{\tau\in\mathcal T_\lambda}
  \linspan\{Y-Y':Y,Y'\in\mathcal U_{\lambda,\tau}^n\}
  =
  \{V\in A_n-A_n:\omega_\tau(V)=0\text{ for every }\tau\in\mathcal T_\lambda\}.
\]
\end{lemma}
\begin{proof}
The first assertion follows from \cref{lem:qtsasm:single-defect-sets-non-empty}.
If $Y,Y'\in\mathcal U_{\lambda,\tau}^n$, then $Y-Y'\in A_n-A_n$ and every boundary coordinate vanishes on $Y-Y'$.
Thus the left-hand side is contained in the right-hand side.
By \cref{prop:qtsasm:height-tangent-space}, $\Psi_n(A_n)-\Psi_n(A_n)$ has basis $\mathcal E_u$, where $u$ ranges over representatives of the four-element $R$-orbits disjoint from $\mathcal B_n\cup\mathcal Z_n$.
By \cref{lem:qtsasm:fixed-defect-height-differences}, for a direction $V=Y-Y'\in A_n-A_n$, the conditions $\omega_\tau(V)=0$ for all $\tau$ are equivalent to requiring the height difference $\Psi_n(Y)-\Psi_n(Y')$ to vanish on $V(\Gamma_\lambda)$.
Since $\Gamma_\lambda$ is $R$-invariant, the added vanishing condition removes precisely the basis directions whose orbits meet $V(\Gamma_\lambda)$.
Therefore the remaining basis directions are precisely those whose orbits are disjoint from $\mathcal B_n\cup V(\Gamma_\lambda)\cup\mathcal Z_n$.
By \cref{lem:qtsasm:local-flips-away-from-ferrers-cycle}, every vector $2\mathcal E_u$ for such an orbit direction is the height difference of two QTSASMs whose cores lie in the same set $\mathcal U_{\lambda,\tau}^n$.
The corresponding core difference therefore lies in the left-hand side.
Injectivity of $\Psi_n$ gives the asserted spanning equality in core coordinates.
\end{proof}

\begin{lemma}\label{lem:qtsasm:single-defect-spanning}
Let $n\geq5$, $n\not\equiv2\pmod4$, and let $\lambda$ be a non-zero Ferrers diagram contained in the staircase.
Let
\[
  F_\lambda^{=}
  =
  \{Y\in P^\core_{\QTSASM(n)}:F_\lambda(Y)=0\}.
\]
Then
\[
  \linspan\{Y-Y':Y,Y'\in F_\lambda^{=}\cap\pi_C(\QTSASM(n))\}
  =
  \{V\in A_n-A_n:F_\lambda(V)=0\}.
\]
\end{lemma}
\begin{proof}
By \cref{lem:qtsasm:ferrers-defect-identity}, on $A_n$ we have
\[
  1+2F_\lambda(Y)=\sum_{\tau\in\mathcal T_\lambda}\omega_\tau(Y).
\]
Thus, for directions $V\in A_n-A_n$,
\[
  2F_\lambda(V)=\sum_{\tau\in\mathcal T_\lambda}\omega_\tau(V).
\]
Let
\[
  W=\linspan\{Y-Y':Y,Y'\in F_\lambda^{=}\cap\pi_C(\QTSASM(n))\}.
\]
If $Y,Y'\in F_\lambda^{=}\cap\pi_C(\QTSASM(n))$, then $Y-Y'\in A_n-A_n$ and
$F_\lambda(Y-Y')=F_\lambda(Y)-F_\lambda(Y')=0$.
Therefore
\[
  W\subseteq\{V\in A_n-A_n:F_\lambda(V)=0\}.
\]
For each $\tau$, choose $Y_\tau\in\mathcal U_{\lambda,\tau}^n$, possible by \cref{lem:qtsasm:fixed-defect-spanning}.
By the equality characterization in \cref{lem:qtsasm:ferrers-defect-identity}, each $Y_\tau$ lies in $F_\lambda^{=}\cap\pi_C(\QTSASM(n))$.
Thus $Y_\tau-Y_\sigma\in W$.
Moreover, $\omega_\tau(Y_\tau-Y_\sigma)=1$, $\omega_\sigma(Y_\tau-Y_\sigma)=-1$, and every other boundary coordinate vanishes on $Y_\tau-Y_\sigma$.
These vectors span all vectors $u\in\R^{\mathcal T_\lambda}$ with $\sum_\tau u_\tau=0$.
If $V\in A_n-A_n$ satisfies $F_\lambda(V)=0$, then $\sum_\tau \omega_\tau(V)=0$.
The differences $Y_\tau-Y_\sigma$ therefore allow us to choose $W_0\in W$ such that $V'=V-W_0$ has every boundary coordinate equal to zero.
By \cref{lem:qtsasm:fixed-defect-spanning}, $V'$ lies in the sum of the spans generated by differences within the individual sets $\mathcal U_{\lambda,\tau}^n$.
Every such difference is a difference of tight QTSASM cores by \cref{lem:qtsasm:ferrers-defect-identity}, so each of these spans is contained in $W$.
Hence $V'\in W$, and then $V=V'+W_0\in W$.
\end{proof}

\begin{lemma}\label{lem:qtsasm:ferrers-facets}
Let $n\geq5$ and $n\not\equiv2\pmod4$.
For every non-zero Ferrers diagram $\lambda$ contained in the staircase, the Ferrers inequality
\[
  F_\lambda(Y)=\sum_{i=1}^{k-1}\sum_{j=1}^{\lambda_i} y_{i+1,j+1}\geq0
\]
defines a facet of $P^\core_{\QTSASM(n)}$.
\end{lemma}
\begin{proof}
Validity follows from \cref{lem:qtsasm:ferrers-defect-identity}.
Since $\lambda$ is non-zero, the left-hand side $F_\lambda$ is non-zero.
The coordinate support of $F_\lambda$ is contained in $T_k$.
To see this, suppose that the term $y_{i+1,j+1}$ occurs in $F_\lambda$.
Then
\[
  (i+1)+(j+1)\leq i+1+\lambda_i+1\leq i+1+(k-1-i)+1=k+1.
\]
By \cref{lem:qtsasm:no-shifted-staircase-affine-hull-equation}, $F_\lambda$ is not constant on $A_n$.
Hence
\[
  \dim\{V\in A_n-A_n:F_\lambda(V)=0\}=\dim A_n-1.
\]
By \cref{lem:qtsasm:single-defect-spanning}, the linear span of differences of tight QTSASM cores is exactly this direction space.
Therefore the tight face has dimension
\[
  \dim A_n-1=\dim P^\core_{\QTSASM(n)}-1,
\]
where the last equality follows from $A_n=\aff(P^\core_{\QTSASM(n)})$ in \cref{thm:qtsasm:dimension}.
Thus the inequality defines a facet.
\end{proof}

\begin{lemma}\label{lem:qtsasm:distinct-ferrers-facets}
Let $n\geq5$ and $n\not\equiv2\pmod4$.
If $\lambda$ and $\mu$ are distinct non-zero Ferrers diagrams, then no scalar $\alpha$ makes $F_\lambda-\alpha F_\mu$ constant on $A_n$.
Consequently, whenever the corresponding Ferrers inequalities define facets of $P^\core_{\QTSASM(n)}$, those facets are distinct.
\end{lemma}
\begin{proof}
The coordinate support of $F_\eta$ is contained in $T_k$ for every Ferrers diagram $\eta$ contained in the staircase.
To see this, suppose that the coordinate $y_{i+1,j+1}$ occurs in $F_\eta$.
Then
\[
  (i+1)+(j+1)\leq i+1+\eta_i+1\leq k+1.
\]
Thus every difference $F_\lambda-\alpha F_\mu$ is supported in $T_k$.
Suppose that, for some scalar $\alpha$, the combination $F_\lambda-\alpha F_\mu$ is constant on $A_n$.
If $\alpha=0$, then $F_\lambda$ is constant on $A_n$, and \cref{lem:qtsasm:no-shifted-staircase-affine-hull-equation} gives $F_\lambda=0$, impossible because $\lambda$ is non-zero.
Hence $\alpha\neq0$.
By \cref{lem:qtsasm:no-shifted-staircase-affine-hull-equation}, the combination $F_\lambda-\alpha F_\mu$ is zero.
Thus $F_\lambda=\alpha F_\mu$.
The coefficient vectors of $F_\lambda$ and $F_\mu$ are non-zero $0/1$-vectors, so $\alpha=1$ and the same coordinates occur in the two sums.
Hence $\lambda_i=\mu_i$ for every $i$, so $\lambda=\mu$, a contradiction.
If two facet-defining Ferrers inequalities determined the same facet, then their supporting hyperplanes in the affine space $A_n$ would coincide.
Their restrictions to $A_n$ would therefore differ by multiplication by a non-zero scalar and addition of a constant.
Equivalently, $F_\lambda-\alpha F_\mu$ would be constant on $A_n$ for some non-zero scalar $\alpha$, which the first part excludes.
\end{proof}

\begin{proposition}\label{prop:qtsasm:ferrers-facet-lower-bound}
Let $n\geq4$, $n\not\equiv2\pmod4$, and put $k=\floor*{n/2}$.
The number of facets of $P_{\QTSASM(n)}$ is at least
\[
  \Catalan_{k-1}-1
  =
  \frac1k\binom{2k-2}{k-1}-1.
\]
Consequently, the number of facets of $P_{\QTSASM(n)}$ is at least
\[
  \frac{4^{k-1}}{k(2k-1)}-1
  \geq
  \frac{2^{n-3}}{n^2}-1,
\]
and is therefore $2^{\Omega(n)}$ as $n\to\infty$ with $n\not\equiv2\pmod4$.
\end{proposition}
\begin{proof}
If $n=4$, then $k=2$ and $\Catalan_{k-1}-1=0$.
Assume from now on that $n\geq5$.

By \cref{lem:qtsasm:ferrers-facets,lem:qtsasm:distinct-ferrers-facets}, non-zero Ferrers diagrams contained in the staircase give distinct facets of $P^\core_{\QTSASM(n)}$.
By \cref{thm:xasm:assembly}, the QTSASM assembly map sends $P^\core_{\QTSASM(n)}$ onto $P_{\QTSASM(n)}$.
The core projection is inverse to the assembly map on $P^\core_{\QTSASM(n)}$, so this restriction is an affine isomorphism from $P^\core_{\QTSASM(n)}$ onto $P_{\QTSASM(n)}$.
Thus these core facets give distinct facets of $P_{\QTSASM(n)}$.

The Ferrers diagrams contained in the staircase with row lengths $k-2,\dots,0$ are counted by the Catalan number $\Catalan_{k-1}$, by the standard lattice-path bijection.
Removing the diagram $(0,\dots,0)$ removes exactly the tautological inequality $0\geq0$.
Therefore the number of facets of $P_{\QTSASM(n)}$ is at least
\[
  \Catalan_{k-1}-1
  =
  \frac1k\binom{2k-2}{k-1}-1.
\]

Using $\binom{2m}{m}\geq 4^m/(2m+1)$ with $m=k-1$, we get
\[
  \Catalan_{k-1}
  =
  \frac1k\binom{2k-2}{k-1}
  \geq
  \frac{4^{k-1}}{k(2k-1)}.
\]
Consequently, the number of facets of $P_{\QTSASM(n)}$ is at least
\[
  \frac{4^{k-1}}{k(2k-1)}-1.
\]
If $n=2k$, then $4^{k-1}=2^{n-2}$ and $k(2k-1)<n^2$.
If $n=2k+1$, then $4^{k-1}=2^{n-3}$ and again $k(2k-1)<n^2$.
In both cases,
\[
  \frac{4^{k-1}}{k(2k-1)}
  \geq
  \frac{2^{n-3}}{n^2}.
\]
This proves the proposition.
\end{proof}

\section{Diagonally symmetric ASMs (DSASMs)}\label{sec:DSASM}
In this class, we impose invariance under reflection across the main (northwest--southeast) diagonal; in other words, the matrix equals its transpose.
The corresponding symmetry subgroup is $\mathcal{G} = \{\mathcal{I}, \mathcal{D}\}$.
Let $P_\DS$ denote the linear subspace of diagonally symmetric real matrices, i.e.,
\[
  P_\DS
  =
  \left\{
    X \in \R^{n \times n} : x_{i,j} = x_{j,i}\ \forall i,j \in [n]
  \right\}.
\]
Clearly, any DSASM satisfies the ASM constraints in~\crefrange{eq:asm:real}{eq:asm:col-sum} and also the symmetry constraints defining $P_\DS$; thus $P_\DSASM \subseteq P_\ASM \cap P_\DS$.
We will prove that these constraints are in fact sufficient, i.e., $P_\DSASM = P_\ASM \cap P_\DS$.

\paragraph{Core and assembly map.}
Let the \emph{core} of a DSASM be its upper-triangular part including the main diagonal, i.e.,
\[
  C = \left\{(i,j) \in [n] \times [n] : i \leq j\right\}
\]
is the set of \emph{core positions}, and let $\pi_C$ be the coordinate-wise projection onto $C$.
Define the affine map $\varphi : \R^C \to \R^{n \times n}$ by
\[
  \varphi(Y)_{i,j} =
  \begin{cases}
    y_{i,j} & \text{if } i \in [j],\\
    y_{j,i} & \text{if } i \in [j+1, n]
  \end{cases}
\]
for $Y \in \R^C$ and $i,j \in [n]$.
Thus $\varphi$ places the core $Y$ on and above the main diagonal and fills the lower-triangular part by reflection across the main diagonal, yielding a diagonally symmetric matrix.
Clearly, the map $\varphi$ is an assembly map: it is affine, satisfies $\pi_C(\varphi(Y)) = Y$ for every $Y \in \R^C$, and $\varphi(\pi_C(X)) = X$ for every $X \in \DSASM(n)$, because $X$ is completely determined by its entries in $C$ together with the imposed diagonal symmetry.

\bigskip
We now describe the core polytope of DSASMs.
\begin{theorem}\label{thm:dsasm:corepolytope}
  Let $n \geq 1$.
  Then the core polytope $P^\core_\DSASM \subseteq \R^C$ of $n \times n$ DSASMs is described by the following system.
  \begin{align}
                                   y&_{i,j} \in \R                            &\forall i \in [n], j \in [i,n], \label{eq:dsasm:real}\\
    0 \leq \sum_{i'=1}^{\min\{i,j\}} y&_{i',i} + \sum_{j'=i+1}^{j} y_{i,j'} \leq 1 &\forall i \in [n], j \in [n-1], \label{eq:dsasm:L-prefix}\\
    \sum_{i'=1}^{i}                 y&_{i',i} + \sum_{j'=i+1}^{n} y_{i,j'} = 1   &\forall i \in [n]. \label{eq:dsasm:L-sum}
  \end{align}
\end{theorem}
\begin{proof}
  We show that the integer solutions to the system~\crefrange{eq:dsasm:real}{eq:dsasm:L-sum} are exactly the cores of DSASMs, and then we argue that the system defines an integral polytope.

  \medskip
  First, let $X$ be an $n\times n$ DSASM, and let $Y=\pi_C(X)$ be its core.
  Since $X$ is diagonally symmetric, we have $X=\varphi(Y)$.
  By the definition of $\varphi$, for every $i\in[n]$ and $j\in[n-1]$,
  \begin{equation}\label{eq:dsasm:phi-prefix-used}
    \sum_{j'=1}^{j} x_{i,j'} = \sum_{i'=1}^{\min\{i,j\}} y_{i',i} + \sum_{j'=i+1}^{j} y_{i,j'},
  \end{equation}
  and, in the case $j=n$,
  \begin{equation}\label{eq:dsasm:phi-sum-used}
    \sum_{j'=1}^{n} x_{i,j'} = \sum_{i'=1}^{i} y_{i',i} + \sum_{j'=i+1}^{n} y_{i,j'}.
  \end{equation}
  Since every DSASM is in particular an ASM, the ASM row-prefix bounds in~\cref{eq:asm:row-prefix}, together with the identity~\cref{eq:dsasm:phi-prefix-used}, yield the constraints in~\cref{eq:dsasm:L-prefix}; the row-sum equations in~\cref{eq:asm:row-sum}, together with the identity~\cref{eq:dsasm:phi-sum-used}, give the equations in~\cref{eq:dsasm:L-sum}; and the real-variable condition~\cref{eq:dsasm:real} follows from the ASM real-variable condition~\cref{eq:asm:real}.
  Thus the cores of DSASMs satisfy the system~\crefrange{eq:dsasm:real}{eq:dsasm:L-sum}.

  \medskip
  Second, let $Y\in\Z^C$ satisfy the system~\crefrange{eq:dsasm:real}{eq:dsasm:L-sum}, and set $X=\varphi(Y)$.
  By construction, $X$ is an $n\times n$ integer matrix with $X=X^\top$ and core $Y$.
  Using the identities~\cref{eq:dsasm:phi-prefix-used,eq:dsasm:phi-sum-used}, the constraints in~\cref{eq:dsasm:L-prefix,eq:dsasm:L-sum} translate exactly into the ASM row-prefix and row-sum constraints for $X$.
  By symmetry, the ASM column-prefix bounds and column sums follow as well.
  Hence $X$ satisfies all ASM constraints from \cref{thm:ASMpolytope} and, being diagonally symmetric, $X$ is a DSASM\@.
  We conclude that the integer solutions to the system~\crefrange{eq:dsasm:real}{eq:dsasm:L-sum} are precisely the cores of DSASMs.

  \medskip
  It remains to prove that the system~\crefrange{eq:dsasm:real}{eq:dsasm:L-sum} defines an integral polytope.
  For $i \in [0,n], j \in [i,n+1]$, let
  \[
    z_{i,j} = \sum_{i'=1}^{i} \sum_{j'=j}^{n} y_{i',j'}.
  \]
  Note that $z_{0,j}=z_{i,n+1}=0$ by empty sums.
  This defines an injective linear map $Y\mapsto Z$ with inverse
  \[
    y_{i,j} = z_{i,j}-z_{i,j+1}-z_{i-1,j}+z_{i-1,j+1}
  \]
  for $i \in [n], j \in [i,n]$.
  For $i \in [n]$ and $j \in [i]$, we have the identity
  \begin{equation}\label{eq:dsasm:core:id1}
    \sum_{i'=1}^{j} y_{i',i} = z_{j,i} - z_{j,i+1},
  \end{equation}
  and, for $i \in [n]$ and $j \in [i,n]$,
  \begin{equation}\label{eq:dsasm:core:id2}
    \sum_{j'=j+1}^{n} y_{i,j'} = z_{i,j+1} - z_{i-1,j+1}.
  \end{equation}
  Now we prove that the constraints in~\cref{eq:dsasm:L-prefix,eq:dsasm:L-sum} are equivalent to
  \begin{align}
    0 \leq z&_{j,i}-z_{j,i+1} \leq 1     & \forall i\in[n], j\in[i-1], \label{eq:dsasm:L-prefix-Z-1}\\
    0 \leq z&_{i,j+1}-z_{i-1,j+1} \leq 1 & \forall i\in[n], j\in[i,n-1], \label{eq:dsasm:L-prefix-Z-2}\\
           z&_{i,i}-z_{i-1,i+1} = 1      & \forall i\in[n]. \label{eq:dsasm:row-sum-Z}
  \end{align}
  Indeed, for $j\in[i-1]$, substituting the identity~\cref{eq:dsasm:core:id1} into the constraints in~\cref{eq:dsasm:L-prefix} yields the bounds in~\cref{eq:dsasm:L-prefix-Z-1}.
  For $j \in [i,n-1]$, rewrite the constraints in~\cref{eq:dsasm:L-prefix} as $0 \leq \sum_{j'=j+1}^{n} y_{i,j'} \leq 1$ using the equations in~\cref{eq:dsasm:L-sum}, and then apply the identity~\cref{eq:dsasm:core:id2} to obtain the bounds in~\cref{eq:dsasm:L-prefix-Z-2}.
  Finally, setting $j=i$ in~\cref{eq:dsasm:core:id1,eq:dsasm:core:id2} and substituting into the equations in~\cref{eq:dsasm:L-sum} gives the equations in~\cref{eq:dsasm:row-sum-Z} for $i \in [n]$.

  \smallskip
  Observe that each constraint in~\crefrange{eq:dsasm:L-prefix-Z-1}{eq:dsasm:row-sum-Z} involves the difference of two variables.
  Therefore, the coefficient matrix is the transpose of the node--arc incidence matrix of a digraph, and hence it is totally unimodular.
  Since the right-hand sides are integers, the system~\crefrange{eq:dsasm:L-prefix-Z-1}{eq:dsasm:row-sum-Z} together with $z_{0,j}=z_{i,n+1}=0$ defines an integral polytope; see \cref{thm:digraphIncidenceTU}.
  The map $Y \mapsto Z$ is a linear bijection that preserves integrality, and thus every vertex $Y$ of the polytope defined by the system~\crefrange{eq:dsasm:real}{eq:dsasm:L-sum} is the preimage of a vertex $Z$ of the polytope defined by the system~\crefrange{eq:dsasm:L-prefix-Z-1}{eq:dsasm:row-sum-Z}.
  Therefore, the system~\crefrange{eq:dsasm:real}{eq:dsasm:L-sum} defines an integral polytope.
\end{proof}

By \cref{thm:xasm:assembly,thm:dsasm:corepolytope}, we obtain the following description of $P_\DSASM$.
\begin{theorem}\label{thm:dsasm:coreDescr}
  Let $n \geq 1$, and $\widehat P^\core_\DSASM = \{X \in \R^{n \times n} : \pi_C(X) \in P^\core_\DSASM\}$.
  Then
  \[
    P_\DSASM = \widehat P^\core_\DSASM \cap P_\DS.
  \]
\end{theorem}

Using the assembly map $\varphi$ and the identities~\cref{eq:dsasm:phi-prefix-used,eq:dsasm:phi-sum-used} to translate the DSASM core constraints in \cref{thm:dsasm:coreDescr} to and from the ASM constraints in \cref{thm:ASMpolytope}, we obtain the following.
\begin{theorem}\label{thm:dsasm:descr}
  For every $n \geq 1$,
  \[
    P_\DSASM = P_\ASM \cap P_\DS.
  \]
\end{theorem}

\begin{theorem}\label{thm:dsasm:dim}
  For every $n \geq 1$, the dimension of $P_\DSASM$ is $\frac{n(n-1)}{2}$.
\end{theorem}
\begin{proof}
  It suffices to prove that the dimension of $P^\core_\DSASM$ is $\frac{n(n-1)}{2}$, because the assembly map $\varphi$ restricts to an affine isomorphism between $P^\core_\DSASM$ and $P_\DSASM$, which preserves dimension.
  First, we give an upper bound.
  The polytope $P^\core_\DSASM \subseteq \R^{C}$ is described by the system~\crefrange{eq:dsasm:real}{eq:dsasm:L-sum}.
  The system of linear equations given in~\cref{eq:dsasm:L-sum} consists of $n$ equations, which are linearly independent since the variable $y_{i,i}$ appears only in the $i^\text{th}$ equation of~\cref{eq:dsasm:L-sum}.
  Thus, the affine subspace defined by the equations in~\cref{eq:dsasm:L-sum} has dimension
  $
    |C|-n = \frac{n(n+1)}{2}-n=\frac{n(n-1)}{2},
  $
  which gives the bound $\dim(P^\core_\DSASM)\leq \frac{n(n-1)}{2}$.

  \medskip
  Second, we construct $\frac{n(n-1)}{2} + 1$ affinely independent cores in $P^\core_\DSASM$ and hence obtain a matching lower bound.
  Let $\ol Y\in\R^C$ be the core with all entries equal to $1/n$.
  It is easy to see that $\ol Y$ satisfies all inequalities in~\cref{eq:dsasm:L-prefix} strictly and also satisfies every equation in~\cref{eq:dsasm:L-sum}.
  For each $i,j \in [n]$ with $i < j$, define
  $
    \ol Y^{i,j}
    = \ol Y
    + \varepsilon \chi_{i,j}
    - \varepsilon \chi_{i,i}
    - \varepsilon \chi_{j,j},
  $
  where $\varepsilon$ is a small positive constant.
  By the definition of~\cref{eq:dsasm:L-sum}, the variable $y_{i,j}$ appears in exactly the $i^\text{th}$ and $j^\text{th}$ equations, and the increment of $y_{i,j}$ in these equations is canceled by the decrement of $y_{i,i}$ and $y_{j,j}$, respectively.
  Hence $\ol Y^{i,j}$ satisfies the equations in~\cref{eq:dsasm:L-sum}.
  Since $\ol Y$ satisfies the inequalities in~\cref{eq:dsasm:L-prefix} strictly, choosing $\varepsilon>0$ small enough guarantees that $\ol Y^{i,j}$ does not violate any inequality in~\cref{eq:dsasm:L-prefix}, and therefore $\ol Y^{i,j}\in P^\core_\DSASM$ for $i,j \in [n]$ with $i < j$.

  The cores $\ol Y$ and $\ol Y^{i,j}$ for $i,j \in [n]$ with $i < j$ are affinely independent: only the difference $\ol Y^{i,j}-\ol Y$ has a non-zero off-diagonal entry at $(i,j)$, so the cores $\{\ol Y^{i,j}-\ol Y : i,j \in [n], i < j\}$ are linearly independent.
  Therefore, the dimension of $P^\core_\DSASM$ is at least $\frac{n(n-1)}{2}$.

  Combining the lower and upper bounds yields $\dim(P^\core_\DSASM)=\dim(P_\DSASM)=\frac{n(n-1)}{2}$.
\end{proof}

We continue with identifying the facets of $P^\core_\DSASM$ and $P_\DSASM$.
\begin{theorem}
  Let $n \geq 3$.
  The facets of $P^\core_\DSASM$ are given by tightening the lower bound in~\cref{eq:dsasm:L-prefix} to equality for $(i,j) \in ([2,n-1] \times [2,n-2]) \cup ([n] \times \{1\})$, and the upper bound for $(i,j) \in ([2,n-1] \times [2,n-2]) \cup ([2,n] \times \{n-1\})$.
  In particular, the number of facets of $P^\core_\DSASM$ is $2(n-2)^2 + 3$.
\end{theorem}
\begin{proof}
  The facets are obtained by tightening a single inequality in~\cref{eq:dsasm:L-prefix} to equality for the index pairs listed in the statement of the theorem.
  We call the instances of the lower bounds in~\cref{eq:dsasm:L-prefix} that are tightened to equality the \emph{facet lower bounds}, and we define the \emph{facet upper bounds} analogously.
  We refer to the union of these two families as the \emph{facet inequalities}.

  We proceed in two steps.
  First, we show that the facet inequalities together with the equations in~\cref{eq:dsasm:L-sum} imply every inequality in~\cref{eq:dsasm:L-prefix}.
  Then, for every facet inequality, we construct a core of an $n \times n$ matrix violating that facet inequality and no other, thereby proving that no facet inequality is redundant.
  The core $\ol Y$ constructed in the second step of the proof of \cref{thm:dsasm:dim} shows that no facet inequality is an implicit equation; thus the two steps together imply that the facet inequalities form a minimal system that, extended with the equations in~\cref{eq:dsasm:L-sum}, describes the convex hull of the cores of DSASMs, which proves the theorem.

  \medskip
  Now we prove that the facet inequalities together with the equations in~\cref{eq:dsasm:L-sum} imply every inequality in~\cref{eq:dsasm:L-prefix}.
  We first derive a few bounds that will be used later.
  The facet lower bounds in~\cref{eq:dsasm:L-prefix} for $i\in[n]$ and $j=1$ give $y_{1,i}\geq 0$ for every $i\in[n]$.
  Moreover, for each $i\in[2,n]$, subtracting the facet upper bound in~\cref{eq:dsasm:L-prefix} for $i$ and $j=n-1$ from~\cref{eq:dsasm:L-sum} for $i$ yields $y_{i,n}\geq 0$.
  Together with the equation in~\cref{eq:dsasm:L-sum} for $n$, this also implies $y_{i,n}\leq 1$ for all $i\in[n]$.

  To derive the non-facet lower bounds in~\cref{eq:dsasm:L-prefix}, we distinguish three cases.
  For $i=1$ and $j\in[2,n-1]$, the left-hand side of~\cref{eq:dsasm:L-prefix} equals $\sum_{t=1}^{j} y_{1,t}$, which is non-negative since $y_{1,t}\geq 0$ for all $t\in[n]$.
  For $i=n$ and $j\in[2,n-1]$, the left-hand side equals $\sum_{t=1}^{j} y_{t,n}$, which is non-negative since $y_{t,n}\geq 0$ for all $t\in[n]$.
  For $i\in[2,n-1]$ and $j=n-1$, the left-hand side equals $1-y_{i,n}$ by~\cref{eq:dsasm:L-sum} for $i$, and this is non-negative because $y_{i,n}\leq 1$.

  To derive the non-facet upper bounds in~\cref{eq:dsasm:L-prefix}, we again distinguish three cases.
  For $i=1$ and $j\in[n-1]$, we have $\sum_{t=1}^{j} y_{1,t}\leq \sum_{t=1}^{n} y_{1,t}=1$ by~\cref{eq:dsasm:L-sum} for $i=1$ and $y_{1,t}\geq 0$.
  For $i=n$ and $j\in[n-2]$, we have $\sum_{t=1}^{j} y_{t,n}\leq \sum_{t=1}^{n} y_{t,n}=1$ by~\cref{eq:dsasm:L-sum} for $i=n$ and $y_{t,n}\geq 0$.
  For $i\in[2,n-1]$ and $j=1$, the left-hand side equals $y_{1,i}$, and $y_{1,i}\leq \sum_{t=1}^{n} y_{1,t}=1$ by~\cref{eq:dsasm:L-sum} for $i=1$ together with $y_{1,t}\geq 0$.
  The rest of the inequalities in~\cref{eq:dsasm:L-prefix} are exactly the facet inequalities, and thus we conclude that the facet inequalities together with the equations in~\cref{eq:dsasm:L-sum} imply every inequality in~\cref{eq:dsasm:L-prefix}.

  \medskip
  It remains to show that no facet inequality is redundant.
  For each $i \in [2,n-1], j \in [2,n-2]$ and $i \in [n], j = 1$, we construct a core $L^n_{i,j} \in \R^C$ that violates the facet lower bound in~\cref{eq:dsasm:L-prefix} for the given $i,j$ and satisfies every other facet inequality as well as the equations in~\cref{eq:dsasm:L-sum}.
  Similarly, for each $i \in [2,n-1], j \in [2,n-2]$ and $i \in [2,n], j = n-1$, we construct a core $U^n_{i,j} \in \R^C$ that violates the facet upper bound in~\cref{eq:dsasm:L-prefix} for the given $i,j$ and satisfies every other facet inequality as well as the equations in~\cref{eq:dsasm:L-sum}.

  Let $r \in [n]$ denote an integer, set $\tilde r = r-1$ and $\tilde n = n-1$.
  For the core $Z$ of an $(n-1) \times (n-1)$ matrix $X'$, we define the \emph{extension operator}
  \begin{center}
    \begin{tikzpicture}[scale=.9]
      \draw[fill=cyan!15] (1,7) -| (0,7) -| (0,8) |- (4,8) |- (4,4) |- (3,4) -| (3,5);
      \ThreeDotsAt{(2-.15,6-.15)}{.21213203}{-45}
      \draw[fill=cyan!15] (6,2) -| (5,2) -| (5,3) |- (8,3) |- (8,0) |- (7,0) -| (7,1);
      \ThreeDotsAt{(6.5-.15,1.5-.15)}{.21213203}{-45}
      \draw[fill=cyan!15] (5,4) -| (5,8) -| (8,8) |- (8,4) -| cycle;
      \draw (4,3) -| (4,8) -| (5,8) |- (5,4) |- (8,4) |- (8,3) |- cycle;
      \node[cell, draw=none] at (4,3) {$1$};
      \node[cell, draw=none] at (4,4) {$0$};
      \node[cell, draw=none] at (5,3) {$0$};
      \node[cell, draw=none] at (4,7) {$0$};
      \node[cell, draw=none] at (7,3) {$0$};
      \ThreeDotsAt{(4.5,6)}{.21213203}{90}
      \ThreeDotsAt{(6.5,3.5)}{.21213203}{0}

      \node[cell, draw=none] at (0,7) {$z_{1,1}$};
      \node[cell, draw=none] at (3,7) {$z_{1,\tilde r}$};
      \node[cell, draw=none] at (3,4) {$z_{\tilde r,\tilde r}$};
      \ThreeDotsAt{(2,7.5)}{.21213203}{0}
      \ThreeDotsAt{(3.5,6)}{.21213203}{90}

      \node[cell, draw=none] at (5,7) {$z_{1,r}$};
      \node[cell, draw=none] at (7,7) {$z_{1,\tilde n}$};
      \node[cell, draw=none] at (5,4) {$z_{\tilde r,r}$};
      \node[cell, draw=none] at (7,4) {$z_{\tilde r,\tilde n}$};
      \ThreeDotsAt{(6.5,7.5)}{.21213203}{0}
      \ThreeDotsAt{(6.5,4.5)}{.21213203}{0}
      \ThreeDotsAt{(5.5,6)}{.21213203}{90}
      \ThreeDotsAt{(7.5,6)}{.21213203}{90}

      \node[cell, draw=none] at (5,2) {$z_{r,r}$};
      \node[cell, draw=none] at (7,2) {$z_{r,\tilde n}$};
      \node[cell, draw=none] at (7,0) {$z_{\tilde n,\tilde n}$};
      \ThreeDotsAt{(7.5,1.5)}{.21213203}{90}
      \ThreeDotsAt{(6.5,2.5)}{.21213203}{0}

      \node[draw=none] at (-1.75,4) {$\ext_r(Z)\ =$};
      \node[draw=none] at (8.15,4) {$.$};
    \end{tikzpicture}
  \end{center}
  More precisely, $\ext_r$ takes the core of an $(n-1) \times (n-1)$ matrix and yields the core of an $n \times n$ matrix by adding a new $r^\text{th}$ row and a new $r^\text{th}$ column with the new entries being uniformly $0$ except for the entry at the intersection of the new row and column, which is set to $1$.
  Note that if $r=1$, then a new first row and column are inserted; if $r=n$, then a new last row and column are inserted.

  Now we are ready to construct $L^n_{i,j}$ and $U^n_{i,j}$ via an inductive approach.
  For the base case $n = 3$, we set
  \[
    L^{3}_{1,1}=
    \vcenter{\hbox{\begin{ytableau}
          -1    & 1     & 1\\
          \none & 0     & 0\\
          \none & \none & 0
        \end{ytableau}}}\,,
    \quad
    L^{3}_{2,1} =
    \vcenter{\hbox{\begin{ytableau}
          1     & -1    & 1\\
          \none & 2     & 0\\
          \none & \none & 0
        \end{ytableau}}}\,,
    \quad
    L^{3}_{3,1} =
    \vcenter{\hbox{\begin{ytableau}
          1     & 1     & -1\\
          \none & -1    & 1\\
          \none & \none & 1
        \end{ytableau}}}\,,
    \quad
    U^{3}_{2,2} =
    \vcenter{\hbox{\begin{ytableau}
          0     & 0     & 1\\
          \none & 2     & -1\\
          \none & \none & 1
        \end{ytableau}}}\,,
    \quad
    U^{3}_{3,2} =
    \vcenter{\hbox{\begin{ytableau}
          0     & 0     & 1\\
          \none & 0     & 1\\
          \none & \none & -1
        \end{ytableau}}}\,.
  \]
  For $n \geq 4$, we set
  \begin{align*}
    L^{n}_{2,n-2} &=
    \begin{cases}
      \vspace{1em}
      \vcenter{\vspace{3pt}\hbox{\begin{ytableau}
            1     & 0     & 0     & 0\\
            \none & -1    & 1     & 1\\
            \none & \none & 0     & 0\\
            \none & \none & \none & 0
          \end{ytableau}}} & \text{if } n=4,\\
      \vspace{1em}
      \vcenter{\hbox{\begin{ytableau}
            0     & 0     & 1     & 0     & 0\\
            \none & 0     & -1    & 1     & 1\\
            \none & \none & 1     & 0     & 0\\
            \none & \none & \none & 0     & 0\\
            \none & \none & \none & \none & 0
          \end{ytableau}}} & \text{if } n=5,\\
      \ext_3(L^{n-1}_{2,n-3}) & \text{if } n\geq6,
    \end{cases}
    &U^{n}_{n-1,2} =
      \begin{cases}
        \vspace{1em}
        \vcenter{\vspace{3pt}\hbox{\begin{ytableau}
              0     & 0     & 1     & 0\\
              \none & 0     & 1     & 0\\
              \none & \none & -1    & 0\\
              \none & \none & \none & 1
            \end{ytableau}}} & \text{if } n=4,\\
        \vspace{1em}
        \vcenter{\hbox{\begin{ytableau}
              0     & 0     & 0     & 1     & 0\\
              \none & 0     & 0     & 1     & 0\\
              \none & \none & 1     & -1    & 1\\
              \none & \none & \none & 0     & 0\\
              \none & \none & \none & \none & 0
            \end{ytableau}}} & \text{if } n=5,\\
        \ext_4(U^{n-1}_{n-2,2}) & \text{if } n\geq6,
      \end{cases}
  \end{align*}
  \begin{align*}
    L^{n}_{n-1,1} &=
    \begin{cases}
      \vcenter{\vspace{3pt}\hbox{\begin{ytableau}
            1     & 0     & -1    & 1 & \none\\
            \none & 0     & 1     & 0\\
            \none & \none & 1     & 0\\
            \none & \none & \none & 0
          \end{ytableau}}} & \text{if } n=4,\\
      \ext_3(L^{n-1}_{n-2,1}) & \text{if } n\geq5,
    \end{cases}
    &U^{n}_{2,n-1} =
    \begin{cases}
      \vspace{1em}
      \vcenter{\vspace{3pt}\hbox{\begin{ytableau}
            0     & 0     & 0     & 1 & \none\\
            \none & 1     & 1     & -1\\
            \none & \none & 0     & 0\\
            \none & \none & \none & 1
          \end{ytableau}}} & \text{if } n=4,\\
      \ext_3(U^{n-1}_{2,n-2}) & \text{if } n\geq5,
    \end{cases}\displaybreak[0]\\[1em]
    L^{n}_{n,1} &=
    \begin{cases}
      \vspace{1em}
      \vcenter{\vspace{3pt}\hbox{\begin{ytableau}
            1     & 0     & 1     & -1 & \none\\
            \none & 0     & 0     & 1\\
            \none & \none & 0     & 0\\
            \none & \none & \none & 1
          \end{ytableau}}} & \text{if } n=4,\\
      \ext_3(L^{n-1}_{n-1,1}) & \text{if } n\geq5.
    \end{cases}
  \end{align*}
  For the remaining indices, we define $L^n_{i,j}$ and $U^n_{i,j}$ recursively.
  In particular, for $i=j=1$ and for $i \in [2,n-2], j \in [n-3]$, we set
  \begin{align*}
    L&^n_{i,j} = \ext_{n}(L^{n-1}_{i,j}).
    \intertext{For $i \in [3,n-2], j = n-2$ and $i = n-1, j \in [2,n-2]$, we set}
    L&^n_{i,j} = \ext_{1}(L^{n-1}_{i-1,j-1}).
    \intertext{For $i \in [2,n-2], j \in [2,n-2]$ and $i = n-1, j = n-2$, we set}
    U&^n_{i,j} = \ext_{n}(U^{n-1}_{i,j}).
    \intertext{For $i \in [3,n], j = n-1$ and $i = n-1, j \in [3,n-3]$, we set}
    U&^n_{i,j} = \ext_{1}(U^{n-1}_{i-1,j-1}).
  \end{align*}

  It is straightforward to verify that the cores $L^n_{i,j}$ and $U^n_{i,j}$ as defined above violate the corresponding facet inequality for the given parameters $i$ and $j$; moreover, the construction ensures that they satisfy the equations in~\cref{eq:dsasm:L-sum} and all remaining facet inequalities.
  Indeed, for $n=3$ and for the explicitly listed cores for $n=4$ and $n=5$, this can be checked directly.
  By definition, $\ext_r(Z)$ is obtained from $Z$ by inserting a new $r^\text{th}$ row and column whose only non-zero entry is $y_{r,r}=1$.
  Consequently, the left-hand side of any inequality in~\cref{eq:dsasm:L-prefix} either coincides with the corresponding left-hand side for $Z$ (possibly after an index shift), or differs only by adding terms that are $0$; in particular, every facet inequality that holds for $Z$ continues to hold for $\ext_r(Z)$.
  Furthermore, since the new row and column contribute only the entry $y_{r,r}=1$, the equations in~\cref{eq:dsasm:L-sum} remain valid under $\ext_r$.
  The recursive constructions of $L^n_{i,j}$ and $U^n_{i,j}$ ensure that, whenever $Z$ violates exactly one facet lower or upper bound, the core $\ext_r(Z)$ violates exactly the corresponding shifted inequality, namely the one for $i,j$, among the facet inequalities for the enlarged dimension.
  It follows by induction on $n$ that, for every admissible index pair $i,j$, the core $L^n_{i,j}$ satisfies the equations in~\cref{eq:dsasm:L-sum} and all facet inequalities except the lower facet inequality for $i,j$, and analogously $U^n_{i,j}$ satisfies the equations in~\cref{eq:dsasm:L-sum} and all facet inequalities except the upper facet inequality for $i,j$.
  Hence no facet inequality is redundant.

  \medskip
  The number of facets obtained by tightening lower and upper bounds to equality is $(n-2)(n-3) + n$ and $(n-2)(n-3) + n - 1$, respectively.
  Since every facet inequality is violated by exactly one of the cores constructed above, the argument implies that these facets are pairwise distinct; thus their total number is $2(n-2)(n-3) + 2n - 1 = 2(n-2)^2 + 3$, as stated in the theorem.
\end{proof}

\section{Diagonally and antidiagonally symmetric ASMs (DASASMs)}\label{sec:DASASM}
In this class, we impose invariance under reflections across both the main (northwest--southeast) diagonal and the antidiagonal (northeast--southwest), and hence also under rotation by $\pi$ induced by composing these two reflections.
The corresponding symmetry subgroup is $\mathcal{G} = \{\mathcal{I}, \mathcal{D}, \mathcal{A}, \mathcal{R}_{\pi}\}$.
Let $P_\DAS$ denote the linear subspace of diagonally and antidiagonally symmetric real matrices, i.e.,
\[
  P_\DAS =
  \left\{
    X \in \R^{n \times n} : x_{i,j} = x_{j,i} = x_{n+1-j,n+1-i}\ \forall i,j \in [n]
  \right\}.
\]
Clearly, any DASASM satisfies the ASM constraints in~\crefrange{eq:asm:real}{eq:asm:col-sum} and also the symmetry constraints defining $P_\DAS$; thus $P_\DASASM \subseteq P_\ASM \cap P_\DAS$.
We will prove that these constraints are in fact sufficient, i.e., $P_\DASASM = P_\ASM \cap P_\DAS$.

\paragraph{Core and assembly map.}
Let $n \geq 1$ and set $k = \floor*{n/2}$.
Let the \emph{core} of a DASASM be the part of the matrix lying on and above both the main diagonal and the antidiagonal, except the middle entry $(k+1,k+1)$ when $n$ is odd, i.e.,
\[
  C = \left\{(i,j) \in [k] \times [n] : i \leq j \leq n+1-i\right\}
\]
is the set of \emph{core positions}, and let $\pi_C$ be the coordinate-wise projection onto $C$.
Equivalently, $C$ consists of the full first row, the second row without its first and last entries, the third row without its first and last two entries, and so on, down to row $k$.
Define the affine map $\varphi : \R^C \to \R^{n \times n}$ by
\[
  \varphi(Y)_{i,j} =
  \begin{cases}
    y_{i,j}                       & \text{if } (i,j) \in C,\\
    1 - 2\sum_{i'=1}^{k} y_{i',k+1}& \text{if } 2 \nmid n, i=j=k+1,\\
    y_{j,i}                       & \text{if } j \in [\min\{i-1, n+1-i\}],\\
    y_{n+1-j,n+1-i}                & \text{if } j \in [\max\{i, n+2-i\},n],\\
    y_{n+1-i,n+1-j}                & \text{if } j \in [n+2-i,i-1]
  \end{cases}
\]
for $Y \in \R^C$ and $i,j \in [n]$.
Thus $\varphi$ places the core $Y$ in the trapezoid-shaped region above both diagonals, assigns the central entry according to the second case when $n$ is odd, and completes the matrix by reflecting across the main diagonal, the antidiagonal, and their composition, yielding a diagonally and antidiagonally symmetric matrix.
Clearly, the map $\varphi$ is an assembly map: it is affine, satisfies $\pi_C(\varphi(Y)) = Y$ for every $Y \in \R^C$, and $\varphi(\pi_C(X)) = X$ for every $X \in \DASASM(n)$, because $X$ is completely determined by its entries in $C$ together with the imposed diagonal and antidiagonal symmetries and the ASM constraints.

\begin{theorem}\label{thm:dasasm:corepolytope}
  Let $n \geq 1$ and $k = \floor*{n/2}$.
  Then the core polytope $P^\core_\DASASM \subseteq \R^C$ of $n \times n$ DASASMs is described by the following system.
  \begin{align}
                                     y&_{i,j} \in \R                                                                           &\forall (i,j) \in C, \label{eq:dasasm:real}\\
    0 \leq \sum_{i'=1}^{\min\{i-1,j\}} y&_{i',i} + \sum_{j'=i}^{\min\{n+1-i,j\}} y_{i,j'}  + \sum_{i'=n-i+2}^{j} y_{n+1-i',n+1-i} \leq 1 &\forall i \in [k], j \in [n-1], \label{eq:dasasm:row-prefix}\\
                   \sum_{i'=1}^{i-1} y&_{i',i} + \sum_{j'=i}^{n+1-i} y_{i,j'}  + \sum_{i'=n-i+2}^{n} y_{n+1-i',n+1-i} = 1             &\forall i \in [k]; \label{eq:dasasm:row-sum}
    \intertext{if $n$ is odd, then we also add the constraint}
    0 \leq \sum_{i'=1}^{j}           y&_{i',k+1} \leq 1                                                                         &\forall j \in [k]. \label{eq:dasasm:mid-col-pref}
  \end{align}
\end{theorem}
\begin{proof}
  We prove that, when $n$ is even, the integer solutions to the system~\crefrange{eq:dasasm:real}{eq:dasasm:row-sum} are exactly the cores of DASASMs, while for odd $n$ the cores arise precisely as the integer solutions to the extended system~\crefrange{eq:dasasm:real}{eq:dasasm:mid-col-pref}.
  We then show that the system defines an integral polytope in both cases.

  \medskip
  First, let $X$ be an $n\times n$ DASASM, and let $Y = \pi_C(X)$ be its core.
  Since $X$ is both diagonally and antidiagonally symmetric, we have $X = \varphi(Y)$.
  By the definition of $\varphi$, for every $i\in[k]$ and $j\in[n]$, the row-prefix sum within row $i$ up to column $j$ can be written in terms of the core entries as
  \begin{equation}\label{eq:dasasm:phi-row-prefix-used}
    \sum_{j'=1}^{j} x_{i,j'}
    = \sum_{i'=1}^{\min\{i-1,j\}} y_{i',i}
    + \sum_{j'=i}^{\min\{n+1-i,j\}} y_{i,j'}
    + \sum_{i'=n-i+2}^{j} y_{n+1-i',n+1-i}.
  \end{equation}
  If $n$ is odd, then $(i,k+1)\in C$ for every $i \in [k]$, so along the middle column we also obtain that, for every $i\in[k]$,
  \begin{equation}\label{eq:dasasm:phi-mid-col-prefix-used}
    \sum_{i'=1}^{i} x_{i',k+1}
    = \sum_{i'=1}^{i} y_{i',k+1}.
  \end{equation}
  Since every DASASM is in particular an ASM, the ASM row-prefix bounds in~\cref{eq:asm:row-prefix}, together with the identity~\cref{eq:dasasm:phi-row-prefix-used} for $j\in[n-1]$, yield the constraints in~\cref{eq:dasasm:row-prefix}.
  The case $j = n$ of the same identity, together with the ASM row-sum constraint in~\cref{eq:asm:row-sum}, gives the equations in~\cref{eq:dasasm:row-sum}.
  In the odd case, combining the ASM column-prefix bounds in~\cref{eq:asm:col-prefix} with the identity~\cref{eq:dasasm:phi-mid-col-prefix-used} yields the middle-column-prefix bounds in~\cref{eq:dasasm:mid-col-pref}.
  Thus the cores of DASASMs satisfy the system~\crefrange{eq:dasasm:real}{eq:dasasm:mid-col-pref}.

  \medskip
  Conversely, let $Y$ be an integer solution to the system and put $X = \varphi(Y)$.
  The identity~\cref{eq:dasasm:phi-row-prefix-used} shows that the constraints in~\cref{eq:dasasm:row-prefix,eq:dasasm:row-sum} are exactly the ASM row-prefix bounds and row-sum constraints for the top half of the matrix, while the identity~\cref{eq:dasasm:phi-mid-col-prefix-used} gives the middle-column prefixes for odd $n$.
  The remaining ASM constraints follow by symmetry.
  Hence $X$ satisfies all ASM constraints, and by construction, it is diagonally and antidiagonally symmetric.
  Therefore, $X$ is a DASASM and $Y = \pi_C(X)$ is its core.
  This proves that the integer solutions to the system are exactly the cores of DASASMs.

  \medskip
  It remains to prove that the system~\crefrange{eq:dasasm:real}{eq:dasasm:row-sum} and the system~\crefrange{eq:dasasm:real}{eq:dasasm:mid-col-pref} define integral polytopes when $n$ is even and odd, respectively.
  For $i \in [k]$ and $j \in [0,n]$, define
  \[
    z_{i,j}
    = \sum_{i'=1}^{\min\{i-1,j\}} y_{i',i}
    + \sum_{j'=i}^{\min\{n+1-i,j\}} y_{i,j'}
    + \sum_{i'=n-i+2}^{j} y_{n+1-i',n+1-i},
  \]
  For $j\in[n-1]$, this is the left-hand side of~\cref{eq:dasasm:row-prefix}, while $z_{i,n}$ is the left-hand side of~\cref{eq:dasasm:row-sum}.
  If $n$ is odd, then, for $j \in [0,k]$, we let
  \[
    z_{k+1,j} = \sum_{i'=1}^{j} y_{i',k+1}.
  \]
  Note that $z_{i,0}=0$ for every $i$ by empty sums.
  By definition, the constraints in~\cref{eq:dasasm:row-prefix,eq:dasasm:row-sum} become
  \begin{align}
    0 \leq z&_{i,j} \leq 1   &\forall i \in [k], j \in [n-1], \label{eq:dasasm:z-box}\\
    z&_{i,n} = 1      &\forall i \in [k], \label{eq:dasasm:z-rowsum}\\
    \intertext{and, for odd $n$,~\cref{eq:dasasm:mid-col-pref} turns into}
    0 \leq z&_{k+1,j} \leq 1 &\forall j \in [k]. \label{eq:dasasm:z-mid-box}
  \end{align}

  Next we express the differences of consecutive $z$-variables in terms of the core entries.
  Fix $i\in[k]$.
  From the definition of $z_{i,j}$ we see that, as $j$ increases by $1$, at most one of the three upper limits changes, so $z_{i,j}-z_{i,j-1}$ is always a single core entry.
  Direct inspection of the three sums gives
  \begin{align}
    z&_{i,j} - z_{i,j-1} = y_{j,i} &\forall j\in[i-1], \label{eq:dasasm:dz-left}\\
    z&_{i,j} - z_{i,j-1} = y_{i,j} &\forall j\in[i,n+1-i], \label{eq:dasasm:dz-middle}\\
    z&_{i,j} - z_{i,j-1} = y_{n+1-j,n+1-i} &\forall j\in[n+2-i,n], \label{eq:dasasm:dz-right}\\
    \intertext{furthermore, if $n$ is odd, then we have}
    z&_{k+1,j}-z_{k+1,j-1}=y_{j,k+1} &\forall j\in[k]. \label{eq:dasasm:dz-midcol}
  \end{align}

  For every core position $(i,j) \in C$, the equation~\cref{eq:dasasm:dz-middle} gives
  \begin{align}
    y&_{i,j} = z_{i,j} - z_{i,j-1}            &\forall (i,j)\in C, \label{eq:dasasm:y-from-z-own}\\
    \intertext{in addition, for core positions that are off both diagonals, the identities~\cref{eq:dasasm:dz-left,eq:dasasm:dz-right} give the alternative expression}
    y&_{i,j} = z_{j,i} - z_{j,i-1}            &\forall (i,j)\in C \text{ with } j \in [k], i \in [j-1], \label{eq:dasasm:y-from-z-col}\\
    y&_{i,j} = z_{n+1-j,n+1-i} - z_{n+1-j,n-i} &\forall (i,j)\in C \text{ with } j \in [n+1-k,n], i \in [n-j], \label{eq:dasasm:y-from-z-anti}\\
    \intertext{and, if $n$ is odd, then we have}
    y&_{i,k+1} = z_{k+1,i} - z_{k+1,i-1}       &\forall i \in [k]. \label{eq:dasasm:y-from-z-midcol}
  \end{align}

  \medskip
  We now eliminate the core variables $y$ to obtain a system in the $z$-variables --- and later define $y$ in terms of $z$ by~\cref{eq:dasasm:y-from-z-own}.
  For every $(i,j)\in C$ that is off both the main diagonal and the antidiagonal, there is exactly one alternative formula for $y_{i,j}$, given by one of~\cref{eq:dasasm:y-from-z-col,eq:dasasm:y-from-z-anti,eq:dasasm:y-from-z-midcol}.
  Subtracting this alternative expression from~\cref{eq:dasasm:y-from-z-own} for the same $(i,j)$ cancels $y_{i,j}$ and produces a linear relation among the $z$-variables.
  Explicitly, for such a pair of indices $i,j$, we obtain an equation of the form
  \begin{equation}\label{eq:dasasm:core:z-system}
    z_{i,j} - z_{i,j-1} - z_{r,s} + z_{r,s-1} = 0,
  \end{equation}
  where $(r,s)$ is the index pair appearing in the alternative expression for $y_{i,j}$ in one of the expressions~\cref{eq:dasasm:y-from-z-col,eq:dasasm:y-from-z-anti,eq:dasasm:y-from-z-midcol}.
  Let $Az = 0$ denote the system of all equations of the form~\cref{eq:dasasm:core:z-system}.
  Each row of $A$ has four non-zero entries in $\{0,\pm1\}$, corresponding to the coefficients of $z_{i,j}$, $z_{i,j-1}$, $z_{r,s}$, and $z_{r,s-1}$.

  We claim that each column of $A$ contains at most one $1$ and at most one $-1$.
  Fix a variable $z_{u,v}$.
  It can only appear in the two consecutive differences $z_{u,v}-z_{u,v-1}$ and $z_{u,v+1}-z_{u,v}$, as right and left endpoint, respectively.
  By construction, every difference $z_{r,t}-z_{r,t-1}$ corresponds to a unique core entry $y_{i,j}$ (either via the expression~\cref{eq:dasasm:y-from-z-own} or via one of the expressions~\cref{eq:dasasm:y-from-z-col,eq:dasasm:y-from-z-anti,eq:dasasm:y-from-z-midcol}), and for each off-diagonal $y_{i,j}$ we introduce at most one consistency equation.
  Thus each difference $z_{r,t}-z_{r,t-1}$ appears in at most one row of $A$.
  Since a variable $z_{u,v}$ can occur only in the two consecutive differences $z_{u,v}-z_{u,v-1}$ and $z_{u,v+1}-z_{u,v}$, and these occurrences have opposite signs by construction, the column of $A$ indexed by $z_{u,v}$ has at most one $1$ and at most one $-1$.
  Hence $A$ is a $\{0,\pm1\}$-matrix with at most one $1$ and at most one $-1$ in each column, and therefore it is a submatrix of the node--arc incidence matrix of a digraph, and hence it is totally unimodular.

  The system in the $z$-variables consists of the bounds in~\cref{eq:dasasm:z-box}, the equations in~\cref{eq:dasasm:z-rowsum}, the system $Az=0$ coming from~\cref{eq:dasasm:core:z-system}, and, when $n$ is odd, the bounds in~\cref{eq:dasasm:z-mid-box}.
  The equations have a totally unimodular coefficient matrix and integer right-hand sides, while the remaining constraints are simple lower and upper bounds on individual variables.
  By the standard integrality theorem for systems with a totally unimodular coefficient matrix, integer right-hand side, and integral variable bounds, the feasible region in the $z$-space is an integral polytope; see \cref{thm:digraphIncidenceTU}.

  \medskip
  The map $z \mapsto y$ defined by $y_{i,j} = z_{i,j} - z_{i,j-1}$  for $(i,j)\in C$ is a linear bijection and it preserves integrality.
  Conversely, from any feasible $y$, the system of equations in~\crefrange{eq:dasasm:dz-left}{eq:dasasm:dz-midcol} recovers a unique feasible $z$, and the constraints in~\crefrange{eq:dasasm:z-box}{eq:dasasm:z-mid-box} are precisely the images of~\crefrange{eq:dasasm:row-prefix}{eq:dasasm:mid-col-pref} under this correspondence.
  In particular, these constructions define integer-linear bijections between the feasible $z$- and $y$-regions.
  Thus the polytopes of the feasible $y$ and $z$ solutions are affinely isomorphic.

  Since the $z$-polytope is integral, the core polytope $P^\core_\DASASM$ is the image of an integral polytope under an integer-linear bijection, and hence is itself integral.
  Therefore, the system~\crefrange{eq:dasasm:real}{eq:dasasm:row-sum} and the system~\crefrange{eq:dasasm:real}{eq:dasasm:mid-col-pref} define integral polytopes for even $n$ and odd $n$, respectively.
\end{proof}

By \cref{thm:xasm:assembly,thm:dasasm:corepolytope}, we obtain the following description of $P_\DASASM$.
\begin{theorem}\label{thm:dasasm:coreDescr}
  Let $n \geq 1$, $k=\floor*{n/2}$, and $\widehat P^\core_\DASASM=\{X\in\R^{n\times n}:\pi_C(X)\in P^\core_\DASASM\}$.
  Then
  \[
    P_\DASASM
    =\widehat P^\core_\DASASM\cap P_\DAS \cap \Bigl\{X\in\R^{n\times n}:\ \sum_{i=1}^n x_{i,k+1}=1\Bigr\}.
  \]
\end{theorem}

Using the assembly map $\varphi$ and the identities~\cref{eq:dasasm:phi-row-prefix-used,eq:dasasm:phi-mid-col-prefix-used} to translate the DASASM core constraints in \cref{thm:dasasm:coreDescr} to and from the ASM constraints in \cref{thm:ASMpolytope}, we obtain the following.
\begin{theorem}\label{thm:dasasm:descr}
  For every $n \geq 1$,
  \[
    P_\DASASM = P_\ASM \cap P_\DAS.
  \]
\end{theorem}

\begin{theorem}\label{thm:dasasm:dim}
  For every $n \geq 1$, the dimension of $P_\DASASM$ is $\floor*{\frac{n^2}{4}}$.
\end{theorem}
\begin{proof}
  It suffices to prove that the dimension of $P^\core_\DASASM$ is $\floor*{\frac{n^2}{4}}$, because the assembly map $\varphi$
  restricts to an affine isomorphism between $P^\core_\DASASM$ and $P_\DASASM$, which preserves dimension.
  First, we give an upper bound.
  Let $k=\floor*{n/2}$ and let $C$ be the set of core positions.
  Then
  \[
    |C|
    = \sum_{i=1}^{k}\bigl((n+1-i)-i+1\bigr)
    = \sum_{i=1}^{k}(n+2-2i)
    = k(n+1-k).
  \]
  By \cref{thm:dasasm:corepolytope}, the polytope $P^\core_\DASASM\subseteq \R^{C}$ is described by the system~\crefrange{eq:dasasm:real}{eq:dasasm:row-sum} if $n$ is even, and by the system~\crefrange{eq:dasasm:real}{eq:dasasm:mid-col-pref} if $n$ is odd.
  In either case, the system of linear equations given in~\cref{eq:dasasm:row-sum} consists of $k$ equations, which are linearly independent, since the variable $y_{i,i}$ appears only in the $i^\text{th}$ equation of~\cref{eq:dasasm:row-sum}.
  Thus, the affine subspace defined by the equations in~\cref{eq:dasasm:row-sum} has dimension
  $
    |C|-k = k(n+1-k)-k = k(n-k) = \floor*{\frac{n^2}{4}},
  $
  which gives the bound $\dim(P^\core_\DASASM)\leq \floor*{\frac{n^2}{4}}$.

  \medskip
  Second, we construct $\floor*{\frac{n^2}{4}}+1$ affinely independent cores in $P^\core_\DASASM$ and hence obtain a matching lower bound.
  Let $\ol Y\in\R^C$ be the core with all entries equal to $1/n$.
  It is easy to see that $\ol Y$ satisfies the equations in~\cref{eq:dasasm:row-sum}, and the inequalities in~\cref{eq:dasasm:row-prefix} strictly; moreover, when $n$ is odd, $\ol Y$ also satisfies the inequalities in~\cref{eq:dasasm:mid-col-pref} strictly.
  Thus $\ol Y\in P^\core_\DASASM$.
  Let $S=\{(i,j)\in C : i<j\}$ be the set of off-diagonal core positions.
  Since the diagonal core positions are exactly $(i,i)$ for $i\in[k]$, we have $|S| = |C|-k = k(n-k)=\floor*{\frac{n^2}{4}}$.
  For each $(i,j)\in S$, define
  \[
    \ol Y^{i,j}=
    \begin{cases}
      \ol Y+\varepsilon\chi_{i,j}-\varepsilon\chi_{i,i}-\varepsilon\chi_{j,j}        & \text{if } j\in[k],\\
      \ol Y+\varepsilon\chi_{i,j}-\varepsilon\chi_{i,i}-\varepsilon\chi_{n+1-j,n+1-j} & \text{if } j\in[n+1-k,n] \text{ and } i\neq n+1-j,\\
      \ol Y+\varepsilon\chi_{i,j}-\varepsilon\chi_{i,i}                              & \text{otherwise,}
    \end{cases}
  \]
  where $\varepsilon$ is a small positive constant.
  Every off-diagonal variable $y_{i,j}$ appears in the $i^\text{th}$ equation in~\cref{eq:dasasm:row-sum}, and it appears in at most one further equation: if $j\leq k$, then it appears also in the $j^\text{th}$ equation, while if $j\in[n+1-k,n]$, then it appears also in the $(n+1-j)^\text{th}$ equation; in particular, when $n$ is odd and $j=k+1$, no such second equation occurs.
  Moreover, in the third case of the definition of $\ol Y^{i,j}$, the only potential second index is $n+1-j$, which either does not lie in $[k]$ or equals $i$, so the two potential equations coincide and we avoid subtracting twice from the same diagonal entry.
  Thus, in every equation of~\cref{eq:dasasm:row-sum} in which $y_{i,j}$ appears, the increment $+\varepsilon$ is cancelled by one of the diagonal decrements, and hence each $\ol Y^{i,j}$ satisfies the equations in~\cref{eq:dasasm:row-sum}.
  Since $\ol Y$ satisfies all defining inequalities strictly, choosing $\varepsilon>0$ small enough guarantees that the inequalities in~\cref{eq:dasasm:row-prefix} and, when $n$ is odd, in~\cref{eq:dasasm:mid-col-pref} are satisfied, and therefore $\ol Y^{i,j}\in P^\core_\DASASM$ for all $(i,j)\in S$.

  The cores $\ol Y$ and $\ol Y^{i,j}$ for $(i,j)\in S$ are affinely independent: only the difference $\ol Y^{i,j}-\ol Y$ has a non-zero off-diagonal entry at $(i,j)$, so the set $\{\ol Y^{i,j}-\ol Y : (i,j)\in S\}$ is linearly independent.
  Therefore, the dimension of $P^\core_\DASASM$ is at least $|S| = \floor*{\frac{n^2}{4}}$.

  Combining the lower and upper bounds yields $\dim(P^\core_\DASASM)=\dim(P_\DASASM)=\floor*{\frac{n^2}{4}}$.
\end{proof}

\begin{theorem}\label{thm:dasasm:facets}
  Let $n \geq 2$.
  The facets of $P^\core_\DASASM$ are given by tightening the lower bound in~\cref{eq:dasasm:row-prefix} to equality for
  $
  (i,j)\in \{(1,1)\} \cup \bigl([2,k]\times[n-2]\bigr)
  $, and the upper bound for
  $
  (i,j)\in \{(1,n-1)\} \cup \bigl([2,k]\times[2,n-1]\bigr)
  $.
  If $n$ is odd, then in addition the facets include those obtained by tightening the lower bound in~\cref{eq:dasasm:mid-col-pref} to equality for
  $
    j\in [k]
  $,
  and the upper bound for
  $
    j \in [2,k]
  $.
  In particular, the number of facets of $P^\core_\DASASM$ is $(n-2)^2 + 2$.
\end{theorem}
\begin{proof}
  The facets are obtained by tightening a single inequality in~\cref{eq:dasasm:row-prefix}, and if $n$ is odd also in~\cref{eq:dasasm:mid-col-pref}, to equality for the index pairs listed in the statement of the theorem.
  We call the instances of the lower bounds that are tightened to equality the \emph{facet lower bounds}, and we define the \emph{facet upper bounds} analogously.
  We refer to the union of these two families as the \emph{facet inequalities}.

  We proceed in two steps.
  First, we show that the facet inequalities together with the equations in~\cref{eq:dasasm:row-sum} imply every inequality in~\cref{eq:dasasm:row-prefix} and, if $n$ is odd, every inequality in~\cref{eq:dasasm:mid-col-pref}.
  Then, for every facet inequality, we construct a core of an $n \times n$ matrix violating that facet inequality and no other, thereby proving that no facet inequality is redundant.
  The core $\ol Y$ constructed in the second step of the proof of \cref{thm:dasasm:dim} shows that no facet inequality is an implicit equation; thus the two steps together imply that the facet inequalities form a minimal system that, extended with the equations in~\cref{eq:dasasm:row-sum}, describes the convex hull of the cores of DASASMs, which proves the theorem.

  \medskip
  Now we prove that the facet inequalities together with the equations in~\cref{eq:dasasm:row-sum} imply every inequality in~\cref{eq:dasasm:row-prefix}.
  Clearly, we need to treat only those inequalities that are non-facet inequalities, namely, the lower bounds in~\cref{eq:dasasm:row-prefix} for
  \[
    (i,j) \in (\{1\} \times [2,n-1]) \cup ([2,k] \times \{n-1\}),
  \]
  and the upper bounds for
  \[
    (i,j) \in (\{1\} \times [n-2]) \cup ([2,k] \times \{1\}).
  \]
  If $n$ is odd, then in addition the only non-facet inequality in~\cref{eq:dasasm:mid-col-pref} is the upper bound for $j = 1$.

  \smallskip
  For $i \in [k]$ and $j \in [n-1]$, let $z_{i,j}$ denote the left-hand side of~\cref{eq:dasasm:row-prefix}, and extend the notation by setting $z_{i,n}=1$ using the equations in~\cref{eq:dasasm:row-sum}.
  We also set $z_{i,0}=0$ for $i \in [k]$.
  We first record that every entry in the first core row is non-negative.
  Indeed, the facet lower bound at $(1,1)$ gives $z_{1,1} \geq 0$, and for each $i \in [2,k]$ the facet lower bound at $(i,1)$ gives $z_{i,1} \geq 0$.
  From the difference equations for the $z$ variables, we have $z_{i,1} - z_{i,0} = y_{1,i}$ for all $i \in [2,k]$, and $z_{1,1} = y_{1,1}$.
  Thus $y_{1,i} \geq 0$ for all $i \in [k]$.
  Moreover, for each $i \in [k]$, the facet upper bound at $(i,n-1)$ yields $z_{i,n-1} \leq 1$.
  Using the equations in~\cref{eq:dasasm:row-sum}, we have $z_{i,n}=1$, and from the difference equations, we obtain $z_{i,n} - z_{i,n-1} = y_{1,n+1-i}$.
  Hence, $y_{1,n+1-i} = 1 - z_{i,n-1} \geq 0$ for all $i \in [2,k]$.
  For odd $n$, non-negativity of the middle entry $y_{1,k+1}$ follows from the facet lower bound in~\cref{eq:dasasm:mid-col-pref} for $j=1$.
  Consequently, $y_{1,t} \geq 0$ for all $t \in [n]$.
  The case $i=1$ of the equations in~\cref{eq:dasasm:row-sum} gives $\sum_{t=1}^n y_{1,t} = z_{1,n} = 1$, which together with non-negativity implies $y_{1,t} \leq 1$ for all $t \in [n]$.

  To derive the non-facet lower bounds in~\cref{eq:dasasm:row-prefix}, we distinguish two cases.
  For $i=1$ and $j \in [2,n-1]$, the left-hand side is $z_{1,j} = \sum_{t=1}^j y_{1,t}$, which is non-negative since $y_{1,t} \geq 0$.
  For $i \in [2,k]$ and $j=n-1$, the left-hand side is $z_{i,n-1} = 1 - y_{1,n+1-i}$, which is non-negative because $y_{1,n+1-i} \leq 1$.
  To derive the non-facet upper bounds in~\cref{eq:dasasm:row-prefix}, we again distinguish two cases.
  For $i=1$ and $j \in [n-2]$, we have $z_{1,j} \leq \sum_{t=1}^n y_{1,t} = 1$ since $y_{1,t} \geq 0$.
  For $i \in [2,k]$ and $j=1$, the left-hand side is $z_{i,1} = y_{1,i}$, which satisfies $y_{1,i} \leq 1$.
  This completes the derivation of all inequalities in~\cref{eq:dasasm:row-prefix}.

  If $n$ is odd, then we must also derive the non-facet inequalities in~\cref{eq:dasasm:mid-col-pref}.
  The only non-facet inequality in this group is the upper bound for $j=1$, which reads $y_{1,k+1} \leq 1$, since the prefix sum of length $1$ is just the first entry.
  As shown above, $y_{1,t} \leq 1$ holds for all $t$, in particular for $t=k+1$, so this inequality is implied.

  \medskip
  It remains to show that no facet inequality is redundant.
  For each $(i,j) \in \{(1,1)\} \cup ([2,k] \times [n-2])$, we construct a core $L^n_{i,j} \in \R^C$ that violates the facet lower bound in~\cref{eq:dasasm:row-prefix} for the given $i,j$ and satisfies every other facet inequality as well as the equations in~\cref{eq:dasasm:row-sum}.
  Similarly, for each $(i,j) \in \{(1,n-1)\} \cup ([2,k] \times [2,n-1])$, we construct a core $U^n_{i,j} \in \R^C$ that violates the facet upper bound in~\cref{eq:dasasm:row-prefix} for the given $i,j$ and satisfies every other facet inequality as well as the equations in~\cref{eq:dasasm:row-sum}.
  If $n$ is odd, we additionally construct a core $L^n_{k+1,j}$ for each $j \in [k]$ that violates the lower bound in~\cref{eq:dasasm:mid-col-pref} for $j$, and a core $U^n_{k+1,j}$ for each $j \in [2,k]$ that violates the upper bound in~\cref{eq:dasasm:mid-col-pref} for $j$, while satisfying all other constraints.

  In order to build the certifying cores of $n\times n$ DASASMs from the certifying cores for smaller sizes, we introduce four \emph{extension operators}.
  Set $\tilde n = n - 1$, $\tilde k = \floor*{\tilde n / 2}$, and let $Z$ be the core of an $\tilde n \times \tilde n$ matrix.
  For odd $n$, i.e., $\tilde n$ even, we define
  \begin{center}
    \begin{tikzpicture}[scale=.95, baseline=(current bounding box.center)]
      \node at (-1.6,1.5) {$\extI(Z)\ =$};
      \node[draw=none] at (7.15,1.5) {$.$};

      \draw (3,3) -- (4,3);
      \draw (3,0) -- (4,0);

      \draw[fill=cyan!15] (2,1) -- (2,0) -- (3,0) -- (3,3) -- (0,3) -- (0,2) -- (1,2);

      \draw[fill=cyan!15] (5,1) -- (5,0) -- (4,0) -- (4,3) -- (7,3) -- (7,2) -- (6,2);

      \node at (0.5,2.5) {$z_{1,1}$};
      \ThreeDotsAt{(1.5,2.5)}{.21213203}{0}
      \node at (2.5,2.5) {$z_{1,\tilde k}$};
      \ThreeDotsAt{(2.5,1.5)}{.21213203}{90}
      \node at (2.5,0.5) {$z_{\tilde k,\tilde k}$};
      \ThreeDotsAt{(1.5-.15,1.5-.15)}{.21213203}{-45}

      \node at (4.5,2.5) {$z_{1,\tilde k+1}$};
      \node at (6.5,2.5) {$z_{1,\tilde n}$};
      \node at (4.5,0.5) {$z_{\tilde k,\tilde k+1}$};
      \ThreeDotsAt{(5.5,2.5)}{.21213203}{0}
      \ThreeDotsAt{(4.5,1.5)}{.21213203}{90}
      \ThreeDotsAt{(5.5+.15,1.5-.15)}{.21213203}{45}

      \node at (3.5,2.5) {$0$};
      \ThreeDotsAt{(3.5,1.5)}{.21213203}{90}
      \node at (3.5,0.5) {$0$};
    \end{tikzpicture}
  \end{center}
  For even $n$, i.e., $\tilde n$ odd, we define
  \begin{center}
    \begin{tikzpicture}[scale=.95, baseline=(current bounding box.center)]
      \node at (-1.6,2.0) {$\extL(Z)\ =$};
      \node[draw=none] at (8.15,2.0) {\phantom{$,$}};

      \draw (3,4) -- (4,4);
      \draw (3,1) -- (3,0) -- (5,0) -- (5,1);

      \draw[fill=cyan!15] (2,2) -- (2,1) -- (3,1) -- (3,4) -- (0, 4) -- (0,3) -- (1,3);

      \draw[fill=cyan!15] (6,2) -- (6,1) -- (4,1) -- (4,4) -- (8,4) -- (8,3) -- (7,3);

      \node at (0.5,3.5) {$z_{1,1}$};
      \ThreeDotsAt{(1.5,3.5)}{.21213203}{0}
      \node at (2.5,3.5) {$z_{1,\tilde k}$};
      \ThreeDotsAt{(2.5,2.5)}{.21213203}{90}
      \node at (2.5,1.5) {$z_{\tilde k,\tilde k}$};
      \ThreeDotsAt{(1.5-.15,2.5-.15)}{.21213203}{-45}

      \node at (4.5,3.5) {$z_{1,\tilde k+1}$};
      \node at (5.5,3.5) {$z_{1,\tilde k+2}$};
      \ThreeDotsAt{(6.5,3.5)}{.21213203}{0}
      \node at (7.5,3.5) {$z_{1,\tilde n}$};
      \node at (4.5,1.5) {$z_{\tilde k,\tilde k+1}$};
      \node at (5.5,1.5) {$z_{\tilde k,\tilde k+2}$};
      \ThreeDotsAt{(6.5+.15,2.5-.15)}{.21213203}{45}
      \ThreeDotsAt{(4.5,2.5)}{.21213203}{90}
      \ThreeDotsAt{(5.5,2.5)}{.21213203}{90}

      \node at (3.5,3.5) {$0$};
      \ThreeDotsAt{(3.5,2.5)}{.21213203}{90}
      \node at (3.5,1.5) {$0$};

      \node at (3.5,0.5) {$\eta$};
      \node at (4.5,0.5) {$0$};
    \end{tikzpicture}
  \end{center}
  and
  \begin{center}
    \begin{tikzpicture}[scale=.95, baseline=(current bounding box.center)]
      \node at (-1.6,2.0) {$\extRL(Z)\ =$};
      \node[draw=none] at (8.15,2.0) {$,$};

      \draw (4,4) -- (5,4);
      \draw (3,1) -- (3,0) -- (5,0) -- (5,1);

      \draw[fill=cyan!15] (2,2) -- (2,1) -- (4,1) -- (4,4) -- (0,4) -- (0,3) -- (1,3);

      \draw[fill=cyan!15] (6,2) -- (6,1) -- (5,1) -- (5,4) -- (8,4) -- (8,3) -- (7,3);

      \node at (0.5,3.5) {$z_{1,1}$};
      \ThreeDotsAt{(1.5,3.5)}{.21213203}{0}
      \node at (2.5,3.5) {$z_{1,\tilde k}$};
      \node at (3.5,3.5) {$z_{1,\tilde k+1}$};
      \ThreeDotsAt{(2.5,2.5)}{.21213203}{90}
      \node at (2.5,1.5) {$z_{\tilde k,\tilde k}$};
      \node at (3.5,1.5) {$z_{\tilde k,\tilde k+1}$};
      \ThreeDotsAt{(1.5-.15,2.5-.15)}{.21213203}{-45}
      \ThreeDotsAt{(3.5,2.5)}{.21213203}{90}

      \node at (4.5,3.5) {$0$}; %
      \node at (5.5,3.5) {$z_{1,\tilde k+2}$};
      \ThreeDotsAt{(6.5,3.5)}{.21213203}{0}
      \node at (7.5,3.5) {$z_{1,\tilde n}$};
      \node at (4.5,1.5) {$0$};
      \node at (5.5,1.5) {$z_{\tilde k,\tilde k+2}$};
      \ThreeDotsAt{(6.5+.15,2.5-.15)}{.21213203}{45}
      \ThreeDotsAt{(4.5,2.5)}{.21213203}{90}
      \ThreeDotsAt{(5.5,2.5)}{.21213203}{90}

      \node at (3.5,0.5) {$\eta$};
      \node at (4.5,0.5) {$0$};
    \end{tikzpicture}
  \end{center}
  where $\eta = 1-\sum_{i=1}^{\tilde k} z_{i,\tilde k+1}$.

  Let $n$ be arbitrary, and set $\tilde n = n - 2$, $\tilde k = \floor*{\tilde n / 2}$.
  For a core $Z$ of an $\tilde n \times \tilde n$ matrix and an index $r \in [k]$, we set $\tilde r = r-1$ and define
  \begin{center}
    \begin{tikzpicture}[scale=.95, baseline=(current bounding box.center)]
      \node at (-1.6,1.0) {$\ext^{\sqcup}_{r}(Z)\ =$};
      \node[draw=none] at (12.15,1.0) {$.$};

      \draw (3,4) -- (5,4);
      \draw (8,4) -- (9,4);
      \draw (3,1) -- (3,0) -- (9,0) -- (9,1);

      \draw[fill=cyan!15] (2,2) -- (2,1) -- (3,1) -- (3,4) -- (0,4) -- (0,3) -- (1,3);

      \draw[fill=cyan!15] (10,2) -- (10,1) -- (9,1) -- (9,4) -- (12,4) -- (12,3) -- (11,3);

      \draw[fill=cyan!15] (4,1) rectangle (8,4);

      \draw[fill=cyan!15,draw=none] (5,-1) -- (4,-1) -- (4,0) -- (8,0) -- (8,-1) -- (7,-1) -- (6,-2);
      \draw (5,-1) -- (4,-1) -- (4,0);
      \draw (7,-1) -- (8,-1) -- (8,0);

      \ThreeDotsAt{(1.5-.15,2.5-.15)}{.21213203}{-45}
      \ThreeDotsAt{(10.5+.15,2.5-.15)}{.21213203}{45}
      \ThreeDotsAt{(2.5,2.5)}{.21213203}{90}
      \ThreeDotsAt{(9.5,2.5)}{.21213203}{90}

      \node at (0.5,3.5) {$z_{1,1}$};
      \ThreeDotsAt{(1.5,3.5)}{.21213203}{0}
      \node at (2.5,3.5) {$z_{1,\tilde r}$};
      \node at (9.5,3.5) {$z_{1,n-r}$};
      \ThreeDotsAt{(10.5,3.5)}{.21213203}{0}
      \node at (11.5,3.5) {$z_{1,\tilde n}$};

      \node at (3.5,3.5) {$0$};
      \node at (4.5,3.5) {$z_{1,r}$};
      \node at (4.5,1.5) {$z_{\tilde r,r}$};
      \node at (7.5,1.5) {$z_{\tilde r,\tilde n - \tilde r}$};
      \ThreeDotsAt{(6,3.5)}{.21213203}{0}
      \ThreeDotsAt{(6,1.5)}{.21213203}{0}
      \node at (7.5,3.5) {$z_{1,\tilde n - \tilde r}$};
      \node at (8.5,3.5) {$0$};
      \ThreeDotsAt{(3.5,2.5)}{.21213203}{90}
      \ThreeDotsAt{(4.5,2.5)}{.21213203}{90}
      \ThreeDotsAt{(7.5,2.5)}{.21213203}{90}
      \ThreeDotsAt{(8.5,2.5)}{.21213203}{90}
      \node at (3.5,1.5) {$0$};
      \node at (8.5,1.5) {$0$};

      \node at (2.5,1.5) {$z_{\tilde r,\tilde r}$};
      \node at (9.5,1.5) {$z_{\tilde r,n-r}$};

      \node at (3.5,0.5) {$1$};
      \node at (4.5,0.5) {$0$};
      \ThreeDotsAt{(6,0.5)}{.21213203}{0}
      \node at (7.5,0.5) {$0$};
      \node at (8.5,0.5) {$0$};

      \node at (4.5,-.5) {$z_{r,r}$};
      \ThreeDotsAt{(6,-0.5)}{.21213203}{0}
      \node at (7.5,-.5) {$z_{r,\tilde n - \tilde r}$};
      \ThreeDotsAt{(5.5-.15,-1.5-.15)}{.21213203}{-45}
      \ThreeDotsAt{(6.5+.15,-1.5-.15)}{.21213203}{45}
    \end{tikzpicture}
  \end{center}

  More precisely, if $n$ is odd, then $\extI$ takes the core of an $(n-1) \times (n-1)$ matrix, inserts an all-zero column of height $k$ in the unique middle position, thereby creating the core of an $n \times n$ matrix.
  If $n$ is even, then $\extL$ takes the core of an $(n-1) \times (n-1)$ matrix, inserts an all-zero column of height $k$ immediately \emph{before} the middle position, whereas $\extRL$ inserts it immediately \emph{after} the middle position; then both operators append a new last core row of length $2$ with first entry
  $
  \eta = 1-\sum_{i=1}^{\tilde k} z_{i,\tilde k+1}
  $
  and second entry $0$, where $\tilde k=\floor*{(n-1)/2}$ and $z_{i,\tilde k+1}$ denotes the $i^\text{th}$ entry in the middle column of the predecessor core.
  Therefore, we again obtain the core of an $n \times n$ matrix.
  For an index $r\in[k]$, $\ext^{\sqcup}_r$ takes the core of an $(n-2) \times (n-2)$ matrix, inserts two new all-zero columns of height $r-1$ in the symmetric positions $r$ and $n+1-r$, and inserts a new all-zero row in position $r$.
  All newly created core entries are $0$, except the entry at position $(r,r)$, which is set to $1$.
  Clearly, this is the core of an $n \times n$ matrix.
  Note that for $r=1$, $\ext^{\sqcup}_r$ inserts only a new first row.
  The superscript indicates the shape of the newly inserted region.

  \medskip
  Now we are ready to construct $L^n_{i,j}$ and $U^n_{i,j}$ via an inductive approach.
  For $n \in [2,4]$, we set
  \begin{align*}
    L^2_{1,1}&=
                           \vcenter{\hbox{\begin{ytableau}
                                 -1 & 2
                               \end{ytableau}}}\,,
    &
      U^2_{1,1}&=
                             \vcenter{\hbox{\begin{ytableau}
                                   2 & -1
                                 \end{ytableau}}}\,,
    \\[12pt]
    L^3_{1,1}&=
                           \vcenter{\hbox{\begin{ytableau}
                                 -1 & 1 & 1
                               \end{ytableau}}}\,,
    &
      L^3_{2,1}&=
                             \vcenter{\hbox{\begin{ytableau}
                                   1 & -1 & 1
                                 \end{ytableau}}}\,,
    &
      U^3_{1,2}&=
                             \vcenter{\hbox{\begin{ytableau}
                                   1 & 1 & -1
                                 \end{ytableau}}}\,,
    \\[12pt]
    L^4_{1,1}&=
                           \vcenter{\hbox{\begin{ytableau}
                                 -1 & 1 & 0 & 1\\
                                 \none & 0 & 0 & \none
                               \end{ytableau}}}\,,
    &
      L^4_{2,1}&=
                             \vcenter{\hbox{\begin{ytableau}
                                   1 & -1 & 1 & 0\\
                                   \none & 1 & 0 & \none
                                 \end{ytableau}}}\,,
    &
      L^4_{2,2}&=
                             \vcenter{\hbox{\begin{ytableau}
                                   0 & 0 & 1 & 0\\
                                   \none & -1 & 1 & \none
                                 \end{ytableau}}}\,,
    \\[6pt]
    U^4_{1,3}&=
                           \vcenter{\hbox{\begin{ytableau}
                                 1 & 0 & 1 & -1\\
                                 \none & 0 & 0 & \none
                               \end{ytableau}}}\,,
    &
      U^4_{2,2}&=
                             \vcenter{\hbox{\begin{ytableau}
                                   0 & 1 & 0 & 0\\
                                   \none & 1 & -1 & \none
                                 \end{ytableau}}}\,,
    &
      U^4_{2,3}&=
                             \vcenter{\hbox{\begin{ytableau}
                                   1 & 0 & -1 & 1\\
                                   \none & 1 & 1 & \none
                                 \end{ytableau}}}\,.
    \intertext{Furthermore, for $n \in [5,7]$, we include the further explicit certificates}
    L^5_{2,2}&=
                           \vcenter{\hbox{\begin{ytableau}
                                 0 & 0 & 0 & 1 & 0\\
                                 \none & -1 & 1 & 0 & \none
                               \end{ytableau}}}\,,
    &
      L^5_{2,3}&=
                             \vcenter{\hbox{\begin{ytableau}
                                   0 & 0 & 1 & 0 & 0\\
                                   \none & 0 & -1 & 2 & \none
                                 \end{ytableau}}}\,,
    &
      L^5_{3,1}&=
                             \vcenter{\hbox{\begin{ytableau}
                                   1 & 0 & -1 & 1 & 0\\
                                   \none & 0 & 1 & -1 & \none
                                 \end{ytableau}}}\,,
    \displaybreak[0]\\[6pt]
    U^5_{2,2}&=
                           \vcenter{\hbox{\begin{ytableau}
                                 0 & 0 & 1 & 0 & 0\\
                                 \none & 2 & -1 & 0 & \none
                               \end{ytableau}}}\,,
    &
      U^5_{2,3}&=
                             \vcenter{\hbox{\begin{ytableau}
                                   0 & 1 & 0 & 0 & 0\\
                                   \none & 0 & 1 & -1 & \none
                                 \end{ytableau}}}\,,
    &
      U^5_{3,2}&=
                             \vcenter{\hbox{\begin{ytableau}
                                   0 & 0 & 1 & 0 & 0\\
                                   \none & 0 & 1 & 0 & \none
                                 \end{ytableau}}}\,,
    \displaybreak[0]\\[12pt]
    L^6_{2,3}&=
                           \vcenter{\hbox{\begin{ytableau}
                                 0 & 0 & 1 & 0 & 0 & 0\\
                                 \none & 0 & -1 & 1 & 1 & \none\\
                                 \none & \none & 0 & 0 & \none & \none
                               \end{ytableau}}}\,,
    &
      U^6_{2,3}&=
                             \vcenter{\hbox{\begin{ytableau}
                                   0 & 0 & 0 & 1 & 0 & 0\\
                                   \none & 1 & 1 & -1 & 0 & \none\\
                                   \none & \none & 0 & 0 & \none & \none
                                 \end{ytableau}}}\,,
    \displaybreak[0]\\[12pt]
    L^7_{2,3}&=
                           \vcenter{\hbox{\begin{ytableau}
                                 0 & 0 & 1 & 0 & 0 & 0 & 0\\
                                 \none & 0 & -1 & 1 & 1 & 0 & \none\\
                                 \none & \none & 0 & 0 & 0 & \none & \none
                               \end{ytableau}}}\,,
    &
      L^7_{2,4}&=
                             \vcenter{\hbox{\begin{ytableau}
                                   0 & 0 & 0 & 1 & 0 & 0 & 0\\
                                   \none & 0 & 0 & -1 & 1 & 1 & \none\\
                                   \none & \none & 0 & 1 & -1 & \none & \none
                                 \end{ytableau}}}\,,
    \\[6pt]
    U^7_{2,3}&=
                           \vcenter{\hbox{\begin{ytableau}
                                 0 & 0 & 0 & 1 & 0 & 0 & 0\\
                                 \none & 1 & 1 & -1 & 0 & 0 & \none\\
                                 \none & \none & -1 & 1 & 0 & \none & \none
                               \end{ytableau}}}\,,
    &
      U^7_{2,4}&=
                             \vcenter{\hbox{\begin{ytableau}
                                   0 & 0 & 0 & 0 & 1 & 0 & 0\\
                                   \none & 0 & 1 & 1 & -1 & 0 & \none\\
                                   \none & \none & 0 & 0 & 0 & \none & \none
                                 \end{ytableau}}}\,,
    &
      U^7_{4,2}&=
                             \vcenter{\hbox{\begin{ytableau}
                                   0 & 0 & 0 & 1 & 0 & 0 & 0\\
                                   \none & 0 & 0 & 1 & 0 & 0 & \none\\
                                   \none & \none & 1 & -1 & 1 & \none & \none
                                 \end{ytableau}}}\,.
  \end{align*}

  For $n \geq 5$, assume that all certifying cores for sizes at most $n-1$ are already given.
  For $n\in[5,7]$, some certificates are given explicitly above.
  All remaining certificates for $n \geq 5$ are obtained by applying the extension operators as follows.

  Let $n$ be even.
  For
  $
    (i,j)\in \{(1,1)\} \cup \bigl([2,k]\times[n-2]\bigr),
  $
  define
  \begin{align*}
    L^{n}_{i,j}&=
                 \begin{cases}
                   \extL \bigl(L^{n-1}_{i,j}\bigr)               & \text{if } (i,j)=(1,1) \text{ or } (i\in[2,k-1] \text{ and } j\in[k-2]),\\
                   \extL \bigl(L^{n-1}_{i,j-1}\bigr)             & \text{if } i\in[2,k-1], j\in[k+1,n-2],\\
                   \extRL \bigl(L^{n-1}_{i,j}\bigr)              & \text{if } \bigl(i\in[2,k-1] \text{ and } j=k-1\bigr) \text{ or } \bigl(i=k \text{ and } j\in[k-1]\bigr),\\
                   \ext^{\sqcup}_{1} \bigl(L^{n-2}_{i-1,j-1}\bigr) & \text{if } \bigl(i\in[3,k-1] \text{ and } j=k\bigr) \text{ or } \bigl(i=k \text{ and } j\in[k,n-3]\bigr),\\
                   \ext^{\sqcup}_{3} \bigl(L^{n-2}_{2,k-1}\bigr)   & \text{if } i=2, j=k,\\
                   \vrefl \bigl(U^{n}_{k,2}\bigr)                & \text{if } i=k, j=n-2,
                 \end{cases}\\
    \intertext{
    and for
    $
    (i,j)\in \{(1,n-1)\} \cup \bigl([2,k]\times[2,n-1]\bigr),
    $
    define}
    U^n_{i,j}&=
               \begin{cases}
                 \extL \bigl(U^{n-1}_{i,j}\bigr)               & \text{if } i\in[2,k-1], j\in[2,k-2],\\
                 \extL \bigl(U^{n-1}_{i,j-1}\bigr)             & \text{if } (i,j)=(1,n-1) \text{ or } (i\in[2,k-1] \text{ and } j\in[k+1,n-1]),\\
                 \extRL \bigl(U^{n-1}_{i,j}\bigr)              & \text{if } \bigl(i\in[2,k-1] \text{ and } j=k-1\bigr) \text{ or } \bigl(i=k \text{ and } j\in[2,k-1]\bigr),\\
                 \ext^{\sqcup}_{1} \bigl(U^{n-2}_{i-1,j-1}\bigr) & \text{if } \bigl(i\in[3,k-1] \text{ and } j=k\bigr) \text{ or } \bigl(i=k \text{ and } j\in[k,n-2]\bigr),\\
                 \ext^{\sqcup}_{3} \bigl(U^{n-2}_{2,k-1}\bigr)  & \text{if } i=2, j=k,\\
                 \vrefl \bigl(L^{n}_{k,1}\bigr)               & \text{if } i=k, j=n-1,
               \end{cases}
  \end{align*}
  where $\vrefl:\R^{C}\to\R^{C}$ denotes the vertical reflection, i.e., it reverses each core row.

  Let $n$ be odd.
  For
  $
    (i,j)\in \{(1,1)\} \cup \bigl([2,k]\times[n-2]\bigr) \cup \bigl(\{k+1\}\times[k]\bigr),
  $
  define
  \begin{align*}
    L^n_{i,j}&=
               \begin{cases}
                 \extI \bigl(L^{n-1}_{i,j}\bigr)               & \text{if } (i,j)=(1,1) \text{ or } (i\in[2,k] \text{ and } j\in[k-1]),\\
                 \extI \bigl(L^{n-1}_{i,j-1}\bigr)             & \text{if } i\in[2,k], j\in[k+2,n-2],\\
                 \ext^{\sqcup}_{1} \bigl(L^{n-2}_{i-1,j-1}\bigr)& \text{if } \bigl(i\in[3,k] \text{ and } j\in\{k,k+1\}\bigr) \text{ or } \bigl(i=k+1 \text{ and } j\in[2,k]\bigr),\\
                 \ext^{\sqcup}_{3} \bigl(L^{n-2}_{k,1}\bigr)    & \text{if } i=k+1, j=1,\\
                 \ext^{\sqcup}_{3} \bigl(L^{n-2}_{2,j-1}\bigr)  & \text{if } i=2, j\in\{k,k+1\},
               \end{cases}\\
    \intertext{
    and for
    $
    (i,j)\in \{(1,n-1)\} \cup \bigl([2,k]\times[2,n-1]\bigr) \cup \bigl(\{k+1\}\times[2,k]\bigr),
    $
    define
    }
    U^n_{i,j}&=
               \begin{cases}
                 \extI \bigl(U^{n-1}_{i,j}\bigr)                & \text{if } i\in[2,k], j\in[2,k-1],\\
                 \extI \bigl(U^{n-1}_{i,j-1}\bigr)              & \text{if } (i,j)=(1,n-1) \text{ or } (i\in[2,k] \text{ and } j\in[k+2,n-1]),\\
                 \ext^{\sqcup}_{1} \bigl(U^{n-2}_{i-1,j-1}\bigr) & \text{if } \bigl(i\in[3,k] \text{ and } j\in\{k,k+1\}\bigr) \text{ or } \bigl(i=k+1 \text{ and } j\in[3,k]\bigr),\\
                 \ext^{\sqcup}_{4} \bigl(U^{n-2}_{k,2}\bigr)     & \text{if } i=k+1, j=2,\\
                 \ext^{\sqcup}_{3} \bigl(U^{n-2}_{2,j-1}\bigr)   & \text{if } i=2, j\in\{k,k+1\}.
               \end{cases}
  \end{align*}

  \medskip
  We claim that the definitions above provide the desired certificates.
  We prove this by induction on~$n$.
  The explicitly listed cores for $n \in [3, 7]$ satisfy the claim by direct inspection.
  Now let $n \geq 5$, and assume that the claim holds for smaller sizes.
  By construction, every certifying core for size $n$ is obtained from a certifying core for size $n-1$ or size $n-2$ by applying one of the operators $\extI$, $\extL$, $\extRL$, $\ext^{\sqcup}_{r}$, or possibly $\vrefl$.

  Consider a facet inequality with parameter $(i,j)$.
  If it involves entries only inside the embedded smaller core, then the sum coincides with the corresponding sum in the smaller certificate after the evident index shift; hence it is satisfied by the induction hypothesis, except for the unique facet inequality violated by the smaller certificate.
  The recursive case distinctions above were chosen precisely so that this unique violated facet inequality is transported to the intended parameter $(i,j)$.

  If an inequality involves only newly inserted entries, then its left-hand side can be read off directly from the definition of the applied extension:
  for $\extI$ and $\ext^{\sqcup}_{r}$, all newly created core entries are $0$ except for a single new diagonal entry equal to $1$, so every such sum belongs to $\{0,1\}$, and hence all corresponding facet inequalities are satisfied.
  For $\extL$ and $\extRL$, the only newly created entry that may be non-zero is the diagonal entry in the new last core row, whose predecessor-coordinate value $\eta$ is the one displayed above; therefore, any facet inequality involving this entry is again controlled by the induction hypothesis for the predecessor certificate, except for the intended transported violation.
  Finally, $\vrefl$ preserves all defining equations and interchanges lower and upper bounds while sending $(i,j)$ to $(i,n+1-j)$, so the certificate property is preserved under $\vrefl$.
  Moreover, each extension operator preserves all defining equations of $P^\core_\DASASM$ by construction; for $\extL$ and $\extRL$, the displayed choice of $\eta$ is precisely what enforces the new row-sum equation.

  For every facet inequality of $P^\core_\DASASM$, we have constructed a core that violates exactly this facet inequality and satisfies every other facet inequality as well as all defining equations.
  This shows that no facet inequality is redundant.
  Together with the first part of the proof and the fact that no facet inequality is an implicit equation, the facet inequalities, together with the equations in~\cref{eq:dasasm:row-sum}, form a minimal description of $P^\core_\DASASM$.
\end{proof}

\section{Totally symmetric ASMs (TSASMs)}\label{sec:TSASM}
In this final symmetry class, we impose invariance under all symmetries of the square, that is, under all reflections and all rotations.
The corresponding symmetry subgroup is $\mathcal{G} = \mathcal{D}_4$.
Since the reflections $\mathcal{V}$ and $\mathcal{D}$ generate $\mathcal{D}_4$, a real matrix is totally symmetric if and only if it is both vertically and diagonally symmetric.
In particular, the linear subspace of totally symmetric real matrices can be written as
\[
  P_\TS = P_\VS \cap P_\DS,
\]
where $P_\VS$ and $P_\DS$ are as defined in \cref{sec:VSASM,sec:DSASM}, respectively.
Clearly, any TSASM satisfies the ASM constraints in~\crefrange{eq:asm:real}{eq:asm:col-sum} and also the symmetry constraints defining $P_\TS$; thus $P_\TSASM \subseteq P_\ASM \cap P_\TS$.
We note, however, that $P_\ASM \cap P_\TS$ does not equal $P_\TSASM$.
In fact, $P_\TSASM \subsetneq P_\ASM \cap P_\TS$ for every odd $n \geq 5$, as witnessed by the fractional vertices constructed for VSASMs in \cref{sec:VSASM}.
Nevertheless, we will show that $P_\TSASM = P_\VSASM \cap P_\DS$.

Since every TSASM is in particular a VHSASM, the properties established in \cref{sec:VHSASM} for VHSASMs continue to hold for TSASMs.
In particular, we immediately obtain the following parity restriction by \cref{lem:vhsasm:noEvenN}.
\begin{lemma}\label{lem:tsasm:noEvenN}
  There is no $n \times n$ TSASM if $n$ is even.
\end{lemma}

Thus, $P_\TSASM = \emptyset$ for even $n$; hence, we assume that $n$ is odd in the rest of the section.
Let $k = \floor*{n/2}$.
Applying \cref{lem:vhsasm:middleColAndRowAlternate} to a TSASM, we obtain that the middle row and the middle column are fixed to the alternating pattern, which we now state formally.
\begin{lemma}\label{lem:tsasm:middleColAndRow}
  Let $n \geq 1$ be odd and set $k = \floor*{n/2}$.
  For every TSASM $X \in \{0,\pm1\}^{n \times n}$, we have $x_{i,k+1} = (-1)^{i+1}$ for every $i \in [n]$ and $x_{k+1,j} = (-1)^{j+1}$ for every $j \in [n]$.
\end{lemma}

\paragraph{Core and assembly map.}
Assume $n$ is odd and let $k = \floor*{n/2}$.
Let the \emph{core} of a TSASM be the upper-triangular part (including the main diagonal) of its upper-left $k \times k$ block, i.e.,
\[
  C = \left\{(i,j) \in [k] \times [k] : i \leq j\right\}
\]
is the set of \emph{core positions}, and let $\pi_C$ be the coordinate-wise projection onto~$C$.
Define the affine map $\varphi : \R^{C} \to \R^{n \times n}$ by
\[
  \varphi(Y)_{i,j} =
  \begin{cases}
    y_{\min\{i,j\},\max\{i,j\}}               & \text{if } i \in [k], j \in [k],\\
    (-1)^{i+1}                             & \text{if } j = k+1,\\
    (-1)^{j+1}                             & \text{if } i = k+1,\\
    y_{\min\{i,n+1-j\},\max\{i,n+1-j\}}        & \text{if } i \in [k], j \in [k+2, n],\\
    y_{\min\{n+1-i,j\},\max\{n+1-i,j\}}        & \text{if } i \in [k+2, n], j \in [k],\\
    y_{\min\{n+1-i,n+1-j\},\max\{n+1-i,n+1-j\}} & \text{if } i \in [k+2, n], j \in [k+2, n]
  \end{cases}
\]
for $Y \in \R^{C}$ and $i,j \in [n]$.
Note that the second and third cases both apply for $(i,j) = (k+1,k+1)$, but they give the same value $(-1)^{k+2}$, so $\varphi(Y)$ is well defined.
By definition, $\varphi$ uses the entries of $Y$ to fill the upper-left $k \times k$ block symmetrically with respect to the main diagonal, fixes the middle row and the middle column to the alternating pattern from \cref{lem:tsasm:middleColAndRow}, and completes the matrix by reflecting this block across the middle column and across the middle row, thereby filling all four quadrants.
In particular, $\varphi(Y)$ is invariant under vertical reflection and reflection across the main diagonal, and hence is totally symmetric.
Clearly, the map $\varphi$ is an assembly map: it is affine, satisfies $\pi_C(\varphi(Y)) = Y$ for every $Y \in \R^{C}$, and $\varphi(\pi_C(X)) = X$ for every $X \in \TSASM(n)$, because $X$ is determined by its core together with the prescribed middle row and column and the imposed vertical and diagonal symmetries.

\bigskip
We now describe the core polytope of TSASMs.
\begin{theorem}\label{thm:tsasm:corepolytope}
  Let $n \geq 1$ be odd, and set $k = \floor*{n/2}$.
  Then the core polytope $P^\core_\TSASM \subseteq \R^C$ of $n \times n$ TSASMs is described by the following system.
  \begin{align}
                                   y&_{i,j} \in \R                              &\forall (i,j) \in C, \label{eq:tsasm:real}\\
    0 \leq \sum_{i'=1}^{\min\{i,j\}} y&_{i',i} + \sum_{j'=i+1}^{j} y_{i,j'} \leq 1   &\forall i \in [k], j \in [k-1], \label{eq:tsasm:L-prefix}\\
    \sum_{i'=1}^{i}                 y&_{i',i} + \sum_{j'=i+1}^{k} y_{i,j'} = \chi_{2 \mid i}   &\forall i \in [k]. \label{eq:tsasm:L-sum}
  \end{align}
\end{theorem}
\begin{proof}
  We show that the integer solutions to the system~\crefrange{eq:tsasm:real}{eq:tsasm:L-sum} are exactly the cores of TSASMs, and then we argue that the system defines an integral polytope.

  \medskip
  First, let $X$ be an $n\times n$ TSASM, and let $Y=\pi_C(X)$ be its core.
  Since $X$ is totally symmetric, we have $X=\varphi(Y)$ for the assembly map $\varphi$ defined above.
  For $i\in[k]$ and $j\in[k-1]$, all entries of row $i$ in the first $j$ columns lie in the upper-left $k\times k$ block, and this block is symmetric with respect to the main diagonal.
  Hence, by the definition of $\varphi$, we obtain for every $i\in[k], j\in[k-1]$ that
  \begin{equation}\label{eq:tsasm:phi-prefix-used}
    \sum_{j'=1}^{j} x_{i,j'}
    = \sum_{i'=1}^{\min\{i,j\}} y_{i',i}
    + \sum_{j'=i+1}^{j} y_{i,j'}.
  \end{equation}
  Likewise, for each $i \in [k]$, the sum of the first $k$ entries within row $i$ can be written as
  \begin{equation}\label{eq:tsasm:phi-sum-used}
    \sum_{j'=1}^{k} x_{i,j'}
    = \sum_{i'=1}^{i} y_{i',i}
    + \sum_{j'=i+1}^{k} y_{i,j'}.
  \end{equation}

  Since every TSASM is in particular an ASM, the ASM row-prefix bounds in~\cref{eq:asm:row-prefix}, together with the identity~\cref{eq:tsasm:phi-prefix-used}, yield the constraints in~\cref{eq:tsasm:L-prefix}.
  Now we consider the row-sum constraints.
  For $i\in[n]$, we have $x_{i,k+1} = (-1)^{i+1}$ by \cref{lem:tsasm:middleColAndRow}.
  For $i \in [k]$, using the ASM row-sum constraint in~\cref{eq:asm:row-sum}, we get
  \[
    \sum_{j'=1}^{n} x_{i,j'}
    = 2 \sum_{j'=1}^{k} x_{i,j'} + x_{i,k+1}
    = 2 \sum_{j'=1}^{k} x_{i,j'} + (-1)^{i+1}
    = 1.
  \]
  Therefore, $\sum_{j'=1}^{k}x_{i,j'}=(1-(-1)^{i+1})/2=\chi_{2 \mid i}$.
  Substituting $\sum_{j'=1}^{k}x_{i,j'}=\chi_{2 \mid i}$ into the identity~\cref{eq:tsasm:phi-sum-used} yields the equations in~\cref{eq:tsasm:L-sum}.
  Thus the cores of TSASMs satisfy the system~\crefrange{eq:tsasm:real}{eq:tsasm:L-sum}.

  \medskip
  Second, let $Y \in \Z^C$ satisfy the system~\crefrange{eq:tsasm:real}{eq:tsasm:L-sum}, and set $X = \varphi(Y)$.
  By construction, $X$ is an $n\times n$ integer matrix, it is invariant under vertical reflection and also under reflection across the main diagonal, and its core is $Y$; in particular, $X$ is totally symmetric.
  For every $i \in [k]$ and $j \in [k-1]$, the identity~\cref{eq:tsasm:phi-prefix-used} expresses the row-prefix sum $\sum_{j'=1}^{j} x_{i,j'}$ in terms of $Y$, and the constraints in~\cref{eq:tsasm:L-prefix} are exactly the ASM row-prefix constraints in~\cref{eq:asm:row-prefix} for those prefixes.
  The case $j = k$ of the identity~\cref{eq:tsasm:phi-sum-used}, together with the equations in~\cref{eq:tsasm:L-sum}, shows that $\sum_{j'=1}^{k} x_{i,j'} = \chi_{2 \mid i}$ for $i \in [k]$.
  By the definition of $\varphi$, the middle column is fixed to $x_{i,k+1} = (-1)^{i+1}$ for every $i$, and horizontal symmetry yields the same row sums and first $k$ prefix sums for the bottom $k$ rows, while the middle row $k+1$ has the alternating pattern $x_{k+1,j} = (-1)^{j+1}$.
  Thus every row of $X$ is palindromic and has total sum $1$, and in each row, the first $k$ prefix sums lie in $\{0,1\}$.
  Applying \cref{lem:vsasm:symmetricHalfExtends} row-wise, we obtain that all row-prefix sums of $X$ lie in $\{0,1\}$, so the ASM row-prefix bounds in~\cref{eq:asm:row-prefix} and row-sum constraints in~\cref{eq:asm:row-sum} hold for $X$.
  Since $X$ is symmetric with respect to the main diagonal, every column of $X$ is also a row of $X$, and the ASM column-prefix bounds in~\cref{eq:asm:col-prefix} and column-sum constraints in~\cref{eq:asm:col-sum} follow from the corresponding row constraints.
  Hence $X$ satisfies all ASM constraints from \cref{thm:ASMpolytope} and, being totally symmetric, $X$ is a TSASM\@.
  We conclude that the integer solutions to the system~\crefrange{eq:tsasm:real}{eq:tsasm:L-sum} are precisely the cores of TSASMs.

  \medskip
  It remains to prove that the system~\crefrange{eq:tsasm:real}{eq:tsasm:L-sum} defines an integral polytope.
  We proceed in a way similar to the proof of \cref{thm:dsasm:corepolytope}.
  For $i \in [0,k]$ and $j \in [i,k+1]$, let
  \[
    z_{i,j} = \sum_{i'=1}^{i} \sum_{j'=j}^{k} y_{i',j'}.
  \]
  Note that $z_{0,j} = z_{i,k+1} = 0$ by empty sums.
  This defines an injective linear map $Y \mapsto Z$ with inverse
  \[
    y_{i,j} = z_{i,j} - z_{i,j+1} - z_{i-1,j} + z_{i-1,j+1}
  \]
  for $i \in [k]$ and $j \in [i,k]$.
  For $i \in [k]$ and $j \in [i]$, we have the identity
  \begin{equation}\label{eq:tsasm:core:id1}
    \sum_{i'=1}^{j} y_{i',i} = z_{j,i} - z_{j,i+1},
  \end{equation}
  and, for $i \in [k]$ and $j \in [i,k]$,
  \begin{equation}\label{eq:tsasm:core:id2}
    \sum_{j'=j+1}^{k} y_{i,j'} = z_{i,j+1} - z_{i-1,j+1}.
  \end{equation}

  Now we prove that the constraints in~\cref{eq:tsasm:L-prefix,eq:tsasm:L-sum} are equivalent to
  \begin{align}
                                 0 \leq z&_{j,i} - z_{j,i+1} \leq 1             & \forall i \in [k], j \in [i-1], \label{eq:tsasm:L-prefix-Z-1}\\
    \chi_{2 \mid i} - 1 \leq z&_{i,j+1} - z_{i-1,j+1} \leq \chi_{2 \mid i}         & \forall i \in [k], j \in [i,k-1], \label{eq:tsasm:L-prefix-Z-2}\\
                                        z&_{i,i} - z_{i-1,i+1} = \chi_{2 \mid i} & \forall i \in [k]. \label{eq:tsasm:row-sum-Z}
  \end{align}
  Indeed, for $j \in [i-1]$, substituting the identity~\cref{eq:tsasm:core:id1} into the constraints in~\cref{eq:tsasm:L-prefix} yields the bounds in~\cref{eq:tsasm:L-prefix-Z-1}.
  For $j \in [i,k-1]$, we use the equations in~\cref{eq:tsasm:L-sum} to write
  \[
    \sum_{i'=1}^{i} y_{i',i} + \sum_{j'=i+1}^{j} y_{i,j'}
    = \chi_{2 \mid i} - \sum_{j'=j+1}^{k} y_{i,j'},
  \]
  so the inequality in~\cref{eq:tsasm:L-prefix} is equivalent to
  \[
    \chi_{2 \mid i} - 1
    \leq \sum_{j'=j+1}^{k} y_{i,j'}
    \leq \chi_{2 \mid i},
  \]
  and the identity~\cref{eq:tsasm:core:id2} gives the bounds in~\cref{eq:tsasm:L-prefix-Z-2}.
  Finally, setting $j=i$ in~\cref{eq:tsasm:core:id1,eq:tsasm:core:id2} and substituting into the equations in~\cref{eq:tsasm:L-sum} yields the equations in~\cref{eq:tsasm:row-sum-Z} for $i \in [k]$.

  \smallskip
  Observe that each constraint in~\crefrange{eq:tsasm:L-prefix-Z-1}{eq:tsasm:row-sum-Z} involves the difference of two $z$-variables with coefficients $+1$ and $-1$.
  Therefore, the coefficient matrix is the transpose of the node--arc incidence matrix of a digraph, and hence it is totally unimodular.
  The right-hand sides are integers, so the system~\crefrange{eq:tsasm:L-prefix-Z-1}{eq:tsasm:row-sum-Z} together with $z_{0,j} = z_{i,k+1} = 0$ defines an integral polytope in the $z$-variables; see \cref{thm:digraphIncidenceTU}.
  The map $Y \mapsto Z$ is a linear bijection that preserves integrality, and thus every vertex $Y$ of the polytope defined by the system~\crefrange{eq:tsasm:real}{eq:tsasm:L-sum} is the preimage of a vertex $Z$ of the polytope defined by the system~\crefrange{eq:tsasm:L-prefix-Z-1}{eq:tsasm:row-sum-Z}.
  Therefore, the system~\crefrange{eq:tsasm:real}{eq:tsasm:L-sum} defines an integral polytope.
\end{proof}

We obtain the following description of the polytope $P_\TSASM$ of TSASMs.
\begin{theorem}\label{thm:tsasm:coreDescr}
  Let $n \geq 1$ be odd, let $k = \floor*{n/2}$, and $\widehat P^\core_\TSASM = \{X \in \R^{n \times n} : \pi_C(X) \in P^\core_\TSASM\}$.
  Then
  \[
    P_\TSASM
    = \widehat P^\core_\TSASM \cap P_\TS \cap \left\{X \in \R^{n \times n} : x_{i,k+1} = x_{k+1,i} = (-1)^{i+1}\ \forall i \in [n]\right\}.
  \]
\end{theorem}
\begin{proof}
  By \cref{thm:xasm:assembly,thm:tsasm:corepolytope}, it suffices to prove that
  $
    \varphi(\R^C)
    = P_\TS \cap \{X \in \R^{n \times n} : x_{i,k+1} = x_{k+1,i} = (-1)^{i+1}\ \forall i \in [n]\}.
  $
  Let $P$ denote the right-hand side.
  By definition, $\varphi$ inserts its argument $Y$ on the core positions $C$ in the upper-left $k \times k$ block, fixes the $i^\text{th}$ entries of the middle row and the middle column to $(-1)^{i+1}$ for every $i$, and fills the remaining entries of the matrix using total symmetry.
  Thus $\varphi(Y) \in P$ for every~$Y \in \R^C$.

  Conversely, take any $X \in P$.
  The equations defining $P_\TS$ together with the prescribed middle row and column imply that $X$ is completely determined by its core $\pi_C(X)$.
  Therefore, $\varphi(\pi_C(X)) = X$, and hence $X \in \varphi(\R^C)$.
  This shows $\varphi(\R^C) = P$, and the statement follows.
\end{proof}

\begin{theorem}\label{thm:tsasm:descr}
  Let $n \geq 1$ be arbitrary, and set $k = \floor*{n/2}$.
  Then
  \[
    P_\TSASM
    = P_\ASM \cap P_\TS \cap \left\{X \in \R^{n \times n} : x_{i,k+1} = x_{k+1,i} = (-1)^{i+1}\ \forall i \in [n]\right\}
    = P_\VSASM \cap P_\DS.
  \]
\end{theorem}
\begin{proof}
  The second equation follows immediately by \cref{thm:vsasm:descr} and the definition of $P_\TS$, so we focus on the first one.

  For odd $n$, we obtain the statement by straightforward transformations of the system given in \cref{thm:tsasm:coreDescr}, in complete analogy with the passage from \cref{thm:vhsasm:coreDescr} to \cref{thm:vhsasm:descr}.

  For even $n$, the polytope $P_\TSASM$ is empty; thus we need to show that the right-hand side is empty as well.
  Notice that
  $
    \{X \in \R^{n \times n} : x_{i,k+1} = x_{k+1,i} = (-1)^{i+1}\ \forall i \in [n]\}
  $
  forces the entries of column $k+1$ to alternate between $+1$ and $-1$, with the first entry equal to $+1$.
  Hence $\sum_{i=1}^n x_{i,k+1} = 0$ because $n$ is even.
  On the other hand, if $X \in P_\ASM$, then we have $\sum_{i=1}^n x_{i,k+1} = 1$ by~\cref{eq:asm:col-sum}, a contradiction.
  Thus, the right-hand side is empty for even $n$.
\end{proof}

\begin{theorem}\label{thm:tsasm:dim}
  For every odd $n \geq 3$, the dimension of $P_\TSASM$ is $\frac{(n-5)(n-3)}{8}$.
\end{theorem}
\begin{proof}
  It suffices to prove that the dimension of $P^\core_\TSASM$ is $\frac{(n-5)(n-3)}{8}$, because the assembly map $\varphi$ restricts to an affine isomorphism between $P^\core_\TSASM$ and $P_\TSASM$, which preserves dimension.
  First, we give an upper bound.
  Let $n \geq 3$ be odd, set $k=\floor*{n/2}$, and recall that $C=\{(i,j)\in[k]\times[k]: i\leq j\}$.
  By~\cref{thm:tsasm:corepolytope}, the inequalities in~\cref{eq:tsasm:L-prefix} with $j=1$ give $0\leq y_{1,i}\leq1$ for every $i\in[k]$, while the equation in~\cref{eq:tsasm:L-sum} with $i=1$ gives $\sum_{j=1}^k y_{1,j}=0$.
  It follows that
  \begin{align}\label{eq:tsasm:dim:firstrow}
    y&_{1,j}=0 &\forall j\in[k].
  \end{align}

  By \cref{thm:tsasm:corepolytope}, the polytope $P^\core_\TSASM\subseteq\R^{C}$ is described by the system~\crefrange{eq:tsasm:real}{eq:tsasm:L-sum}.
  The system of linear equations in~\cref{eq:tsasm:L-sum} consists of $k$ equations, and these are linearly independent since the variable $y_{i,i}$ appears only in the $i^\text{th}$ equation.
  Moreover, after imposing the equations in~\cref{eq:tsasm:dim:firstrow}, the $i=1$ equation in~\cref{eq:tsasm:L-sum} becomes redundant.
  For each $i\in[2,k]$, the variable $y_{1,i}$ appears in the $i^\text{th}$ equation of~\cref{eq:tsasm:L-sum}.
  Using the equations $y_{1,i}=0$ from~\cref{eq:tsasm:dim:firstrow}, we eliminate these terms as follows: for each $i\in[2,k]$, we replace the $i^\text{th}$ equation of~\cref{eq:tsasm:L-sum} by the difference of that equation and $y_{1,i}=0$.
  This elementary row operation does not change the solution set and thus the rank of the equation system.
  After this replacement, none of the modified equations in~\cref{eq:tsasm:L-sum} for $i\in[2,k]$ involves a variable of the form $y_{1,j}$, and hence their support is disjoint from that of~\cref{eq:tsasm:dim:firstrow}.
  Therefore, the combined system has $k+(k-1)=2k-1$ independent equations and defines an affine subspace of dimension
  $
    |C|-(2k-1)
    =\frac{k(k+1)}{2}-2k+1
    =\frac{(k-1)(k-2)}{2}
    =\frac{(n-5)(n-3)}{8}.
  $
  Since $P^\core_\TSASM$ is contained in this affine subspace, we obtain the upper bound $\dim(P^\core_\TSASM)\leq \frac{(n-5)(n-3)}{8}$.

  \medskip
  Second, we construct $\frac{(n-5)(n-3)}{8}+1$ affinely independent cores in $P^\core_\TSASM$ and hence obtain a matching lower bound.
  Let $\ol Y$ denote the average of the cores of all TSASMs.
  Then $\ol Y \in P^\core_\TSASM$ and it satisfies the equations in~\cref{eq:tsasm:dim:firstrow}.
  We claim that $\ol Y$ reaches neither the lower nor the upper bound in~\cref{eq:tsasm:L-prefix} for any $i\in[2,k-1]$ and $j\in[i,k-1]$.
  It suffices to show that, for each such $i,j$, there exists a TSASM core for which the left-hand side of~\cref{eq:tsasm:L-prefix} equals $1$, and there exists a TSASM core for which it equals $0$.
  To see this, set $y_{i,i}=1$ for every even $i\in[2,k]$ and set all remaining core entries to $0$.
  This yields an integer solution to the system~\crefrange{eq:tsasm:real}{eq:tsasm:L-sum}, and hence the core of a TSASM by \cref{thm:tsasm:corepolytope}.
  For this core, the left-hand side of~\cref{eq:tsasm:L-prefix} equals $1$ for every even $i\in[2,k-1]$ and every $j\in[i,k-1]$, while for every odd $i\in[3,k-1]$ it equals $0$ for every $j\in[i,k-1]$.
  Next, set $y_{i,k}=(-1)^i$ for every $i\in[2,k]$, set $y_{i,i}=1$ for every odd $i\in[3,k-1]$, and fill the rest of the core entries with $0$.
  Again we obtain an integer solution to the system~\crefrange{eq:tsasm:real}{eq:tsasm:L-sum}, and hence the core of a TSASM by \cref{thm:tsasm:corepolytope}.
  For this core, the left-hand side of~\cref{eq:tsasm:L-prefix} equals $1$ for every odd $i\in[3,k-1]$ and every $j\in[i,k-1]$, while for every even $i\in[2,k-1]$ it equals $0$ for every $j\in[i,k-1]$.
  Thus, for every $i\in[2,k-1]$ and $j\in[i,k-1]$, the left-hand side in~\cref{eq:tsasm:L-prefix} attains both values $0$ and $1$ on TSASM cores, and therefore $\ol Y$ does not reach equality in~\cref{eq:tsasm:L-prefix} for any such $i,j$.
  Let
  $
    S=\{(i,j)\in C : 2\leq i < j\leq k\}.
  $
  For each $(i,j)\in S$, define
  $
    \ol Y^{i,j}=\ol Y+\varepsilon\chi_{i,j}-\varepsilon\chi_{i,i}-\varepsilon\chi_{j,j},
  $
  where $\varepsilon$ is a small positive constant.
  By definition, $\ol Y^{i,j}$ satisfies the equations in~\cref{eq:tsasm:L-sum,eq:tsasm:dim:firstrow}.
  The left-hand side of an inequality in~\cref{eq:tsasm:L-prefix} decreases by $\varepsilon$ when its indices are $i,s$ with $s\in[i,j-1]$, increases by $\varepsilon$ when its indices are $j,s$ with $s\in[i,j-1]$, and is unchanged otherwise.
  The inequalities whose left-hand sides decrease are strict at $\ol Y$ by the argument above.
  For those whose left-hand sides increase, the first TSASM core constructed above has left-hand side $0$, so the corresponding value at $\ol Y$ is strictly smaller than $1$.
  Therefore, choosing $\varepsilon>0$ small enough ensures that $\ol Y^{i,j}\in P^\core_\TSASM$ for all $(i,j)\in S$.
  The cores $\ol Y$ and $\ol Y^{i,j}$ for $(i,j)\in S$ are affinely independent: only the difference $\ol Y^{i,j}-\ol Y$ has a non-zero entry at $(i,j)$, so the cores $\{\ol Y^{i,j}-\ol Y : (i,j)\in S\}$ are linearly independent.
  Therefore, the dimension of $P^\core_\TSASM$ is at least $|S| = \frac{(n-5)(n-3)}{8}$.

  Combining the lower and upper bounds yields $\dim(P^\core_\TSASM)=\dim(P_\TSASM)=\frac{(n-5)(n-3)}{8}$.
\end{proof}

\begin{theorem}\label{thm:tsasm:facets}
  Let $n \geq 9$ be odd, and set $k=\floor*{n/2}$.
  The facets of $P^\core_\TSASM$ are given by tightening the lower bound in~\cref{eq:tsasm:L-prefix} to equality for
  $
    (i,j)\in
    \{(2,2)\}
    \cup
    \{(i,j) : i \in [3,k-1], j \in [2,k-1-\chi_{2 \mid i}]\}
    \cup
    \bigl(\{k\} \times \{j\in[3,k-1-\chi_{2 \nmid k}] : 2 \nmid j\}\bigr)
  $,
  and the upper bound for
  $
    (i,j)\in
    \{(i,j) : i \in [4,k-1], j \in [4,k-1-\chi_{2 \nmid i}]\}
    \cup
    \bigl(\{k\} \times \{j\in[4,k-1-\chi_{2 \mid k}] : 2 \mid j\}\bigr)
  $.
  In particular, the number of facets of $P^\core_\TSASM$ is $\frac{n^2 -15n + 62}{2}$.
\end{theorem}
\begin{proof}
  The facets are obtained by tightening a single inequality in~\cref{eq:tsasm:L-prefix} to equality for the index pairs listed in the statement of the theorem.
  We call the instances of the lower bounds that are tightened to equality the \emph{facet lower bounds}, and we define the \emph{facet upper bounds} analogously.
  We refer to the union of these two families as the \emph{facet inequalities}.

  We proceed in two steps.
  First, we show that the facet inequalities together with the equations in~\cref{eq:tsasm:L-sum,eq:tsasm:dim:firstrow} imply every inequality in~\cref{eq:tsasm:L-prefix}.
  Then, for every facet inequality, we construct a core of an $n \times n$ matrix violating that facet inequality and no other, thereby proving that no facet inequality is redundant.
  The core $\ol Y$ constructed in the second step of the proof of \cref{thm:tsasm:dim} shows that no facet inequality is an implicit equation; thus the two steps together imply that the facet inequalities form a minimal system that, extended with the equations in~\cref{eq:tsasm:L-sum,eq:tsasm:dim:firstrow}, describes the convex hull of the cores of TSASMs, which proves the theorem.

  \medskip
  Now we prove that the facet inequalities together with the equations in~\cref{eq:tsasm:L-sum,eq:tsasm:dim:firstrow} imply every inequality in~\cref{eq:tsasm:L-prefix}.
  Clearly, we need to treat only those inequalities that are non-facet inequalities, namely, the lower bounds in~\cref{eq:tsasm:L-prefix} for
  \begin{align*}
    (i,j)\in\;&(\{1\}\times[k-1])
                \ \cup\ \bigl(\{2\}\times(\{1\}\cup[3,k-1])\bigr)
                \ \cup\ \bigl([3,k-1]\times\{1\}\bigr)\\
              &\cup\ \bigl(\{i\in[4,k-1]:2 \mid i\}\times\{k-1\}\bigr)
                \ \cup\ \Bigl(\{k\}\times\bigl(\{1,2\}\ \cup\ \{j\in[4,k-1-\chi_{2 \mid k}]:2 \mid j\}\bigr)\Bigr),
  \end{align*}
  and the upper bounds for
  \begin{align*}
    (i,j)\in\;&([3]\times[k-1])
              \ \cup\ \bigl([4,k-1]\times[3]\bigr)
                \ \cup\ \bigl(\{i\in[4,k-1]:2 \nmid i\}\times\{k-1\}\bigr)\\
              &\cup\ \Bigl(\{k\}\times\bigl([3]\ \cup\ \{j\in[4,k-1-\chi_{2 \nmid k}]:2 \nmid j\}\bigr)\Bigr).
  \end{align*}

  For $i\in[k]$ and $j\in[k-1]$, let $z_{i,j}$ denote the left-hand side of~\cref{eq:tsasm:L-prefix}, and extend the notation in accordance with the equations in~\cref{eq:tsasm:L-sum} by setting $z_{i,k} = \chi_{2 \mid i}$ for $i\in[k]$.
  For $j \in [2,k]$, we have
  \[
    z_{i,j}-z_{i,j-1}=
    \begin{cases}
      y_{j,i} & \text{if } 2\leq j\leq i,\\
      y_{i,j} & \text{if } i+1\leq j\leq k,
    \end{cases}
  \]
  by the definition of $z_{i,j}$ and, for $j=k$, by the extension through the equations in~\cref{eq:tsasm:L-sum}.
  We also have $z_{1,j}=\sum_{t=1}^{j} y_{1,t}$ for $j\in[k-1]$, and $z_{i,1}=y_{1,i}$ for $i\in[k]$, again by the definition of $z_{i,j}$.
  The equations in~\cref{eq:tsasm:dim:firstrow} give $y_{1,t}=0$ for all $t\in[k]$, hence $z_{1,j}=0$ for every $j\in[k-1]$ and also $z_{i,1}=0$ for every $i\in[k]$.
  Therefore,~\cref{eq:tsasm:L-prefix} holds whenever $i=1$ or $j=1$.

  We continue with $i=2$.
  For each $t\in[2,k-1]$, the facet lower bound at $(t,2)$ yields $0\leq z_{t,2}=y_{2,t}$, hence $y_{2,t}\geq 0$.
  Moreover,
  $
    y_{2,k}=z_{k,3}-y_{3,k}=z_{k,3}-(z_{3,k}-z_{3,k-1})=z_{k,3}+z_{3,k-1},
  $
  where $z_{3,k}=0$ by~\cref{eq:tsasm:L-sum}.
  Since $z_{k,3}\geq0$ and $z_{3,k-1}\geq0$ are facet lower bounds, this gives $y_{2,k}\geq0$ and thus $y_{2,t}\geq 0$ for all $t\in[2,k]$.
  Consequently, for every $j\in[3,k-1]$,
  $
    0\leq z_{2,j}=\sum_{t=2}^j y_{2,t}\leq \sum_{t=2}^k y_{2,t}=z_{2,k}=1,
  $
  which proves the non-facet bounds in~\cref{eq:tsasm:L-prefix} for $i=2$ and $j\in[3,k-1]$.
  Since $z_{i,2}=y_{2,i}$ for every $i\in[2,k]$, we also obtain the non-facet upper bounds for $j=2$.

  Next we derive the non-facet upper bounds for $i\in[4,k]$, $j=3$.
  For every $t\in[3,k]$, the facet lower bound at $(t,3)$ gives $0\leq z_{t,3}=y_{2,t}+y_{3,t}$, where we use $y_{1,t}=0$.
  In particular, $y_{2,3}+y_{3,3}\geq 0$, and $y_{3,t}\geq -y_{2,t}$ for every $t\in[4,k]$.
  The case $i=3$ of the equations in~\cref{eq:tsasm:L-sum} gives
  $
    0=z_{3,k}=(y_{2,3}+y_{3,3})+\sum_{t=4}^k y_{3,t}.
  $
  Hence, for every $i\in[4,k]$,
  $
    y_{3,i}=-(y_{2,3}+y_{3,3})-\sum_{t\in[4,k]\setminus\{i\}} y_{3,t} \leq \sum_{t\in[4,k]\setminus\{i\}} y_{2,t}
  $,
  and therefore
  $
    z_{i,3}=y_{2,i}+y_{3,i}\leq \sum_{t=4}^k y_{2,t}\leq \sum_{t=2}^k y_{2,t}=1
  $.
  This proves the non-facet upper bounds in~\cref{eq:tsasm:L-prefix} for $j=3$ and $i\in[4,k]$.

  We now derive the remaining non-facet upper bounds in row $i=3$, namely those with $j\in[3,k-1]$.
  For such $j$, the facet lower bounds at $(t,3)$ for $t\in[j+1,k]$ give $0\leq y_{2,t}+y_{3,t}$, hence $-y_{3,t}\leq y_{2,t}$.
  Using $z_{3,k}=0$ from~\cref{eq:tsasm:L-sum} (and $y_{1,3}=0$), we can write
  $
    z_{3,j}
    = y_{2,3} + y_{3,3} + \sum_{t=4}^j y_{3,t}
    =-\sum_{t=j+1}^k y_{3,t}
    \leq \sum_{t=j+1}^k y_{2,t}
    \leq \sum_{t=2}^k y_{2,t}
    = z_{2,k}
    = 1
  $,
  where the second inequality uses $y_{2,t}\geq 0$ for $t\in[2,k]$.
  This proves the remaining non-facet upper bounds in row $i=3$.

  We continue with the case $i\in[4,k-1]$, $j=k-1$.
  By the difference relations and the definition of $z_{i,k}$, we have
  $
    z_{i,k-1}=z_{i,k}-y_{i,k}=\chi_{2 \mid i}-y_{i,k}
  $.
  Moreover, the same difference relations give
  $
    y_{i,k}=z_{k,i}-z_{k,i-1}
  $.
  If $i$ is even, then the facet upper bound at $(k,i)$ gives $z_{k,i}\leq 1$ and the facet lower bound at $(k,i-1)$ gives $z_{k,i-1}\geq 0$, hence $y_{i,k}\leq 1$, and therefore $z_{i,k-1}=1-y_{i,k}\geq 0$, proving the non-facet lower bound.
  If $i$ is odd, then the facet lower bound at $(k,i)$ gives $z_{k,i}\geq 0$ and the facet upper bound at $(k,i-1)$ gives $z_{k,i-1}\leq 1$, hence $y_{i,k}\ge-1$ and thus $z_{i,k-1}=-y_{i,k}\leq 1$, proving the non-facet upper bound.

  Finally, we treat the remaining non-facet bounds in row $i=k$.
  We first record the sign of $y_{t,k}$ for $t\in[2,k-1]$.
  For odd $t\in[3,k-1]$, we have $z_{t,k}=0$ and the facet lower bound at $(t,k-1)$ yields $z_{t,k-1}\geq 0$, hence $y_{t,k}=z_{t,k}-z_{t,k-1}\leq 0$.
  For even $t\in[4,k-1]$, we have $z_{t,k}=1$ and the facet upper bound at $(t,k-1)$ yields $z_{t,k-1}\leq 1$, hence $y_{t,k}=z_{t,k}-z_{t,k-1}\geq 0$.
  Moreover, we already showed $y_{2,k}\geq 0$ above.
  Let $j\in[2,k-1]$ be even.
  Then $y_{j,k}\geq 0$, and we also have $z_{k,j-1}\geq 0$ because either $j-1=1$ (so $z_{k,1}=0$) or $j-1\geq 3$ is odd, in which case the lower bound at $(k,j-1)$ is a facet inequality.
  Using $z_{k,j}=z_{k,j-1}+y_{j,k}$, we obtain $z_{k,j}\geq 0$, proving the non-facet lower bounds in row $k$.

  For the remaining non-facet upper bounds in row $k$, it remains only to consider odd endpoints at least $5$.
  First, let $j\in[5,k-2]$ be odd.
  Then $j+1$ is even and $y_{j+1,k}\geq 0$, and the facet upper bound at $(k,j+1)$ gives $z_{k,j+1}\leq 1$.
  Using $z_{k,j}=z_{k,j+1}-y_{j+1,k}$, we obtain $z_{k,j}\leq 1$.
  If $k\geq 6$ is even, then the last remaining endpoint is $j=k-1$.
  In this case, $k-1$ is odd and we use $z_{k,k-1}=z_{k,k-2}+y_{k-1,k}\leq z_{k,k-2}\leq 1$, because $z_{k,k-2}\leq 1$ is a facet inequality and $y_{k-1,k}\leq 0$ as shown above.

  This completes the derivation of all inequalities in~\cref{eq:tsasm:L-prefix} from the facet inequalities and the equation system.

  \medskip
  It remains to show that none of the facet inequalities in~\cref{eq:tsasm:L-prefix} is redundant.
  For every horizontal facet lower bound indexed by $(i,j)$, we construct a core $L^n_{i,j}\in\R^C$ that violates this lower bound while satisfying all other facet inequalities.
  Analogously, for every horizontal facet upper bound indexed by $(i,j)$, we construct a core $U^n_{i,j}\in\R^C$ that violates this upper bound while satisfying all other facet inequalities.

  In order to build the certifying cores of $n\times n$ TSASMs from the certifying cores for smaller sizes, we define two \emph{extension operators}.
  For $\tilde n = n - 2$, $\tilde k = \floor*{\tilde n / 2}$, and a core $Z$ of an $\tilde n \times \tilde n$ matrix, we define
  \begin{center}
    \begin{tikzpicture}[scale=.9, baseline=(current bounding box.center)]
      \node at (-1.6,1) {$\extI(Z)\ =$};
      \node[draw=none] at (4.15,1) {$;$};

      \draw (3,3) -- (4,3) -- (4,-1) -- (3,-1) -- (3,0);

      \draw[fill=cyan!15] (2,1) -- (2,0) -- (3,0) -- (3,3) -- (0,3) -- (0,2) -- (1,2);

      \node at (0.5,2.5) {$z_{1,1}$};
      \ThreeDotsAt{(1.5,2.5)}{.21213203}{0}
      \node at (2.5,2.5) {$z_{1,\tilde k}$};
      \ThreeDotsAt{(2.5,1.5)}{.21213203}{90}
      \node at (2.5,0.5) {$z_{\tilde k,\tilde k}$};
      \ThreeDotsAt{(1.5-.15,1.5-.15)}{.21213203}{-45}

      \node at (3.5,2.5) {$0$};
      \ThreeDotsAt{(3.5,1.5)}{.21213203}{90}
      \node at (3.5,0.5) {$0$};
      \node at (3.5,-0.5) {$\chi_{2 \mid k}$};
    \end{tikzpicture}
  \end{center}
  furthermore, for $\tilde n = n - 4$, $\tilde k = \floor*{\tilde n / 2}$, a core $Z$ of an $\tilde n \times \tilde n$ matrix, and an index $r \in [\tilde k + 1]$, we define
  \begin{center}
    \begin{tikzpicture}[scale=.9]
      \draw[fill=cyan!15] (1,7) -| (0,7) -| (0,8) |- (4,8) |- (4,4) |- (3,4) -| (3,5);
      \ThreeDotsAt{(2-.15,6-.15)}{.21213203}{-45}
      \draw[fill=cyan!15] (7,1) -| (6,1) -| (6,2) |- (9,2) |- (9,-1) |- (8,-1) -| (8,0);
      \ThreeDotsAt{(7.5-.15,0.5-.15)}{.21213203}{-45}
      \draw[fill=cyan!15] (6,4) -| (6,8) -| (9,8) |- (9,4) -| cycle;
      \draw (4,4) -- (4,3) -- (5,3) -- (5,2) -- (6,2);
      \draw (4,8) -- (6,8);
      \draw (9,4) -- (9,2);
      \node[cell, draw=none] at (8,3) {$0$};
      \node[cell, draw=none] at (8,2) {$0$};
      \node[cell, draw=none] at (6,3) {$0$};
      \node[cell, draw=none] at (6,2) {$0$};
      \node[cell, draw=none] at (5,2) {$\chi_{2 \nmid r}$};
      \node[cell, draw=none] at (5,7) {$0$};
      \node[cell, draw=none] at (5,3) {$0$};
      \node[cell, draw=none] at (5,4) {$0$};
      \node[cell, draw=none] at (4,7) {$0$};
      \node[cell, draw=none] at (4,4) {$0$};
      \node[cell, draw=none] at (4,3) {$\chi_{2 \mid r}$};
      \ThreeDotsAt{(4.5,6)}{.21213203}{90}
      \ThreeDotsAt{(5.5,6)}{.21213203}{90}
      \ThreeDotsAt{(7.5,3.5)}{.21213203}{0}
      \ThreeDotsAt{(7.5,2.5)}{.21213203}{0}

      \node[cell, draw=none] at (0,7) {$z_{1,1}$};
      \node[cell, draw=none] at (3,7) {$z_{1,\tilde r}$};
      \node[cell, draw=none] at (3,4) {$z_{\tilde r,\tilde r}$};
      \ThreeDotsAt{(2,7.5)}{.21213203}{0}
      \ThreeDotsAt{(3.5,6)}{.21213203}{90}

      \node[cell, draw=none] at (6,7) {$z_{1,r}$};
      \node[cell, draw=none] at (8,7) {$z_{1,\tilde k}$};
      \node[cell, draw=none] at (6,4) {$z_{\tilde r,r}$};
      \node[cell, draw=none] at (8,4) {$z_{\tilde r,\tilde k}$};
      \ThreeDotsAt{(7.5,7.5)}{.21213203}{0}
      \ThreeDotsAt{(7.5,4.5)}{.21213203}{0}
      \ThreeDotsAt{(6.5,6)}{.21213203}{90}
      \ThreeDotsAt{(8.5,6)}{.21213203}{90}

      \node[cell, draw=none] at (6,1) {$z_{r,r}$};
      \node[cell, draw=none] at (8,1) {$z_{r,\tilde k}$};
      \node[cell, draw=none] at (8,-1) {$z_{\tilde k,\tilde k}$};
      \ThreeDotsAt{(8.5,0.5)}{.21213203}{90}
      \ThreeDotsAt{(7.5,1.5)}{.21213203}{0}

      \node[draw=none] at (-1.75,3.5) {$\ext^{\llcorner}_{r}(Z)\ =$};
      \node[draw=none] at (9.15,3.5) {$.$};
    \end{tikzpicture}
  \end{center}
  More precisely, $\extI$ takes the core of an $(n-2)\times(n-2)$ matrix and adjoins a new last column of core entries that is identically $0$ except the new diagonal entry, which is set to $y_{k,k}=\chi_{2 \mid k}$ as prescribed by~\cref{eq:tsasm:L-sum}.
  Hence $\extI$ yields the core of an $n\times n$ matrix.
  The operator $\ext^{\llcorner}_r$ takes the core of an $(n-4)\times(n-4)$ matrix and produces a core of an $n\times n$ matrix by inserting two new rows at positions $r$ and $r+1$, and two new columns at positions $r$ and $r+1$ so that the triangular shape of cores is retained.
  All newly created core entries are set to $0$, except for the two new diagonal entries: the one with even index is set to $1$ and the other is set to $0$ in accordance with the equations in~\cref{eq:tsasm:L-sum}, i.e.,
  $
  y_{r,r}=\chi_{2 \mid r}
  $
  and
  $
    y_{r+1,r+1}=\chi_{2 \nmid r}
  $.

  Now we are ready to construct $L^n_{i,j}$ and $U^n_{i,j}$ via an inductive approach.
  For $n \in \{9,11\}$, we set
  \begin{align*}
    L^9_{2,2}&=\vcenter{\hbox{$
               \begin{ytableau}
                 0 & 0 & 0 & 0\\
                 \none & -1 & 1 & 1\\
                 \none & \none & -1 & 0\\
                 \none & \none & \none & 0\\
               \end{ytableau}
               $}}\,,
    &
      L^9_{3,2}&=\vcenter{\hbox{$
                 \begin{ytableau}
                   0 & 0 & 0 & 0\\
                   \none & 1 & -1 & 1\\
                   \none & \none & 1 & 0\\
                   \none & \none & \none & 0\\
                 \end{ytableau}
                 $}}\,,
    &
    L^9_{3,3}&=\vcenter{\hbox{$
               \begin{ytableau}
                 0 & 0 & 0 & 0\\
                 \none & 1 & 0 & 0\\
                 \none & \none & -1 & 1\\
                 \none & \none & \none & 0\\
               \end{ytableau}
               $}}\,,
    &
      L^9_{4,3}&=\vcenter{\hbox{$
                 \begin{ytableau}
                   0 & 0 & 0 & 0\\
                   \none & 0 & 1 & 0\\
                   \none & \none & 0 & -1\\
                   \none & \none & \none & 2\\
                 \end{ytableau}
                 $}}\,,
  \end{align*}
  \begin{align*}
    L^{11}_{2,2} &= \vcenter{\hbox{$
                   \begin{ytableau}
                     0 & 0 & 0 & 0 & 0\\
                     \none & -1 & 1 & 1 & 0\\
                     \none & \none & -1 & 0 & 0\\
                     \none & \none & \none & 0 & 0\\
                     \none & \none & \none & \none & 0\\
                   \end{ytableau}
                   $}}\,,
    &
      L^{11}_{3,2} &= \vcenter{\hbox{$
                     \begin{ytableau}
                       0 & 0 & 0 & 0 & 0\\
                       \none & 1 & -1 & 1 & 0\\
                       \none & \none & 1 & 0 & 0\\
                       \none & \none & \none & 0 & 0\\
                       \none & \none & \none & \none & 0\\
                     \end{ytableau}
                     $}}\,,
    &
      L^{11}_{3,3} &= \vcenter{\hbox{$
                     \begin{ytableau}
                       0 & 0 & 0 & 0 & 0\\
                       \none & 1 & 0 & 0 & 0\\
                       \none & \none & -1 & 1 & 0\\
                       \none & \none & \none & 0 & 0\\
                       \none & \none & \none & \none & 0\\
                     \end{ytableau}
                     $}}\,,
    \\[2mm]
    L^{11}_{3,4} &= \vcenter{\hbox{$
                   \begin{ytableau}
                     0 & 0 & 0 & 0 & 0\\
                     \none & 0 & 0 & 1 & 0\\
                     \none & \none & 0 & -1 & 1\\
                     \none & \none & \none & 1 & 0\\
                     \none & \none & \none & \none & -1\\
                   \end{ytableau}
                   $}}\,,
    &
      L^{11}_{4,2} &= \vcenter{\hbox{$
                     \begin{ytableau}
                       0 & 0 & 0 & 0 & 0\\
                       \none & 1 & 0 & -1 & 1\\
                       \none & \none & 0 & 1 & -1\\
                       \none & \none & \none & 1 & 0\\
                       \none & \none & \none & \none & 0\\
                     \end{ytableau}
                     $}}\,,
    &
      L^{11}_{4,3} &= \vcenter{\hbox{$
                     \begin{ytableau}
                       0 & 0 & 0 & 0 & 0\\
                       \none & 0 & 1 & 0 & 0\\
                       \none & \none & 0 & -1 & 0\\
                       \none & \none & \none & 1 & 1\\
                       \none & \none & \none & \none & -1\\
                     \end{ytableau}
                     $}}\,,
    \\[2mm]
    L^{11}_{5,3} &= \vcenter{\hbox{$
                   \begin{ytableau}
                     0 & 0 & 0 & 0 & 0\\
                     \none & 0 & 1 & 0 & 0\\
                     \none & \none & 0 & 0 & -1\\
                     \none & \none & \none & 0 & 1\\
                     \none & \none & \none & \none & 0\\
                   \end{ytableau}
                   $}}\,,
    &
      U^{11}_{4,4} &= \vcenter{\hbox{$
                     \begin{ytableau}
                       0 & 0 & 0 & 0 & 0\\
                       \none & 0 & 0 & 0 & 1\\
                       \none & \none & 0 & 0 & 0\\
                       \none & \none & \none & 2 & -1\\
                       \none & \none & \none & \none & 0\\
                     \end{ytableau}
                     $}}\,,
    &
      U^{11}_{5,4} &= \vcenter{\hbox{$
                     \begin{ytableau}
                       0 & 0 & 0 & 0 & 0\\
                       \none & 0 & 0 & 0 & 1\\
                       \none & \none & 0 & 0 & 0\\
                       \none & \none & \none & 0 & 1\\
                       \none & \none & \none & \none & -2\\
                     \end{ytableau}
                     $}}\,.
  \end{align*}

  \newpage
  For $n\geq 13$, define\vspace{-6mm}
  \begin{center}
    \begin{minipage}[t]{0.49\linewidth}
      \begin{align*}
        L^n_{3,k-1} \;&=\;
                        \begin{cases}
                          \vspace{1em}
                          \vcenter{\vspace{3pt}\hbox{$
                          \begin{ytableau}
                            0&0&0&0&0&0\\
                            \none&0&0&0&1&0\\
                            \none&\none&0&0&-1&1\\
                            \none&\none&\none&1&0&0\\
                            \none&\none&\none&\none&0&0\\
                            \none&\none&\none&\none&\none&0
                          \end{ytableau}$}}                    & \text{if } n=13,\\
                          \vspace{1em}
                          \vcenter{\hbox{$
                          \begin{ytableau}
                            0&0&0&0&0&0&0\\
                            \none&0&0&0&0&1&0\\
                            \none&\none&0&0&0&-1&1\\
                            \none&\none&\none&0&0&1&0\\
                            \none&\none&\none&\none&0&0&0\\
                            \none&\none&\none&\none&\none&0&0\\
                            \none&\none&\none&\none&\none&\none&-1
                          \end{ytableau}$}}                    & \text{if } n=15,\\
                          \ext^{\llcorner}_{4}\bigl(L^{n-4}_{3,k-3}\bigr) & \text{if } n\geq 17,
                        \end{cases}
        \\[2.5ex]
        L^n_{k-1,2} \;&=\;
                        \begin{cases}
                          \vspace{1em}
                          \vcenter{\vspace{3pt}\hbox{$
                          \begin{ytableau}
                            0&0&0&0&0&0\\
                            \none&1&0&0&-1&1\\
                            \none&\none&0&0&1&-1\\
                            \none&\none&\none&0&0&1\\
                            \none&\none&\none&\none&0&0\\
                            \none&\none&\none&\none&\none&0&\none
                          \end{ytableau}$}}                     & \text{if } n=13,\\
                          \ext^{\llcorner}_{4}\bigl(L^{n-4}_{k-3,2}\bigr) & \text{if } n\geq 15,
                        \end{cases}
        \\[2.5ex]
        L^n_{k-1,3} \;&=\;
                        \begin{cases}
                          \vspace{1em}
                          \vcenter{\vspace{3pt}\hbox{$
                          \begin{ytableau}
                            0&0&0&0&0&0\\
                            \none&0&1&0&0&0\\
                            \none&\none&0&0&-1&0\\
                            \none&\none&\none&0&1&0\\
                            \none&\none&\none&\none&0&0\\
                            \none&\none&\none&\none&\none&1
                          \end{ytableau}$}}                    & \text{if } n=13,\\
                          \vspace{1em}
                          \vcenter{\hbox{$
                          \begin{ytableau}
                            0&0&0&0&0&0&0\\
                            \none&0&1&0&0&0&0\\
                            \none&\none&0&0&0&-1&0\\
                            \none&\none&\none&0&0&1&0\\
                            \none&\none&\none&\none&0&0&0\\
                            \none&\none&\none&\none&\none&1&0\\
                            \none&\none&\none&\none&\none&\none&0
                          \end{ytableau}$}}                    & \text{if } n=15,\\
                          \ext^{\llcorner}_{5}\bigl(L^{n-4}_{k-3,3}\bigr) & \text{if } n\geq 17,
                        \end{cases}
      \end{align*}
    \end{minipage}\hspace{-3mm}
    \begin{minipage}[t]{0.49\linewidth}
      \begin{align*}
        L^n_{4,k-2} \;&=\;
                        \begin{cases}
                          \vspace{1em}
                          \vcenter{\vspace{3pt}\hbox{$
                          \begin{ytableau}
                            0&0&0&0&0&0\\
                            \none&1&0&0&0&0\\
                            \none&\none&0&0&0&0\\
                            \none&\none&\none&-1&1&1\\
                            \none&\none&\none&\none&-1&0\\
                            \none&\none&\none&\none&\none&0
                          \end{ytableau}$}}                    & \text{if } n=13,\\
                          \vspace{1em}
                          \vcenter{\hbox{$
                          \begin{ytableau}
                            0&0&0&0&0&0&0\\
                            \none&0&0&0&1&0&0\\
                            \none&\none&0&0&0&0&0\\
                            \none&\none&\none&0&-1&1&1\\
                            \none&\none&\none&\none&0&0&0\\
                            \none&\none&\none&\none&\none&0&0\\
                            \none&\none&\none&\none&\none&\none&-1
                          \end{ytableau}$}}                    & \text{if } n=15,\\
                          \vspace{1em}
                          \vcenter{\hbox{$
                          \begin{ytableau}
                            0&0&0&0&0&0&0&0\\
                            \none&0&0&0&0&1&0&0\\
                            \none&\none&0&0&0&0&0&0\\
                            \none&\none&\none&0&0&-1&1&1\\
                            \none&\none&\none&\none&0&0&0&0\\
                            \none&\none&\none&\none&\none&1&0&0\\
                            \none&\none&\none&\none&\none&\none&-1&0\\
                            \none&\none&\none&\none&\none&\none&\none&0
                          \end{ytableau}$}}                    & \text{if } n=17,\\
                          \ext^{\llcorner}_{5}\bigl(L^{n-4}_{4,k-4}\bigr) & \text{if } n\geq 19,
                        \end{cases}
        \\[2.5ex]
        L^n_{k,3} \;&=\;
                      \begin{cases}
                        \vspace{1em}
                        \vcenter{\vspace{3pt}\hbox{$
                        \begin{ytableau}
                          0&0&0&0&0&0\\
                          \none&0&1&0&0&0\\
                          \none&\none&0&0&0&-1\\
                          \none&\none&\none&0&0&1\\
                          \none&\none&\none&\none&0&0\\
                          \none&\none&\none&\none&\none&1&\none&\none
                        \end{ytableau}$}}                      & \text{if } n=13,\\
                        \ext^{\llcorner}_{5}\bigl(L^{n-4}_{k-2,3}\bigr)   & \text{if } n\geq 15.
                      \end{cases}
      \end{align*}
    \end{minipage}
  \end{center}
  Furthermore, for $n \geq 13$, assume that the certificates for the facet lower bounds are already given for sizes at most $n-2$.
  For the rest of the indices, we define
  \begin{equation}
    L^n_{i,j} \;=\;
    \begin{cases}
      \extI\bigl(L^{n-2}_{i,j}\bigr)        & \text{if } (i,j)=(2,2) \text{ or } (i \in [3, k-2] \text{ and } j \in [2, k-2-\chi_{2 \mid i}]),\\[1mm]
      \ext^{\llcorner}_{1}\bigl(L^{n-4}_{i-2,j-2}\bigr) & \text{if }
        \begin{aligned}[t]
          &i \in [5, k-1] \text{ and } j=k-1-\chi_{2 \mid i}, \text{ or}\\
          &(i,j)=(k,k-1) \text{ and } 2 \mid k, \text{ or}\\
          &i=k, j \in [5, k-2-\chi_{2 \mid k}] \text{ and } 2 \nmid j, \text{ or}\\
          &i=k-1, j \in [4, k-2+\chi_{2 \nmid k}].
        \end{aligned}
    \end{cases}
  \end{equation}

  Now we turn to the upper bounds.
  For $n\geq 13$, define\vspace{-6mm}
  \begin{center}
    \begin{minipage}[t]{0.49\linewidth}
      \begin{align*}
        U^n_{4,k-1} \;&=\;
                        \begin{cases}
                          \vspace{1em}
                          \vcenter{\vspace{3pt}\hbox{$
                          \begin{ytableau}
                            0&0&0&0&0&0\\
                            \none&0&0&0&0&1\\
                            \none&\none&0&0&0&0\\
                            \none&\none&\none&1&1&-1\\
                            \none&\none&\none&\none&-1&0\\
                            \none&\none&\none&\none&\none&1
                          \end{ytableau}$}}                    & \text{if } n=13,\\
                          \vspace{1em}
                          \vcenter{\vspace{3pt}\hbox{$
                          \begin{ytableau}
                            0&0&0&0&0&0&0\\
                            \none&0&0&0&0&0&1\\
                            \none&\none&0&0&0&0&0\\
                            \none&\none&\none&1&0&1&-1\\
                            \none&\none&\none&\none&0&0&0\\
                            \none&\none&\none&\none&\none&0&0\\
                            \none&\none&\none&\none&\none&\none&0&\none
                          \end{ytableau}$}}                    & \text{if } n=15,\\
                          \ext^{\llcorner}_{5}\bigl(U^{n-4}_{4,k-3}\bigr) & \text{if } n\geq 17,
                        \end{cases}
        \\[2.5ex]
        U^n_{k-1,4} \;&=\;
                        \begin{cases}
                          \vspace{1em}
                          \vcenter{\vspace{3pt}\hbox{$
                          \begin{ytableau}
                            0&0&0&0&0&0&0\\
                            \none&0&0&0&0&1&0\\
                            \none&\none&0&0&0&0&0\\
                            \none&\none&\none&0&0&1&0\\
                            \none&\none&\none&\none&1&-1&0\\
                            \none&\none&\none&\none&\none&0&0\\
                            \none&\none&\none&\none&\none&\none&0
                          \end{ytableau}$}}                    & \text{if } n=15,\\
                          \vspace{1em}
                          \vcenter{\vspace{3pt}\hbox{$
                          \begin{ytableau}
                            0&0&0&0&0&0&0&0\\
                            \none&0&0&0&0&0&1&0\\
                            \none&\none&0&0&0&0&0&0\\
                            \none&\none&\none&0&0&0&1&0\\
                            \none&\none&\none&\none&1&0&-1&0\\
                            \none&\none&\none&\none&\none&1&0&0\\
                            \none&\none&\none&\none&\none&\none&-1&0\\
                            \none&\none&\none&\none&\none&\none&\none&1
                          \end{ytableau}$}}                    & \text{if } n=17,\\
                          \ext^{\llcorner}_{6}\bigl(U^{n-4}_{k-3,4}\bigr) & \text{if } n\geq 19,
                        \end{cases}
      \end{align*}
    \end{minipage}\hspace{-3mm}
    \begin{minipage}[t]{0.49\linewidth}
      \begin{align*}
        U^n_{5,k-2} \;&=\;
                        \begin{cases}
                          \vspace{1em}
                          \vcenter{\vspace{3pt}\hbox{$
                          \begin{ytableau}
                            0&0&0&0&0&0\\
                            \none&0&0&0&0&1\\
                            \none&\none&0&0&1&-1\\
                            \none&\none&\none&0&1&0\\
                            \none&\none&\none&\none&-2&0\\
                            \none&\none&\none&\none&\none&1
                          \end{ytableau}$}}                    & \text{if } n=13,\\
                          \vspace{1em}
                          \vcenter{\vspace{3pt}\hbox{$
                          \begin{ytableau}
                            0&0&0&0&0&0&0\\
                            \none&0&0&0&0&1&0\\
                            \none&\none&0&0&0&0&0\\
                            \none&\none&\none&0&0&0&1\\
                            \none&\none&\none&\none&2&-1&-1\\
                            \none&\none&\none&\none&\none&1&0\\
                            \none&\none&\none&\none&\none&\none&0
                          \end{ytableau}$}}                    & \text{if } n=15,\\
                          \vspace{1em}
                          \vcenter{\vspace{3pt}\hbox{$
                          \begin{ytableau}
                            0&0&0&0&0&0&0&0\\
                            \none&0&0&0&0&0&1&0\\
                            \none&\none&0&0&0&0&0&0\\
                            \none&\none&\none&0&0&0&0&1\\
                            \none&\none&\none&\none&1&1&-1&-1\\
                            \none&\none&\none&\none&\none&0&0&0\\
                            \none&\none&\none&\none&\none&\none&0&0\\
                            \none&\none&\none&\none&\none&\none&\none&1
                          \end{ytableau}$}}                    & \text{if } n=17,\\
                          \vspace{1em}
                          \vcenter{\vspace{3pt}\hbox{$
                          \begin{ytableau}
                            0&0&0&0&0&0&0&0&0\\
                            \none&0&0&0&0&0&0&1&0\\
                            \none&\none&0&0&0&0&0&0&0\\
                            \none&\none&\none&0&0&0&0&0&1\\
                            \none&\none&\none&\none&1&0&1&-1&-1\\
                            \none&\none&\none&\none&\none&1&0&0&0\\
                            \none&\none&\none&\none&\none&\none&-1&0&0\\
                            \none&\none&\none&\none&\none&\none&\none&1&0\\
                            \none&\none&\none&\none&\none&\none&\none&\none&0
                          \end{ytableau}$}}                    & \text{if } n=19,\\
                          \ext^{\llcorner}_{6}\bigl(U^{n-4}_{5,k-4}\bigr) & \text{if } n\geq 21.
                        \end{cases}
      \end{align*}
    \end{minipage}
  \end{center}

  \noindent
  \begin{center}
    \begin{minipage}[t]{0.49\linewidth}
      \vspace{0pt}
      \begin{align*}
        U^n_{k-1,5} \;&=\;
                        \begin{cases}
                          \vspace{1em}
                          \vcenter{\vspace{3pt}\hbox{$
                          \begin{ytableau}
                            0&0&0&0&0&0&0\\
                            \none&0&0&0&0&0&1\\
                            \none&\none&0&0&0&0&0\\
                            \none&\none&\none&0&0&1&0\\
                            \none&\none&\none&\none&0&1&-1\\
                            \none&\none&\none&\none&\none&-1&0\\
                            \none&\none&\none&\none&\none&\none&0
                          \end{ytableau}$}}                    & \text{if } n=15,\\
                          \vspace{1em}
                          \vcenter{\vspace{3pt}\hbox{$
                          \begin{ytableau}
                            0&0&0&0&0&0&0&0\\
                            \none&0&0&0&0&0&0&1\\
                            \none&\none&0&0&0&0&0&0\\
                            \none&\none&\none&0&0&0&1&0\\
                            \none&\none&\none&\none&0&0&1&-1\\
                            \none&\none&\none&\none&\none&1&-1&1\\
                            \none&\none&\none&\none&\none&\none&-1&0\\
                            \none&\none&\none&\none&\none&\none&\none&0
                          \end{ytableau}$}}                    & \text{if } n=17,\\
                          \vspace{1em}
                          \vcenter{\vspace{3pt}\hbox{$
                          \begin{ytableau}
                            0&0&0&0&0&0&0&0&0\\
                            \none&0&0&0&0&0&0&0&1\\
                            \none&\none&0&0&0&0&0&0&0\\
                            \none&\none&\none&0&0&0&0&1&0\\
                            \none&\none&\none&\none&0&0&0&1&-1\\
                            \none&\none&\none&\none&\none&1&0&-1&1\\
                            \none&\none&\none&\none&\none&\none&0&0&0\\
                            \none&\none&\none&\none&\none&\none&\none&0&0\\
                            \none&\none&\none&\none&\none&\none&\none&\none&-1
                          \end{ytableau}$}}                    & \text{if } n=19,\\
                          \ext^{\llcorner}_{7}\bigl(U^{n-4}_{k-3,5}\bigr) & \text{if } n\geq 21,
                        \end{cases}
      \end{align*}
    \end{minipage}\hspace{-3mm}
    \begin{minipage}[t]{0.49\linewidth}
      \vspace{0pt}
      \begin{align*}
        U^n_{k,4} \;&=\;
                      \begin{cases}
                        \vspace{1em}
                        \vcenter{\vspace{3pt}\hbox{$
                        \begin{ytableau}
                          0&0&0&0&0&0\\
                          \none&0&0&0&0&1\\
                          \none&\none&0&0&0&0\\
                          \none&\none&\none&0&0&1\\
                          \none&\none&\none&\none&1&-1\\
                          \none&\none&\none&\none&\none&0
                        \end{ytableau}$}}                      & \text{if } n=13,\\
                        \vspace{1em}
                        \vcenter{\vspace{3pt}\hbox{$
                        \begin{ytableau}
                          0&0&0&0&0&0&0\\
                          \none&0&0&0&0&0&1\\
                          \none&\none&0&0&0&0&0\\
                          \none&\none&\none&0&0&0&1\\
                          \none&\none&\none&\none&0&1&-1\\
                          \none&\none&\none&\none&\none&0&0\\
                          \none&\none&\none&\none&\none&\none&-1
                        \end{ytableau}$}}                      & \text{if } n=15,\\
                        \vspace{1em}
                        \vcenter{\vspace{3pt}\hbox{$
                        \begin{ytableau}
                          0&0&0&0&0&0&0&0\\
                          \none&0&0&0&0&0&0&1\\
                          \none&\none&0&0&0&0&0&0\\
                          \none&\none&\none&0&0&0&0&1\\
                          \none&\none&\none&\none&0&1&0&-1\\
                          \none&\none&\none&\none&\none&0&0&0\\
                          \none&\none&\none&\none&\none&\none&0&0\\
                          \none&\none&\none&\none&\none&\none&\none&0
                        \end{ytableau}$}}                      & \text{if } n=17,\\
                        \ext^{\llcorner}_{6}\bigl(U^{n-4}_{k-2,4}\bigr) & \text{if } n\geq 19.
                      \end{cases}
      \end{align*}
    \end{minipage}
  \end{center}
  Furthermore, for $n \geq 13$, assume that the certificates for the facet upper bounds are already given for sizes at most $n-2$.
  For the rest of the indices, we define
  \[
    U^n_{i,j} \;=\;
    \begin{cases}
      \extI\bigl(U^{n-2}_{i,j}\bigr)        & \text{if } i \in [4, k-2] \text{ and } j \in [4, k-2-\chi_{2 \nmid i}],\\
      \ext^{\llcorner}_{1}\bigl(U^{n-4}_{i-2,j-2}\bigr) & \text{if }
        \begin{aligned}[t]
          &i \in [6, k-1] \text{ and } j=k-1-\chi_{2 \nmid i}, \text{ or}\\
          &i=k-1 \text{ and } j \in [6, k-2-\chi_{2 \mid k}], \text{ or}\\
          &i=k, j \in [6, k-2-\chi_{2 \nmid k}] \text{ and } 2 \mid j, \text{ or}\\
          &(i,j) = (k, k-1).
        \end{aligned}
    \end{cases}
  \]

  \medskip
  We claim that the cores $L^n_{i,j}$ and $U^n_{i,j}$ defined above are certificates for the corresponding facet lower and facet upper bounds, respectively, which we verify by induction on~$n$.
  The claim holds by direct inspection for all cores that are listed explicitly in the construction: each such core satisfies the defining equations in~\cref{eq:tsasm:L-sum,eq:tsasm:dim:firstrow} and all facet inequalities in~\cref{eq:tsasm:L-prefix} other than the indicated inequality, which it violates.

  Now let $n\geq 13$ be odd and assume that the claim holds for all smaller odd sizes.
  Let $Y$ be one of the cores $L^n_{i,j}$ or $U^n_{i,j}$ defined recursively, and let $Z$ be the predecessor core appearing on the right-hand side of its definition.
  Then we have either $Y=\extI(Z)$ for a core $Z$ of an $(n-2) \times (n-2)$ matrix, or $Y=\ext^{\llcorner}_r(Z)$ for a core $Z$ of an $(n-4) \times (n-4)$ matrix.
  By the definition of $\extI$ and $\ext^{\llcorner}_r$, the core $Y$ contains an embedded copy of the core~$Z$, and every newly created core entry of~$Y$ is~$0$ except possibly for the prescribed newly created diagonal entries.
  In particular, the equations in~\cref{eq:tsasm:L-sum} corresponding to indices in the embedded copy are inherited from~$Z$, and the equations corresponding to the newly created indices are satisfied by the setting of the new entries.
  Hence $Y$ satisfies all equations in~\cref{eq:tsasm:L-sum,eq:tsasm:dim:firstrow}.

  \smallskip
  We now check the facet inequalities for size~$n$.
  Every facet inequality is one of the lower or upper bounds in~\cref{eq:tsasm:L-prefix}.
  Fix such a facet inequality and consider the defining prefix sum on its left-hand side.

  If the endpoint of the defining prefix lies in the embedded copy of~$Z$ inside~$Y$, then any newly created entries occurring in that defining sum are equal to~$0$ by construction: both extension operators create only diagonal non-zero entries, which never contribute to such prefixes.
  Hence the defining prefix sum on~$Y$ coincides with the corresponding defining prefix sum on~$Z$ after the evident index shift induced by $\extI$ or $\ext^{\llcorner}_r$.
  Therefore, the facet inequality holds for~$Y$ by the induction hypothesis, except for the unique facet inequality violated by~$Z$.
  Moreover, the case distinctions in the recursive definitions were chosen precisely so that this uniquely violated facet inequality of~$Z$ is transported to the intended facet inequality for~$Y$.

  If the endpoint of the defining prefix lies in a newly created row or column, then either the defining prefix involves only newly created entries, or otherwise newly created entries occurring in that defining sum are equal to~$0$.
  In the former case, all such entries are~$0$ except possibly for a single prescribed new diagonal entry, and the construction ensures that all other entries in the same defining prefix are~$0$; hence the defining prefix sum belongs to~$\{0,1\}$ and in particular satisfies both the facet lower and the facet upper bound.
  In the latter case, the defining prefix sum on~$Y$ coincides with the corresponding defining prefix sum on~$Z$ after the evident index shift, and the induction hypothesis ensures that the latter is non-violating.

  For every facet inequality in~\cref{eq:tsasm:L-prefix}, we have constructed a core that violates exactly this inequality and satisfies all other facet inequalities in~\cref{eq:tsasm:L-prefix} as well as the equations in~\cref{eq:tsasm:L-sum,eq:tsasm:dim:firstrow}.
  Therefore, no inequality in~\cref{eq:tsasm:L-prefix} is redundant.
  Together with the first part of the proof and the fact that no facet inequality in~\cref{eq:tsasm:L-prefix} is an implicit equation, this implies that the facet inequalities in~\cref{eq:tsasm:L-prefix}, together with the equations in~\cref{eq:tsasm:L-sum,eq:tsasm:dim:firstrow}, form a minimal description of~$P^\core_{\TSASM}$.
\end{proof}

\section{Linear optimization and the integer Carath\'eodory property}\label{sec:linear-optimization}

We conclude the main results with a linear\nobreakdash-optimization result for every non-empty symmetry class and an integer Carath\'eodory result for $P_\ASM$ and, at every admissible size, for the full-matrix and core polytopes of every non-trivial class other than the quarter-turn class.

\subsection{Linear optimization}\label{subsec:linear-optimization}
For each non-empty symmetry class, linear optimization reduces to the standard combinatorial optimization problems underlying the integrality proofs.
We state the result only for the non-empty cases; as discussed above, these are easy to characterize: VSASMs, VHSASMs, and TSASMs require odd $n$, QTSASMs require $n\not\equiv2\pmod4$, and the remaining classes are non-empty for every~$n$.

\begin{theorem}\label{thm:linear-optimization}
  For every dihedral symmetry class $\XASM$ and every size $n$ with $\XASM(n)\neq\emptyset$, a minimum\nobreakdash-cost $n\times n$ $\XASM$ for any rational cost matrix can be found in strongly polynomial time.
\end{theorem}
\begin{proof}
  We first identify the standard optimization problems to which the unrestricted ASM description and the descriptions of the core polytopes reduce.
  The unrestricted ASM description and the VSASM, VHSASM, and HTSASM core descriptions have constraint matrices coming from two laminar families of index sets; the standard construction from inclusion forests gives minimum\nobreakdash-cost flow formulations for their linear\nobreakdash-optimization problems.
  For~QTSASMs, linear optimization is carried out on the polynomial-size auxiliary bidirected system used in the proof of \cref{thm:qtsasm:corepolytope}; the parity inequalities in~\cref{eq:qtsasm:cut} need not be enumerated, and the resulting problem falls under the Edmonds--Johnson weighted non-bipartite $b$-matching framework.
  For DSASMs and TSASMs, the substitutions used in the proofs of \cref{thm:dsasm:corepolytope,thm:tsasm:corepolytope} turn the core descriptions into bounded\nobreakdash-tension problems, equivalently the duals of minimum\nobreakdash-cost circulation problems.
  For DASASMs, the $z$-substitution used in the proof of \cref{thm:dasasm:corepolytope} yields a system with a node--arc incidence matrix and lower and upper bounds, hence a minimum\nobreakdash-cost circulation formulation with bounds.
  These formulations admit strongly polynomial combinatorial optimization algorithms; see, for example, Schrijver~\cite[Chapters~12, 19, 36, and~41]{schrijver2003combinatorial}.

  It remains to transfer a full-matrix objective to core coordinates.
  For unrestricted ASMs, we optimize directly over $P_\ASM$.
  For every other symmetry class, write the assembly map as $\varphi(Y)=L(Y)+B$, where $L$ is linear and $B$ is fixed.
  If $M\in\Q^{n\times n}$ is the given cost matrix and $X=\varphi(Y)$, then $\langle M,X\rangle = \langle M,L(Y)\rangle+\langle M,B\rangle$.
  Hence minimizing $\langle M,X\rangle$ over $P_\XASM=\varphi(P^\core_\XASM)$ is equivalent, up to the additive constant $\langle M,B\rangle$, to minimizing the induced rational linear objective $Y\mapsto \langle M,L(Y)\rangle$ over $P^\core_\XASM$.
  This induced objective is obtained by substituting the affine assembly formula, and is therefore computable in polynomial time from $M$ and $\varphi$.

  In the symmetry classes whose integrality proofs use auxiliary variables or substitutions, the induced core objective is transferred through the same construction.
  For~QTSASMs, if $c$ denotes the induced objective on the $y$-variables, we optimize $(c,0)$ over the integer hull of the auxiliary $(y,z)$-system from the proof of \cref{thm:qtsasm:corepolytope}.
  For DSASMs and TSASMs, the objective is written in the tension variables used in the proofs of \cref{thm:dsasm:corepolytope,thm:tsasm:corepolytope}.
  For DASASMs, the objective is written in the $z$-variables from the proof of \cref{thm:dasasm:corepolytope}.
  In each case, the transfer has a polynomial-size description and changes the objective only by an explicitly computable linear transformation and, where boundary values are fixed, by an additive constant.

  The resulting objective can therefore be optimized in strongly polynomial time over $P_\ASM$, over the appropriate core polytope, or, in the QTSASM case, over the integer hull of the auxiliary $(y,z)$-system.
  Since the relevant feasible region is non-empty, bounded, and integral, an optimal integral point may be chosen.
  In the unrestricted ASM case and in the cases optimized directly over a core polytope, such points are exactly the corresponding ASMs or ASM cores.
  In the QTSASM case, the $y$-projections of integral points of the auxiliary $(y,z)$-system are exactly the QTSASM cores.
  Applying the assembly map, with the identity map in the unrestricted case, yields a minimum\nobreakdash-cost $n\times n$ $\XASM$.
\end{proof}

\subsection{The integer Carath\'eodory property}\label{subsec:integer-caratheodory}

Following Gijswijt and Regts~\cite{gijswijt2012polyhedra}, we say that a polyhedron $Q\subseteq\R^d$ has the \emph{integer Carath\'eodory property} if the following holds.
For every positive integer $r$ and every $x\in rQ\cap\Z^d$, there are affinely independent points $v_1,\dots,v_m\in Q\cap\Z^d$ and positive integers $\lambda_1,\dots,\lambda_m$ such that $x=\sum_{i=1}^m\lambda_i v_i$ and $\sum_{i=1}^m\lambda_i=r$.
The polyhedron $Q$ has the \emph{integer decomposition property} if, for every positive integer $r$, every $x\in rQ\cap\Z^d$ is a sum of $r$ points of $Q\cap\Z^d$.
Thus the integer Carath\'eodory property implies the integer decomposition property by taking each $v_i$ with multiplicity~$\lambda_i$.
For a full-matrix polytope, integer points are taken in $\Z^{n\times n}$; for a core polytope in $\R^C$, they are taken in~$\Z^C$.
We first record that the integer Carath\'eodory property is preserved by the affine maps between core and full-matrix coordinates.
\begin{lemma}\label{lem:integer-caratheodory:affine-isomorphism}
  Let $Q\subseteq\R^d$ and $Q'\subseteq\R^{d'}$ be polytopes, and let $f(x)=Lx+b$ be an affine isomorphism from $Q$ onto~$Q'$.
  Assume that $L$ and $b$ have integer entries and that the inverse of $f$ on $Q'$ is the restriction of an affine map $g(y)=My+c$ with $M$ and $c$ having integer entries.
  Then $Q$ has the integer Carath\'eodory property if and only if $Q'$ does.
  The same equivalence holds for the integer decomposition property.
\end{lemma}
\begin{proof}
  For every positive integer $r$, the map $x\mapsto Lx+rb$ restricts to a bijection from $rQ\cap\Z^d$ onto $rQ'\cap\Z^{d'}$, with inverse $y\mapsto My+rc$.
  If $x=\sum_{i=1}^m\lambda_i v_i$ and $\sum_{i=1}^m\lambda_i=r$, then $Lx+rb=\sum_{i=1}^m\lambda_i f(v_i)$.
  The points $f(v_i)$ are integral, and an affine isomorphism preserves affine independence, so a decomposition for $x$ gives one for $Lx+rb$.
  The converse follows in the same way from~$g$.
  The same calculation, without the affine-independence condition, proves the assertion for the integer decomposition property.
\end{proof}

Gijswijt and Regts proved that a polyhedron given by a totally unimodular constraint matrix and an integral right-hand side has the integer Carath\'eodory property~\cite{gijswijt2012polyhedra}.
The descriptions obtained earlier allow us to apply this result to $P_\ASM$ and, at every admissible size, to the core polytopes of the six non-trivial classes other than the quarter-turn class.
The assembly maps then transfer the property to the corresponding full-matrix polytopes.
\begin{corollary}\label{cor:xasm:integer-caratheodory}
  Let $\XASM$ be a dihedral symmetry class other than the quarter-turn class, and let $n$ be an admissible size.
  Then $P_\XASM$ has the integer Carath\'eodory property.
  If $\XASM$ is non-trivial and $C$ is its set of core positions, then $P^\core_\XASM\subseteq\R^C$ also has the integer Carath\'eodory property.
  Consequently, all these polytopes have the integer decomposition property.
\end{corollary}
\begin{proof}
  In each system from~\cref{thm:ASMpolytope,thm:vsasm:corepolytope,thm:vhsasm:corepolytope,thm:htsasm:corepolytope}, the sets of variables in the row constraints form one laminar family and those in the column constraints form another.
  Thus~\cref{thm:edmonds-lemma} shows that the coefficient matrix is totally unimodular after each equation is written as two inequalities, and all right-hand sides are integral.
  The theorem of Gijswijt and Regts therefore gives the integer Carath\'eodory property for $P_\ASM$ and for the VSASM, VHSASM, and HTSASM core polytopes.
  The proofs of~\cref{thm:dsasm:corepolytope,thm:dasasm:corepolytope,thm:tsasm:corepolytope} give linear bijections from the DSASM, DASASM, and TSASM core polytopes to polytopes defined by node--arc incidence systems with integral right-hand sides and bounds, and the maps and their inverses have integer coefficients.
  Adding the bound rows and writing each equation as two inequalities preserves total unimodularity, so the theorem of Gijswijt and Regts gives the integer Carath\'eodory property for these transformed polytopes.
  By~\cref{lem:integer-caratheodory:affine-isomorphism}, the corresponding core polytopes have the integer Carath\'eodory property as well.
  By~\cref{thm:xasm:assembly}, each assembly map restricts to an affine isomorphism from the corresponding core polytope onto the full-matrix polytope, with inverse the core projection.
  These maps have integer coefficients and integer constant terms, so~\cref{lem:integer-caratheodory:affine-isomorphism} transfers the integer Carath\'eodory property to the full-matrix polytopes.
  Repeating each point according to its multiplicity in an integer Carath\'eodory decomposition gives the integer decomposition property.
\end{proof}

The quarter-turn class is not covered by~\cref{cor:xasm:integer-caratheodory}.
The proof of~\cref{thm:qtsasm:corepolytope} uses an auxiliary bidirected system rather than a system with a totally unimodular constraint matrix.
Integer hulls of bidirected systems need not have the integer decomposition property and therefore need not have the integer Carath\'eodory property.
For example, the perfect matching problem in a non-bipartite graph has such a formulation.
The node--edge incidence matrix is bidirected, and the integer solutions of the equations $\sum_{e\ni v}x_e=1$ for all vertices $v$, together with the bounds $0\leq x_e\leq1$ for all edges $e$, are exactly the incidence vectors of perfect matchings.
Thus the integer hull of this system is the perfect matching polytope.
The Petersen graph has six perfect matchings, and every edge belongs to two of them, so the vector with value $1/3$ on every edge belongs to its perfect matching polytope.
Hence the all-ones edge vector belongs to three times this polytope.
It cannot be written as a sum of three perfect matchings, because such a sum would partition the edges into three matchings and hence give a $3$-edge-coloring of the Petersen graph; see Schrijver~\cite[Chapters~28 and~38]{schrijver2003combinatorial}.
The perfect matching polytope of the Petersen graph therefore fails the integer decomposition property and hence also the integer Carath\'eodory property.
Consequently, the bidirected description does not settle either property for the quarter-turn polytopes; both questions remain open, as recorded in~\cref{prob:qtsasm:integer-caratheodory}.

\section{Open problems}\label{sec:open-problems}

We conclude by recording several open problems suggested by the preceding results.

\begin{problem}[Quarter-turn facets]\label{prob:qtsasm-facets}
Determine the full facet structure of the quarter-turn symmetric ASM polytope.
\Cref{thm:qtsasm:corepolytope} gives a complete inequality description of $P^\core_\QTSASM$, but not a minimal facet description.
It remains to decide exactly which prefix bounds in~\cref{eq:qtsasm:prefix} and which parity inequalities in~\cref{eq:qtsasm:cut} are facet-defining.
Since the assembly map is an affine isomorphism between $P^\core_\QTSASM$ and $P_\QTSASM$, this is equivalent to the corresponding full-matrix question.
The exact facet counts reported in \cref{rem:qtsasm:facet-data} for small values of $n$, together with the Catalan lower bound in \cref{prop:qtsasm:ferrers-facet-lower-bound}, leave the precise asymptotic growth unresolved.
The strongest form of a solution would be an explicit formula for the number of facets, but even determining the asymptotic growth would be valuable.
\end{problem}

\begin{problem}[Quarter-turn decompositions]\label{prob:qtsasm:integer-caratheodory}
Determine whether $P_\QTSASM$ has the integer Carath\'eodory property for every admissible size~$n$.
More weakly, determine whether it has the integer decomposition property for every admissible size~$n$.
By~\cref{thm:xasm:assembly}, the QTSASM assembly map restricts to an affine isomorphism from $P^\core_\QTSASM$ onto $P_\QTSASM$, with inverse the core projection.
Both maps have integer coefficients and integer constant terms, so~\cref{lem:integer-caratheodory:affine-isomorphism} shows that the full-matrix and core versions of both questions are equivalent.
The bidirected formulation used in the proof of~\cref{thm:qtsasm:corepolytope} does not resolve either question, as explained in~\cref{subsec:integer-caratheodory}.
\end{problem}

\begin{problem}[Ehrhart theory]\label{prob:ehrhart-theory}
For the unrestricted ASM polytope, Behrend and Knight identified the integer points in $rP_\ASM$ with higher-spin ASMs of line sum~$r$~\cite{behrend2007higher}.
Knight's 2009 thesis established Ehrhart polynomial or quasi-polynomial results for the fixed-point polytopes $P_\ASM\cap P_{\mathcal{G}}$ and computed examples for small sizes~\cite[Chapter~4]{knightPhD2009alternating}.
Izanloo's 2019 thesis later derived a formula for the Ehrhart polynomial of $P_\ASM$ and obtained its relative volume from the leading coefficient~\cite{izanloo2019volume}.
Knight's fixed-point polytopes agree with $P_\XASM$ in some classes~\cite[Section~6.2.3]{knightPhD2009alternating}, but are strict relaxations in others, so their Ehrhart functions do not in general determine the Ehrhart polynomial of $P_\XASM$.
For each non-trivial dihedral symmetry class, determine the Ehrhart polynomial of $P_\XASM$ for all admissible sizes, or equivalently that of its core polytope $P^\core_\XASM$, where integer points are taken in $\Z^C$.
It would also be interesting to compute the normalized volumes and $h^*$-polynomials of these hulls and to understand how these invariants depend on the symmetry class and on the parity restrictions on~$n$.
\end{problem}

\begin{problem}[Triangulations and toric ideals]\label{prob:idp-triangulations}
By~\cref{cor:xasm:integer-caratheodory}, at every admissible size, the full-matrix polytope of each symmetry class other than the quarter-turn class, and the core polytope of each non-trivial such class, have the integer Carath\'eodory property and hence the integer decomposition property.
Determine whether these polytopes admit unimodular triangulations, and describe their toric ideals and affine semigroups explicitly.
Whether the quarter-turn full-matrix and core polytopes have even the integer decomposition property for every admissible size remains open by~\cref{prob:qtsasm:integer-caratheodory}.
\end{problem}

\begin{problem}[Adjacency and graph structure]\label{prob:adjacency}
The face structure beyond facets is also largely open for the symmetric classes.
For each dihedral symmetry class $\XASM$, characterize when two $\XASM$s are adjacent vertices of $P_\XASM$.
More generally, determine the graph of $P_\XASM$, its diameter, and perhaps the full face lattice or $f$-vector, and decide whether these structures admit descriptions in terms of symmetry-preserving local moves on ASMs or on their height matrices.
\end{problem}

\begin{problem}[Prescribed and restricted entries]\label{prob:restricted-entries}
Brualdi and Kim studied special completion problems for ordinary and symmetric ASMs~\cite{brualdi2014symmetric,brualdi2015completions}.
The following position-wise allowed-set problem is substantially more general.
Given subsets $A_{i,j}\subseteq\{0,\pm1\}$ of allowed values, characterize when there exists an $\XASM$ such that $x_{i,j}\in A_{i,j}$ for every $i,j\in[n]$.
\Cref{thm:linear-optimization} decides the special case in which each $A_{i,j}$ is one of $\{-1\}$, $\{1\}$, or $\{0,\pm1\}$, but it does not provide a structural characterization.
Even for ordinary ASMs, the case in which $A_{i,j}=\{\pm1\}$ is imposed at some positions does not seem to follow from the linear\nobreakdash-optimization result.
\end{problem}

\section*{Acknowledgments}

The author thanks Nóra A.\ Borsik and Tamás Takács for helpful comments on an earlier version of the manuscript, and is especially grateful to Nóra A.\ Borsik for discussions on an earlier approach to the proof of \cref{thm:htsasm:corepolytope} and for suggestions regarding the facet characterizations.

This work was supported by the Ministry of Innovation and Technology NRDI Office within the framework of the Artificial Intelligence National Laboratory Program; by the Ministry of Innovation and Technology of Hungary from the National Research, Development and Innovation Fund, financed under the ELTE TKP 2021-NKTA-62 funding scheme; and by the grant NKFIH-154121.

OpenAI ChatGPT was used to assist with literature search and language editing.

\bibliographystyle{plain}
\bibliography{bibliography}

\end{document}